\documentclass[a4paper, reqno]{article}

\usepackage{a4wide}

\normalfont

\usepackage{amsmath}
\usepackage{amsthm}
\usepackage{amssymb}
\usepackage{mathtools}
\mathtoolsset{centercolon} 

\usepackage{xcolor}

\usepackage[arrow, matrix, curve]{xy}

\usepackage{tikz}
\usetikzlibrary{calc,matrix,arrows}

\usepackage{enumitem}

\setenumerate[1]{label={\textrm{(\roman*)}}}
\setenumerate[2]{label={\textrm{(\alph*)}}}

\numberwithin{equation}{section}

\usepackage[alphabetic,backrefs]{amsrefs}

\theoremstyle{definition}
	\newtheorem{definition}{Definition}[section]
	\newtheorem{defn}[definition]{Definition} 
	\newtheorem*{definition*}{Definition}
	\newtheorem{example}[definition]{Example}
        \newtheorem{notation}[definition]{Notation}
\theoremstyle{plain}
	\newtheorem{lemma}[definition]{Lemma}
	
	\newtheorem{convention}[definition]{Convention}
	\newtheorem{proposition}[definition]{Proposition}
	\newtheorem{prop}[definition]{Proposition}
	\newtheorem{theorem}[definition]{Theorem}
	\newtheorem{thm}[definition]{Theorem}
	\newtheorem{fact}[definition]{Fact}
	\newtheorem*{theorem*}{Theorem}
	\newtheorem{corollary}[definition]{Corollary}

	\newtheorem{claim}{Claim}
        \newtheorem{observation}[definition]{Observation}
        
\theoremstyle{remark}

	\newtheorem{remark}[definition]{Remark}

\newenvironment{axioms}[1]{%
\begin{enumerate}[%
 label={\bfseries({#1}\arabic*)},
 ref={({#1}\arabic*)},
 leftmargin=*]}%
{\end{enumerate}}

\newenvironment{axiomsZero}[1]{%
\begin{enumerate}[%
 start=0,
 label={\bfseries({#1}\arabic*)},
 ref={({#1}\arabic*)},
 leftmargin=*]}%
{\end{enumerate}}

\DeclareMathOperator{\ad}{ad}
\DeclareMathOperator{\Ad}{Ad}
\DeclareMathOperator{\Aut}{Aut}
\DeclareMathOperator{\AutEff}{Aut_{\mathrm{eff}}}

\DeclareMathOperator{\Fix}{Fix}
\DeclareMathOperator{\Hol}{Hol}

\DeclareMathOperator{\id}{id}
\DeclareMathOperator{\Id}{Id}
\DeclareMathOperator{\Inn}{Inn}
\DeclareMathOperator{\Out}{Out}
\DeclareMathOperator{\proj}{proj}
\DeclareMathOperator{\rk}{rk}
\DeclareMathOperator{\Stab}{Stab}
\DeclareMathOperator{\Sym}{Sym}
\DeclareMathOperator{\Trans}{Trans}

\newcommand{\Geff}{G_{\mathrm{eff}}}

\newcommand{\C}{\mathbb{C}}
\newcommand{\E}{\mathbb{E}} 
\newcommand{\F}{\mathbb{F}}
\newcommand{\HH}{\mathbb{H}}
\newcommand{\K}{\mathbb{K}}
\newcommand{\N}{\mathbb{N}}
\newcommand{\Q}{\mathbb{Q}}
\newcommand{\R}{\mathbb{R}}
\newcommand{\Z}{\mathbb{Z}}

\newcommand{\AAA}{\mathcal{A}}
\newcommand{\CCC}{\mathcal{C}}

\newcommand{\EEE}{\mathcal{E}}
\newcommand{\FFF}{\mathcal{F}}
\newcommand{\GGG}{\mathcal{G}}
\newcommand{\LLL}{\mathcal{L}}

\newcommand{\TTT}{\mathcal{T}}
\newcommand{\UUU}{\mathcal{U}}
\newcommand{\VVV}{\mathcal{V}}

\newcommand{\XXX}{\mathcal{X}}

\newcommand{\GL}{\mathrm{GL}}

\newcommand{\SL}{\mathrm{SL}}

\newcommand{\SO}{\mathrm{SO}}

\DeclarePairedDelimiter\gen{\langle}{\rangle}

\newcommand{\xto}{\xrightarrow}
\newcommand{\into}{\hookrightarrow}

\let\ol\overline

\renewcommand{\iff}{\;\Leftrightarrow\;}

\newcommand{\disjoint}{\bigsqcup}

\newcommand\mtr[4]{\left(\begin{matrix} {#1} & {#2} \\ {#3} & {#4} \end{matrix}
\right) }

\renewcommand\phi\varphi
\renewcommand\epsilon\varepsilon
\newcommand\eps\varepsilon

\newcommand{\mX}{\widehat{\mu}}  
\newcommand{\mK}{\mu}  
\newcommand{\mT}{\widetilde{\mu}}  

\newcommand{\Defn}[1]{{\textbf{#1}}}

\begin{document}

\title{Kac--Moody symmetric spaces}
\author{Walter Freyn \and Tobias Hartnick \and Max Horn \and Ralf K\"ohl}

\maketitle

\begin{abstract}
  In the present article we introduce and study a class of topological reflection spaces that we call Kac--Moody symmetric spaces. These are associated with split real Kac--Moody groups and generalize Riemannian symmetric spaces of non-compact split type. 
  
  Based on work by the third-named author we observe that in a non-spherical Kac--Moody symmetric space there exist pairs of points that do not lie on a common geodesic; however, any two points can be connected by a chain of geodesic segments. We moreover classify maximal flats in Kac--Moody symmetric spaces and study their intersection patterns, leading to a classification of global and local automorphisms. Some of our methods apply to general topological reflection spaces beyond the Kac--Moody setting.
  
  Unlike Riemannian symmetric spaces, non-spherical non-affine irreducible Kac--Moody symmetric spaces also admit an invariant causal structure. For causal and anti-causal geodesic rays with respect to this structure we find a notion of asymptoticity, which allows us to define a future and past boundary of such Kac--Moody symmetric space. We show that these boundaries carry a natural polyhedral cell structure and are cellularly isomorphic to geometric realizations of the two halves of the twin buildings of the underlying split real Kac--Moody group. We also show that every automorphism of the symmetric space is uniquely determined by the induced cellular automorphism of the future and past boundary. 
  
  The invariant causal structure on a non-spherical non-affine irreducible Kac--Moody symmetric space gives rise to an invariant pre-order on the underlying space, and thus to a subsemigroup of the Kac--Moody group. 
  
  We conclude that while in some aspects Kac--Moody symmetric spaces closely resemble Riemannian symmetric spaces, in other aspects they behave similarly to {\em masures}, their non-Archimedean cousins. 
\end{abstract}

\section{Introduction}

Kac--Moody groups over a local field $\K$ as for instance studied in \cite{Rousseau06}, \cite{GaussentRousseau08}, \cite{HartnickKoehlMars}, \cite{GaussentRousseau14}, \cite{HartnickKoehl},  \cite{Bardy-PanseGaussentRousseau} are infinite-dimensional generalizations of the groups of $\K$-points of (split) semisimple algebraic groups. From a geometric point of view, semisimple groups over local fields arise as subgroups of the isometry groups of Riemannian symmetric spaces (in the Archimedean case) and Euclidean buildings (in the non-Archimedean case). It is thus natural to ask whether Kac--Moody groups over local fields admit a similar geometric interpretation.

For Kac--Moody groups over non-Archimedean local fields such a geometric interpretation is described in \cite{Rousseau:2011}, where the author discusses the notion of a {\em masure affine ordonn\'ee}  (sometimes translated as \emph{ordered affine hovel} into English, e.g.\ in \cite{GaussentRousseau08}). {\em Masures} are certain generalizations of Euclidean buildings that admit an action by a Kac--Moody group over a non-Archimedean local field $\K$, generalizing the notion of a Bruhat--Tits building endowed with the action of the $\K$-points of a split semisimple group. 

In the present article we investigate the Archimedean situation, focussing on the split real case. We introduce a generalization of Riemannian symmetric spaces of non-compact split type,  which we call Kac--Moody symmetric spaces and on which split real Kac--Moody groups act in a way that generalizes the action of semisimple split real Lie groups on their Riemannian symmetric spaces. It turns out that in this setting one can observe both phenomena that one is familiar with from the finite-dimensional theory and phenomena that are specific to the infinite-dimensional situation; some of these infinite-dimensional phenomena in fact have non-Archimedean analogs in the theory of {\em masures}. 
%

A key structural problem that one has to face when generalizing the notion of a Riemannian symmetric space, is that the latter is originally defined in terms of a smooth Riemannian metric on a manifold; we are unaware of any reasonable notion of smoothness on the kind of homogeneous spaces on which a (non-spherical and non-affine) real Kac--Moody group naturally acts, nor are these spaces metrizable with respect to their natural topologies. Our starting point is thus an alternative characterization of affine symmetric spaces, due to Ottmar Loos \cite{Loos68, Loos67}.

\begin{fact}[Loos \cite{Loos68, Loos67}] Let $\XXX$ be an affine symmetric space, and given $x, y \in \XXX$ denote by $x \cdot y$ the point reflection of $y$ at $x$. Then
$\mu: \XXX \times \XXX \to \XXX$, $\mu(x,y) := x \cdot y$ is a $C^1$-map satisfying the following axioms:
\begin{axioms}{RS}
\item[{\bfseries(RS1)}] for any $x\in \XXX$ we have $x\cdot x=x$,
\item[{\bfseries(RS2)}] for any pair of points $x,y\in \XXX$ we have $x\cdot(x\cdot y)=y$,
\item[{\bfseries(RS3)}] for any triple of points $x,y,z\in \XXX$ we have  $x\cdot( y\cdot z) = (x\cdot y) \cdot (x\cdot z)$,
\item[{\bfseries(RS4${}_{\rm loc}$)}] every $x \in \XXX$ has a neighbourhood $U$ such that $x\cdot y=y$ implies $y=x$ for all $y \in U$.
\end{axioms}
Conversely, if $\XXX$ is a smooth manifold and $\mu:  \XXX \times \XXX \to \XXX$ is a $C^1$-map subject to (RS1)--(RS4${}_{\rm loc}$) above, then $\XXX$ is an affine symmetric space, and $\mu(x,y)$ is the point reflection of $y$ at $x$. If $\XXX$ is a Riemannian symmetric space, then the isometries of $\XXX$ are exactly the $C^1$-maps $\alpha: \XXX \to \XXX$ satisfying $\alpha(x \cdot y) = \alpha(x) \cdot \alpha(y)$. If $\XXX$ is moreover of the non-compact type, then instead of the local condition (RS4${}_{\rm loc}$) it satisfies the global condition
\begin{axioms}{RS}
\item[{\bfseries(RS4)}] $x\cdot y=y$ implies $y=x$ for all $y \in \XXX$.
\end{axioms}
\end{fact}

Since we are interested in generalizations of Riemannian symmetric spaces of  non-compact type, we define the following:

\begin{definition} A pair $(\XXX, \mu)$ is called a \Defn{topological symmetric space} provided $\XXX$ is a topological space and $\mu: \XXX \times \XXX \to \XXX$, $\mu(x,y) := x \cdot y$ is a continuous map subject to the axioms (RS1)--(RS4) above. The \Defn{automorphism group} $\Aut(\XXX, \mu)$ of $(\XXX, \mu)$ is defined as
\[
\Aut(\XXX, \mu) := \{\alpha: \XXX \to \XXX\mid \alpha \text{ homeomorphism}, \,\alpha(x \cdot y) = \alpha(x) \cdot \alpha(y)\}.
\]
\end{definition}

Loos' theorem strongly uses the differentiability of $\mu$, and not much is known about general topological symmetric spaces without any smoothness assumption. For example, it is not even known to us whether a topological symmetric space which is homeomorphic to a finite-dimensional manifold necessarily arises from an affine symmetric space.

\medskip
We pursue three goals in the present article:
\begin{enumerate}
\item to develop a basic theory of topological symmetric spaces in the absence of any smoothness assumption;
\item to associate a topological symmetric space to a large class of Kac--Moody groups over an Archimedean local field (focusing on the split real case for simplicity);
\item to develop the structure theory of such Kac--Moody symmetric spaces, studying their geodesics, maximal flats, (local and global) automorphisms, causal structures and boundaries.
\end{enumerate}
Our results concerning (i) might actually be of interest beyond Kac--Moody theory. 

The three concepts of flats, geodesics and one-parameter subgroups of the isometry group are of fundamental nature in the study of Riemannian symmetric spaces. The former two are usually defined using the curvature tensor, and the existence of the latter is derived from an existence theorem for solutions of ordinary differential equations. In our topological setting we need to define flats and geodesics without reference to the curvature tensor, and to establish the existence of one-parameter subgroups without analytic tools.

Given a topological symmetric space $(\XXX, \mu)$ we call a subset $\gamma \subset \XXX$ a \Defn{geodesic} if there exists a bijection $\varphi: \R \to \gamma$ such that $\varphi(2x-y) = \mu(\varphi(x), \varphi(y))$ for all $x,y \in \R$. Compact connected subsets of geodesics will be called \Defn{geodesic segments}. We will now explain how geodesics in topological symmetric spaces give rise to one-parameter subgroups of $\Aut(\XXX, \mu)$. To formulate our result we first observe that by (RS3) every $x \in \XXX$ defines an automorphism $s_x \in \Aut(\XXX, \mu)$ by $s_x(y) := x \cdot y$ called the \Defn{point reflection} at $x$. The subgroup of $ \Aut(\XXX, \mu)$ generated by these point reflections will be denoted by $G(\XXX, \mu)$ and called the \Defn{main group} of $\XXX$.

\begin{thm}[Existence of one-parameter subgroups without differentiability assumptions; cf.\ Proposition~\ref{PropTranslationGroups}] \label{thm1.3} Let $(\XXX, \mu)$ be a topological symmetric space. Given a geodesic $\gamma \subset \XXX$ let
\[
T_\gamma := \{ s_p \circ s_q\mid p, q \in \gamma \} \subset G(\XXX, \mu).
\]
\begin{enumerate}
\item $T_\gamma \cong (\R, +)$ is a one-parameter subgroup of $G(\XXX, \mu)$ (and in particular of $\Aut(\XXX, \mu)$).
\item $T_\gamma$ acts sharply transitively on $\gamma$ by Euclidean translations.
\item If $t_1, t_2 \in T_\gamma$ and $t_1|_\gamma = t_2|_\gamma$, then $t_1 = t_2$.
\item If any two points in $\XXX$ can be connected by a finite chain of geodesic segments, then the one-parameter subgroups $T_\gamma$ generate a subgroup of $G(\XXX, \mu)$ of index $\leq 2$.
\end{enumerate}
\end{thm}

As for the definition of a flat in a topological symmetric space, we offer two notions, which we will later show to lead to the same concept in Kac--Moody symmetric spaces. Firstly, we have the following purely synthetic definition:

\begin{definition}
A closed subset $F \subset \XXX$ of cardinality $\geq 2$ is called a \Defn{weak flat} if it satisfies the following properties:
\begin{axioms}{F}
\item $F$ is a \Defn{reflection subspace}, i.e.\ if $x, y \in F$, then $x \cdot y \in F$.
\item $F$ is \Defn{midpoint convex}, i.e.\ if $x,y \in F$ then there exists $z \in F$ with $z \cdot x = y$ (and thus $z \cdot y = x$).
\item $F$ is \Defn{weakly abelian}, i.e.\ for all $x,y,z \in F$ one has
 \[ x \cdot (z \cdot (y\cdot z))=y\cdot (z\cdot (x\cdot z)).\]
\end{axioms}
\end{definition}
Denote Euclidean symmetric space with multiplication $\mu(x,y):= 2x-y$ by $\E^n = (\R^n, \mu)$. A closed reflection subspace $F$ of a topological symmetric space $\XXX$ is called a \Defn{Euclidean flat} of \Defn{rank} $n$ if it is isomorphic to $\E^n$ as a topological reflection space. With this notion a geodesic is just a Euclidean flat of rank $1$, and every Euclidean flat is a weak flat, see Figure \ref{figure:algebraic_flat}.

\begin{figure}[h]
\label{figure:algebraic_flat}
\centering
\begin{tikzpicture}[scale=0.7]

\coordinate (x) at (1,2);
\coordinate (y) at (2,4);
\coordinate (z) at (0,4);

\coordinate (sx-z) at ($(z)!2.0!(x)$);
\coordinate (sy-z) at ($(z)!2.0!(y)$);
\coordinate (sz-sx-z) at ($(sx-z)!2.0!(z)$);
\coordinate (sz-sy-z) at ($(sy-z)!2.0!(z)$);
\coordinate (sy-sz-sx-z) at ($(sz-sx-z)!2.0!(y)$);
\coordinate (sx-sz-sy-z) at ($(sz-sy-z)!2.0!(x)$);

\draw [fill=blue] (x) circle (2pt) node [above right] {$x$};
\draw [fill=blue] (y) circle (2pt) node [above right] {$y$};
\draw [fill=blue] (z) circle (2pt) node [above right] {$z$};

\draw [fill=blue] (sx-z) circle (2pt) node [above right] {$x\cdot z$};
\draw [fill=blue] (sy-z) circle (2pt) node [above right] {$y\cdot z$};

\draw [fill=blue] (sz-sx-z) circle (2pt) node [left] {$z\cdot (x\cdot z)$};
\draw [fill=blue] (sz-sy-z) circle (2pt) node [above] {$z\cdot (y\cdot z)$};

\draw [fill=blue] (sy-sz-sx-z) circle (2pt) node [below right] {$y\cdot (z\cdot (x\cdot z)) = x\cdot (z\cdot (y\cdot z))$};

\draw (z) -- (sx-z) -- (sz-sx-z) -- (sy-sz-sx-z);
\draw (z) -- (sy-z) -- (sz-sy-z) -- (sx-sz-sy-z);

\draw[dashed] (sy-sz-sx-z) -- (sy-z);
\draw[dashed] (sx-sz-sy-z) -- (sx-z);

\end{tikzpicture}
\caption{Euclidean space is weakly abelian.}
\label{fig:alg-flat}
\end{figure}
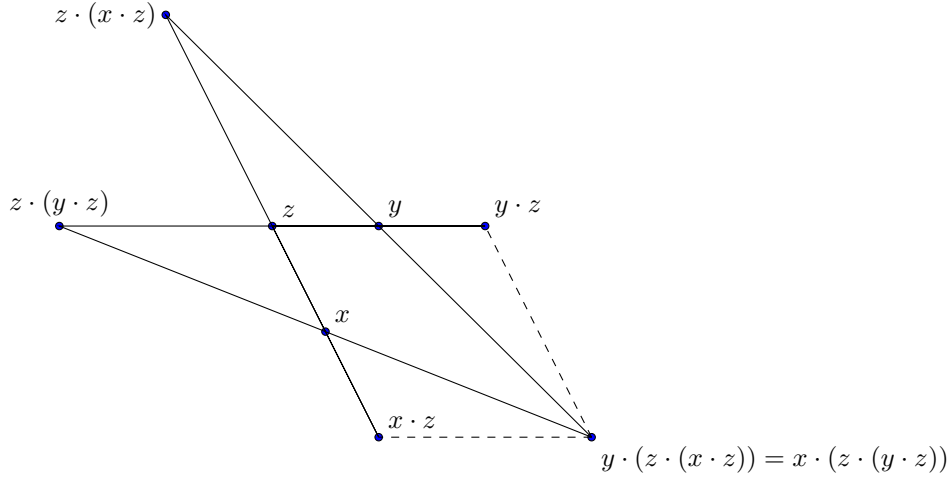

\medskip
We now turn to the main objects of our interest in the present article and introduce Kac--Moody symmetric spaces, a class of topological symmetric spaces associated with (split real) Kac--Moody groups. Given a generalized Cartan matrix {\bf A}  (see Definition~\ref{GCMdef}) we denote by $G = G({\bf A})$ the corresponding \Defn{simply connected centred split real Kac--Moody group} of type {\bf A} (see Definition~\ref{DefTopSplitKM}). Throughout this article we will assume that {\bf A} is \Defn{irreducible} and \Defn{symmetrizable} (see Definition~\ref{GCMdef}), and we will consider $G$ as a topological group with the \Defn{Kac--Peterson topology} (see Definition~\ref{DefTopSplitKM}). For some of our results we will need additional assumptions on {\bf A} (e.g.\ non-spherical, non-affine, on a few occasions two-spherical), but for the basic definitions we do not need any of these assumptions.

There exists a canonical continuous involution $\theta$ of $G$ which on each standard rank one subgroup restricts to the contragredient automorphism $g \mapsto g^{-\top}$. Any conjugate of the involution $\theta$ in the semidirect product $G \rtimes \gen{\theta}$ is called a \Defn{Cartan--Chevalley involution}. The group $G$ acts transitively by conjugation on the set $\XXX_G$ of Cartan--Chevalley involutions, and we equip $\XXX_G$ with the quotient topology with respect to this action.

\begin{proposition}[Cf.\ Corollary \ref{TopRefSpaceExternal}] \label{thm1.5} The space $\XXX_G$ is a topological symmetric space with respect to
\[ \mu: \XXX_G \times \XXX_G \to \XXX_G, \quad \mu(\alpha,\beta) := \alpha\circ \beta \circ \alpha.\]
\end{proposition}

\begin{definition} The symmetric space $(\XXX_G, \mu)$ is called the \Defn{unreduced Kac--Moody symmetric space} of the split real Kac--Moody group $G$.
\end{definition}

In the spherical case, i.e., if the Kac--Moody group $G$ actually is a Lie group, this is the (involution model of the) associated Riemannian symmetric space.

If the Cartan matrix ${\bf A}$ is non-invertible, then the center $Z(G)$ of $G$ has positive dimension, given by the corank of ${\bf A}$. In this case, the unreduced Kac--Moody symmetric space $\XXX_G$ fibers over a topological symmetric space $\ol{\XXX}_G$ with fiber given by a Euclidean space of dimension equal to the corank of $\mathbf{A}$, and the adjoint quotient $\Ad(G)$ of $G$ acts on $\ol{\XXX}_G$. We refer to $\ol{\XXX}_G$ as the \Defn{reduced Kac--Moody symmetric space} of $G$. In the case where ${\bf A}$ is non-invertible, it is this reduced version that resembles most closely a Riemannian symmetric space.

\medskip
The following result describes flats in Kac--Moody symmetric spaces.
\begin{theorem}[Flats in Kac--Moody symmetric spaces; cf.\ Section~\ref{classmaxflats}]\label{IntroFlats} Let $\XXX_G$ be an unreduced Kac--Moody symmetric space, and let $\ol{\XXX}_G$ be its reduced quotient.
\begin{enumerate}
\item Every weak flat in $\XXX_G$ or $\ol{\XXX}_G$ is Euclidean. In particular, all weak flats are finite-dimensional and locally compact.
\item Every weak flat in $\XXX_G$ or $\ol{\XXX}_G$ is contained in a maximal weak flat.
\item The projection $\XXX_G \to \ol{\XXX}_G$ induces a bijection between maximal weak flats in $\XXX_G$ and maximal weak flats in $\ol{\XXX}_G$.
\item $G$ acts transitively on pairs $(p, F)$ where $F$ is a maximal weak flat in $\XXX_G$ (or $\ol{\XXX}_G$) and $p \in F$. In particular, all maximal weak flats in $\XXX_G$ (respectively $\ol{\XXX}_G$) are Euclidean spaces of the same dimension $r(\XXX_G)$ (respectively $r(\ol{\XXX}_G)$).
\item $r(\XXX_G)$ equals the number of rows of ${\bf A}$, and $r(\ol{\XXX}_G)$ equals the rank of ${\bf A}$.
\end{enumerate}
\end{theorem}

The integers $r(\XXX_G)$ and $r(\ol{\XXX}_G)$ are called the \Defn{rank} of $\XXX_G$ and $\ol{\XXX}_G$ respectively. In the sequel we refer to a maximal weak flat simply as a \Defn{maximal flat}, and to a pair $(p, F)$ as in (iii) as a \Defn{pointed maximal flat}. Besides maximal flats, we are also interested in minimal non-trivial flats, i.e. geodesics. 

\begin{theorem}[Geodesic connectedness of Kac--Moody symmetric spaces; cf.\ Section~\ref{geodesicconnectedness}] \label{thm1.8} The Kac--Moody symmetric spaces  $\XXX_G$ and $\ol{\XXX}_G$ have the following properties:
\begin{enumerate}
\item $\XXX_G$ and $\ol{\XXX}_G$ are \Defn{geodesically connected}, i.e., any two points in $\XXX_G$ or $\ol{\XXX}_G$ can be connected by a finite chain of geodesic segments.
\item If ${\bf A}$ is not spherical, then  $\XXX_G$ and $\ol{\XXX}_G$ are \Defn{non-geodesic}, i.e.\ there exist points $x, y \in \XXX_G$ (and also in $\ol{\XXX}_G$) which do not lie on a common geodesic (and hence are not contained in a common maximal flat).
\end{enumerate}
\end{theorem}

Note that (ii) is in stark contrast to the case of Riemannian symmetric spaces, which are always geodesic. It is, however, reminiscent of the corresponding property of {\em masures}: not every pair of points is contained in a common apartment. In fact, this property is the key feature that separates the class of {\em masures} from the class of buildings. 

By construction, the group $G$ acts by automorphisms on $\XXX_G$ and thus on its quotient $\ol{\XXX}_G$. The latter action (but in general not the former) factors through a faithful action of $\Ad(G)$. As in the spherical case, the full automorphism group of $\ol{\XXX}_G$ is slightly larger than $\Ad(G)$.

\begin{theorem}[Automorphisms of reduced Kac--Moody symmetric spaces; cf.\ Proposition~\ref{AutBuilding}, Proposition~\ref{StructureOfAutG} and Theorem~\ref{ThmGlobalAut}] \label{thm1.9}
The group  $\Ad(G)$ is a finite index subgroup of the automorphism group $\Aut(\ol{\XXX}_G)$. 
More precisely, $\Aut(\ol{\XXX}_G)$ is isomorphic to $\Aut(G) \cong \Aut(\Ad(G))$, and every automorphism in $ \Aut(\Ad(G))$ can be written as a product of an inner automorphism, a diagonal automorphism, a power of a fixed Cartan--Chevalley involution and an automorphism of the Dynkin diagram $\Gamma_{\bf A}$. Moreover, $\Aut(\ol{\XXX}_G)$ embeds into the automorphism group of the twin building associated with $G$ and if ${\bf A}$ is non-spherical then
\[
\Aut(\ol{\XXX}_G) = \Aut^+(\ol{\XXX}_G) \rtimes \langle s_o \rangle,
\]
where $\Aut^+(\ol{\XXX}_G) < \Aut(\ol{\XXX}_G) $ is the index two subgroup preserving the two halves of the twin building (instead of interchanging the two halves).
\end{theorem}

\begin{convention} For the rest of this introduction we assume that {\bf A} is non-spherical and non-affine (on top of our standing assumptions that {\bf A} be irreducible and symmetrizable).
\end{convention}

Besides the global automorphisms in $\Aut(\ol{\XXX}_G)$ one can also consider local automorphisms of $\ol{\XXX}_G$ in the following sense. Fix a pointed maximal flat $(p, F)$ and let 
\begin{align*}
\Stab(p, F) &:= \{g \in \Aut(\ol{\XXX}_G) \mid g.F = F, g.p = p\}, &
\Stab_G(p, F) &:= \{g \in G \mid g.F = F, g.p = p\}, \\
\Fix(p, F) &:=\{g \in \Aut(\ol{\XXX}_G) \mid \forall f \in F: g.f = f\}, &
\Fix_G(p, F) &:=\{g \in G \mid \forall f \in F: g.f = f\}.
\end{align*}
Then $W( \Aut(\ol{\XXX}_G \curvearrowright \ol{\XXX}_G)) := \Stab(p, F)/\Fix(p, F)$ is independent of the choice of $(p, F)$ and acts on $F$ fixing $p$, and the same is true for the group $W(G \curvearrowright \ol{\XXX}_G) :=\Stab_G(p, F)/\Fix_G(p, F)$. The groups $W( \Aut(\ol{\XXX}_G \curvearrowright \ol{\XXX}_G))$ and $W(G \curvearrowright \ol{\XXX}_G)$ are called the \Defn{geometric Weyl groups} of $ \Aut(\ol{\XXX}_G)$ and $G$ respectively. Note that the action of these groups on $F$ preserve the subset $F^{\rm sing}(p) \subset F$ of points of $F$ which are contained in more than one maximal flat containing $p$. Moreover, if $\phi: \E^r \to F$ is an arbitrary isomorphism of reflection spaces with $\phi(0) = p$, then $\phi$ intertwines the elements of the geometric Weyl groups with linear automorphisms of $\E^r$. It thus follows that the geometric Weyl groups are contained in the group
\[
\GL(p, F, F^{\rm sing}(p)) :=  \{\alpha: F \to F\mid \alpha(F^{\rm sing}(p)) = F^{\rm sing}(p) \text{ and }\widehat{\alpha} := \phi^{-1} \circ \alpha \circ \phi \in \GL_n(\R)\}.
\]
We refer to elements of $\GL(p, F, F^{\rm sing}(p))$ as \Defn{local transformations} of $(p, F)$; one can show that the notion of a local transformation does not depend on the choice of $\phi$.

It turns out that the subset $\phi^{-1}(F^{\rm sing}(p)) \subset \E^r$ is a hyperplane arrangement, and hence every homothety of $\E^r$ gives rise to a local transformation. More generally, there exists a non-degenerate bilinear form on $\E^r$ (unique up to multiples) such that elements of $\GL(p, F, F^{\rm sing}(p))$ act by similarities with respect to this bilinear form (see Corollary \ref{CanonicalFormpF}). This yields a splitting
\[
\GL(p, F, F^{\rm sing}(p)) = \R^{>0} \times \Aut(p,F),
\]
where $\R^{>0}$ is the group of positive homotheties and $\Aut(p,F)$ is the subgroup of $\GL(p, F, F^{\rm sing}(p))$ which preserves the canonical bilinear form. Elements of $\Aut(p, F)$ are called \Defn{local automorphisms} of $(p, F)$, and by the following theorem the geometric Weyl groups acts by local automorphisms. Concerning the statement of the theorem we recall that one can associate to the generalized Cartan matrix ${\bf A}$ a Coxeter system $(W,S)$ whose Coxeter diagram $\Gamma_{(W,S)}$ has the same underlying graph as the Dynkin diagram $\Gamma_{\bf A}$ of {\bf A}, but whose labelling carries less information (see Subsection \ref{appendixweylgroup}); hence the automorphism group $\Aut(\Gamma_{\bf A})$ can be considered as a subgroup of the automorphism group $\Aut(W,S) :=\Aut(\Gamma_{(W,S)})$ of the Coxeter diagram.
\begin{theorem}[Local vs.\ global automorphisms; cf.\ Theorem~\ref{ExtendLocalAutomorphisms}] \label{thm1.11} Let $\ol{\XXX}_G$ be a reduced Kac--Moody symmetric space of irreducible non-spherical, non-affine type and let $(p, F)$ be a maximal pointed flat in $\ol{\XXX}_G$.
\begin{enumerate}
\item The action of the geometric Weyl group $W( \Aut(\ol{\XXX}_G) \curvearrowright \ol{\XXX}_G) $ on $(p, F)$ is by local automorphisms.
\item The local automorphism group $\Aut(p, F)$ is isomorphic to $(W \rtimes \Aut(W,S)) \times \Z/2\Z$, and hence the group $\GL(p, F, F^{\rm sing}(p))$ is isomorphic to $\R^{>0} \times (W \rtimes \Aut(W,S)) \times \Z/2\Z$.
\item Under the isomorphism from (ii) the subgroup $W( \Aut(\ol{\XXX}_G) \curvearrowright \ol{\XXX}_G) < \Aut(p, F)$ corresponds to the finite index subgroup $(W \rtimes \Aut(\Gamma_{\bf A})) \times \Z/2\Z < (W \rtimes \Aut(W,S)) \times \Z/2\Z$. 
\item Every local automorphisms is the restriction of a global automorphism if and only if $\Aut(W,S)=\Aut(\Gamma_{\bf A})$.
\end{enumerate}
\end{theorem}
In fact, the group $W \rtimes \Aut(W,S)$ appearing in (ii) is nothing but the (simplicial) automorphism group of the simplicial Coxeter complex $\Sigma(W,S)$ associated with the Coxeter system $(W,S)$. Concerning the Weyl group of $G$ we observe:
\begin{corollary}[Algebraic Weyl group equals geometric Weyl group; cf.\ Corollary \ref{GeometricWeylGroup2}] \label{thm1.12} The geometric Weyl group $W(G \curvearrowright \ol{\XXX}_G)$ of $G$ is isomorphic to the algebraic Weyl group $W$.
\end{corollary}
With the notable exception of Theorem \ref{thm1.8}.(ii) the part of the theory of Kac--Moody symmetric spaces described so far follows closely the classical theory of Riemannian symmetric spaces. On the other hand, it turns out that (irreducible, non-spherical, non-affine) Kac--Moody symmetric spaces also carry additional structure, which is not shared by Riemannian symmetric spaces of the non-compact type, but which is shared by a different class of affine symmetric spaces called \Defn{causal symmetric spaces} (see \cite{HilgertOlafsson}).
\begin{proposition}[Existence of an invariant causal structure, cf.\ Section \ref{SecCausalStructureEx}] \label{PropCausalIntro}There exist families $(\ol{C}^+_x)_{x \in \ol{\XXX}_G}$ and $(\ol{C}^-_x)_{x \in \ol{\XXX}_G}$ of subsets of $\ol{\XXX}_G$ with the following properties:
\begin{enumerate}
\item $(\ol{C}^+_x)_{x \in \ol{\XXX}_G}$ and $(\ol{C}^-_x)_{x \in \ol{\XXX}_G}$ are \Defn{cone fields}, i.e.,  for every $x \in \ol{\XXX}_G$ the subsets $\ol{C}^+_x \subset \ol{\XXX}_G$ and $\ol{C}^-_x \subset \ol{\XXX}_G$ each intersect every flat containing $x$ in an open cone with tip $x$.
\item $(\ol{C}^+_x)_{x \in \ol{\XXX}_G}$ and $(\ol{C}^-_x)_{x \in \ol{\XXX}_G}$ are \Defn{invariant} under $\Aut^+(\ol{\XXX}_G)$, i.e., $\alpha(\ol{C}^+_x) = \ol{C}^+_{\alpha(x)}$ and $\alpha(\ol{C}^-_x) = \ol{C}^-_{\alpha(x)}$ for all $\alpha \in \Aut^+(\ol{\XXX}_G)$ and $x \in \ol{\XXX}_G$.
\end{enumerate}
\end{proposition}
In analogy with the theory of causal symmetric spaces we refer to the invariant cone fields $(\ol{C}^+_x)_{x \in \ol{\XXX}_G}$ and $(\ol{C}^-_x)_{x \in \ol{\XXX}_G}$ as \Defn{causal structures} on $\ol{\XXX}_G$. Roughly speaking, the causal structures are a global version of the Tits cone, resp.\ its negative in the underlying Kac--Moody Lie algebra. We refer to Subsection \ref{SecCausalStructureEx} for a precise definition. 

From the choice of a \Defn{positive causal structure} $(\ol{C}^+_x)_{x \in \ol{\XXX}_G}$ we infer a notion of \Defn{causal} (or ``time-like\footnote{In the study of Lorentzian causal structures, causal curves are also called \Defn{time-like curves}. Since the causal structures investigated here need not be Lorentzian, we will not use this terminology.}'') \Defn{curve} in $\ol{\XXX}_G$. Namely, we say that a continuous curve $\gamma: [S, T] \to \ol{\XXX}$ with $0< S < T < \infty$ is \Defn{causal} if for every $t\in [S,T)$ there exists $\epsilon>0$ such that 
\[
\gamma((t, t + \epsilon)) \subset \ol{C}^+_{\gamma(t)}.
\]
The notion of an \Defn{anti-causal curve} is defined analogously via the \Defn{negative causal structure} $(\ol{C}^-_x)_{x \in \ol{\XXX}_G}$. Using causal geodesic rays in $\ol{\XXX}_G$ we associate two further structures with $\ol{\XXX}_G$ which have no counterpart in the theory of Riemannian symmetric spaces, but which are reminiscent to classical objects in the theory of causal symmetric spaces: the \Defn{causal boundary} of $\ol{\XXX}_G$ and the \Defn{causal pre-order} on $\ol{\XXX}_G$.

The causal boundary can be constructed as follows. Denote by $\partial_\bullet \ol{\XXX}_G$ the collection of geodesic rays in $\ol{\XXX}_G$, and by $\Delta^\pm_\bullet \subset  \partial_\bullet \ol{\XXX}_G$ the subset of all causal/anti-causal geodesic rays. By invariance of the causal structure, the subsets $\Delta^\pm_\bullet$ are invariant under $\Aut^+(\ol{\XXX})$ and their union $\Delta_\bullet$ is invariant under $\Aut(\ol{\XXX})$. Points in the causal boundary will be defined as equivalence classes of causal or anti-causal rays by an equivalence relation which mimics asymptoticity of geodesic rays in Riemannian symmetric spaces.

Recall that if $\XXX$ is a non-compact Riemannian symmetric space, then two geometric rays in $\XXX$ are called \Defn{asymptotic}, if they are at bounded Hausdorff distance. For example, two geodesic rays $r_1, r_2$ in Euclidean space $\E^n$ are asymptotic if and only if they are parallel and point in the same direction, i.e.\ they are of the form $r_1(t) = x + tv$ and $r_2(t) = y + tv$ for some $x,y \in \R^n$ and a unit vector $v$, and two geodesic rays  in the hyperbolic plane are asymptotic if they have the same endpoint in the boundary. In Subsection \ref{SubsecAsymptotic} we construct equivalence relations $\parallel$ on $\Delta^\pm_\bullet$  with the following properties:
\begin{axioms}{A}
\item If $r \in \Delta^\pm_\bullet$ and $x \in \ol{\XXX}_G$, then there exists a unique $r' \in \Delta^\pm_\bullet$ emanating from $x$ with $r \parallel r'$.
\item $\parallel$ is invariant under $\Aut^+(\ol{\XXX})$, i.e., if $r_1 \parallel r_2$, then $\alpha(r_1)\parallel \alpha(r_2)$ for all $\alpha \in \Aut^+(\ol{\XXX})$.
\item  If $r_1, r_2 \in \Delta^\pm$ are contained in a common embedded hyperbolic plane in $\ol{\XXX}_G$, which arises as the orbit of a rank one subgroup of $G$, 
then $r_1 \parallel r_2$ if and only if they are asymptotic in the hyperbolic sense.
\item If $r_1, r_2 \in \Delta^\pm$ are contained in a common maximal flat $F$, then $r_1 \parallel r_2$ if and only if they are asymptotic in the Euclidean sense.
\end{axioms}
In view of these properties we call $r_1, r_2 \in \Delta^\pm$ \Defn{asymptotic} if $r_1 \parallel r_2$.
\begin{definition} The set $\Delta^+_\parallel := \Delta^+ _\bullet/ \parallel$ of asymptoticity classes of causal rays is called the \Defn{future boundary} of the Kac--Moody symmetric space $\ol{\XXX}_G$, and the set $\Delta^-_\parallel := \Delta^- _\bullet/ \parallel$ is called its \Defn{past boundary}. The union $\Delta_\parallel := \Delta^+_\parallel \sqcup \Delta^-_\parallel$ is called the \Defn{causal boundary}.
\end{definition}
By (A2) the $\Aut^+(\ol{\XXX}_G)$-action on causal/anti-causal curves induces an action on the future/past boundary. In Subsection \ref{SecMuni} we equip the boundaries $\Delta^+_\parallel$ with the structure of an \Defn{ideal polyhedral complex}. Here, a polyhedral complex is a topological space obtained by glueing polyhedra along faces, and an ideal polyhedral complex is obtained from a polyhedral complex by removing some faces of codimension $\geq 2$ (see Subsection \ref{SecIdealPoly}). We then show that $\Aut^+(\ol{\XXX}_G)$ acts on these boundaries by polyhedral automorphisms. Unlike the spherical case, the ideal polyhedral structure on the boundary will in general not be simplicial. 

In Subsection \ref{TwinBuildinRealization} we construct an ideal polyhedral complex $|\Delta|_{\ol{\mathfrak a}}$, whose associated chamber system is given by the twin building $\Delta$ of $G$. This complex is combinatorially isomorphic to the Davis realization of $\Delta$ in the sense that the underlying cell posets are isomorphic, but in general the cells may have a different geometry and may even be of smaller dimension.
\begin{theorem}[Twin building vs.\ causal boundary; cf.\ Corollary~\ref{theend}] \label{thm1.14} The causal boundary $\Delta_\parallel$ is $\Aut(\ol{\XXX}_G)$-equivariantly geometrically isomorphic to $|\Delta|_{\ol{\mathfrak a}}$, and the past and future boundary are $\Aut^+(\ol{\XXX}_G)$-equivariantly geometrically isomorphic to the halves of $|\Delta|_{\ol{\mathfrak a}}$.
\end{theorem}
 Theorem \ref{thm1.14} should be compared to the classical fact that the geometric boundary of an irreducible Riemannian symmetric space of non-compact type, i.e.\ the collection of all geodesic rays modulo asymptoticity, carries a natural polyhedral (in fact, simplicial) structure which is isomorphic to the geometric realization of the corresponding spherical building (see, e.g., \cite{Kleiner/Leeb:1997}). This analogy is meaningful, since in the finite-dimensional case, the Tits cone is given by the whole Cartan subalgebra, and hence the canonical causal structure is the trivial causal structure in which every curve is causal. 

In the case of a hyperbolic Kac--Moody group, Theorem \ref{thm1.14} can be seen as a global version of the lightcone embedding of the twin building as described in \cite{CFF}. The analogous construction of a twin building at infinity for {\em masures} can be found in \cite[Section~3]{Rousseau:2011}; by \cite[Theorem~1]{CMR} this twin building at infinity of a {\em masure} actually carries a natural topology that turns it into a weak topological twin building in the sense of \cite{HartnickKoehlMars}. 

As in the finite-dimensional case, each asymptoticity class of causal rays in a Kac--Moody symmetric space forms an orbit under the action of an appropriate parabolic subgroup of $G$ (see Proposition~\ref{parabolicatinfinity}). Geometrically this means that if $r$ is a causal ray, which is regular in the sense that it is contained in a unique maximal flat, then all the causal rays asymptotic to $r$ can be obtained by parallel-translating $r$ in this flat and then sliding the resulting rays along suitable horospheres.

To push the analogy with the Riemannian case even further,  recall that every automorphism of a Riemannian symmetric space is uniquely determined by its action on the geometric boundary, i.e., the spherical building at infinity. In the Kac--Moody setting a similar statement is true: The automorphism is uniquely determined by its action on the causal boundary, i.e., the twin building at infinity.

\begin{theorem}[Causal boundary rigidity; cf.\ Corollary~\ref{theend}] \label{thm1.15} 
Every automorphism of $\ol{\XXX}_G$ is uniquely determined by the induced combinatorial automorphism of the causal boundary. Every automorphism in $\Aut^+(\ol{\XXX}_G)$ is uniquely determined by the induced combinatorial automorphism of the future (or past) boundary.
\end{theorem}

Having discussed the causal boundary of Kac--Moody symmetric spaces, we now turn to the second structure on $\ol{\XXX}_G$ induced by the canonical causal structure: We write $x \prec y$ and say that $x$ \Defn{strictly causally precedes} $y$ if there exists a piecewise geodesic causal curve $\gamma: [S, T] \to \ol{\XXX}_G$ with $\gamma(S) = x$ and $\gamma(T) = y$, and we define the \Defn{causal pre-order} $\preceq$ on $\ol{\XXX}_G$ by setting $x \preceq y$ if $x \prec y$ or $x=y$. Invariance of the causal structure implies that the pre-order $\preceq$ is invariant under $\Aut^+(\ol{\XXX}_G)$. It is not clear from the definition whether the causal pre-order is anti-symmetric, i.e.\ a partial order.
\begin{proposition}[Order dichotomy, cf.\ Proposition \ref{OrderDichotomy}] \label{dichintro} Either the causal pre-order on $\ol{\XXX}_G$ is the trivial pre-order, i.e.\ any point in $\ol{\XXX}_G$ causally proceeds any other point, or the causal pre-order is a partial order.
\end{proposition}
Currently we do not know for any irreducible, non-spherical, non-affine Kac--Moody symmetric space whether its causal pre-order is trivial or a partial order, but we believe that it is not always trivial. The problem of establishing such a result is related to a more classical problem in Kac--Moody theory, namely whether Kostant's classical convexity theorem \cite[Theorem 4.1]{Kostant73} can be extended to general Kac--Moody groups. An infinitesimal version was established by Kac and Peterson in \cite{KacPeterson84}, but there is no global version available so far.

\medskip
The focus of this article by design is on non-spherical and non-affine Kac--Moody symmetric spaces. We refer to \cite{Heintze}, \cite{Freyn} for literature focusing more on the affine case. If in our article one replaces the (derived) centered Kac--Moody group by the full Kac--Moody group with larger torus corresponding to the enlargened generalized Cartan matrix in the non-invertible situation, then it is likely to be possible to carry over our results also to the affine case. An additional advantage of that alternative approach should be that all involved polyhedral cell structures were actually simplicial, at the cost that in the non-affine situation with non-invertible generalized Cartan matrix the dimension of the maximal flats were larger than necessary.

\medskip \noindent {\bf Acknowledgements:}
The authors thank the Mathematisches Forschungsinstitut Oberwolfach for the hospitality during the mini-workshops {\em Symmetric varieties and involutions of algebraic groups} in 2008 and {\em Generalizations of symmetric varieties} in 2012 and, furthermore, the Lorentz Center Leiden and the Center for Mathematical Sciences at the Technion, Haifa for hosting the subsequent international workshops on that topic in 2013 and 2015.  
The authors express their gratitude to the participants of these workshops for numerous discussions on the topic of this article; they particularly thank Bernhard M\"uhlherr for a wealth of very helpful conversations and Pierre-Emmanuel Caprace and Guy Rousseau for several deep discussions concerning the question whether Kac--Moody symmetric spaces admit a meaningful canonical pre-order. The authors also thank an anonymous referee and Guy Rousseau for extremely helpful very detailed criticism on preliminary versions of this article; these comments tremendously helped improve the quality of the structure and of the results of this article. In addition the authors thank Julius Gr\"uning and Timoth\'ee Marquis for further very useful comments on one of these preliminary versions.

The first-named author thanks the IHES in Bures-sur-Yvette for the hospitality during several extended research visits.

The second-named author thanks the JLU Gie{\ss}en for the hospitality during several extended research visits since 2014.

The last-named author acknowledges partial funding from EPSRC via the grant EP/H02283X and from DFG via the grant KO4323/13. The last-named author also thanks the Hausdorff Institute for Mathematics in Bonn for the hospitality during the trimester programs {\em Rigidity} in 2009 and {\em Algebra and Number Theory} in 2010, the IHES in Bures-sur-Yvette for the hospitality during a research visit in 2013, the Max Planck Institute for Gravitational Physics in Golm for the hospitality during two research visits in 2013 and 2015, and the Technion in Haifa for the hospitality during several extended research visits since 2014.

\tableofcontents

\section{Concepts from synthetic geometry}
\subsection{Reflection spaces} \label{reflectionspaces}

\begin{definition}\ \label{defn:reflection_space_Loos}
Let $\XXX$ be a set and $\mu: \XXX \times \XXX \to \XXX$, $(x,y) \mapsto x \cdot y$ be a map.
\begin{enumerate}
\item $(\XXX, \mu)$ is called a \Defn{reflection space} if it satisfies the following axioms:
\begin{axioms}{RS}
  \item\label{RS1} for any $x\in \XXX$ one has $x\cdot x=x$,
  \item\label{RS2} for any pair of points $x,y\in \XXX$ one has $x\cdot(x\cdot y)=y$,
  \item\label{RS3} for any triple of points $x,y,z\in \XXX$ one has  $x\cdot( y\cdot z) = (x\cdot y) \cdot (x\cdot z)$.
\end{axioms}
\item A reflection space is called \Defn{symmetric} or a \Defn{symmetric space} if it satisfies the additional axiom:
\begin{axioms}{RS} \setcounter{enumii}{3}
  \item\label{RS4} $x\cdot y=y$ implies $y=x$ for all $x,y \in \XXX$.
\end{axioms}
\end{enumerate}
The \Defn{category of reflection spaces} has the class of reflection spaces as objects; a morphism between two objects $(\XXX_1, \mu_1)$ and $(\XXX_2, \mu_2)$ is a map $\phi:\XXX_1\to \XXX_2$ such that $\phi(x \cdot y) = \phi(x) \cdot \phi(y)$ for all $x,y \in \XXX_1$. The \Defn{category of symmetric spaces} is the full subcategory whose objects are symmetric spaces.
\end{definition}

\begin{remark} Our definition of a reflection space is taken from \cite{Loos69I}. However, Loos defines a symmetric space as a smooth reflection space, in which a local version of \ref{RS4} holds. Our definition of a symmetric space is more demanding, but does not require a topology on $\XXX$. 
  An alternative definition of a discrete symmetric space can be found in \cite{Caprace05}. In view of \eqref{InvolutionPicture} in Lemma~\ref{CapraceAxiom} below, the definition of a symmetric space given in \cite{Caprace05} is equivalent to what we call a reflection space in this article.
  
In the literature the concept of a reflection space is also known as an {\em (involutory) quandle}.
\end{remark}

\begin{example}\ \label{firstexample}
\begin{enumerate}
\item For any group $G$, the pair $(G,\mu_G)$ with $\mu_G(x,y):=xy^{-1}x$ is a reflection space.
\item For $n\in\N$, the \Defn{$n$-dimensional Euclidean space} $\E^n$ is the symmetric space $(\R^n,\mu_\E)$
  with $\mu_\E(x,y):=2x-y = x-y+x$. Geometrically, $\mu_\E(x,\cdot)$ is the point reflection at $x$.

  Note that this example, of course, is just the example of part (i) for the group $(\R^n,+)$.
\item Similar to (ii), spheres and hyperbolic spaces are reflection spaces, where $\mu(x, \cdot)$ is defined as the spherical/hyperbolic point reflection at $x$.
\end{enumerate}
\end{example}

In view of the previous examples, given a reflection space $(\XXX, \mu)$ the map
\[
s_x: \XXX \to \XXX, \quad y \mapsto x\cdot y.
\]
is called the \Defn{point reflection} at $x$; a product of two point reflections is called a \Defn{transvection}. By Axiom \ref{RS2} all point reflections are involutions, and Axiom \ref{RS3} states that point reflections (and hence transvections) are automorphisms.

In the sequel denote by $\Aut(\XXX, \mu)$ the \Defn{automorphism group} of $\XXX$ and by \[S(\XXX, \mu) := \{ s_x \mid x \in \XXX \} \subset \Aut(\XXX, \mu)\] the subset of all point reflections. The subgroup  \[G(\XXX, \mu) := \langle S(\XXX, \mu) \rangle < \Aut(\XXX, \mu)\] generated by the set $S(\XXX, \mu)$ of point reflections is called the \Defn{main group} of $(\XXX, \mu)$, and the subgroup  \[\Trans(\XXX, \mu) := \langle s_x \circ s_y \mid x, y \in \XXX \rangle < G(\XXX, \mu)\] generated by all transvections is called the  \Defn{transvection group}.
By definition, $\Trans(\XXX, \mu)$ has index at most $2$ in $G(\XXX, \mu)$. The reflection space $(\XXX, \mu)$ is called \Defn{homogeneous} if $\Aut(\XXX, \mu)$ acts transitively on $\XXX$, and \Defn{reflection-homogeneous} if $G(\XXX, \mu)$ acts transitively on $\XXX$.

\medskip
The following formula describes the behavior of point reflections under conjugation.

\begin{lemma}[{see \cite[p.~64, line 15]{Loos69I}}]\label{CapraceAxiom} Let $(\XXX, \mu)$ be a reflection space, $x,y \in \XXX$ and $\alpha \in \Aut(\XXX, \mu)$. Then
\[
\alpha \circ s_y \circ \alpha^{-1} = s_{\alpha(y)}.
\]
In particular,
\begin{equation} \label{InvolutionPicture}
s_x \circ s_y \circ s_x = s_{s_x(y)}.
\end{equation}
\end{lemma}

\begin{proof} 
For $z \in \XXX$ one has
\[
(\alpha \circ s_y \circ \alpha^{-1})(z) = \alpha(y \cdot \alpha^{-1} z) = \alpha(y) \cdot z = s_{\alpha(y)}(z),
\]
which proves the first statement. The second statement then follows from the first and the fact that point reflections are involutive automorphisms.
\end{proof}

\begin{remark}\label{AutEmbedding} The lemma implies that both $G(\XXX, \mu)$ and $\Trans(\XXX, \mu)$ are normal in $\Aut(\XXX, \mu)$. In particular, if one denotes by \[c_\alpha(g) := \alpha \circ g \circ \alpha^{-1}\] the conjugation by an element $\alpha \in \Aut(\XXX, \mu)$, then the assignment $\alpha \mapsto c_\alpha$ induces group homomorphisms
\[
c: \Aut(\XXX, \mu) \to \Aut(G(\XXX, \mu)) \quad \text{and} \quad \widehat{c}: \Aut(\XXX, \mu) \to \Aut(\Trans(\XXX, \mu)).
\]
Note that if $\alpha \in \ker(c)$ then for all $x \in \XXX$ one has
\[
s_{\alpha(x)} = c_\alpha(s_x) = s_x.
\] 
Thus if $\XXX$ is symmetric, or more generally $s_x \neq s_y$ for all $x \neq y$ in $\XXX$, then \[c: \Aut(\XXX, \mu) \to \Aut(G(\XXX, \mu))\] is injective.
\end{remark}

\subsection{Involution model and quadratic representation}

The following example provides an important construction of reflection spaces. In fact, by Lemma~\ref{GSIso} below, every symmetric space arises from this construction.

\begin{example} \label{involutionmodel}
Let $G$ be a group, let $S \subset G$ be a conjugation-invariant generating subset of involutions, and define a map
\[
\psi: S \times S \to S, \quad \psi(s,r) := s \cdot r = srs.
\]
Then $(S, \psi)$ is a reflection space, called the \Defn{reflection space associated with the pair $(G,S)$}.
Indeed, for all $x, y \in S$ one has $x \cdot x = xxx = x$ and $x \cdot( x\cdot y ) = xxyxx = y$ and, finally,
\[
x \cdot (y \cdot z) = xyzyx = xyxxzxxyx = xyx \cdot xzx = (x \cdot y) \cdot (x \cdot z),
\]
i.e., axioms \ref{RS1}, \ref{RS2}, \ref{RS3} hold.

The group $G$ acts by automorphisms on $(S, \psi)$ via conjugation and its center $Z(G)$ lies in the kernel of this action. Conversely, any $g \in G$ that acts trivially by conjugation on $S$ necessarily has to be central in $G$, because $S$ generates $G$.

One concludes that the main group of $(S, \psi)$, i.e., the group generated by the point reflections of $(S,\psi)$, is isomorphic to $G/Z(G)$. Furthermore, $(S, \psi)$ is symmetric if and only if $S$ does not contain any pair of distinct commuting involutions; and it is reflection-homogeneous if and only if $S$ consists of a single conjugacy class in $G$.
\end{example}

A version of the following lemma has been established in \cite{Caprace05} for primitive reflection spaces. Essentially the same proof applies to symmetric spaces.

\begin{lemma}\label{GSIso} Let $(\XXX, \mu)$ be a symmetric space, let $S := S(\XXX, \mu)$ be the set of its point reflections, and let $G := G(\XXX, \mu)$ be the main group generated by the point reflections. Then the following assertions hold:
\begin{enumerate}
\item $S \subset G$ is a conjugation-invariant subset of $G$.
\item If $(S, \psi)$ is the reflection space associated with the pair $(G,S)$, then 
\[
s: (\XXX, \mu) \to (S, \psi), \quad x \mapsto s_x
\]
is a $G$-equivariant isomorphism of reflection spaces.
\item $G$ has trivial center and $S$ does not contain any pair of distinct commuting involutions.
\item $\XXX$ is reflection-homogeneous if and only if $S$ consists of a single conjugacy class.
\end{enumerate}
\end{lemma}
\begin{proof} By \eqref{InvolutionPicture} on page \pageref{InvolutionPicture}, the set $S$ is invariant under conjugation by elements in $S$. Since $S$ generates $G$, it is therefore invariant under conjugation by elements in $G$. This shows (i) and makes it meaningful to consider the reflection space $(S, \psi)$ introduced in Example~\ref{involutionmodel}. Concerning (ii), the map $s$ is surjective by definition, and it is also injective, for, if  $s_x = s_y$, then by \ref{RS1} one has $s_x(y) = s_y(y) = y$, which by \ref{RS4} implies $x=y$. By \eqref{InvolutionPicture} the map $s$ is an $S$-equivariant and hence $G$-equivariant morphism, proving (ii).

 In particular, since $G$ is the main group of $(\XXX, \mu)$, it is also the main group of $(S, \psi)$. This, however, implies that $G$ has trivial center by the argument given in Example~\ref{involutionmodel}. Also, since $(S, \psi) \cong (\XXX, \mu)$ is symmetric, no two involutions in $S$ commute. This shows (iii). Assertion (iv) follows again from  $(S, \psi) \cong (\XXX, \mu)$.
\end{proof}

The reflection space $(S, \psi)$ defined in (ii) is referred to as the \Defn{involution model} of $(\XXX, \mu)$. By the lemma, every symmetric space admits an involution model.

\begin{remark}\label{TransvecRealization}
  Rather than realizing a reflection-homogeneous symmetric space $(\XXX, \mu)$ by a suitable generating conjugacy class of involutions of its main group, one can also realize it as a suitable subset of its transvection group.

  This embedding, which depends on a choice of basepoint $o \in \XXX$, is referred to as the \Defn{quadratic representation} of $\XXX$ in \cite[Section II.1]{Loos69I} (see also \cite[Lemma 2.3]{Caprace05}). Given $x \in \XXX$ one defines $t_x := s_x \circ s_o \in \Trans(\XXX, \mu)$ and sets $T(\XXX, \mu, o) := \{t_x \mid x \in \XXX\}$. Then the map 
\[
t: \XXX \to  T(\XXX, \mu, o), \quad x \mapsto t_x
\]
is a bijection; indeed, injectivity follows from $t_x \circ s_o = s_x$. This bijection induces on $T(\XXX, \mu, o)$ the structure of a symmetric space. Now by \eqref{InvolutionPicture} on page \pageref{InvolutionPicture} for all $x,y \in \XXX$ one has
\[
t_{x}t_y^{-1}t_x = s_x \circ s_o \circ (s_y \circ s_o)^{-1} s_x \circ s_o  =  s_x \circ s_y \circ s_x \circ s_o =  s_{s_x(y)} \circ s_o = t_{s_x(y)},
\]
whence the induced multiplication in this model is given by
\begin{equation}\label{stInverses}
T(\XXX, \mu, o) \times T(\XXX, \mu, o) \to T(\XXX, \mu, o), \quad (s, t) \mapsto s \cdot t = s t^{-1} s.
\end{equation}
Note that $T(\XXX, \mu, o)$ is a reflection subspace of the group $\Trans(\XXX, \mu)$, where the latter is equipped with its canonical reflection space structure as given by Example~\ref{firstexample}(i).

As another consequence of \eqref{InvolutionPicture} observe that for all $x, y\in \XXX$ one has
\[
s_x\circ s_y = s_x \circ s_o \circ  (s_o \circ s_y \circ s_o) \circ s_o = s_x \circ s_o \circ  s_{s_o(y)}  \circ s_o  = t_x \circ t_{s_o(y)}.
\]
In particular, $T(\XXX, \mu, o)$ actually generates the transvection group.
\end{remark}

\subsection{Topological reflection spaces} All the concepts introduced in the previous subsection make sense in a topological setting.

\begin{definition}
Let $\XXX$ be a topological space and let $\mu: \XXX \times \XXX \to \XXX$, $(x,y) \mapsto x \cdot y$ be a continuous map.
\begin{enumerate}
\item $(\XXX, \mu)$ is called a \Defn{topological reflection space} if it satisfies axioms \ref{RS1}--\ref{RS3}, and a \Defn{topological symmetric space} if it satisfies axioms \ref{RS1}--\ref{RS4}.
\item The \Defn{categories of topological reflection spaces} and \Defn{of topological symmetric spaces} are defined by requiring morphisms to be continuous in addition to preserving the product.
\item The \Defn{automorphism group} $\Aut(\XXX, \mu)$, the \Defn{main group} ${G}(\XXX, \mu)$ and the \Defn{transvection group} $\Trans(\XXX, \mu)$ are defined as in the abstract setting with the additional requirement that automorphisms be homeomorphisms.
\end{enumerate}
\end{definition}

The following topological variants of Examples~\ref{firstexample} and \ref{involutionmodel} provide examples for topological reflection spaces. 

\begin{example}\ 
\begin{itemize}
\item For any topological group $G$, the pair $(G,\mu_G)$ with $\mu_G(x,y):=xy^{-1}x$ is a topological reflection space.
\item The $n$-dimensional Euclidean space $\E^n =(\R^n,\mu_\E)$ is a topological symmetric space with its canonical vector space topology. Similarly, spheres and hyperbolic spaces are topological reflection spaces with their standard topologies.
\item Given a topological group $G$ and a conjugation-invariant generating subset $S$ of involutions, then $S$ is a topological reflection space with respect to the multiplication $r \cdot s = rsr$.
\end{itemize}
\end{example}
\begin{remark} We emphasize that Lemma~\ref{GSIso} does not have counterpart in the setting of general topological reflection spaces. More precisely, if $(\XXX, \mu)$ is a topological symmetric space, then the abstract reflection space underlying $(\XXX, \mu)$ can of course be realized as a subset of its main group (or inside its transvection group), but finding a group topology on either of these groups which restricts to the given topology on $(\XXX, \mu)$ is difficult without additional hypotheses on the structure of the topological symmetric space.
\end{remark}

\subsection{Flats in topological reflection spaces}
Throughout this section let $(\XXX, \mu)$ be a topological reflection space and let $x,y,z \in \XXX$. Since point reflections are involutions one has $s_x(y) = z$ if and only if $s_x(z) = y$. In this situation one calls $x$ a \Defn{midpoint} of $y$ and $z$.

In \cite{LawsonLim2007} the authors develop a rich structure theory of reflection spaces in which any pair of points has a unique midpoint, see \cite[Section~2, Axiom (S4)]{LawsonLim2007}. We will see in Corollary \ref{CorHoles} that every non-spherical Kac--Moody symmetric space contains pairs of points that do not admit a midpoint, hence it is important for us to develop the basic theory of reflection spaces without assuming the existence of midpoints. Note also that, in general topological reflection spaces, midpoints, if they exist, need not be unique, as is already clear from the example of spheres.
\begin{definition}
Let $(\XXX, \mu)$ be a topological reflection space and $\UUU \subset \XXX$ a subspace.
\begin{enumerate}
\item
$\UUU\subseteq \XXX$ is a \Defn{reflection subspace} if for $p,q\in \UUU$ also $s_p(q)\in \UUU$.
\item
$\UUU\subseteq \XXX$ is \Defn{midpoint convex} if for all $p,q\in \UUU$ there is a midpoint of $p$ and $q$ in $\UUU$.
\end{enumerate}
\end{definition}
Note that a reflection subspace of a topological reflection space $(\XXX, \mu)$ is itself a topological reflection space with respect to the restriction of $\mu$ and the subspace topology. Also note that the closure of a reflection subspace $\UUU$ is a reflection subspace\footnote{For, if $x,y$ are contained in the closure $\ol{\UUU}$, then there exist nets $(x_\alpha), (y_\alpha)$ in $\UUU$ converging to $x$ and $y$ respectively, whence $x \cdot y = \lim x_\alpha \cdot y_\alpha \in \ol{\UUU}$ by joint continuity of multiplication.}, whereas generally it is unclear to us whether the closure of a midpoint convex subset is midpoint convex, if $\XXX$ is not locally compact.\footnote{In case $\XXX$ actually is locally compact, one can argue as follows. Let $\UUU$ be a midpoint convex subset of $\XXX$ and let $\ol{\UUU}$ be its closure in $\XXX$. Then $\UUU$ contains nets $(x_\alpha)$ converging to $x$ and $(y_\alpha)$ converging to $y$. By local compactness the net $(z_\alpha)$ consisting of the midpoints $z_\alpha$ of $x_\alpha$ and $y_\alpha$ contains a subnet that in $\ol{\UUU}$ converges to some point $z$. By continuity, the reflection $s_z$ interchanges $x$ and $y$, i.e., $z \in \ol{\UUU}$ is a midpoint of $x$ and $y$.}

\begin{example}\label{ExampleFlat}
The $n$-dimensional Euclidean space $\E^n$ is midpoint convex. Moreover, $(\Z^n,\mu_\E)$ is a reflection subspace of $\E^n$ which is not midpoint convex, whereas $(\Q^n,\mu_\E)$ is a midpoint convex reflection subspace of $\E^n$, albeit not closed. The closed midpoint convex reflection subspaces of $\E^n$ are exactly the affine subspaces, i.e., the translates of $\R$-vector subspaces of the underlying $\R^n$.
\end{example}

\begin{definition}
Let $(\XXX, \mu)$ be a topological reflection space and $F \subseteq \XXX$ a reflection subspace.
\begin{enumerate}
\item $x,y\in \XXX$ \Defn{weakly commute} if for all $z\in \XXX$ one has
 \[ x \cdot (z \cdot (y\cdot z))=y\cdot (z\cdot (x\cdot z)).\]

\item $x,y\in \XXX$ \Defn{commute} if for all $a,b\in \XXX$ one has
\[ x \cdot (a \cdot (y\cdot b))=y\cdot (a\cdot (x\cdot b)).\]

\item $F$ is \Defn{(weakly) abelian} if all its points (weakly) commute.

\item $F$ is called a \Defn{(weak) flat} if it is closed, midpoint convex, (weakly) abelian, and contains at least two points.

\item $F$ is called a \Defn{Euclidean flat} of \Defn{rank} $n$ if it is closed and isomorphic to $\E^n$ as a topological reflection space.
\end{enumerate}
\end{definition} 

\begin{lemma}\label{LemmaBasicFlats} Let $(\XXX, \mu)$ be a reflection space.
\begin{enumerate}
\item Every Euclidean flat is a flat, and every flat is a weak flat.
\item Every $g \in \Aut(\XXX, \mu)$ preserves the collection of weak flats, and the subcollections of flats, Euclidean flats and Euclidean flats of a given rank $n$.
\item Every weak subflat of a Euclidean flat is Euclidean.
\end{enumerate}
\end{lemma}

\begin{proof} The first statement of (i) is contained in \cite[Proposition III.2.5]{Loos69I}, and the second statement of (i) is obvious, (ii) is immediate from the definitions, and (iii) follows from Example~\ref{ExampleFlat}.
\end{proof}

For an illustration that Euclidean flats are weakly abelian see  Figure \ref{fig:alg-flat}.
\begin{remark}
Theorem~\ref{maximalflatsareeuclidean} below states that in Kac--Moody symmetric spaces every weak flat is Euclidean, whence all three notions of flats coincide in that situation.

The notion of an abelian reflection subspace is taken from \cite[III.2.2, p.~134ff]{Loos69I}. Note that spheres and hyperbolic spaces are not weakly abelian, thus among constant curvature smooth examples, being weakly abelian is equivalent to flatness in the sense of zero curvature. In the smooth homogeneous context, being abelian is equivalent to the vanishing of the curvature tensor by \cite[Proposition III.2.5]{Loos69I}.
\end{remark}

Assume now that $\XXX$ is a topological reflection space and that $F \subset \XXX$ is a Euclidean flat of rank $n$. By definition this means that there exists a homeomorphism $\varphi: \R^n \to F$ which is an isomorphism of reflection spaces, where $\R^n$ carries the Euclidean reflection structure. Any such map will be referred to as a \Defn{chart} of $F$, and if $p := \varphi(0)$ then we say that the chart is \Defn{centred} at $p$. If $n=1$, then  a chart is also called a \Defn{parametrization}.

Now let $F \subset \XXX$ be a topological reflection space and $\varphi: \R^n \to F$ be a chart of $F$. Then every automorphism $\alpha \in \Aut(\XXX)$ that stabilizes the set $F$ induces a map $\widehat{\alpha} := \varphi \circ \alpha \circ \varphi^{-1}: \R^n \to \R^n$ and we observe:

\begin{proposition}\label{PropLocalLinearity}
  If 
  $\alpha \in \Aut(\XXX)$ preserves $F$, then $\widehat{\alpha} := \varphi \circ \alpha \circ \varphi^{-1}$ is an affine transformation, i.e., $\widehat{\alpha}$ is linear-by-translation.
\end{proposition}

\begin{proof}
  The map $\widehat{\alpha} : \R^n \to \R^n$ is a topological isomorphism of reflection spaces. In particular, for all $x, y \in \R^n$ one has \[\widehat{\alpha}(2x-y)=\widehat{\alpha}(\mu(x,y))=\mu(\widehat{\alpha}(x),\widehat{\alpha}(y))=2\widehat{\alpha}(x)-\widehat{\alpha}(y).\] The group of translations acts transitively on $\R^n$, so up to composition of $\widehat{\alpha}$ with an appropriate translation one may assume $\widehat{\alpha}(0) = 0$.
By setting $y=0$ one then concludes that $\widehat{\alpha}$ is homogeneous with respect to powers of $2$ and by setting $x=0$ one concludes that $\widehat{\alpha}$ is homogeneous with respect to $-1$. Replacing $x$ by $\frac{1}{2}x$ and $y$ by $-y$ then implies that $\widehat{\alpha}$ is additive. 
Since $\Z[\frac{1}{2}]$ is dense in $\R$, this implies $\R$-linearity of $\widehat{\alpha}$. 
\end{proof}

By abuse of language one says that $\alpha$ \Defn{acts affine-linearly} on $F$.

\subsection{Geodesics and translation groups}

In this section we prove Theorem~\ref{thm1.3}.

\begin{definition}  Let $(\XXX, \mu)$ be a topological reflection space. A Euclidean flat $\gamma \subset \XXX$ of rank $1$ is called a \Defn{geodesic}, and the subset
\[
T_\gamma := \{ s_p \circ s_q\mid p, q \in \gamma \} \subset \Trans(\XXX, \mu).
\]
is called the associated \Defn{translation group}.
\end{definition}

It is not obvious a priori that $T_\gamma$ is a group. However, one can show the following:

\begin{proposition}\label{PropTranslationGroups} Let $(\XXX, \mu)$ be a topological reflection space and $\gamma \subset \XXX$ a geodesic.
\begin{enumerate}
\item $T_\gamma \cong (\R, +)$ is a one-parameter subgroup of $\Trans(\XXX, \mu)$.
\item $T_\gamma$ acts sharply transitively on $\gamma$ by Euclidean translations.
\item If $t_1, t_2 \in T_\gamma$ and $t_1|_\gamma = t_2|_\gamma$, then $t_1 = t_2$.
\end{enumerate}
\end{proposition}

For the proof of Proposition~\ref{PropTranslationGroups} use the following notation: Fix a parametrization $\phi: \R \to \gamma$ of $\gamma$ so that in particular $\phi(2x-y) = s_{\phi(x)}(\phi(y))$, given $x \in \R$ abbreviate $s_x := s_{\phi(x)}$, and given $x,y \in \R$ define a transvection
\begin{equation}\label{Deftxy}
t_\gamma[x,y] := t[x, y] := s_{y} \circ s_{(x+y)/2}.
\end{equation}
By construction, $t[x,y]$ is a transvection along $\gamma$ which maps $\phi(x)$ to $\phi(y)$, hence the notation. Note that the restriction of this transvection to $\gamma$ corresponds via $\phi$ to the translation by $y-x$ in $\R$. With this notation one has $T_\gamma = \{t_\gamma[x,y]\mid x, y \in \R\}$ and, thus, Proposition~\ref{PropTranslationGroups} is a consequence of the following lemma:

\begin{lemma} With the notation just introduced the following hold.
\begin{enumerate}
\item For every $x \in \R$ the map $t_{\gamma, x}: (\R, +) \to T_\gamma$, $y \mapsto t[x, x+y]$ is an injective group homomorphism.
\item For all $x, y\in \R$ one has $t[x,x+y] = t[0, y]$. In particular, $t_{\gamma, x}$ is onto for every $x \in \R$.
\end{enumerate}
\end{lemma}

\begin{proof}
\begin{enumerate}
\item By Lemma~\ref{CapraceAxiom} and the formula for Euclidean reflections in $\R$ one has
\begin{equation}\label{LineReflection}
s_x \circ s_y \circ s_x = s_{2x-y}\quad (x,y \in \R).
\end{equation}
This allows one to rewrite \eqref{Deftxy} as
\begin{equation}\label{txyrewritten}
t[x,y] = s_{y} \circ s_{(x+y)/2} = s_{(x+y)/2}\circ (s_{(x+y)/2} \circ s_{y} \circ s_{(x+y)/2}) = s_{(x+y)/2} \circ s_{x},
\end{equation}
which yields in particular $t[x, x+y] = s_{x+y/2} \circ s_x$ and, thus,
\begin{equation}\label{txInverse}
t[x,x+y] \circ t[x, x-y] = s_{x+y/2} \circ (s_x \circ s_{x-y/2} \circ s_x) = s_{x+y/2} \circ s_{x+y/2} = \Id.
\end{equation}
It also follows from \eqref{LineReflection} and \eqref{txyrewritten} that
\[
t[ x, x+y/2]^2 =  (s_{x+y/4} \circ s_{x} \circ s_{x+y/4}) \circ s_{x} = s_{x+y/2} \circ s_x = t[x,x+y],
\]
whence by induction
\begin{equation}\label{tHomEven}
t[ x, x+y] =  t[x, x+2^{-n}y]^{2^n}.
\end{equation}
Next one observes that
\[
(s_{x+y/2} \circ s_{x})(\phi(x+y/2)) = s_{x+y/2}(\phi(x-y/2))=\varphi(x+3y/2),
\]
and inductively one obtains for all $k \in \N$,
\[
(s_{x+y/2} \circ s_{x})^k(\phi(x+y/2))=\varphi(x+(2k+1)y/2)
\]
Thus, if $n=2k+1$ is an odd positive integer, then
\begin{alignat*}{3}
t[x, x+y]^n &= (s_{x+y/2} \circ s_{x})^{2k+1} &
&= (s_{x+y/2} \circ s_{x})^k \circ s_{x+y/2} \circ (s_x \circ s_{x+y/2})^k \circ s_{x}\\
&= s_{x+(2k+1)y/2} \circ s_x &
&= t[x, x+ny].
\end{alignat*}
Combining this with \eqref{txInverse} and \eqref{tHomEven} it follows that the restriction of $t_{\gamma, x}$ to the dense subset $\Z[\tfrac12]$ of $\R$ is a homomorphism, whence continuity of $\mu$ implies that $t_{\gamma, x}$ is a homomorphism. Since $t[x,x+y](x) = \phi(x+y)$ it is injective.
\item Let $x, y \in \R$ and set $z := x/2+y/4$ so that
\[
s_{z}(\phi(0)) = \phi(x+y/2), \quad s_z(\phi(x)) = \phi(y/2).
\]
It follows from (i) that $s_0 \circ s_x$ and $s_z \circ s_x$ commute and hence
\[
s_0 \circ s_x = (s_z \circ s_x) \circ (s_0 \circ s_x) \circ  (s_z \circ s_x)^{-1} = s_z \circ s_x \circ s_0 \circ s_z = s_{y/2} \circ s_{x+y/2}.
\]
One deduces that $s_{y/2} \circ s_0 = s_{x+y/2} \circ s_x$, and thus taking inverses
\[
t[0,y] =  s_{y/2} \circ s_0 =  s_{x+y/2} \circ s_x = t[x, x+y].\qedhere
\]
\end{enumerate}
\end{proof}

\begin{remark}
Proposition~\ref{PropTranslationGroups} generalizes \cite[Proposition~XIII.5.5]{Lang1999} as well as \cite[Theorem~3.6(iv)]{Neeb2002} to arbitrary topological reflection spaces: any geodesic in any topological reflection space defines a one-parameter subgroup of its automorphism group.
It is quite remarkable that this property relies purely on group theory and elementary Euclidean geometry and does not require any differentiable structure whatsoever.

After the dissemination of our results in late 2016 and early 2017 variants of our Proposition~\ref{PropTranslationGroups} and the underlying lemma were published in \cite{Neeb2017} and \cite{Oeh2017}. 
\end{remark}

\subsection{Geodesically connected reflection spaces}

\begin{definition} Let $(\XXX, \mu)$ be a topological reflection space and $\gamma \subset \XXX$ a geodesic. A compact connected subset $\sigma \subset \gamma$ with non-empty relative interior is called a \Defn{geodesic segment}. A triple $\vec \sigma = (\sigma, s(\vec \sigma), t(\vec \sigma))$, where $\sigma$ is a geodesic segment and $s(\vec \sigma)$ and $t(\vec \sigma)$ are the endpoints of $\sigma$ is called an \Defn{oriented geodesic segment} from $s(\vec \sigma)$ to $t(\vec \sigma)$. Given an oriented geodesic segment $\vec \sigma$ in $\gamma$ the \Defn{parallel transport along $\vec \sigma$} is defined as the unique transvection $t[\vec \sigma] \in T_\gamma$ mapping $s(\vec \sigma)$ to $t(\vec \sigma)$.

  An \Defn{oriented piecewise geodesic curve} is a sequence $\vec \sigma =  (\vec \sigma_1, \vec \sigma_2, \dots, \vec \sigma_n)$ of oriented geodesic segments with $t(\sigma_i) = s(\sigma_{i+1})$. Then set $s(\vec \sigma) := s(\vec \sigma_1)$ and $t(\vec \sigma) := t(\vec \sigma_n)$ and say that $\vec \sigma$ is a curve from $s(\vec \sigma)$ to $t(\vec \sigma)$. Also define \Defn{parallel transport along $\vec \sigma$} as the transvection
\[t[\vec\sigma] :=  t[\vec \sigma_n] \circ \cdots \circ t[\vec \sigma_2] \circ t[\vec \sigma_1].\]
$(\XXX, \mu)$ is \Defn{geodesically connected} if for all $p, q \in \XXX$ there exists an oriented piecewise geodesic curve from $p$ to $q$.
\end{definition}

Recall that in a finite-dimensional Riemannian symmetric space any pair of points lies on a common geodesic. This is no longer the case for Kac--Moody symmetric spaces by Corollary~\ref{CorHoles} below. Nevertheless, Kac--Moody symmetric spaces still satisfy the weaker property of being geodesically connected by Lemma~\ref{lem:geod-conn}; as it turns out this is enough to deduce various basic structural features such as the following information concerning the transvection group:

\begin{proposition} \label{PropTransGeneratedByOneParam}
Let $(\XXX, \mu)$ be a geodesically connected topological reflection space.
\begin{enumerate}
\item $\Trans(\XXX, \mu)$ acts transitively on $\XXX$. In particular, $\XXX$ is reflection-homogeneous.
\item $\Trans(\XXX, \mu)$ is generated by the one-parameter subgroups $T_\gamma$, where $\gamma$ runs through all geodesics in $\XXX$.
\end{enumerate}
\end{proposition}

\begin{proof}
\begin{enumerate}
\item If $p,q$ are distinct points in $\XXX$ and $\vec \sigma$ is an oriented piecewise geodesic curve from $p$ to $q$, then $t[\vec \sigma] \in \Trans(\XXX, \mu)$ maps $p$ to $q$.
\item Let $p$ and $q$ be distinct points in $\XXX$ and let $\vec \sigma =  (\vec \sigma_1, \vec \sigma_2, \dots, \vec \sigma_n)$ be a piecewise oriented geodesic curve between $p$ and $q$. It suffices to show that $s_q \circ s_p \in \Trans(\XXX, \mu)$ can be written as a product of elements of the translation groups corresponding to the geodesics involved in the above curve. To this end set $p_i = t(\vec \sigma_i)$ and $p_0 := p$, $q_i := s_{p_i}(p_{i-1})$. Then $t_i := s_{q_{i}}\circ s_{p_i} \in T_{\gamma_i}$ where $\gamma_i$ is the geodesic containing $\vec{\sigma}_i$ and $t_i(p_{i-1}) = q_i$. Thus (ii) follows from the computation
\begin{align*}
s_{q} \circ s_p
&= s_{p_n} \circ s_{p_0} =  (s_{p_n} \circ s_{p_{n-1}}) \circ (s_{p_{n-1}} \circ s_{p_{n-2}}) \dots (s_{p_1} \circ s_{p_0})\\
&= ((s_{p_n} \circ s_{p_{n-1}} \circ  s_{p_n}) \circ s_{p_n}) \circ \dots \circ ((s_{p_1} \circ s_{p_{0}} \circ  s_{p_1}) \circ s_{p_1})\\
&= (s_{q_n} \circ s_{p_n}) \circ \dots \circ (s_{q_1} \circ s_{p_1})\\
&= t_n \circ \dots \circ t_1. \qedhere
\end{align*}
\end{enumerate}
 \end{proof}
 \begin{remark} Part (ii) of the proposition provides an obstruction for a group to occur as the transvection group of some geodesically connected topological reflection space: Any such group has to be generated by a family of subgroups isomorphic to $(\R, +)$.
 \end{remark}

 \subsection{Local transformations of strongly transitive reflection spaces}

In a general topological reflection space, it is unclear to us whether every flat is contained in a maximal flat. Indeed, while every midpoint convex abelian reflection subspace certainly is contained in a maximal midpoint convex abelian reflection subspace, there is no reason for this maximal space to be closed. As it is unclear to us whether closures of midpoint convex subsets are again midpoint convex, we are unable to guarantee even the existence of a single maximal flat in this generality. However, if maximal flats exist, then they often give a major insight into the structure of the topological reflection space, since every automorphism has to preserve maximal flats and their intersection patterns. 
 
 \begin{definition} Let $\XXX$ be a topological reflection space which admits maximal flats.
 \begin{enumerate}
 \item A pair $(p, F)$ where $F$ is a maximal flat in $\XXX$ and $p \in F$ is a point is called a \Defn{pointed maximal flat}.
 \item Let $G$ be a group acting on $\XXX$ by automorphisms. We say that the action is \Defn{strongly transitive} if $G$ acts transitively on pointed maximal flats.
 \item $\XXX$ is called \Defn{strongly transitive} if $\Aut(\XXX)$ acts strongly transitively on $\XXX$.
\end{enumerate}
 \end{definition}

 The following observation is often useful for checking strong transitivity. It will be used, for instance, in Corollary~\ref{ConsequencesFlatClassification} below in order to show that Kac--Moody symmetric spaces are strongly transitive.

 \begin{proposition}\label{PropAutomaticallyStrong} 
Let $\Trans(\XXX) < G <\Aut(\XXX)$. If $G$ acts transitively on maximal flats in $\XXX$ and if one, whence all, of these are Euclidean, then $G$ acts strongly transitively on $\XXX$.
 \end{proposition}

 \begin{proof} Any maximal flat $F \subset \XXX$ by definition is a reflection subspace and, hence, each point reflection of $F$ is induced by a point reflection of $\XXX$. Since $G$ acts transitively on the set of maximal flats of $\XXX$, the flat $F$ is Euclidean, i.e., $F \cong \E^n$ for some $n$. It follows that the stabilizers of $F$ in $\Trans(\XXX)$ and, thus, in $G$ contain $\Trans(F)$, and hence act transitively on $F$. This implies the proposition.
 \end{proof}

 \begin{remark}
If in the situation of the preceding proposition the maximal Euclidean flats of $\XXX$ have rank $k$, then what we call strong transitivity in the present article coincides with the notion of {\em $k$-flat homogeneity} in the literature.
 \end{remark}

Let $\XXX$ be a topological reflection space which contains a maximal flat, which moreover is Euclidean and let $G$ be a group with
$\Trans(\XXX) < G <\Aut(\XXX)$. Moreover, assume that $G$ acts transitively on maximal flats of $\XXX$, and let $(p,F)$ be a pointed maximal flat in $\XXX$. By Proposition~\ref{PropAutomaticallyStrong}, $G$ acts strongly transitively on $\XXX$ and $F$ is Euclidean. Denote by \[
\Stab_{G}(p, F) := \{g \in G\mid g.F = F, g.p = p\} \quad \text{and} \quad \Fix_{G}(p, F) :=\{g \in G \mid \forall f \in F: g.f = f\}
\]
the stabilizer, respectively fixator of $(p,F)$ in $G$.

\begin{definition} \label{defsingularregular}
Let  $\XXX$ be a topological reflection space which contains a maximal flat $F$, which moreover is Euclidean, and let $p \in F$. 

A point $q \in F$ is called \Defn{singular with respect to $p$} if there exists a second maximal flat distinct from $F$ containing both $p$ and $q$, and \Defn{regular with respect to $p$} otherwise. Denote by $F^{\rm reg}(p) \subset F$ the subset of regular points in $F$ with respect to $p$, and by $F^{\rm sing}(p) \subset F$ the subset of singular points in $F$ with respect to $p$.

A map $f: F \to F$ is called \Defn{linear at p} if for some (hence any) chart $\varphi: \R^n \to F$ which is centred at $p$ we have $\varphi \circ f \circ \varphi^{-1} \in \GL_n(\R)$.
It is called a \Defn{local transformation} of the pointed flat $(p,F)$ if it is linear at $p$ and preserves the decomposition $F=F^{\rm reg}(p) \sqcup F^{\rm sing}(p)$. Denote by $\GL(p,F, F^{\rm sing}(p))$ the group of local transformations of $(p,F)$.
\end{definition}

\begin{proposition} \label{2.29} Let $\XXX$ be a topological reflection space which contains a maximal flat, which moreover is Euclidean, and assume that $G$ acts transitively on maximal flats of $\XXX$.
\begin{enumerate}
\item The group $W(G\curvearrowright\XXX) :=  \Stab_{G}(p, F)/ \Fix_{G}(p, F)$ is independent of the choice of pointed flat $(p,F)$ up to conjugation.
\item There is a homomorphism $\rho_F: W(G\curvearrowright\XXX) \to \GL(p, F, F^{\rm sing}(p))$, $\rho_F([f]) := f|_F$, which is independent of the choice of pointed flat $(p,F)$ up to conjugation.
\end{enumerate}
\end{proposition}

\begin{proof} By Proposition~\ref{PropAutomaticallyStrong} the group $G$ acts strongly transitively on $\XXX$. Assertion (i) and the second statement of assertion (ii) are immediate from strong transitivity. The first statement of assertion (ii) follows from Proposition~\ref{PropLocalLinearity} and Lemma~\ref{LemmaBasicFlats}.
\end{proof}

\begin{definition}\ \label{DefLocalAction}
\begin{enumerate}
\item The group $W(G \curvearrowright \XXX)$ is called the \Defn{(geometric) Weyl group} of the action $G \curvearrowright \XXX$. 
\item The
homomorphism $\rho_F: W(G\curvearrowright\XXX) \to \GL(p, F, F^{\rm sing}(p))$ is the \Defn{local action} of $G$ on $\XXX$. 
\end{enumerate}
\end{definition}
 
 \section{Split real Kac--Moody groups and their Iwasawa decompositions}

\subsection{Groups with RGD systems} \label{RGDBN}

This subsection provides some necessary background concerning groups with RGD systems (see \cite[Chapter~8]{AbramenkoBrown2008}); for the definitions of  a prenilpotent pair of roots as well as the definitions of the ``closed'' interval $[\alpha,\beta]$ and the ``open'' interval $]\alpha,\beta[$ of roots $\alpha, \beta$ used therein see \cite[Sections~8.5.2, 8.5.3]{AbramenkoBrown2008}.
\begin{definition} \label{defrgd}
Let $(W,S)$ be a Coxeter system with root system $\Phi$ and let $\Phi^+$ be a subset of positive roots. An \Defn{RGD system} is a triple $(G, \{U_\alpha\}_{\alpha \in \Phi}, T)$, where $G$ is a group, $T < G$ a subgroup and $\{U_\alpha\}_{\alpha \in \Phi}$ is a family of subgroups of $G$ subject to the following axioms:
\begin{axiomsZero}{RGD}
\item For each root $\alpha \in \Phi$, one has $U_\alpha \neq \{1\}$.
\item For each \index{prenilpotent pair of roots}prenilpotent pair 
$\{\alpha, \beta\} \subseteq \Phi$ of distinct roots, one has 
$[U_\alpha, U_\beta] \subseteq \langle U_\gamma \mid \gamma \in ]\alpha,
\beta[ \rangle$. 
\item\label{RGD3}
For each $s \in S$ there exists a function $\mu_s : U_{\alpha_s} \backslash
\{ 1 \} \to G$ such that for all $u \in U_{\alpha_s} \backslash \{1\}$ and
$\alpha \in \Phi$ one has $\mu_s(u) \in U_{-\alpha_s}uU_{-\alpha_s}$ and
$\mu_s(u)U_\alpha\mu_s(u)^{-1} = U_{s(\alpha)}$. Moreover, $\mu_s(y)^{-1}\mu_s(v) \in T$.
\item\label{RGD4} For each $s \in S$ one has $U_{-\alpha_s} \nsubseteq U_+ := 
\langle U_\alpha \mid \alpha \in \Phi^+ \rangle$.
\item $G =  T\langle U_\alpha \mid \alpha \in \Phi \rangle$.
\item \label{RGD6} The group $T$ normalizes every $U_\alpha$.
\end{axiomsZero}
The groups $U_\alpha$ are called \Defn{root subgroups}, the group $T$ is called \Defn{maximal (split) torus}, as are its conjugates.
Following \cite{Caprace:2007}, an RGD-system is \Defn{centred} if $G$ is generated by its root subgroups.
 \end{definition}
 Every centred RGD system $(G, \{U_\alpha\}_{\alpha \in \Phi}, T)$ gives rise to a \Defn{saturated twin BN pair} $(B_+, B_-, N)$ in the sense of Tits as follows (cf.~\cite[Theorem~8.80]{AbramenkoBrown2008}). If $\mu_s: U_{\alpha_s} \backslash \{ 1 \} \to U_{-\alpha_s}U_{\alpha_s} U_{-\alpha_s}$ is the map provided by \ref{RGD3}, the group $U_+$ is as in \ref{RGD4} and $U_- :=  \langle U_\alpha \mid \alpha \in -\Phi^+ \rangle$, then $T$ normalizes both $U_+$ and $U_-$ and one obtains a twin $BN$-pair $(B_+, B_-, N)$ by
\begin{align*}
	N   & := T . \langle \mu_s(u) \mid u \in U_{\alpha_s} \backslash \{1\},
 s \in S \rangle, \\
	B_+ & := T \ltimes U_+,\\
	B_- & := T \ltimes U_-.
\end{align*}
This twin $BN$-pair satisfies the saturation property $B_+ \cap B_- = T$ (cf.\ \cite[Corollary~8.78]{AbramenkoBrown2008}) and $T = \bigcap_{\alpha \in \Phi}N_G(U_\alpha)$ (cf.\ \cite[Corollary~8.79]{AbramenkoBrown2008});
note that $N = N_G(T)$ by \cite[Theorem~6.87(2) and Theorem~8.80]{AbramenkoBrown2008}.

The twin $BN$-pair $(B_+, B_-, N)$ then gives rise to two buildings with respective chamber sets $\Delta_\pm := G/B_\pm$ and a twinning between them (\cite[Section~8.9]{AbramenkoBrown2008}), which leads to a twin building.

The theory of twin buildings is an invaluable tool for studying groups with an RGD-system. Refer to \cite[Section~6.3 and Chapter~8]{AbramenkoBrown2008} for general background information on twin buildings endowed with a group action and to \cite{HartnickKoehlMars} for a setup of twin buildings that has been specifically tailored to suit the properties of topological Kac--Moody groups.

\subsection{Complex and split real topological Kac--Moody groups}\label{Subsection3B}

\begin{definition} \label{GCMdef}
  A  \Defn{generalized Cartan matrix} is an integral square matrix ${\mathbf{A}} = (a_{ij})_{1 \leq i, j \leq n} \in M_n(\Z)$ satisfying
$a_{ii} = 2$, $a_{ij} \leq 0$ for $i \neq j$, and $a_{ij} = 0$ if and only
  if $a_{ji} = 0$. (Cf.~\cite[\S1.1]{Kac90}.)

  The \Defn{Dynkin diagram} $\Gamma_{\mathbf{A}}$ of ${\mathbf{A}}$ is the edge-labelled graph with vertex set $\VVV = \{1, \dots, n\}$ and edge set $\EEE := \{\{i,j\} \subset \VVV \mid i\neq j, \; a_{ij}a_{ji} \neq 0\}$. If $e \in \EEE$ joins the vertices $i$ and $j$ and $a_{ij}>a_{ji}$, then $e$ is labelled by the number $a_{ij}a_{ji}$, by an arrow from $i$ to $j$ and, if $a_{ij}a_{ji}$ is not prime, by the values $a_{ij}$ and $a_{ji}$. The matrix ${\mathbf{A}}$ and the diagram $\Gamma_\mathbf{A}$ are called \Defn{irreducible} if $\Gamma_{\mathbf{A}}$ is connected, \Defn{two-spherical} if $\Gamma_{\mathbf{A}}$ has no labels $a_{ij}a_{ji}>3$, \Defn{spherical} if $\mathbf{A}$ is the Cartan matrix of a finite-dimensional Lie group, and \Defn{non-spherical} otherwise.
The \Defn{Coxeter diagram} is induced by the Dynkin diagram $\Gamma_{\mathbf{A}}$ by removing all arrows and all values $a_{ij}$ and $a_{ji}$ and replacing labels equal to one by three, labels equal to two by four, labels equal to three by six, and labels greater than three by $\infty$ (see also Section~\ref{appendixweylgroup}).
  
 The generalized Cartan matrix $\mathbf{A}$ is called \Defn{symmetrizable} if there exist a symmetric matrix $B = (b_{ij}) \in M_{n}(\R)$ and diagonal matrix $D = {\rm diag}(\epsilon_1, \dots, \epsilon_n) \in M_n(\R)$ with $\epsilon_j > 0$ such that ${\mathbf{A}} =DB$.
The matrix $D$ is not unique, but one can choose $D$ to be \Defn{minimal} in the sense of \cite[Definition~1.5.1]{Kumar02}: Each $\epsilon_i$ is a positive integer, and if ${\rm diag}(\epsilon'_1, \dots, \epsilon_n')$ is another such matrix, then $\epsilon_i \leq \epsilon_i'$ for all $i$.
\end{definition}

The key results of this article concerning Kac--Moody symmetric spaces hold in the presence of the following general hypotheses. 

\begin{convention}\label{ConventionCartanMatrix}
  In this article ${\mathbf{A}} \in M_n(\Z)$ denotes an
  irreducible symmetrizable generalized Cartan matrix.
\end{convention}

A generalized Cartan matrix ${\mathbf{A}}$ is the key ingredient for defining a topological split Kac--Moody group over $\K \in \{\R, \C\}$. Assume first that ${\mathbf{A}}$ is two-spherical. Under this additional assumption there is a very efficient way of defining these groups as colimits of diagrams of groups as described in \cite{AbramenkoMuehlherr}:
 For each vertex $i \in \VVV$ of the Dynkin diagram $\Gamma_{\mathbf{A}}$ define $G_i(\K) := \SL_2(\K)$. For every pair $\{i,j\} \subset \VVV$ ($i\neq j$) define $G_{\{i,j\}}(\K)$ as the split Lie group over $\K$ of rank two whose Dynkin diagram is the full labelled subgraph of $\Gamma_{\mathbf{A}}$ on vertices $i,j$. A fixed choice of a root basis provides natural inclusion maps $\iota_i: G_i(\K) \into G_{\{i,j\}}(\K)$.

Consider the amalgam $\AAA_{{\mathbf{A}}}$ of topological groups formed by the Lie groups $G_i(\K)$, $i \in \VVV$, and $G_{\{i,j\}}(\K)$, $i \neq j$, together with the canonical inclusions. The colimit of this amalgam in the category of topological groups turns out to be a Hausdorff topological group $G_\K({\mathbf{A}})$, which is moreover a $k_\omega$ space in the sense of Definition \ref{DefKOmega} below (see \cite[Theorem~7.22]{HartnickKoehlMars}). This colimit is abstractly isomorphic to the quotient of the free group generated by the elements of the groups $G_i(\K)$ modulo the relations given as products of conjugates of the relations contained in $G_{\{i,j\}}(\K)$; its topology equals the finest group topology such that the natural embeddings of the Lie groups $G_i(\K)$ are continuous.

\begin{definition}\label{DefTopSplitKM}
  The group $G_\R({\mathbf{A}})$ (respectively $G_\C({\mathbf{A}})$) is called the \Defn{simply connected centred split real (resp.\ complex) Kac--Moody group} of type ${\mathbf{A}}$. The topology on $G_\K({\mathbf{A}})$ defined above is called the \Defn{Kac--Peterson topology}.
  
Given a subset $I \subset \VVV$ the subgroup $G_I(\K) := \langle G_i(\K) \mid i \in I \rangle$ is called a \Defn{standard rank $|I|$ subgroup} of $G_\K(\bf A)$. Denote by $\phi_I: G_I(\K) \to G_\K(\mathbf A)$ the canonical inclusion; if $|I|=1$ one simply writes $\phi_i$ and $G_i$ instead of $\phi_{\{i\}}$ and $G_{\{i\}}$ respectively.
 \end{definition}

The embedding $\R \into \C$ induces embeddings $G_i(\R) \into G_i(\C)$ and $G_{\{i,j\}}(\R) \into G_{\{i,j\}}(\C)$ and hence an embedding $G_\R({\mathbf{A}}) \into G_\C({\mathbf{A}})$. Since our main focus lies on the real case, we will subsequently write $G := G_\R(\bf A)$, $G_i:=G_i(\R)$, etc.

The topological Kac--Moody groups $G_\R({\mathbf{A}})$ and $G_\C({\mathbf{A}})$ and all of the notions pertaining to these groups as defined in this subsection can also be defined without the assumption that ${\mathbf A}$ be two-spherical, and the results in this article are valid without the assumption of two-sphericity unless explicitly stated otherwise. However, in this more general setting the amalgamation results from \cite{AbramenkoMuehlherr} are not available, and thus the definitions become substantially more technical. We refer the reader to \cite{Tits87},  \cite{Remy02}, \cite{Caprace09}, \cite[Chapter~7]{HartnickKoehlMars} for the general definitions.

\subsection{The adjoint quotient and the semisimple adjoint quotient} \label{sectionadjoint}

The group $G = G_\R(\bf A)$ can be considered as an infinite-dimensional generalization of a finite-dimensional semisimple split real Lie group. In fact, if ${\mathbf{A}}$ is a spherical irreducible (generalized) Cartan matrix, then the resulting Kac--Moody group $G$ is an algebraically simply connected simple split real Lie group. In particular, the center of $G$ is $0$-dimensional. In this case $\mathbf{A}$ is automatically symmetrizable and, in fact, invertible.

A non-spherical irreducible symmetrizable generalized Cartan matrix $\mathbf{A}$ on the other hand need not be invertible, as for instance is the case for any generalized Cartan matrix of affine type. In this situation the group $G$ admits a positive-dimensional center $Z(G)$, which leads to some complications in our study of Kac--Moody symmetric spaces.
One way to resolve this issue is to consider instead of $G$ its \Defn{adjoint quotient} \[\Ad(G) := G/Z(G).\] This group, however, has the slight disadvantage that its maximal torus is not isomorphic to a direct product of several copies of the multiplicative group $(\R^\times,\cdot)$, i.e., it is not an algebraically simply connected split torus. We thus introduce an intermediate object, that we call the \Defn{semisimple adjoint quotient} $\ol{G}$ of $G$. By Proposition~\ref{kernelsemisimpleadjoint} below $\ol{G}$ is the unique group which admits surjections with central kernel
\[
G \to \ol{G} \to \Ad(G)
\]
such that the kernel of the former epimorphism is a product of copies of the multiplicative group $(\R^\times,\cdot)$ and the kernel of the latter epimorphism is finite.

\medskip
The construction of $\ol{G}$ relies on some key properties of the \Defn{adjoint representation} of $G$ and the \Defn{exponential function} of $G$. Both relate the complex Kac--Moody group $G_\C(\mathbf{A})$ to the (derived) complex Kac--Moody algebra $\mathfrak{g}$ associated with $\mathbf{A}$, whose basic structure theory is discussed in Section~\ref{AppendixAlgebrasAndWeylAB} in the appendix.

Symmetrizability of the generalized Cartan matrix as required in Convention~\ref{ConventionCartanMatrix} allows one to apply the Gabber--Kac Theorem \ref{GabberKac} which implies that the Lie algebra $\mathfrak g$ is the direct limit of its standard subalgebras of ranks one and two. These are the Lie algebras of the standard rank one and two subgroups of $G_\C({\mathbf{A}})$ and hence the latter groups act on them by the respective adjoint actions. It turns out that the adjoint actions of these subgroups combine into an adjoint representation
\[
\Ad_\C: G_\C({\mathbf{A}}) \to \GL(\mathfrak g),
\]
see \cite[Proposition~6.2.11]{Kumar02}. This restricts to a representation \[\Ad: G \to \GL(\mathfrak g),\] whose image is isomorphic to $\Ad(G) = G/Z(G)$.

As discussed in Section~\ref{AppendixAlgebrasAndWeylAB} the Lie algebra $\mathfrak g$ contains a canonical subalgebra \[\mathfrak h = \sum_{i=1}^n \C \check \alpha_i\] (see formula \eqref{hdef}), which intersects each of the standard rank one Lie algebras $\mathfrak g_i \cong \mathfrak{sl}(2,\C)$ of $\mathfrak g$ in the standard diagonal Cartan subalgebra $\mathfrak h_i := \C \check \alpha_i$ (see Theorem~\ref{GabberKac}). For each $i \in \{1, \dots, n\}$ there exists a natural exponential function \[\exp_i: \mathfrak h_i \hookrightarrow \mathfrak{g}_i \cong \mathfrak{sl}(2,\C) \to \SL(2,\C) \cong G_i(\C),\] whose image is denoted by $H_i$. The groups $H_i < G_\C({\bf A})$ generate the direct product \[ H_\C := \prod_{i=1}^n H_i < G_\C({\mathbf{A}})\] and one obtains a natural exponential function
\begin{eqnarray*}
\exp_\C:  \mathfrak h = \bigoplus_{i=1}^n \mathfrak h_i & \to & \prod H_i = H_\C \\ (X_1, \dots, X_n) & \mapsto & \prod_{i=1}^n \exp_i(X_i).
\end{eqnarray*}
Under the standard identifications $\mathfrak h  \cong \C^n$ and $H_\C \cong (\C^\times)^n$ this map corresponds to the usual exponential map.
Recall from \eqref{KMAcenter} on page \pageref{KMAcenter} and \eqref{dimofh} on page \pageref{dimofh} that the center $\mathfrak c = \mathfrak{z}(\mathfrak g)$ is contained in $\mathfrak h$ and has complex dimension
\begin{equation}
\dim_\C \mathfrak c = n-\rk({\mathbf{A}}). \label{mistake}
\end{equation}

\begin{definition} \label{semisimpleadjoint}
Set $C_\C := \exp_\C(\mathfrak c)$ and $C := C_\C \cap G$ and define the  \Defn{semisimple adjoint quotient} of $G$ by
\[
\ol{G} :=  G/C.
\]
The \Defn{standard maximal (split) torus} of $G = G_\R(\mathbf{A})$ is defined as $T := H_\C \cap G \cong (\R^\times)^n$. Its image $\ol{T}$ in $\ol{G}$ is called the \Defn{standard maximal (split) torus} of $\ol{G}$.
\end{definition}

Let now $\mathfrak a$  be the real form of $\mathfrak h$ defined in Notation~\ref{defa} on page \pageref{defa}. It is an immediate consequence of the definitions that $\exp_\C$ restricts to an injective map
\[
\exp: \mathfrak a \to T,
\]
whose image is denoted by $A:= \exp(\mathfrak a) \cong (\R_{>0})^n$. Moreover, the image of $A$ in $\ol{G}$ is denoted by $\ol{A}$.
The map $\exp: \mathfrak a \to A$ is a bijection which maps $\mathfrak c \cap \mathfrak a$ to $C \cap A$. 
Denoting $\ol{\mathfrak a}:= \mathfrak{a} / (\mathfrak c \cap \mathfrak a)$ as in Notation~\ref{defofa}, this
 induces a bijection \[\exp: \ol{\mathfrak a} \to \ol{A}.\] The inverse maps are denoted by $\log: A \to \mathfrak a$, respectively $\log: \ol{A} \to \ol{\mathfrak a}$. Note that, as vector spaces,
 \[
 \mathfrak a \cong \R^n \quad \text{and} \quad \ol{\mathfrak a} \cong \R^{\rk({\mathbf{A}})}.
 \]

\begin{remark} \label{erratum}
  Before continuing we point out an error in \cite{HartnickKoehlMars}. The statement of \cite[Lemma~7.5]{HartnickKoehlMars} is inaccurate, as becomes obvious from \eqref{mistake} above. The problem is that its proof only applies to $\ol{G}$ (and its analogs over other fields) but not to $G$ (or its analogs over other fields).

  As a consequence, also \cite[Proposition~7.18]{HartnickKoehlMars} has only been established for center-free Kac--Moody groups over local fields and central quotients of $\ol{G}$ (and its analogs over other local fields) instead of central quotients of $G$ (or its analogs over other local fields). That is, the results from \cite{HartnickKoehlMars} only enable us to control the topology on $H_\C/C_\C$ instead of the topology on $H_\C$. 

However, a variation of the embedding argument as used in \cite[Proposition~7.10]{HartnickKoehlMars} in fact allows one to also control the topology on $H_\C$ as follows.  
\end{remark}

\begin{proposition}\label{PropexpHomeomorphismC}
The exponential map $\exp_\C: \mathfrak h \to H_\C$ is a quotient map, where $\mathfrak h$ is equipped with its topological vector space topology and $H_\C$ with the Kac--Peterson topology.
\end{proposition}

\begin{proof}
  It suffices to prove that the Kac--Peterson topology induces the standard topology on $H_\C \cong (\C^\times)^n$. Let $\mathbf{B}$ be an invertible generalized Cartan matrix that contains $\mathbf{A}$ as a principal submatrix. Then $G_\C(\mathbf{A})$ admits a natural topological embedding into $G_\C(\mathbf{B})$ as a closed subgroup with respect to the Kac--Peterson topology and the subgroup $H_\C$ of $G_\C(\mathbf{A})$ embeds topologically as a closed subgroup into the corresponding subgroup $H^\mathbf{B}_\C$ of $G_\C(\mathbf{B})$. Since $\mathbf{B}$ is invertible, the associated Kac--Moody algebra and group have zero-dimensional center (see formula \eqref{mistake}) and so, in fact, \cite[Proposition~7.18]{HartnickKoehlMars} applies to $G_\C(\mathbf{B})$; that is, $H^\mathbf{B}_\C$ endowed with the Kac--Peterson topology is homeomorphic to $(\C^\times)^{\rk(\mathbf{B})}$ endowed with its standard topology.
Consequently, the closed subgroup $H_\C$ is homeomorphic to $(\C^\times)^{n}$ endowed with its standard topology.
\end{proof}

 The group $A$ carries a natural group topology induced by the Kac--Peterson topology on $H_\C$, which by Proposition~\ref{PropexpHomeomorphismC} makes $A$ homeomorphic to $(\R_{>0})^n$ with its standard topology.
Moreover, one obtains the following immediate consequences:
 
\begin{proposition}\ \label{PropexpHomeomorphism}
\begin{enumerate}
\item The exponential maps $\exp: \mathfrak a \to A$ and $\exp: \ol{\mathfrak a} \to \ol{A}$ are homeomorphisms, if one endows $\mathfrak{a}$ and $\ol{\mathfrak{a}}$ with their standard vector space topologies and $A$ and $\ol{A}$ with the Kac--Peterson topology, respectively the induced quotient topology. In particular, the maps $\log: A \to \mathfrak a$ and $\log: \ol{A} \to \ol{\mathfrak a}$ are continuous.
\item The groups $T$ and $\ol{T}$ are isomorphic as topological groups to $(\R^\times)^n$ and $(\R^\times)^{\rk({\mathbf{A}})}$, respectively, and their respective identity components equal $A$ and $\ol{A}$. \qed
\end{enumerate}
\end{proposition}

Since $T \cong (\R^\times)^n$, its \Defn{torsion subgroup} $M$, i.e., its unique maximal finite subgroup, is of order $2^n$. As topological groups one has $T \cong M \times A$, where $M$ is equipped with the discrete topology. Similarly $\ol{T} \cong \ol{M} \times {\ol{A}}$, where $\ol{M}$ is the image of $M$ in $\ol{G}$, which is the torsion subgroup of $\ol{T}$ of order $2^{\rk({\mathbf{A}})}$.

\begin{proposition}\ \label{kernelsemisimpleadjoint}
\begin{enumerate}
\item The kernel $C$ of the surjection $G \to \ol{G}$ is isomorphic to $(\R^\times)^{n-\rk({\mathbf{A}})}$ as a topological group.
\item The kernel of the map $\ol{G} \to \Ad(G)$ is finite and, in fact, isomorphic to $(\Z/2\Z)^k$ for some $k < n$. In particular, it is contained in $\ol{M}$.
\end{enumerate}
\end{proposition}

\begin{proof} (i) follows by construction (cf.\ \cite[Proposition~1.6]{Kac90}).
(ii) Since $1$ and $-1$ are the only roots of unity contained in the real numbers $\R$, this follows from the proof of \cite[Lemma 7.5]{HartnickKoehlMars}. (Note Remark~\ref{erratum}.) 
\end{proof}

\subsection{The extended Weyl group} \label{extendedWeylgroup}

As discussed in Subsections \ref{AppendixAlgebrasAndWeylAB} and \ref{appendixweylgroup} in the appendix, the generalized Cartan matrix ${\mathbf{A}}$ gives rise to a quadruple $(\mathfrak{g}(\mathbf{A}), \mathfrak{h}(\mathbf{A}), \Pi, \check \Pi)$ (see \eqref{quadruple}) and a Coxeter datum $(W, S, \Phi, \Pi)$ (see Definition~\ref{Weylgroup}). One way to define $W$ is as the subgroup of $\GL(\mathfrak h({\bf A}))$ generated by the set $S = \{\check r_{\alpha_1}, \dots, \check r_{\alpha_n}\}$ of reflections given by
\[
\check r_{\alpha_i}(h) = h - \alpha_i(h)\check \alpha_i \quad (i=1, \dots, n),
\]
see \eqref{WeylReflDef} and also \cite[Lemma~1.2]{KacPeterson85c}. From this definition it is immediate that $W$ acts on $\mathfrak h({\bf A})$, and as pointed out in Proposition~\ref{KMRepWeyl} this action preserves the subspace $\mathfrak a$, and descends further to the quotient $\ol{\mathfrak a}$ of $\mathfrak a$. The two resulting representations are discussed further in Subsection \ref{appendixweylgroup} where they are denoted by  ${\rho}_{KM}: W \to \GL(\mathfrak a)$ and $\ol{\rho}_{KM}: W \to \GL(\ol{\mathfrak a})$ and referred to as the \Defn{Kac--Moody representation}\footnote{The terminology varies in the literature; both the Kac--Moody representation and its dual are sometimes called the \Defn{geometric representation} or the \Defn{canonical linear representation}, but we will not use these terms here.}, respectively the \Defn{reduced Kac--Moody representation} of $W$. The Kac--Moody representation is faithful and moreover, if $\mathbf{A}$ is non-affine, then the reduced Kac--Moody representation is faithful as well (see Corollary~\ref{CorollaryRedKMRepFaithful}). All these representations are constructed purely in terms of Lie algebra data; there is, however, also an alternative description of the Weyl group in terms of the group $G=G_\R({\mathbf{A}})$, which we discuss in this subsection. Our main sources here are \cite{Tits87}, \cite{KacPeterson85c}, where the corresponding results are established for $G_\C({\mathbf{A}})$ instead of $G$; one can show that the proofs carry over to the split real case. 

Consider the normalizer $N_G(T)$ of $T$ in $G$; it acts by conjugation on $T$, preserving the identity component $A$, and hence $\Ad(N_G(T))$ preserves $\mathfrak a$. Since $T$ is abelian, this action factors through $T$, hence induces a homomorphism
\[
\rho: N_G(T)/T \to \GL(\mathfrak a), \quad nT \mapsto \Ad(n)|_{\mathfrak a}.
\]
By \cite[Lemma 5.4.3(iii)]{Tits87} the representation $\rho$ is faithful with image $\rho(N_G(T)/T) = \rho_{KM}(W)$. Since the Kac--Moody representation is faithful, this establishes an isomorphism $N_G(T)/T \cong W$. In fact, there is a proper subgroup of $N_G(T)$ which still surjects onto $\rho_{KM}(W)$ and can be defined as follows (see \cite[Corollary~2.3(b)(ii)]{KacPeterson85c}): For every $i \in \{1, \dots, n\}$ define $\widetilde{s}_{\alpha_i} \in G_i <  G=G_\R({\mathbf{A}})$ by
\[
\widetilde{s}_{\alpha_i} := \phi_i \left(\begin{matrix} 0 & 1\\ -1 & 0\end{matrix}\right).
\]
and define the \Defn{extended Weyl group} by $\widetilde{W} := \langle \widetilde{s}_{\alpha_1}, \dots, \widetilde{s}_{\alpha_n} \rangle$.
By {\cite[Formula (2.6)]{KacPeterson85c}}  the extended Weyl group normalizes $T$ (and hence $A$), and by \cite[Proposition~2.1]{KacPeterson85c} the map $\rho$ restricts to an isomorphism
\[
\widetilde{W}/(\widetilde{W} \cap T) \cong W.
\]
More explicitly, since $\Ad(\widetilde{s}_{\alpha_i})|_{s_{\alpha_i}})|_{\mathfrak a} = \rho_{KM}(\check r_{\alpha_i})$ we have a canonical surjection
 \[\widetilde{W} \to W, \quad  \widetilde{s}_{\alpha_i} \mapsto \check r_{\alpha_i}\] 
with kernel $\widetilde{W} \cap T$. By \cite[Corollary~2.3.(a)]{KacPeterson85c} the elements of $\widetilde{W} \cap T$ all have order $\leq 2$; in particular $\widetilde{W} \cap T$ is contained in the torsion subgroup $M$ of $T$, which is of order $2^n$. On the other hand, the elements $\widetilde{s}_{\alpha_i}^2 \in \widetilde{W}$ are contained in $T$ and generate a subgroup of $\widetilde{W} \cap T$ of order $2^n$. We deduce that they generate $M$, and hence the torsion subgroup $M = \widetilde{W} \cap T$ of $T$ can be characterized as the kernel of the canonical surjection $\widetilde{W} \to W$.

\subsection{The twin $BN$ pair and the twin building} \label{sec:twin-BN-and-building}

Let ${\mathbf{A}}$ be an irreducible symmetrizable generalized Cartan matrix, let $G = G_\R({\mathbf{A}})$ as in Definition~\ref{DefTopSplitKM}, and let $\ol{G} = G/C$ be the semisimple adjoint quotient from Definition~\ref{semisimpleadjoint}.
Both  $G$ and $\ol{G}$ act strongly transitively on the same twin building and, hence, admit twin $BN$ pairs (see \cite[Theorem~8.9]{AbramenkoBrown2008}).

\medskip
The group $G$ in fact admits a centred RGD system $(G, \{U_\alpha\}_{\alpha \in \Phi}, T)$ in the sense of Definition~\ref{defrgd}, called the
 \Defn{canonical centred RGD system} and defined as follows, cf.\ \cite[Proposition~8.4.1]{Remy02}: The underlying set of roots $\Phi$ equals the set of real roots of the Kac--Moody algebra  $\mathfrak{g}({\mathbf{A}})$, see Section~\ref{appendixweylgroup}. The group $T$ is generated by the images of the diagonal subgroups $T_0 \subset \SL_2(\R)$ under the maps $\phi_i$ from Definition~\ref{DefTopSplitKM} and, given a simple root $\alpha_i$, one defines
\[
U_{\alpha_i} := \phi_i\left(\left\{\left(\begin{matrix} 1 & t\\ 0 & 1 \end{matrix}\right)\mid t \in \R\right\}\right).
\]
For an arbitrary real root $\alpha \in \Phi$ one writes $\alpha = w.\alpha_i$ (see Section~\ref{appendixweylgroup}) and defines \[U_\alpha := \widetilde w U_{\alpha_i} \widetilde w^{-1},\] where $w = \check r_{\alpha_{j_1}} \cdots \check r_{\alpha_{j_n}} \in W$ and $\widetilde w := \widetilde s_{\alpha_{j_1}} \cdots \widetilde s_{\alpha_{j_n}} \in \widetilde W$ as in Section~\ref{extendedWeylgroup}.

As in Section~\ref{RGDBN} denote by $(B_+, B_-, N)$ the twin $BN$ pair of $G$ induced by this RGD system and by $\Delta_\pm := G/B_\pm$ the sets of chambers or the corresponding positive and negative halves of the associated twin building (cf.~\cite[Section~8.9]{AbramenkoBrown2008}). 
 
\medskip
  The group $\ol{G}$ inherits an \Defn{induced centred RGD system} $(\ol{G}, \{\ol{U}_\alpha\}_{\alpha \in \Phi}, \ol{T})$, where $\ol{U}_\alpha \cong U_\alpha$ and $\ol{T}$, respectively denote the images of $U_\alpha$ and $T$ in $\ol{G}$. Denote by $(\ol{B}_+, \ol{B}_-, \ol{N})$ the twin $BN$ pair of $\ol{G}$ associated with the induced centred RGD system. Then, by construction, $\ol{B}_\pm = \ol{T} \ltimes \ol{U}_\pm$ where $\ol{U_\pm} := \langle \ol{U}_\alpha \mid \alpha \in \pm \Phi^+ \rangle$ as in Section~\ref{RGDBN}.

  Since $C \subset T \subset B_\pm$, one has $G/B_\pm = (G/C)/(B_\pm/C) = \ol{G}/\ol{B}_\pm$. That is, the halves of the twin buildings associated with $G$ and $\ol{G}$ coincide. In other words, the action of $G$ on $\Delta_\pm$ induces an action of $\ol{G}$ on $\Delta_\pm$. Note, furthermore, that $U_\pm \cong \ol{U_\pm}$, as by \cite[Lemma~8.31, Corollary~8.32]{AbramenkoBrown2008} both act sharply transitively on the set of chambers opposite the respective fundamental chambers in $\Delta_\mp$. 

\medskip
In general, given a group with a centred RGD system, the kernel of the action of that group on either half of the associated twin building equals the center of the group (\cite[Proposition~8.82]{AbramenkoBrown2008}). In particular, by Proposition~\ref{kernelsemisimpleadjoint} the action of $\ol{G}$ on $\Delta_\pm$ has a finite kernel, whereas the action of $G$ on $\Delta_\pm$ has an infinite kernel, if $\mathbf{A}$ is not invertible.

\medskip
We will use the following refinement of the Birkhoff decomposition. (Note that it is different from
what is known as the \Defn{refined} Birkhoff decomposition in the literature).
The spherical case is argued to hold in \cite[Remark 6.5]{HelminckWang93} by referring to \cite[Theorem 5.15]{BorelTits65}.
\begin{lemma}
\label{fine birkhoff}
 $G$ and $\ol{G}$ can be written as disjoint unions
  \[ G = \disjoint_{n\in N_G(T)} U_+ n U_-, \quad \ol{G} = \disjoint_{n\in N_{\ol{G}}(\ol{T})} \ol{U}_+ n \ol{U}_-. \]
\end{lemma}
\begin{proof}
For $G$ this is \cite[Proposition~3.3(a), p.~181]{KacPeterson85c}, also \cite[Theorem~5.2.3(g)]{Kumar02}. Note that in the latter this is proved for a \Defn{refined} Tits system
as defined in \cite{KacPeterson83}, but by \cite[1.5.4]{Remy02}, the Tits system for a group with an RGD system is indeed refined. The same argument applies to the refined Tits system for $\ol{G}$.
\end{proof}

\begin{definition}
Given a real root $\alpha \in \Phi$ define the \Defn{rank one subgroup} as \[G_\alpha := \langle U_\alpha, U_{-\alpha} \rangle.\] Note that the standard rank one subgroups of $G$ introduced in Definition~\ref{DefTopSplitKM} are the rank one subgroups associated with the simple roots.
\end{definition}

By \cite[Proposition~7.15]{HartnickKoehlMars} (see also \cite[Section~2E]{KacPeterson83}) the subgroups $B_\pm$ are closed in $G$ with respect to the Kac--Peterson topology and hence $\Delta_\pm$ are Hausdorff $k_\omega$-spaces with respect to the quotient topology by \cite[Assertion 11, p.~116f]{FranklinThomas77}. 

The following proposition summarizes further topological properties of the various subgroups defined above:

\begin{proposition}\ \label{TopologyGSummary}
\begin{enumerate}
\item  $T$ is closed in $G$ and isomorphic to $(\R^\times)^n$ as a topological group. Similarly, $\ol{T}$ is closed in $\ol{G}$ and isomorphic to $(\R^\times)^{\rk({\mathbf{A}})}$ as a topological group.
\item Multiplication induces isomorphisms of topological groups $M \times A \to T$ and $\ol{M} \times \ol{A} \to \ol{T}$, where $M$ and $\ol{M}$ are the torsion subgroups and $A$ and $\ol{A}$ are the connected components of $T$ and $\ol{T}$, respectively. Furthermore, the center of $\ol{G}$ is contained in $\ol M$.
\item Every rank one subgroup in $G$ or $\ol{G}$ is isomorphic as a topological group to $\mathrm{(P)SL}_2(\R)$ with its unique connected Lie group topology, and every root subgroup is isomorphic as a topological group to $(\R, +)$ endowed with its standard topology.
\item Multiplication induces homeomorphisms $M \times A \times U_\pm \to B_\pm$ and $\ol{M} \times \ol A \times \ol{U}_\pm \to \ol{B}_\pm$.
\end{enumerate}
\medskip
If the generalized Cartan matrix $\mathbf{A}$ is two-spherical, then moreover the following hold:
\begin{enumerate}\setcounter{enumi}{4}
\item $B_+B_-$ is open in $G$ and multiplication defines a homeomorphism $U_+ \times T \times U_- \to B_+B_-$;\\
$\ol{B}_+\ol{B}_-$ is open in $\ol{G}$ and multiplication induces a homeomorphism $U_+ \times \ol{T} \times U_- \to \ol{B}_+\ol{B}_-$. 
\item $U_+AU_-$ is open in $G$ and multiplication defines a homeomorphism $U_+ \times A \times U_- \to U_+AU_-$ ;\\
 $\ol{U}_+\ol{A}\,\ol{U}_-$ is open in $\ol G$ and multiplication induces a homeomorphism $\ol{U}_+ \times \ol{A} \times \ol{U}_- \to \ol{U}_+\ol{A}\,\ol{U}_-$. 
\end{enumerate}
\end{proposition}
\begin{proof}
\begin{enumerate}
\item $T$ is closed in $G$ by \cite[Corollary~7.17(iii)]{HartnickKoehlMars}, and so is $\ol{T}$ in $\ol{G}$. The remaining statements follow from Proposition~\ref{PropexpHomeomorphism}.
\item This follows from the discussion after Proposition~\ref{PropexpHomeomorphism}, together with Proposition~\ref{kernelsemisimpleadjoint}.
\item is immediate by \cite[Corollary~7.16(iii)]{HartnickKoehlMars} and \cite[Corollary~7.17(ii)]{HartnickKoehlMars}. 
\item follows from \cite[Proposition~7.27(ii)]{HartnickKoehlMars} plus assertion (ii).

\item follows from  \cite[Lemma~6.1, Proposition~6.6, Proposition~7.31]{HartnickKoehlMars}.
\item follows from (i) and (v): Since $T = A \times M$ with $M$ finite, $A$ is open in $T$ and thus $U_+ \times A \times U_- \subset U_+ \times T \times U_-$ is open. Consequently, the restriction of the open map $U_+ \times T \times U_- \to B_+B_-$ to the open subset $U_+ \times A \times U_-$ is also open, in particular its image is open.
For $\ol{G}$ one argues similarly. \qedhere
\end{enumerate}
\end{proof}

\begin{remark} \label{twospherical1}
It is an interesting question whether for general Cartan matrices ${\mathbf A}$ the map $U_+ \times \ol{T} \times U_- \to \ol{B}_+\ol{B}_-$ is open. Currently this is only known under the additional hypothesis that $\mathbf{A}$ be two-spherical \cite[Proposition~7.31]{HartnickKoehlMars}, but we expect that it is possible to remove this hypothesis;
 in fact, already Kac and Peterson had this expectation in \cite[Section~4G]{KacPeterson83}. If this expectation can be confirmed, then one can remove the assumption of two-sphericity in Proposition~\ref{TopologyGSummary} and consequently in a number of results below. Our suggested approach towards proving the conjecture makes use of an unfolding argument as described in \cite[Definition~1.10]{Hainke/Koehl/Levy:2015} that is very likely to allow one to embed an arbitrary symmetrizable split real Kac--Moody group $G$ as a closed subgroup into a simply laced split real Kac--Moody group $G'$ in such a way that the RGD systems are compatible with one another (see also \cite[Theorem~E]{Marquis:2015}). The fact that \cite[Proposition~7.31]{HartnickKoehlMars} applies to the ambient simply-laced Kac--Moody group $G'$ should allow one to prove the analogous statement for the original Kac--Moody group $G$ via (co)restrictions of the multiplication map.

  Note here that (co)restrictions of open maps of course frequently fail to be open. However, since one is dealing with a bijection in this situation, one can as well establish the continuity of the inverse map, a property that behaves very well under (co)restrictions.     
\end{remark}

\subsection{The Cartan--Chevalley involution and the twist map}\label{sec:twist-map}

Each of the standard rank one subgroups $\mathrm{(P)SL}_2(\R) \cong G_i < G$ admits a continuous involution $\theta_i$ induced by $g \mapsto (g^{-1})^T$.
By \cite[Section~8.2]{Caprace09} (also \cite[Section~2]{KacPeterson85c}), for suitable choices of the given isomorphisms $\mathrm{(P)SL}_2(\R) \cong G_i$ these involutions $\theta_i$ extend uniquely to an involution $\theta : G \to G$, called the \Defn{Cartan--Chevalley involution} of $G$.

The fixed point set of $\theta$  is denoted by
\[
K := G^\theta = \{k \in G\mid \theta(k) = k\}.
\]
Since $\theta$ is continuous by \cite[Lemma~7.20]{HartnickKoehlMars}, the group $K$ is a closed subgroup of $G$ and therefore a $k_\omega$-topological group  (cf.\ \cite[p.~118]{FranklinThomas77}). 

\begin{proposition}\label{ExtWeylInK} The extended Weyl group $\widetilde{W}$ introduced in Section~\ref{extendedWeylgroup} is contained in $N_K(T) < K$.
\end{proposition}
\begin{proof} For $i \in \{1, \dots, n\}$ 
\[
\widetilde{s}_{\alpha_i} := \phi_i \left(\begin{matrix} 0 & 1\\ -1 & 0\end{matrix}\right) \in G_i^{\theta_i} < K,
\]
and these generate $\widetilde{W}$, hence $\widetilde{W} < K$. Moreover, we have seen in Section~\ref{extendedWeylgroup} that $\widetilde{W}$ normalizes $T$, hence $\widetilde{W} \subset N_K(T)$ as claimed.
\end{proof}
We will actually see that $\widetilde{W} = N_K(T)$ in Corollary \ref{AbstractExtendedWeylGroup} below.

\begin{lemma}\label{CCinvB} The Cartan--Chevalley involution stabilizes $T$ and maps $U_+$ to $U_-$. In particular, $\theta(B_+) = B_-$.
\end{lemma}
\begin{proof} This follows from the observation that on each of the rank one subgroups, $\theta$ preserves the diagonal subgroup and interchanges the groups  $U_{\alpha_i} = \phi_i\left(\left\{\left(\begin{matrix} 1 & t\\ 0 & 1 \end{matrix}\right)\mid t \in \R\right\}\right)$ and $U_{-\alpha_i} = \phi_i\left(\left\{\left(\begin{matrix} 1 & 0\\ t & 1 \end{matrix}\right)\mid t \in \R\right\}\right)$.
\end{proof}

\begin{proposition}\label{CCIdescends}
  The Cartan-Chevalley involution preserves $C$ and hence induces a continuous involution $\ol{\theta}$ of $\ol{G}$, which stabilizes $\ol{T}$ and maps $\ol{B}_+$ to $\ol{B}_-$. 
\end{proposition} 
\begin{proof} Let $d\theta: \mathfrak g \to \mathfrak g$ be the involution of $\mathfrak g$ which on the rank one subalgebras $\mathfrak g_i \cong \mathfrak{sl}_2(\C)$ is given by $X \mapsto -X^\ast$. This satisfies $d\theta(\mathfrak g_\alpha) = \mathfrak g_{-\alpha}$ (cf.\ \cite[p.~7]{Kac90}) for every root $\alpha$, and in particular preserves $\ker(\alpha_i)\subset \mathfrak h$ for every $i \in \{1, \dots, n\}$. It thus follows from the definition of $\mathfrak c$ in \eqref{KMAcenter} on page \pageref{KMAcenter} that the latter is $d\theta$-invariant. Since $\exp_\C$ intertwines $d\theta$ and $\theta$ (the latter considered as an automorphism of $G_\C(\mathbf{A}))$, it follows that $\theta$ preserves $C$. The other statements now follow from Lemma~\ref{CCinvB}.
\end{proof}
Note that the image of $K$ in $\ol{G}$ is equal to $\ol{K} := \ol{G}^{\ol{\theta}}$, as both groups are generated by the panel stabilizers $\ol{G_i}^{\ol{\theta}} \ol{T}^{\ol{\theta}}$, $1 \leq i \leq n$ (cf. \cite[Theorem 1.2]{deMedtsGramlichHorn09}).

\medskip
Let us recall and adjust to our setting some of the notions introduced in \cite[Section~2]{Richardson1982}; see also \cite[Section~6]{HelminckWang93}, \cite[Section~5]{KacWang92}.

\begin{defn}\label{DefTwist} Let $G = G_\R(\mathbf{A})$ be the simply connected split real Kac--Moody group of type $\mathbf{A}$, let $\theta$ be its Cartan--Chevalley involution, let $\ol{G}$ be the semisimple adjoint quotient of $G$, and let $\ol{\theta}$ be the involution of $\ol{G}$ induced by $\theta$.
\begin{enumerate}
\item The maps  
\[ G\times G \to G, \quad (g,x) \mapsto g*x := g x \theta(g)^{-1} \quad \text{and} \quad \ol{G}\times \ol{G} \to \ol{G}, \quad (g,x) \mapsto g*x := g x \ol{\theta}(g)^{-1} \]
are called the \Defn{twisted conjugation maps} of $G$ and $\ol{G}$, respectively.
\item
The \Defn{twist maps} of $G$, respectively $\ol{G}$ are the continuous map
\[\tau: G \to G, \quad g\mapsto g*e = g\theta(g)^{-1} \quad \text{and} \quad \ol{\tau}: G \to G, \quad g\mapsto g*e = g\ol{\theta}(g)^{-1}.\]
\end{enumerate}
\end{defn}
Note that twisted conjugation defines a left-action of $G$ on itself, since
\[
g*(h*x) = g*(h x \theta(h)^{-1}) = gh x\theta(h)^{-1}\theta(g)^{-1} = (gh)x \theta(gh)^{-1} = (gh)\ast x,
\]
while $\tau$ is an orbit map of this group action; a similar statement holds for $\ol{G}$. The following lemma summarizes various basic properties of the twist map.

\begin{lemma}\ \label{lem:tau-K-saturation}
\begin{enumerate}
\item For $g \in \tau(G)$ one has $\theta(g) = g^{-1}$ and $\tau(g) = g^2$.
\item For $g,h\in G$ one has $\tau(gh)=g*\tau(h)$.
\item $\tau^{-1}(e)=K$.
\item For $g,h\in G$ one has $gK=hK \iff \tau(g) = \tau(h) \iff \tau(h^{-1}g)=e$.
\item For every $S\subseteq G$ one has $\tau^{-1}(\tau(S))=SK$.
\item $\tau$ factors through $G/K$, yielding a surjective map
\[ \hat\tau: G/K \to \tau(G), \quad gK \mapsto \tau(g). \]
\end{enumerate}
Analogous statements hold for $\ol{G}$ instead of $G$.
\end{lemma}

\begin{remark}
In fact, Definition \ref{DefTwist} makes sense for an arbitrary group $G$ with involution $\theta \in \Aut(G)$, and Lemma~\ref{lem:tau-K-saturation} remains valid in this generality for $K := G^\theta$. In this broader context, one sees that
the twist map from Definition~\ref{DefTwist} can be considered as a non-Galois version of the famous Lang map from \cite[Section~2]{Lang1956}.

Furthermore, even in the case of real Kac--Moody groups there exist involutions $\theta$ different from the Cartan--Chevalley involution that lead to symmetric spaces $G/G^\theta$ worthwhile of further study; we refer to \cite{KacWang92} and \cite{GramlichHornMuehlherr} for a discussion of abstract involutions of Kac--Moody algebras and Kac--Moody groups that might provide a starting point for studying these more general Kac--Moody symmetric spaces.
\end{remark}

\begin{proof}[Proof of Lemma~\ref{lem:tau-K-saturation}]
\begin{enumerate}
\item For $g = h\theta(h)^{-1} \in \tau(G)$ one computes $\theta(g) = \theta(h)\theta(\theta(h)^{-1})
= \theta(h)h^{-1} = g^{-1}$ and $\tau(g) = h\theta(h)^{-1}\theta(h\theta(h)^{-1})^{-1}= (h\theta(h)^{-1})^2 = g^2$.
\item One has $\tau(gh) = gh \ast e = g \ast (h \ast e) = g\ast \tau(h)$.
\item
For $g\in G$, one has $\tau(g)=e\iff g\theta(g)^{-1}=e \iff g=\theta(g) \iff g\in K$.
\item
One computes
\begin{align*}
gK = hK &\iff \exists k\in K: g=hk
\implies \tau(g)=\tau(hk)=\tau(h)
\implies g\theta(g)^{-1} = h\theta(h)^{-1} \\
&\implies h^{-1}g=\theta(h)^{-1}\theta(g)=\theta(h^{-1}g)
\implies h^{-1}g\in K
\implies gK = hK.
\end{align*}
Moreover, by (iii), one has $h^{-1}g\in K \iff \tau(h^{-1}g)=e$.
\item
Let $B:=\tau(S)$. Then $x\in\tau^{-1}(B) \iff \tau(x)\in B \iff \exists s\in S : \tau(x)=\tau(s) \iff \exists s\in S: xK=sK \iff x\in SK$.
\item follows from (v).
\qedhere
\end{enumerate}
\end{proof}

Concerning the statement of the following lemma we recall that $M$ is the torsion subgroup of $T$, that $A$ is the identity component of $T$, and that $T=MA\cong M\times A$.

\begin{lemma}\ \label{tauonT}
\begin{enumerate}
\item $\tau(t) = t^2$ for all $t\in T$.
\item $A = \tau(T) = \tau(A)$. 
\item $B_\pm \cap K =T \cap K = M$ and $A \cap K= \{ e \}$.
\end{enumerate}
\end{lemma}

\begin{proof} The key observation is that $T$ is the direct product of the diagonal subgroups $T_i \cong \R^\times$ in $G_{i}$, and on each of the $T_i$ the involution $\theta$ acts by inversion. In particular, $\tau(t) = t\theta(t)^{-1} = t^2$ for all $t \in T_i$, whence (i) follows. Since the set of squares in $\R^\times$ is given by $\R^{>0}$, and every element in $\R^{>0}$ has a positive square root, (ii) follows from (i). Concerning (iii), observe first that, if $g \in B_\pm \cap K$, then $\theta(g) = g$. Since $\theta(B_\pm) = B_{\mp}$, this implies $g \in B_+ \cap B_- = T$, so $B_\pm \cap K =T \cap K$. Now let $t \in T$ and write $t=ma$ with $a \in A \cong (\R^\times)^n$, $m \in M = (\Z/2\Z)^n$ (see Definition~\ref{semisimpleadjoint}). Then $\tau(t) = m^2a^2 = a^2$, and thus $\tau(t) = e$ if and only if $a = e$, as $A$ is torsion-free. Hence $T \cap K = M$ and $A \cap K = \{e\}$ by Lemma~\ref{lem:tau-K-saturation}(iii).
\end{proof}

In the following proof, we use the abstract Iwasawa decomposition $G = KB_\pm = K U_\pm T = U_\pm T K$ (cf. \cite{deMedtsGramlichHorn09}). While we will in short order actually refine this (see Theorem~\ref{ThmIwasawa}), the next lemma only needs this basic form.

\begin{lemma}\ \label{lem:NGT=A NKT} \label{lem:T=N-cap-tauG}
\begin{enumerate}
\item $N_G(T) \cap \tau(G) = A$.
\item $N_G(T) = A \rtimes N_K(T)$.
\end{enumerate}
\end{lemma}

\begin{proof}
\begin{enumerate}
\item
By Lemma~\ref{tauonT}, one has $A=\tau(T) \subset N_G(T)\cap\tau(G)$. It remains to show the other inclusion.
Let $g\in N_G(T) \cap \tau(G)$. Since $g\in \tau(G) = \tau(U_+TK)$, there exist
 $u\in U_+$,  $t\in T$, $k\in K$ such that
 \[ g=\tau(utk)= \tau(ut) = ut^2 \theta(u)^{-1} \in U_+ t^2 U_-.\]
Since also $g\in N_G(T)$, the refined Birkhoff decomposition (see Lemma~\ref{fine birkhoff})
 yields $g = t^2 \in \tau(T) = A$ as claimed.

\item
First show that $N_G(T)=N_K(T)T$. Indeed, the inclusion $\supseteq$ is clear. Let $g\in N_G(T)$. By the Iwasawa decomposition, there exist $u\in U_+$,  $t\in T$, $k\in K$ such that $g=kut$. As $T=T^g$, one concludes $T^{k^{-1}} = T^u$ and therefore
\[ T^u = T^{k^{-1}} = \theta(T^{k^{-1}}) = \theta(T^u) = T^{\theta(u)}
\implies \tau(u) = u\theta(u)^{-1} \in N_G(T).
\]
But by (i), one has $\tau(u)\in A$, hence by the refined Birkhoff decomposition, $u=1$.
Thus $T^k=T$, hence $k\in N_K(T)$ and $g=kt\in N_K(T)T$ as claimed.
 
Furthermore, one has $T=MA$, and by Lemma~\ref{tauonT} also $M = K \cap T \subset N_K(T)$.
Hence $N_G(T)=N_K(T)T=N_K(T)MA = N_K(T)A$. Since $A\subset T$ is normalized by $N_K(T)$ and $A \cap N_K(T) \subset A\cap K=\{e\}$ (see Lemma~\ref{tauonT}), one arrives at $N_G(T) = A \rtimes N_K(T)$.
\qedhere
\end{enumerate}
\end{proof}
As an application of the previous lemma one can now concretely identify the extended Weyl group $\widetilde{W}$, which was introduced in Subsection \ref{extendedWeylgroup} by means of a set of generators.
\begin{corollary}\label{AbstractExtendedWeylGroup} The extended Weyl group satisfies $\widetilde{W} = N_K(T)$, its image in $\overline{G}$ is given by $N_{\overline{K}}(\overline{T})$.
\end{corollary}
\begin{proof} By Proposition \ref{ExtWeylInK} one has $\widetilde{W} < N_K(T)$. Towards the opposite inclusion recall from Subsection \ref{extendedWeylgroup} that the canonical surjection $\pi: N_G(T) \to N_G(T)/T$ is still surjective when restricted to $\widetilde{W}$ and that $\widetilde{W} \cap T = M$, the torsion subgroup of $T$. Since $\widetilde{W} < N_K(T)$, the first assertion implies that $\pi(N_K(T)) = N_G(T)/T = \pi(\widetilde{W})$; by the second assertion it suffices to show that $\ker(\pi|_{N_K(T)}) = M$. Now by Lemma \ref{lem:NGT=A NKT} one has
  \[
N_G(T)/T = (A \rtimes N_K(T))/(A \times M) \cong N_K(T)/M,
\]
where the isomorphism is induced by the inclusion $N_K(T) \hookrightarrow N_G(T)$. Thus $\ker(\pi|_{N_K(T)}) = M$, which proves the first statement, and the second statement follows from the first one.
\end{proof}

The following technical observation depends heavily on the language of twin buildings. Refer to \cite[Sections~5.8 and 6.3]{AbramenkoBrown2008} and \cite{Horn:Decomp} for the necessary background information. 
Note that the automorphism $\theta$ of $G$ acts on the twin building $\Delta$ by Proposition~\ref{AutBuilding} below. 
A twin apartment of $\Delta$ is called \Defn{$\theta$-stable} if it is invariant as a set under the action of $\theta$.

\begin{lemma}[{\cite[Lemma~4.2]{Horn:Decomp}}] \label{lem:sym-ss}
 Suppose $g\in G$ is symmetric, i.e., $\theta(g) = g^{-1}$. Then the following assertions concerning the action of $g$ on the twin building of $G$ are equivalent:
 \begin{enumerate}
 \item\label{prop:enum:theta-apt} $g$ fixes a $\theta$-stable twin apartment chamberwise.
 \item\label{prop:enum:twin-apt} $g$ fixes a twin apartment chamberwise.
 \item\label{prop:enum:apt} $g$ fixes an apartment chamberwise.
 \item\label{prop:enum:cham} $g$ stabilizes a chamber.
 \item\label{prop:enum:orb} $g$ has a bounded (with respect to the gallery metric) orbit.
 \item\label{prop:enum:res} $g$ stabilizes a spherical residue.
 \qed
 \end{enumerate}
\end{lemma}

\subsection{The topological Iwasawa decomposition}

The goal of this subsection is to prove the following decomposition results for $G$ and $\ol{G}$.
For this, we use the topological structure of these groups as well as the twin buildings $\Delta_\pm$ as discussed in Section~\ref{sec:twin-BN-and-building}.

\begin{theorem}[Topological Iwasawa decomposition]\label{ThmIwasawa}
  Let $G = G_\R(\mathbf{A})$ be the simply connected split real Kac--Moody group of type $\mathbf{A}$ and let $\ol{G}$ be its semisimple adjoint quotient. 
\begin{enumerate}
\item $K \cap B_\pm = M$ and $\ol{K} \cap \ol{B}_\pm = \ol{M}$. In particular, the center of $\ol{G}$ is contained in $\ol{K}$.
\item Multiplication induces continuous bijections $m_1: U_\pm\times A \times K \to G$, $m_2: K\times A \times U_\pm \to G$ and homeomorphisms $\ol{m}_1: \ol{U}_\pm\times \ol{A} \times \ol{K} \to \ol{G}$ and $\ol{m}_2:  \ol{K} \times \ol{A} \times  \ol{U}_\pm\to \ol{G}$.
\item The action of $K$ on both halves of the twin building factors through $\ol{K}$, which acts transitively on both halves of the twin building. Moreover, $\Delta_\pm \cong \ol{K}/\ol{M}$, where $\ol{K}/\ol{M}$ carries the quotient topology.
\end{enumerate}
\end{theorem}
\begin{proof}[Proof of Theorem \ref{ThmIwasawa}, discrete version] First establish the results concerning $G$. (i) follows from Lemma~\ref{tauonT}. Concerning (iii), recall from \cite{deMedtsGramlichHorn09} that 
  $G = KB_\pm$. In particular, $K$ acts transitively on $\Delta_\pm$.

  Now consider the map $m_1$ from (ii). Since $B_\pm =MAU_\pm$ and $G = KB_\pm$, one has $G = KMAU_\pm = KAU_\pm$, i.e.\ $m_1$ is surjective. Injectivity of $m_1$ follows from $B_\pm \cap K = M$, so that $m_1$ is a bijection. Since inversion intertwines $m_1$ and $m_2$, it follows that also $m_2$ is bijective, establishing the discrete part of Theorem \ref{ThmIwasawa} for $G$.

  Concerning $\ol{G}$, since the action of $K$ on $\Delta_\pm$ factors through $\ol{K}$, the latter acts transitively on $\Delta_\pm$, i.e.\ $\ol{G} = \ol{K}\,\ol{B}_\pm$.
The fact $K \cap B_\pm = M < T$ implies $\ol{K} \cap \ol{B}_\pm = \ol{M} < \ol{T}$. In particular, $\Delta_\pm \cong \ol{K}/\ol{M}$ as sets. This in turn implies bijectivity of $\ol{m}_1$ by the same argument used to show bijectivity of $m_1$ and, thus, also of $\ol{m}_2$. The statement about the center of $\ol{G}$ follows from Proposition~\ref{TopologyGSummary}(ii).
\end{proof}

\begin{remark}
Note that the automorphism $\theta$ of $G$ acts on the twin building $\Delta$ by Proposition~\ref{AutBuilding} below. An immediate consequence of the transitive action of $K$ resp. $\ol{K}$ on the chambers of $\Delta_\pm$ therefore is that every chamber $c\in\Delta_\pm$ is opposite its image $\theta(c)\in\Delta_\mp$. Thus they define a unique $\theta$-stable twin apartment, 
\end{remark}

Before turning to the topological version of the theorem, recall some basic facts about $k_\omega$-spaces, cf.\ \cite{FranklinThomas77}. 
\begin{definition}\label{DefKOmega}
A Hausdorff topological space $X$ is called a \Defn{$k_\omega$-space}, if it is the direct limit of an increasing family of compact subspaces $(X_n)_{n \in \N}$, i.e., if $X = \bigcup_n X_n$ and a subset $Y$ of $X$ is open in $X$ if and only if each intersection $Y \cap X_n$ is open in $X_n$; the increasing family $(X_n)_{n \in \N}$ is called a \Defn{$k_\omega$-sequence} for $X$ and the pair $(X, (X_n))$ is called a \Defn{$k_\omega$-pair}.
\end{definition}
\begin{lemma}\label{komegahomeo} Let $(X, (X_n))$ and $(Y, (Y_n))$ be $k_\omega$-pairs and let $f: X \to Y$ be a continuous bijection such that
\[
\forall n \in \N \quad \exists m \in \N: f(X_m) \supset Y_n.
\]
Then $(f(X_n))$ is a $k_\omega$-sequence for $Y$ and $f$ is a homeomorphism.
\end{lemma}
\begin{proof} Since $Y$ is Hausdorff, the sets $f(X_n)$ are compact. Hence by \cite[Assertion~7, p.~114]{FranklinThomas77} for every $n \in \N$ there exists $m \in \N$ such that $f(X_n) \subset Y_m$. The hypothesis therefore implies that the sequences $(f(X_n))$ and $(Y_n)$ define the same limit topology on $Y$, i.e.\ $(f(X_n))$ is a $k_\omega$-sequence for $Y$.  Now for each $n$ the map $f: X_n \to f(X_n)$ is a homeomorphism, and hence $f$ yields a homeomorphism
\[
f: X =  \lim_{\to} X_n \to  \lim_{\to} f(X_n) = Y.\qedhere
\]
\end{proof}

\begin{lemma}\label{Coveringkomega} Let $(X, (X_n))$ be  a $k_\omega$-pair and let $\pi: \widetilde{X} \to X$ be a finite-sheeted covering. Then $(\widetilde{X}, \pi^{-1}(X_n))$ is a $k_\omega$-pair.
\end{lemma}

\begin{proof} Since $\pi$ is a finite-sheeted covering, it is proper and, hence, $\widetilde{X}_n := \pi^{-1}(X_n)$ is compact for every $n \in \N$. Now let $\widetilde x \in \widetilde{X}$
 and $x := \pi(\widetilde{x})$. Then there exist open neighbourhoods $\widetilde{V}$ of $\widetilde{x}$ and $V$ of $x$ such that $\pi$ restricts to a homeomorphism $\widetilde{V} \to V$. Now let $\widetilde{U}$ be a subset of $\widetilde{X}$ containing $\widetilde{x}$ and $U := \pi(\widetilde{U})$ Then one has the following chain of equivalences:
 \begin{align*}
 \widetilde{U} \text { is a neighbourhood of }\widetilde{x} & \iff \widetilde{U} \cap \widetilde{V} \text{ is a neighbourhood of }\widetilde{x}\\
 & \iff U \cap V \text{ is a neighbourhood of } x\\
  & \iff U \cap V \cap X_n \text{ is a neighbourhood of } x \text{ for all sufficiently large }n \in \N\\
    & \iff \widetilde{U} \cap \widetilde{V} \cap \widetilde{X}_n \text{ is a neighbourhood of } \widetilde{x} \text{ for all sufficiently large }n \in \N\\
        & \iff \widetilde{U} \cap \widetilde{X}_n \text{ is a neighbourhood of } \widetilde{x} \text{ for all sufficiently large }n \in \N.
 \end{align*}
This shows in particular that a subset of $\widetilde{X}$ is open if and only if its intersection with $\widetilde{X_n}$ is open for all sufficiently large $n \in \N$. We deduce that $(\widetilde{X},  (\widetilde{X}_n))$ is a $k_\omega$-pair.
\end{proof}

Let $\Delta_\pm = \ol{G}/\ol{B}_\pm$ denote one half of the twin building $\Delta$. Recall from Proposition~\ref{TopologyGSummary}(iv) that $\ol{B}_+$ has the decomposition $\ol{B}_+ = \ol{M}\,\ol{A}\,\ol{U}_+$, where $\ol{M} = \ol{K} \cap \ol{B}_+$ is a finite group. Denote by $\widetilde{\Delta}_\pm$ the quotient $\ol{G}/\ol{A}\,\ol{U}_\pm$. Then the canonical projections
\begin{equation}\label{MapPiFiniteCovering}
\pi_\pm: \widetilde{\Delta}_\pm \to \Delta_\pm
\end{equation}
are finite-sheeted covering maps with fiber $\ol{M}$.

\begin{proposition}\label{TopIwasawaMain} The maps 
\[
\iota_\pm: \ol{K} \to  \widetilde{\Delta}_\pm, \quad k \mapsto k\ol{A}\,\ol{U}_\pm
\]
are homeomorphisms.
\end{proposition}
\begin{proof} It follows from the abstract Iwasawa decomposition that $\iota_\pm$ are continuous bijections. Let
\[\ol{G}^{\pm}_k := \bigcup_{w \in W,\, l(w) \leq k} \ol{B}_\pm w \ol{B}_\pm,\]
denote by $\widetilde{\Delta}_{k, \pm}$ and $\Delta_{k, \pm}$ the respective image of $\ol{G}^\pm_k$ in $\widetilde{\Delta}_\pm$ and $\Delta_\pm$, and let $\ol{K}^\pm_k :=\ol{K} \cap G^{\pm}_k$. 
Then by \cite[Corollary~7.11]{HartnickKoehlMars} and the observation that direct limits commute with quotients one has
\begin{equation}
\ol{G} = \lim_\to \ol{G}^{\pm}_k,
\end{equation}
and, thus, \[\ol{K} =  \lim_\to \ol{K}^+_k.\]

The subsets $\ol{K}^\pm_k \subset \ol{K}$ are compact: Indeed,  by \cite[Theorem~1.2]{deMedtsGramlichHorn09} $\ol{K}^\pm_k$ equals the finite union of products of the form $MK_{\alpha_1} \cdots MK_{\alpha_k}$, where $M = T \cap K$ is finite (see Lemma~\ref{tauonT}) and each $K_{\alpha_i} \cong \SO_2(\R)$ is compact. Since multiplication is continuous and $\ol{K}$ is Hausdorff, this implies that $\ol{K}^{\pm}_k$ are compact, and hence $(\ol{K}, (\ol{K}^{\pm}_k))$ is a $k_\omega$-pair.

By the discrete version of Theorem~\ref{ThmIwasawa}, the group $\ol{K}$ acts transitively on $\widetilde{\Delta}_\pm$ and one has $\iota_\pm(\ol{K}^\pm_k)  = \widetilde{\Delta}_{k, \pm}$. In particular, the spaces $\widetilde{\Delta}_{k, \pm}$ are compact. Therefore $(\widetilde{\Delta}_\pm, (\widetilde{\Delta}_{k, \pm}))$ is a $k_\omega$-pair
%
%
%
and the proposition follows from Lemma~\ref{komegahomeo}.
\end{proof}

\begin{proof}[Proof of Theorem \ref{ThmIwasawa}] Assertion (i) has already been proved for the discrete version of the theorem. Concerning (iii), the finite-sheeted coverings $\pi_\pm :  \widetilde{\Delta}_\pm \to \Delta_\pm$ from \eqref{MapPiFiniteCovering} are continuous and open. By Proposition~\ref{TopIwasawaMain} this implies that the orbit maps $\ol{K} \mapsto \Delta_\pm$ are continuous and open, hence $\Delta_\pm \cong \ol{K}/\ol{M}$ as topological spaces as claimed.

  In order to prove (ii), it is clear that the maps under consideration are continuous, since they are induced by the group multiplication. It thus remains to show that $\ol{m}_2$, and hence $\ol{m}_1$, are open. Given $g \in \ol{G}$, define $k(g) := \iota_\pm^{-1}(g\ol{A}\,\ol{U}_\pm)$, where $\iota_\pm$ is as in Proposition~\ref{TopIwasawaMain} and let $b(g) := k(g)^{-1}g$. Since $\iota_\pm$ is open and $\ol G$ is a topological group, one obtains a continuous map
\[
i_\pm: \ol G \to \ol K \times \ol A\,\ol{U}_\pm, \quad g \mapsto (k(g), b(g))
\]
such that $g = k(g)b(g)$. This map is inverse to the multiplication map $m: \ol K \times \ol A\,\ol{U}_\pm \to \ol G$, showing that $m$ is a homeomorphism. It remains to see that the multiplication map $\ol A \times \ol{U}_\pm \to \ol A\,\ol{U}_\pm$ is open; this however follows from \cite[Proposition~7.27]{HartnickKoehlMars}. This finishes the proof of Theorem \ref{ThmIwasawa}.
\end{proof}

\subsection{The image of the twist map}

The goal of this subsection is to understand the images of the twist maps inside their ambient groups.

\begin{proposition} \label{lem:1=G-cap-tauG}
 $K\cap\tau(G)=\{e\}$ and $\ol{K} \cap \ol{\tau}(\ol{G}) = \{e\}$.
\end{proposition}

\begin{proof}
 Suppose $g\in K\cap\tau(G)$. Then $g=\theta(g)=g^{-1}$ by Lemma~\ref{lem:tau-K-saturation}(i), so $g$ has order 1 or 2.
 Hence its orbits are bounded, and so by Lemma~\ref{lem:sym-ss}
 it stabilizes a chamber $c$ in the twin building of $G$. But then also $\theta(c)=\theta(g.c)=g.\theta(c)$,
 so $g$ fixes chamberwise the (unique) $\theta$-stable twin apartment containing the two opposite chambers $c$ and $\theta(c)$ and, thus, is contained
 in the corresponding $\theta$-split torus $T'$ of $G$ (where $\theta$-split means that $\theta$ leaves $T'$ invariant and acts via inversion on $T'$). Since $K = G^\theta$ acts
 transitively on each half of the twin building, there exists $k\in K$ with ${}^kT'=T$.
 Thus $k * g = {}^kg \in T \cap \tau(G) \cap K$ and, by Lemma~\ref{lem:T=N-cap-tauG}, in fact ${}^kg\in A$ .
 But then ${}^kg\in A\cap K = \{e\}$, hence $g=e$, i.e., $K\cap\tau(G)=\{e\}$.

Similarly, one proves $\ol{K} \cap \ol{\tau}(\ol{G}) = \{e\}$.
\end{proof}

\begin{proposition}\label{tauGgenerates} The group $G$ (respectively, $\ol{G}$) is generated by its subset $\tau(G)$ (respectively, $\ol{\tau}(\ol{G})$).
\end{proposition}

\begin{proof} The map $\tau$ preserves each of the fundamental rank one subgroups $G_i \cong \mathrm{(P)SL}_2(\R)$. A simple computation in $\mathrm{(P)SL}_2(\R)$ shows that $\tau(G_i)$ generates $G_i$ (the matrix group $\SL_2(\R)$ is generated by the set of positive definite symmetric matrices). Thus $\langle \tau(G)\rangle \leq G$ contains each of the fundamental rank one subgroups, whence coincides with $G$. The proof for $\ol{G}$ is the same.
\end{proof}

\begin{proposition}\label{tauGexplicit}
The following assertions hold.
  \begin{enumerate}
\item $\tau(G) = \tau(U_+A) = U_+ \ast A$ and  $\ol{\tau}(\ol{G}) = \ol{\tau}(\ol{U}_+\ol A) = \ol{U}_+ \ast \ol A$.
\item $\tau(G) \subset U_+AU_- \subset G$ and  $\ol{\tau}(\ol{G}) \subset \ol{U}_+\ol{A}\,\ol{U}_- \subset \ol{G}$; more precisely,
\[
\tau(G) = \{u_+au_- \in U_+AU_-\mid u_- = \theta(u_+)^{-1}\}, \quad \ol{\tau}(\ol{G})= \{u_+au_- \in \ol{U}_+\ol{A}\,\ol{U}_-\mid u_- = \theta(u_+)^{-1}\}.
\]
\item Every $g \in \tau(G)$ (respectively, $g \in \ol{\tau}(\ol{G})$) can be written as $g = \tau(u_1 \cdots u_mt)$ with $t 
\in A$ (respectively, $t \in \ol{A}$) and $u_i\in U_{\beta_i}$ (respectively, $u_i\in \ol{U}_{\beta_i}$) for some $\beta_i \in \Phi^+$.
\item If the generalized Cartan matrix $\mathbf{A}$ is two-spherical, then every $g \in \tau(G)$ (respectively, $g \in \ol{\tau}(\ol{G})$) can be written as $g = \tau(u_1 \cdots u_mt)$ with $t 
\in A$ (respectively, $t \in \ol{A}$) and $u_i\in U_{\beta_i}$ (respectively, $u_i\in \ol{U}_{\beta_i}$) for some $\beta_i \in \Pi$.
\end{enumerate}
\end{proposition}
\begin{proof} 
By the Iwasawa decomposition (see Theorem~\ref{ThmIwasawa}), every $g \in G$ can be written as $g = uhk$ with $u \in U_+$, $h \in A$ and $k \in K$. Then $x := \tau(g) = \tau(uh) = u \ast \tau(h)$ by Lemma~\ref{lem:tau-K-saturation}. Now $\tau(A) = A$ by Lemma~\ref{tauonT}, and hence $\tau(G) = U_+ \ast A$. Assertion (i) follows. 

If $u_+ \in U_+$ and $h \in A$, then $\tau(u_+h) = u_+ * \tau(h) = u_+h^2\theta(u_+)^{-1}$ by Lemma~\ref{tauonT}. Moreover, $h^2 \in A$ and $\theta(u_+)^{-1} \in U_-$
 by Lemma~\ref{CCinvB}. Thus (ii) follows from (i) and the fact that every element of $A$ has a square root in $A$.

 Finally, since $A$ normalizes $U_+$, it follows from (i) that $\tau(G) = \tau(U_+A) = \tau(AU_+)$. Then (iii) and (iv) follow from the fact that $U_+$ is generated by the $(U_{\alpha})_{\alpha \in \Phi^+}$ (see \cite[Theorem~8.84]{AbramenkoBrown2008}) and even by the $(U_{\alpha})_{\alpha \in \Pi}$ in the two-spherical case (see \cite[Corollary~1.2]{DevillersMuehlherr} and note from its proof that two-sphericity suffices for the generation result, only the validity of the given presentation requires three-sphericity).

The proofs for $\ol{G}$ are similar.
\end{proof}
It follows from Proposition \ref{tauGexplicit} and continuity of $\theta$ that the map
\begin{equation}\label{U+AToTauG}
h: U_+ \times A \to \tau(G), \quad (u_+,h)  \mapsto  u_+h\theta(u_+)^{-1}.
\end{equation}
is a continuous bijection. We do not currently know whether it is always a homeomorphism. This problem is closely related to the problem whether the continuous bijection $m: U_+ \times A \times U_- \to U_+AU_-$ is always a homeomorphism (as is the case if $\mathbf{A}$ is two-spherical by Proposition~\ref{TopologyGSummary}(vi), but probably holds in much greater generality as discussed Remark \ref{twospherical1}). 
\begin{corollary}\label{TopologytauGexplicit} If the map  $m: U_+ \times A \times U_- \to U_+AU_-$ is a homeomorphism, then the continuous bijection $h: U_+ \times A \to \tau(G)$ from \eqref{U+AToTauG} is a homeomorphism whose inverse is given explicitly by
\[
h^{-1}: \tau(G) \into U_+AU_- \xto{m^{-1}} U_+ \times A \times U_- \to U_+ \times A,
\]
where the first map is the inclusion and the last map is the canonical projection that forgets the last component. In particular, this holds if $\mathbf{A}$ is two-spherical.
\end{corollary}
\begin{proof} Since $h$ is a continuous bijection, only its openness remains to show. It is immediate from the definitions that $h^{-1} \circ h$ is the identity, hence $h^{-1}$ is indeed the inverse of $h$ and openness of $h$ is equivalent to continuity of $h^{-1}$, which follows from continuity of $m^{-1}$.
\end{proof}
The same argument also shows that there is a homeomorphism 
\[
 \ol{\tau}(\ol{G})\to \ol{U}_+ \times \ol{A},
\]
given by the same formula.

\section{Models for Kac--Moody symmetric spaces}

\subsection{Topological symmetric spaces from involutions}
Let $G$ be an arbitrary topological group, let $\theta \in \Aut(G)$ be a continuous involution and let $K = G^\theta$. In this generality one can introduce a twist map
\begin{eqnarray*}
  \tau: G & \to & G \\ g & \mapsto & g\theta(g)^{-1}
\end{eqnarray*}
  as in Definition~\ref{DefTwist}, which will satisfy the properties described in Lemma~\ref{lem:tau-K-saturation}. Since $\theta$ is continuous, $K$ is a closed subgroup of $G$, and thus $G/K$ is a Hausdorff topological space with respect to the quotient topology. Using the involution $\theta$ and the associated twist map $\tau$ one defines a multiplication map
\begin{eqnarray}
\mK : G/K \times G/K & \to & G/K \label{mugeneral} \\ (gK,\,hK) & \mapsto & \tau(g)\theta(h)K. \notag
\end{eqnarray}

Note that $\mK$ is continuous, since $\tau$, $\theta$ and the group multiplication are.

\begin{proposition}\label{PropCosetModel}
If
\begin{equation}\label{ConditionForSymmetryOfReflSpace}
K\cap\tau(G)=\{e\},
\end{equation}
then the pair $(G/K,\mK)$ is a topological symmetric space and the natural action
\begin{eqnarray*}
G & \to & \Sym(G/K) \\ g & \mapsto & (aK \mapsto gaK)
\end{eqnarray*}
is by automorphisms.
\end{proposition}

\begin{proof} For  $a,b,c\in G$ one computes:
\begin{enumerate}[label={(RS\arabic*)},leftmargin=*]
\item $\mK(aK,\,aK) = \tau(a)\theta(a)K = aK$,
\item $\mK(aK,\,\mK(aK,\,bK))
= \mK(aK,\,\tau(a)\theta(b)K)
= \tau(a)\theta(\tau(a)\theta(b)) K
= bK$,
\item $\begin{aligned}[t]
\mK(aK,\,\mK(bK,\,cK))
&= \mK(aK,\,\tau(b)\theta(c)K)
= \tau(a) \theta(\tau(b)\theta(c))K \\
&= \tau(a) \theta(b) b^{-1} \theta( \theta(c))K \\
&= \tau(a)\theta(b) b^{-1} \tau(a) \theta( \tau(a)\theta(c)) K \\
&= \tau( \tau(a)\theta(b)) \theta( \tau(a)\theta(c)) K \\
&= \mK(\tau(a)\theta(b)K,\,\tau(a)\theta(c)K)
= \mK(\mK(aK,\,bK),\,\mK(aK,\,cK)),
\end{aligned}$
\item $\begin{aligned}[t]
\mK(aK,bK)=bK
&\iff \tau(a)\theta(b)K = bK
\iff  b^{-1}a\theta(a)^{-1}\theta(b) = \tau(b^{-1}a)\in K \\
&\overset{\eqref{ConditionForSymmetryOfReflSpace}}{\iff} \tau(b^{-1}a)=e
\iff \tau(a)=\tau(b)
\iff aK = bK.
\end{aligned}$
\end{enumerate}
Since $\mK$ is continuous, this establishes that $(G/K,\mK)$ is a topological symmetric space. The second statement follows from the fact that for $a,b, g\in G$ one has
\[ \mK(gaK,\,gbK)
= \tau(ga)\theta(gb)K
= g \tau(a) \theta(g)^{-1} \theta(g)\theta(b)K
= g \mK(aK,\,bK).
\qedhere
\]
\end{proof}

\subsection{Reduced and unreduced Kac-Moody symmetric spaces}

We are now ready to associate symmetric spaces with a large class of Kac--Moody groups. We choose to work in the following general setting.

\begin{convention}\label{ConventionRealKM}
   The matrix  ${\mathbf{A}} \in M_n(\Z)$ denotes a generalized Cartan matrix of size $n \times n$ and rank $l \leq n$, subject to the restrictions given in Convention~\ref{ConventionCartanMatrix}. That is, $\mathbf{A}$ is assumed to be irreducible and symmetrizable.

  The group $G := G_\R({\mathbf{A}})$ denotes the associated simply connected centred split real Kac--Moody group, and $\ol{G}$ denotes its semisimple adjoint quotient, cf. Definition~\ref{semisimpleadjoint}. $\theta$ and $\ol{\theta}$ denote the Cartan--Chevalley involutions on $G$, respectively $\ol{G}$, and $K$ and $\ol{K}$ denote their respective fixed point groups.
\end{convention}

Recall from Proposition~\ref{lem:1=G-cap-tauG} that $K\cap\tau(G)=\{e\}$ and $\ol{K}\cap\ol{\tau}(\ol{G})=\{e\}$. It thus follows from Proposition~\ref{PropCosetModel}  that both $G/K$ and $\ol{G}/\ol{K}$ carry the structure of a topological symmetric space given by $(gK,hK) \mapsto \mu(gK,hK) = \tau(g)\theta(h)K$.

\begin{definition}
\begin{enumerate}
\item $(G/K, \mK)$ is called the \Defn{unreduced Kac-Moody symmetric space} associated with ${\mathbf{A}}$.
\item  $(\ol{G}/\ol{K}, \ol{\mK})$ is called the \Defn{reduced Kac-Moody symmetric space} associated with ${\mathbf{A}}$.
\end{enumerate}
\end{definition}

If ${\mathbf{A}}$ is invertible, then by Proposition~\ref{kernelsemisimpleadjoint}(i) both versions of the Kac--Moody symmetric space coincide; in this case they are referred to as \emph{the} Kac--Moody symmetric space associated with ${\mathbf{A}}$. In general, however, these two spaces behave quite differently.
Note that $\ol{G}/\ol{K} = \Ad(G)/\Ad(K)$, since the center of $\ol{G}$ is contained in $\ol{K}$ by Theorem~\ref{ThmIwasawa}(i), i.e., the three different groups $G$, $\ol{G}$, $\Ad(G)$ do not lead to a third version of a Kac--Moody symmetric space.

\medskip
A first observation is that the unreduced Kac--Moody symmetric space $(G/K, \mK)$ fibers over the reduced Kac--Moody symmetric space with fiber $\E^{n-l}$.

\begin{proposition} 
\begin{enumerate}
\item The canonical projection $\pi_{{\mathbf{A}}}: G/K \to \ol{G}/\ol{K}$ is a morphism of topological reflection spaces.
\item The fiber $\pi_{{\mathbf{A}}}^{-1}(e\ol{K})$ is isomorphic to $\E^{n-l}$ as a topological reflection space. 
\end{enumerate}
\end{proposition}
\begin{proof} (i) Denote the projection $G \to \ol{G}$ by $g \mapsto [g]$ so that $\pi_{{\mathbf{A}}}$ is given by $\pi_{{\mathbf{A}}}(gK) = [g]\ol{K}$. Then for all $g,h \in G$ one has
\[
\pi_{{\mathbf{A}}}(gK) \cdot \pi_{{\mathbf{A}}}(hK)= [g]\ol{K} \cdot [h]\ol{K} \stackrel{\eqref{mugeneral}}{=} \ol{\tau}([g])\ol{\theta}([h])\ol{K} = [\tau(g)\theta(h)]\ol{K} = \pi_{{\mathbf{A}}}(\tau(g)\theta(h)K) = \pi_{{\mathbf{A}}}(gK\cdot hK). 
\]
(ii) By definition, $\pi_{{\mathbf{A}}}^{-1}(e\ol{K}) = CK/K \cong C/(C\cap K) \cong (\R_{>0})^{n-\rk(\mathbf{A})}$, where the second isomorphism follows from Proposition~\ref{kernelsemisimpleadjoint} and Lemma~\ref{tauonT}(iii). One can parametrize this fiber via
\[
\phi_o: \mathfrak c \cap \mathfrak a \to C/(C \cap K), \quad X \mapsto \exp(X) (C \cap K).
\]
By endowing the vector space $\mathfrak c \cap \mathfrak a$ with its Euclidean reflection space structure this map becomes an isomorphism of reflection spaces. Indeed, if $X, Y \in \mathfrak c \cap \mathfrak a$, then
\begin{eqnarray*}
\phi_o(X) \cdot \phi_o(Y) &=& \exp(X)(C \cap K) \exp(Y) (C \cap K) \quad = \quad \tau(\exp(X))\theta(\exp(Y))(C \cap K)\\
&=& \exp(X)^2\exp(Y)^{-1}(C \cap K) \quad = \quad  \exp(2X-Y) (C \cap K)\\ 
&=& \phi_o(X\cdot Y). 
\end{eqnarray*}
Thus the parametrization is an abstract isomorphism of reflection spaces and, in fact, a topological isomorphism by Proposition~\ref{PropexpHomeomorphism}.
\end{proof}

\begin{lemma}\label{LemmaKernelofAction} The kernel of the action of $G$ on $G/K$ equals the centralizer $C_K(G)$ of $G$ in $K$ and the kernel of the action of $\ol{G}$ on $\ol{G}/\ol{K}$ equals the center $Z(\ol{G})$ of $\ol{G}$.
\end{lemma}

\begin{proof}
  Since $C_K(G) < K$, it acts trivially on $G/K$: for all $g \in C_K(G)$, $a \in G$ one has $gaK=agK=aK$. On the other hand, if $g \in G$ acts trivially on $G/K$, then for all $h \in G$ one has $ghK = hK$. In particular $gK = K$, i.e.\ $g \in K$ and, thus, $\theta(g) = g$. Lemma~\ref{lem:tau-K-saturation} implies
\begin{alignat*}{3}
hgh^{-1}\in K
&\Rightarrow \tau(h^{-1}gh) = e &
&\Rightarrow h^{-1} \ast \tau (gh) = e\\
&\Rightarrow g \ast \tau(h) =h \ast e &
&\Rightarrow g \ast \tau(h) = \tau(h)\\
&\Rightarrow g\tau(h)g^{-1} = \tau(h).
\end{alignat*}
Thus $g$ centralizes $\tau(G)$. Since $\tau(G)$ generates $G$ (see Proposition~\ref{tauGgenerates}), the element $g$ therefore centralizes $G$, i.e., $g \in C_K(G)$. The same argument shows that $g \in \ol{G}$ acts trivially on $\ol{G}/\ol{K}$ if and only if $g \in C_{\ol{K}}(\ol{G}) = Z(\ol{G})$ (cf.\ Theorem~\ref{ThmIwasawa}(i)).
\end{proof}

\begin{definition} \label{Geff}
Define \[\Geff := G/C_K(G).\] By Lemma~\ref{LemmaKernelofAction}  the group $\Geff$ then acts effectively (i.e., faithfully) on $G/K$. Similarly, $\Ad(G) = \ol{G}/Z(\ol{G})$ acts effectively on $\ol{G}/\ol{K}$.
\end{definition}

\begin{remark}
  By the topological Iwasawa decomposition Theorem \ref{ThmIwasawa} there exists a homeomorphism
\[
U_+\times {\ol{A}} \to \ol{G}/\ol{K}, \quad (u, a) \mapsto ua\ol{K}.
\]
This allows one to define the structure of a topological symmetric space on $U_+ \times {\ol{A}}$ by transporting the multiplication map via this homeomorphism. Unfortunately, at the moment we do not know of any good way of describing this induced multiplication map intrinsically, nor do we have an intrinsic description for the induced $G$-action on $U_+ \times {\ol{A}}$.

The key problem is to derive a formula of how to decompose a product $(k_1a_1u_1)(k_2a_2u_2)$ with respect to $\ol{K} \times {\ol{A}} \times U_+$. In the finite-dimensional situation this is achieved in \cite{Kostant73}.
\end{remark}

\subsection{Reflections, transvections and reflection-homogeneity} \label{model1}

Since $\theta$ stabilizes both $C_K(G)$ and $K$, it induces an involutive automorphism of $\Geff$ and an involutive permutation $\theta: G/K \to G/K$ via $\theta(gK) := \theta(g)K$. Defining the basepoint of $G/K$ as $o := eK$, one in fact has $s_{o}(gK) = \tau(e)\theta(g)K = \theta(gK)$, i.e., $\theta$ coincides with the point reflection $s_o$ of the symmetric space $G/K$ at $o$.
In particular, one obtains a subgroup
\[
\Geff \rtimes \gen{\theta} < \Aut(G/K, \mK).
\]

Similarly, by Proposition~\ref{CCIdescends} the Cartan--Chevalley involution $\theta$ induces an involution $\ol{\theta}: \ol{G} \to \ol{G}$ which in turn yields an involutive automorphism $\ol{\theta}: \ol{G}/\ol{K} \to \ol{G}/\ol{K}$ and a subgroup 
\[
\Ad(G)  \rtimes \gen{\ol\theta} < \Aut(\ol{G}/\ol{K},\ol{\mK}),
\]
where $\ol{\theta}$ corresponds to the point reflection at $\ol{o} := e\ol{K}$.

\begin{proposition}\ \label{PropMainTransKM} 
\begin{enumerate}
\item
The set of point reflections of $G/K$ (respectively, $\ol{G}/\ol{K}$) equals the conjugacy class of $s_o$ (respectively, $s_{\ol{o}}$) in $\Geff \rtimes \gen{\theta}$ (respectively, $\Ad(G)  \rtimes \gen{\ol\theta}$).
\item The set of transvections of $G/K$ (respectively, $\ol{G}/\ol{K}$) is given by $\tau(G)^2 C_K(G)$ (respectively, $\ol{\tau}({\ol{G}})^2Z(\ol{G})$), where $\tau(G)^2=\tau(G)\tau(G)$ (and analogously for $\ol{\tau}({\ol{G}})$).
\item The respective transvection groups of $G/K$ and $\ol{G}/\ol{K}$ are \[\Trans(G/K, \mK) = \Geff \qquad \text{ and } \qquad \Trans(\ol{G}/\ol{K}, \ol{\mK}) = \Ad(G).\] The main groups of $G/K$, respectively $\ol{G}/\ol{K}$ are given by \[G(G/K,\mu) = \Geff \rtimes \gen{\theta} \qquad \text{ and } \qquad G(\ol{G}/\ol{K},\ol{\mu})=\Ad(G)  \rtimes \gen{\ol\theta}.\]
\item $G/K$ and $\ol{G}/\ol{K}$ are reflection-homogeneous.
\end{enumerate}
\end{proposition}

\begin{proof}
For $g, h \in G$ one has
\begin{eqnarray}
s_{gK}(hK) &=& \mK(gK, hK) \quad = \quad  \tau(g)\theta(h)K \quad =\quad  g\theta(g)^{-1}\theta(h)K \notag \\
&=& g\theta(g^{-1}h)K = (gC_K(G) \circ s_o \circ g^{-1}C_K(G))(hK) \notag \\
& = & (gC_K(G) \circ \theta \circ g^{-1}C_K(G))(hK),    \label{factorsthroughGK}
\end{eqnarray}
i.e., $s_{gK}$ is conjugate to $s_o$ via $gC_K(G) \in \Geff$. 
Furthermore, observe that for $g \in G$ one has
\begin{equation}\label{QuadraticTwist}
s_{gK} \circ s_o = gC_K(G) \circ s_o \circ g^{-1}C_K(G) \circ s_o = g\theta(g)^{-1}C_K(G)= \tau(g)C_K(G),
\end{equation}
Given $g,h \in G$ therefore 
\[s_{gK}s_{hK} = (s_{gK}s_o)(s_{hK}s_o)^{-1} = \tau(g)\tau(h)^{-1}C_K(G) = \tau(g) \tau(\theta(h))C_K(G),\]
whence the transvections are exactly the elements of $\tau(G)^2 C_K(G) \supset \tau(G) C_K(G)$.
The other claims concerning $G$ now follow readily, using Proposition~\ref{tauGgenerates} and Lemma~\ref{GSIso}. The claims concerning $\ol{G}$ are shown analogously.
\end{proof}

\subsection{Models for Kac-Moody symmetric spaces} \label{model23}

Recall from Section~\ref{reflectionspaces} that every reflection-homogeneous symmetric space can be realized as a subset of its main group (the ``involution model'' from Lemma~\ref{GSIso}) and as a subset of its transvection group (the ``quadratic representation'' from Remark~\ref{TransvecRealization}) with suitably defined multiplications.

\medskip
In view of Example~\ref{involutionmodel} and Proposition~\ref{PropMainTransKM} the \Defn{involution model} of the reflection-homogeneous symmetric space $(G/K,\mu)$ is given by the pair $(\XXX,\mX)$ where \[\XXX := \{ {}^g\theta \in \Geff \rtimes \gen{\theta} \mid g \in \Geff \rtimes \gen{\theta} \} \qquad \text{ and } \qquad \mX: \XXX \times \XXX \to \XXX, \quad (\alpha, \beta) \mapsto \alpha \beta \alpha.\] 
The map $\pi: G \to \XXX, \quad g \mapsto {}^g \theta = gC_K(G) \circ \theta \circ g^{-1}C_K(G)$ by \eqref{factorsthroughGK} factors through $\hat\pi: G/K \to \XXX$, which is an isomorphism of reflection spaces.

\medskip
The \Defn{quadratic representation} of $G/K$ depends on the choice of a basepoint $o \in G/K$. For $o = eK$ by Proposition~\ref{PropMainTransKM} the quadratic representation is given by the map \[
t: G/K \to \Geff, \quad gK \mapsto s_{gK} \circ s_o.
\]
By \eqref{QuadraticTwist} one has $s_{gK} \circ s_o = \tau(g)C_K(G)$. Thus the image  $\TTT = T(G/K, \mK, o) \subset \Trans(G/K, \mK)$ 
of the quadratic representation of $G/K$ is given by the image of $\tau(G)$ in $\Geff$, and the product on $\TTT$ is given by $\widetilde{m}(s,t):= st^{-1}s$ by Remark~\ref{TransvecRealization}. Note that $\tau$ induces an isomorphism of reflection spaces $(G/K, \mK) \to (\TTT, \widetilde{m})$.

By definition, the canonical projection $G \to \Geff$ restricts to a surjection $\tau(G) \to \TTT$. Since the kernel of the projection $G \to \Geff$ is contained in $K$, it intersects $\tau(G)$ trivially by Proposition~\ref{lem:1=G-cap-tauG}. It follows that the projection $\tau(G) \to \TTT$ is actually bijective and so by transport of structure the multiplication \[\mT: \tau(G) \times \tau(G) \to \tau(G), \quad \mT(x,y) = xy^{-1}x\] provides a symmetric space such that \[(\tau(G), \mT) \cong (\TTT, \widetilde{m}) \cong (G/K, \mK).\] This symmetric space $(\tau(G), \mT)$ is called the \Defn{group model} of $G/K$.

 \medskip
The left-multiplication action of $G$ on $G/K$ translates into $G$-actions on $\TTT$ and $\tau(G)$ by automorphisms. Since $t(ghK) = \tau(gh)C_K(G) = g\ast\tau(h)C_K(G)$, the induced $G$-action on $\tau(G)$ is given by twisted conjugation. It follows that the isomorphisms $G/K \to \tau(G)$ is explicitly given by \[\hat\tau: G/K \to \tau(G) : gK = geK \mapsto g\ast \tau(e) = \tau(g).\] 
Combining the isomorphisms $\hat\tau: G/K \to \tau(G)$ and $\hat\pi: G/K \to \XXX$ one also obtain an isomorphism $\rho: \tau(G) \to \XXX$ making the diagram in Figure \ref{fig:presym-maps} commute. 
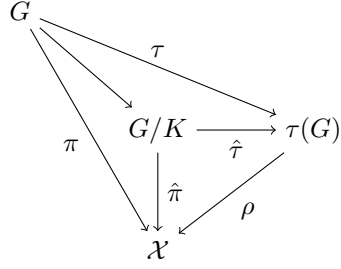
\begin{figure}[h]
\centering
\begin{tikzpicture}
\matrix (m) [matrix of math nodes, row sep=3em,
column sep=3em]
{
|[name=G]| G \\&|[name=GK]| G/K &|[name=tauG]| \tau(G) \\
 & |[name=X]| \XXX \\
};
\path[->] (G) edge node[above] {} (GK);
\path[->] (G) edge node[above]  {$\tau$} (tauG);
\path[->] (G) edge node[below left] {$\pi$} (X);
\path[->] (GK) edge node[right] {$\hat\pi$} (X);
\path[->] (GK) edge node[below] {$\hat\tau$} (tauG);
\path[->] (tauG) edge node[below right] {$\rho$} (X);
\end{tikzpicture}
\caption{Isomorphisms between the different models.}
\label{fig:presym-maps}
\end{figure}
Denoting by $[h]$ the image of $h \in G$ under the projection $G \to \Geff$ this isomorphism is explicitly given as follows.

\begin{lemma}\label{rhoFormula} 
Let $\rho : \tau(G) \to \XXX : h\mapsto [h]\theta$. Then $\rho$ makes the diagram in Figure \ref{fig:presym-maps} commute. In particular, it is an isomorphism of reflection spaces.
\end{lemma}
\begin{proof} It suffices to check that $\rho \circ \tau = \pi$. For this one computes
\[
\rho \circ \tau(g) = \rho(g\theta(g)^{-1}) = [g\theta(g)^{-1}]\theta = [g]\theta[g]^{-1}.\qedhere
\]
\end{proof}

\begin{remark} \label{modelGbar}
For each of the three models of the unreduced symmetric space there is a corresponding model of the reduced symmetric space. The \Defn{coset model} $\ol{G}/\ol{K}$ was already discussed above.
The \Defn{involution model} of $\ol{G}/\ol{K}$ is given by the conjugacy class $\ol{\XXX}$ of $s_{\ol o} = \ol{\theta}$ in $\Ad(G) \rtimes \langle \ol{\theta}\rangle$. Since the latter group can be embedded as a subgroup into the automorphism groups $\Aut(G) < \Aut(\Delta)$ of the group $G$ and\footnote{See Proposition~\ref{AutBuilding} below for the fact that $\Aut(G)$ embeds into $\Aut(\Delta)$.}  its twin building $\Delta$, one can consider $\ol{\XXX}$ both as a set of involutions of the group $G$ and of the twin building $\Delta$. In either of these pictures, the multiplication is given by \[\mX(\alpha, \beta) = \alpha\circ \beta^{-1} \circ \alpha.\] 

The \Defn{group model} of $\ol{G}/\ol{K}$ is given by $(\ol{\tau}(\ol{G}), \mT)$ with multiplication given by \[\mT(x,y) = xy^{-1}x.\] 

As in the unreduced model one has isomorphism between these models as depicted in Figure~\ref{fig:presym-maps2}. Here the isomorphism
$\ol{\rho} :  \ol{\tau}(\ol{G}) \to\ol{\XXX} \subset \Aut(G)$ is given by
\begin{equation} \label{formularho}
  \ol{\rho}(g) = c_g \circ \theta,
\end{equation}
  where $c_g$ denotes the inner automorphism defined by $g$.
\end{remark}
  
\begin{figure}[h]
\centering
\begin{tikzpicture}
\matrix (m) [matrix of math nodes, row sep=3em,
column sep=3em]
{
|[name=G]| \ol{G} \\&|[name=GK]| \ol{G}/\ol{K} &|[name=tauG]| \ol{\tau}(\ol{G}) \\
 & |[name=X]| \ol{\XXX} \\
};
\path[->] (G) edge node[above] {} (GK);
\path[->] (G) edge node[above]  {$\ol{\tau}$} (tauG);
\path[->] (G) edge node[below left] {$\ol{\pi}$} (X);
\path[->] (GK) edge node[right] {$\hat{\ol{\pi}}$} (X);
\path[->] (GK) edge node[below] {$\hat{\ol{\tau}}$} (tauG);
\path[->] (tauG) edge node[below right] {$\ol{\rho}$} (X);
\end{tikzpicture}
\caption{Isomorphisms between the reduced models.}
\label{fig:presym-maps2}
\end{figure}
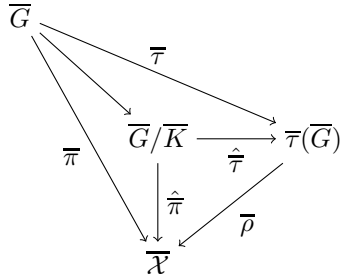

\subsection{Comparison of topologies}

Sections~\ref{model1} and \ref{model23} provided three mutually isomorphic models of the reduced and unreduced Kac--Moody symmetric space --- the coset models $(G/K, \mK)$ and $(\ol{G}/\ol{K}, \ol{\mK})$, the involution models $({\XXX}, \mX)$ and $(\ol{\XXX}, \mX)$, and the group models $(\tau({G}), \mT)$ and $(\ol{\tau}(\ol{G}), \mT)$. 
\begin{convention}\label{ConventionTopologies}
In the sequel we will equip the reflections spaces above with the quotient topologies with respect to the canonical projections $G \to G/K$ and $\ol{G} \to {\ol{G}}/\ol{K}$, respectively the maps $\pi$ and $\ol{\pi}$, respectively the maps $\tau$ and $\ol{\tau}$, unless explicitly stated otherwise. We refer to these topologies as the \Defn{external topologies} on the reflection spaces in question.
\end{convention}
Proposition~\ref{thm1.5} from the introduction now is an immediate consequence of Proposition~\ref{PropCosetModel}, Lemma~\ref{rhoFormula} and Remark~\ref{modelGbar}:
\begin{corollary}\label{TopRefSpaceExternal} With the external topologies from Convention \ref{ConventionTopologies} the reflection spaces $(G/K, \mK)$, $({\XXX}, \mX)$ and $(\tau({G}), \mT)$ (respectively $(\ol{G}/\ol{K}, \ol{\mK})$, $(\ol{\XXX}, \mX)$ and $(\ol{\tau}(\ol{G}), \mT)$) are mutually isomorphic topological reflection spaces.\qed
\end{corollary}
One may ask whether one can describe the canonical topologies of the coset and involution model in more intrinsic terms, without reference to the quotient maps above. We discuss this here for reduced symmetric spaces.
\begin{definition}
\begin{enumerate}
\item The \Defn{internal topology} on $\ol{\tau}(\ol{G})$ is defined as the subspace topology via the embedding $\ol{\tau}(\ol{G}) \hookrightarrow \ol{G}$.
\item The \Defn{internal topology} on $\ol{\XXX}$ is defined as follows: Equip $\Ad(G)$ with the quotient topology with respect to the canonical projection $G \to \Ad(G)$  or, equivalently, $\ol{G} \to \Ad(G)$. Then equip $\Ad(G) \rtimes \langle \ol\theta \rangle$ with the unique group topology in which the finite index subgroup $\Ad(G)$ is open and carries the quotient topology just defined. Finally, equip $\ol{\XXX} \subset \Ad(G) \rtimes \langle \ol\theta \rangle$ with the subspace topology.
\end{enumerate}
\end{definition}
\begin{proposition} \label{topologiescoincide} Equip $\ol{G}/\ol{K}$ with its external topology and $\ol{\tau}(\ol{G})$ and $\ol{\XXX}$ with their internal topologies. 
\begin{enumerate}
\item The maps $\hat{\ol{\pi}}$ and $\hat{\ol{\tau}}$ in Figure~\ref{fig:presym-maps2} are continuous and the map $\ol{\rho}$ in Figure~\ref{fig:presym-maps2} is a homeomorphism.
\item If the multiplication map $m: \ol{U}_+ \times \ol{A} \times \ol{U}_- \to \ol{U}_+\ol{A}\ol{U}_-$ is a homeomorphism, then each of the maps $\hat{\ol{\pi}}$, $\hat{\ol{\tau}}$ and $\ol{\rho}$ in Figure \ref{fig:presym-maps2} is a homeomorphism.
\item If the multiplication map $m: \ol{U}_+ \times \ol{A} \times \ol{U}_- \to \ol{U}_+\ol{A}\ol{U}_-$ is a homeomorphism, then the spaces $(\ol{\XXX}, \mX)$ and $(\ol{\tau}(\ol{G}), \mT)$ are topological reflection spaces with respect to their internal topologies. Moreover, the internal and external topologies on these spaces coincide, and they are isomorphic as topological reflection spaces to each other and to $(\ol{G}/\ol{K}, \ol{\mK})$.
\end{enumerate}
\end{proposition}
\begin{proof} By the commuting diagram in Figure \ref{fig:presym-maps2} it suffices to investigate the maps $\hat{\ol{\tau}}$ and $\ol{\rho}$.
 
(i) The map $\widehat{\ol{\tau}}$ is continuous, since the twist map is continuous. Similarly, continuity of $\ol{\rho}$ follows from formula \eqref{formularho} for $\ol{\rho}$ in Remark~\ref{modelGbar}. It remains to show that $\ol{\rho}$ is open. Proposition~\ref{lem:1=G-cap-tauG} and Theorem~\ref{ThmIwasawa}(i) imply $\ol{\tau}(\ol{G}) \cap Z(\ol{G}) \leq \ol{\tau}(\ol{G}) \cap \ol{K} = \{ e \}$. One concludes that $\ol{\tau}(\ol{G})$ embeds into $\Ad(G)$. 
After identifying $\ol{\tau}(\ol{G})$ with its image in $\Ad(G)$ according to \eqref{formularho} the map $\ol{\rho}^{-1} : \XXX \to \tau(G)$ is given by $\psi \mapsto \psi \circ \theta^{-1}$. Since $\Ad(G) \rtimes \langle \ol\theta \rangle$ is a topological group, $\ol{\rho}^{-1}$ is continuous, and hence $\ol{\rho}$ is open, i.e., a homeomorphism. 

From now on we assume that the map $m: \ol{U}_+ \times \ol{A} \times \ol{U}_- \to \ol{U}_+\ol{A}\ol{U}_-$ is a homeomorphism.

(ii) The map $\widehat{\ol{\tau}}$ is continuous, since the twist map is continuous. For the openness of $\hat{\ol{\tau}}$ note that by the topological Iwasawa decomposition (Theorem~\ref{ThmIwasawa}) there is a homeomorphism $h_1: \ol{U}_+ \times {\ol{A}} \to \ol{G}/\ol{K}$ given by $(u_+, a) \mapsto u_+a\ol{K}$. On the other hand, by Corollary \ref{TopologytauGexplicit} there is a homeomorphism $h: \ol{U}_+ \times {\ol{A}} \to  \ol{\tau}(\ol{G})$ given by $h(u_+, a) = u_+a\theta(u_+)^{-1}$. It thus suffices to show that the composition
\[
h_2: \ol{U}_+ \times {\ol{A}} \xto{h_1} \ol{G}/\ol{K} \xto{\widehat{\ol{\tau}}} \ol{\tau}(\ol{G}) \xto{{h^{-1}}} \ol{U}_+ \times {\ol{A}}
\]
is open. Now $\widehat{\ol \tau} \circ h_1(u_+, a) = u_+a^2\theta(u_+)^{-1}$ and, hence, $h_2(u_+, a) = (u_+, a^2)$. Now openness of $h_1$, and hence of $\widehat{\ol{\tau}}$, follows from the fact that the map ${\ol{A}} \to {\ol{A}} : a \mapsto a^2$ is open.

(iii) This is immediate from (ii) and Corollary \ref{TopRefSpaceExternal}.
\end{proof}

Note that the assumption in (ii) and (iii) is satisfied in the two-spherical case, but probably holds more generally (see Remark \ref{twospherical1}). In order to establish a version of the proposition for unreduced Kac--Moody symmetric spaces, one would need to extend the topological Iwasawa decomposition to the unreduced case.


\section{Flats and geodesics in Kac--Moody symmetric spaces}

Throughout this section $G$ denotes a simply connected centred split real Kac--Moody group of irreducible symmetrizable type, the group $\ol{G}$ denotes its semisimple adjoint quotient, and $\Ad(G)$ its adjoint quotient. Moreover, $\Delta = \Delta^-\sqcup \Delta^+$ denotes the twin building associated to the RGD systems of these groups.

The purpose of this section is to investigate the flats of the Kac--Moody symmetric spaces $G/K$ and $\ol{G}/\ol{K}$. 

\subsection{Standard flats}
We start by constructing explicit examples of Euclidean flats in Kac--Moody symmetric spaces. We will see in Theorem \ref{maximalflatsareeuclidean} below that these are exactly the maximal flats. Recall from Proposition~\ref{PropexpHomeomorphism}(i) that we have homeomorphisms $\exp: \mathfrak a \to A$ and $\exp: \ol{\mathfrak a} \to \ol{A}$. 
\begin{proposition}\label{StdFlatsEuclidean} 
Equip $\mathfrak a$ (respectively $\ol{\mathfrak a}$) with its Euclidean reflection space structure. Then for every $g \in G$ (respectively, $\ol{g} \in \ol{G}$) the map
\[
\varphi_g: \mathfrak a \to gAK, \quad  X\mapsto g\exp(X)K \quad (\text{respectively}, \;\varphi_{\ol{g}}: \ol{\mathfrak a} \to \ol{g}\ol{A}\ol{K}, \quad  X\mapsto g\exp(X)\ol{K} )
\]
is an isomorphism of topological reflection spaces. Moreover, the subset $gAK \subset G/K$ (respectively, $\ol{g}\ol{A}\,\ol{K} \subset \ol{G}/\ol{K}$) is closed, hence a Euclidean flat of dimension $\dim \mathfrak a = n$ (respectively, $\dim \ol{\mathfrak a} = \rk(\bf A)$).
\end{proposition}

\begin{proof} First observe that the subsets $gAK \subset G/K$ (respectively, $\ol{g}\ol{A}\,\ol{K} \subset \ol{G}/\ol{K}$) are closed. By Theorem~\ref{ThmIwasawa}, multiplication $\ol{U}_\pm \times \ol{A} \times \ol{K} \to G$ induces a homeomorphism. Therefore, $\ol{A}\,\ol{K}$ and any of its translates $\ol{g}\ol{A}\,\ol{K}$ are closed in $\ol{G}/\ol{K}$, and so are the preimages $gAK$ in $G/K$. It remains to show that the maps $\varphi_g$ are isomorphisms of reflection spaces. Since both $G$ and $\ol{G}$ act by automorphisms, one may assume that $g = e$, respectively $\ol{g} = e$. Thus let $X, Y \in \mathfrak a$. Using that $\theta(t) = t^{-1}$ for all $t \in A = \tau(A)$ (see Lemmas~\ref{lem:tau-K-saturation} and \ref{tauonT}) and that $\exp$ is a group homomorphism one computes
\begin{eqnarray*}
\mK(\varphi_e(X), \varphi_e(Y))
&= &\tau(\exp(X))\theta(\exp(Y))K\\
&= &\exp(X)\theta(\exp(X))^{-1}\theta(\exp(Y))K\\
&= &\exp(X)\exp(X) \exp(-Y)K\\
&= &\exp(2X-Y)K\\
&= &\varphi_e(X\cdot Y),
\end{eqnarray*}	
and the computation for the reduced case is identical. 
\end{proof}
\begin{definition} For every $g \in G$ (respectively, $\ol{g} \in \ol{G}$) the flat $gAK \subset G/K$ (respectively, $\ol{g}\ol{A}\,\ol{K} \subset \ol{G}/\ol{K}$) is called a \Defn{standard flat}. 
\end{definition}
The following proposition describes images of standard flats under the various isomorphisms of models. By abuse of language we will also refer to these images as standard flats in the respective models.
\begin{proposition}\  \label{propthreemodelsflats}
\begin{enumerate}
\item The image of the standard flat $gAK$ under the isomorphism $\hat \pi: G/K \to \XXX$ is given by 
\[
\XXX_{{}^gT} := \{\alpha \in \XXX\,|\,\alpha({}^gT) \subseteq {}^gT\} =  \{\alpha \in \XXX\,|\,\alpha({{}^gT})= {{}^gT}\}.
\]
\item The image of the standard flat $gAK$ under the isomorphism $\hat \tau: G/K \to \tau(G)$ is given by
\[
F[g] := g \ast A = g \ast \tau(A) \subset \tau(G).
\]
\end{enumerate}
The analogous statements hold for $G$ replaced by $\ol{G}$.
\end{proposition}

\begin{proof}
  Observe that the two descriptions of $\XXX_{{}^gT} $ indeed coincide because $\XXX$ consists of involutions.
Moreover, the maps $g \mapsto gAK$ and $g \mapsto \XXX_{{}^gT}$ and $g \mapsto F[g]$ are all equivariant under the respective $G$-actions. It therefore suffices to show that \begin{equation}\label{StdFlatToShow}
\hat\pi(AK) = \XXX_T \quad \text{and} \quad \hat \tau(TK) = A.
\end{equation}
Certainly, $\hat\pi(AK) \subseteq \pi(T) \subseteq \XXX_T$. Conversely, let ${}^h\theta \in \XXX_T$. This means that
\[
 T = {}^h\theta(T) = (h \circ \theta \circ h^{-1})(T) = h\theta(h^{-1}Th)h^{-1} = \tau(h) \theta(T) \tau(h)^{-1} = \tau(h) T \tau(h)^{-1}.
\]
Hence $\tau(h)\in N_G(T)$. By Corollary~\ref{lem:T=N-cap-tauG} and Lemma~\ref{tauonT} one has $N_G(T)\cap\tau(G)=A=\tau(A)$, so there is $t\in A$ such that $\tau(h) = \tau(t)$ and, therefore, $tK=hK$ by Lemma~\ref{lem:tau-K-saturation}. Thus $hK = tK \in AK$, showing that ${}^h\theta = {}^t\theta = \hat\pi(tK) \in \hat\pi(AK)$ and hence
\[
 \hat\pi(AK) = \pi(T) = \XXX_T.
\]
Finally, $\hat\tau(TK) = \tau(T) = A$. This establishes~\eqref{StdFlatToShow} and finishes the proof.
\end{proof}
\begin{remark} \label{5.4} Denote by $\FFF_{\rm std}(G/K)$ the set of standard flats in $G/K$. By definition, $G$ acts transitively on $\FFF_{\rm std}(G/K)$ via left-multiplication. Recall from Lemma~\ref{lem:NGT=A NKT} that $N_G(T) = A \rtimes N_K(T)$. Since $A$ is the identity component of $T$ (see Definition~\ref{semisimpleadjoint}), one has $N_K(T)\leq N_K(A)$, since conjugation in $G$ is continuous. Conversely, by \cite[Lemma~4.9]{Caprace09} the torus $T$ is the unique torus of $G$ containing $A$, so any element normalizing $A$ necessarily has to normalize $T$, and one deduces that
\begin{equation}
N_K(A) = N_K(T).
\end{equation} 
Thus every $g = N_G(T)$ can be written as $g = ak$ with $a\in A$ and $k \in N_K(A)$, and thus $gAK = akAK = a(kAk^{-1})kK = aAK = AK$. In other words, $N_G(T)$ stabilizes $AK$.

The coset space $G/N_G(T)$ can be identified with the set $\TTT(G)$ of maximal tori of $G$ via the map $gN_G(T) \mapsto {}^gT$. One thus obtains a $G$-equivariant surjection
\begin{equation}\label{TorusParametrization}
\TTT(G) \to \FFF_{\rm std}(G/K), \quad {}^gT \mapsto gAK.
\end{equation}
In other words, the standard flats are parametrized by the maximal tori. The same argument applies to $\ol{G}$ instead of $G$. 

Assertion (ii) of Proposition~\ref{propthreemodelsflats} implies that the parametrization map in \eqref{TorusParametrization} is actually a bijection: Indeed, the standard flat associated with ${}^gT$ in the group model is given by $F[g] = g \ast A = gA\theta(g)^{-1}$, and one has
\[
F[g]\theta(F[g]) = gA\theta(A)g^{-1} = gAg^{-1}.
\]
One can therefore recover $gAg^{-1}$ from the associated flat. Now by \cite[Lemma~4.9]{Caprace09} the group $gAg^{-1}$ is contained in a unique maximal torus of $G$, and this maximal torus is exactly ${}^gT$. Thus $F[g]$ determines $^{g}T$, and the map \eqref{TorusParametrization} is thus bijective.

The same argument applies to maximal tori in $\ol{G}$, as $C < T$ (cf.\ Definition~\ref{semisimpleadjoint}) is central in $G$, whence contained in any $G$-conjugate of $T$ and, moreover, stabilized by any conjugate of $\theta$.
\end{remark}
Note that maximal tori in $G$ are precisely the chamberwise stabilizers of the twin apartments of the twin building $\Delta$, as are the maximal tori in $\ol{G}$. Altogether one observes the following:
\begin{corollary}\label{StdFlatTori}
  The following objects are in $G$-equivariant bijection with the elements of $G/N_G(T)= \ol{G}/N_{\ol{G}}(\ol{T})$:
\begin{enumerate}
\item twin apartments of $\Delta$,
\item maximal tori of $\ol{G}$,
\item maximal tori of $G$,
\item standard flats in $G/K$,
\item standard flats in $\ol{G}/\ol{K}$.
\end{enumerate}
In particular, $G$ acts transitively on these objects, every standard flat in $G/K$ projects to a standard flat in $\ol{G}/\ol{K}$, and every standard flat in $\ol{G}/\ol{K}$ lifts uniquely to a standard flat in $G/K$. \qed
\end{corollary}
By Theorem~\ref{maximalflatsareeuclidean} below the standard flats in either of the two Kac--Moody symmetric spaces are exactly the maximal flats. This in turn implies that the maximal flats in $\ol{G}/\ol{K}$ are in one-to-one correspondence to the maximal flats in $G/K$.

\subsection{Midpoint convex subsets and geodesic connectedness} \label{geodesicconnectedness}

Our next goal is to characterize midpoint convex subsets of Kac--Moody symmetric spaces.
The following definition borrowed from \cite[Section~4.2.2]{Caprace09} is key to this characterization.

\begin{defn}
  An element $g\in G$ (or $\ol{g} \in \ol{G}$) is called \Defn{diagonalizable} if it stabilizes a pair of opposite chambers in $\Delta$ and, hence, stabilizes a twin apartment chamberwise.
\end{defn}

The following example shows that, in the non-spherical case, elements of $\ol{\tau}(\ol{G})$ need not be diagonalizable. The reader is referred to \cite{Horn:Decomp} for a more detailed discussion of this theme.

\begin{example}\label{Hole}
Let $n\geq 1$ and consider the affine example $\ol{G}:=\SL_{n+1}(\R[t,t^{-1}])$ of type $\tilde{A}_n$ with the Cartan--Chevalley involution $\theta(x):=((x^{-1})^T)^\sigma$, where $\sigma$ is the ring automorphism of $\R[t,t^{-1}]$ which fixes $\R$ and interchanges $t$ and $t^{-1}$. Then let
\begin{align*}
u :=  \left(\begin{smallmatrix} 1 & 1+t \\ 0 & 1 \\ && \ddots \\ &&& 1 \end{smallmatrix}\right ) \in B_+,\qquad
v:=\tau(u) &=  u\theta(u)^{-1} =
\left(\begin{smallmatrix} 1 & 1+t \\ 0 & 1 \\ && \ddots \\ &&& 1\end{smallmatrix}\right )\cdot
\left(\begin{smallmatrix} 1 & 0 \\ 1+t^{-1} & 1 \\ && \ddots \\ &&& 1\end{smallmatrix}\right ) \\
&=
\left(\begin{smallmatrix} 1 + (1+t)(1+t^{-1}) & 1+t \\ 1+t^{-1} & 1
 \\ && \ddots \\ &&& 1\end{smallmatrix}\right )
\end{align*}
and the characteristic polynomial of $v$ is
\begin{align*}
c_\lambda(v)
&= \left( (\lambda - (1 + (1+t)(1+t^{-1}))(\lambda-1) - (1+t)(1+t^{-1}) \right) \cdot (\lambda-1)^{n-1} \\
&= \left(\lambda^2 - ( t + 4 +  t^{-1}) \lambda + 1 \right) \cdot (\lambda-1)^{n-1}.
\end{align*}
However, the polynomial $c_\lambda(v)$ does not split into linear factors over $\R[t,t^{-1}]$, whence $v$ is not conjugate within $\ol{G}$ to an element of the torus $\ol{T}$, which consists of diagonal matrices with entries from $\R$.
\end{example}

The following result demonstrates that the behaviour described in the preceding example is not merely an affine but instead a general non-spherical phenomenon: 

\begin{theorem}[{\cite[Theorem~5.7 and Proposition~6.3]{Horn:Decomp}}]\label{HoleThm}
The set $Q:=\bigcap_{i=1}^\infty \tau^i(G)$ (respectively, $\ol{Q}:=\bigcap_{i=1}^\infty \ol{\tau}^i(\ol{G})$) equals the set of diagonalizable elements in $\tau(G)$ (respectively, $\ol{\tau}(\ol{G})$). Moreover, if $G$ is of non-spherical type, then $Q\neq\tau(G)$ and $\ol{Q} \neq \ol{\tau}(\ol{G})$, i.e, both $\tau(G)$ and $\ol{\tau}(\ol{G})$ contain elements which are not diagonalizable.
\end{theorem}
 
 The description of the set of diagonalizable elements in $\tau(G)$, respectively $\ol{\tau}(\ol{G})$ has the following implication:
 
\begin{corollary} \label{lem:mc-ss}
If $F\subseteq\tau(G)$ (or $F \subseteq \ol{\tau}(\ol{G})$) is midpoint convex and $e\in F$, then any $x\in F$ is diagonalizable.
\end{corollary}

\begin{proof}
Let $x\in F$. Then, by midpoint convexity, there is $x'\in F$ such that $x=s_{x'}(e) = \tilde\mu(x',e) = x'^2 = \tau(x')$, where $\tilde\mu$ is the multiplication map of the group model from Section~\ref{model23} and the last equality holds by Lemma~\ref{lem:tau-K-saturation}(i). Iteration of this argument implies that for every $n\in\N$ there is $z_n\in F$ such that $z_n^{2^n}=\tau^n(z_n)=x$. Hence $x\in \bigcap_{i=1}^\infty \tau^i(G)$ (respectively, $x\in \bigcap_{i=1}^\infty \ol{\tau}^i(\ol{G})$ ) and, thus, $x$ is diagonalizable by the preceding theorem.
\end{proof}

\begin{corollary}\label{CorHoles} In every non-spherical Kac--Moody symmetric space (reduced or unreduced) there exists a pair of points that do not admit a midpoint and therefore do not lie on a common geodesic. \qed
\end{corollary}

For instance, the elements $\id=\tau(\id)$ and $\tau(u)$ from Example~\ref{Hole} do not admit a midpoint.

\begin{remark}
The preceding corollary illustrates that Kac--Moody symmetric spaces suffer from exactly the same deficits as the {\em masures} introduced in \cite[Section~3]{GaussentRousseau08} and discussed in detail in \cite{Rousseau:2011}. 
\end{remark}

Despite the lack of geodesics expressed by Corollary \ref{CorHoles} one nevertheless has the following:

\begin{prop}[{cf.\ \cite[Proposition~6.4]{Horn:Decomp}}] \label{lem:geod-conn}
Kac--Moody symmetric spaces are geodesically connected. In particular, \[G = \bigcup_{n \in \N} (KAK)^n.\]
\end{prop}

\begin{proof}
  One needs to show that any pair $x, y \in \tau(G)$ can be connected by a piecewise geodesic curve. The resulting geodesic connectedness of $G/K$ then implies that of $\ol{G}/\ol{K}$.
  
By transitivity of the action of the group $G$ on the symmetric space $\tau(G)$ one may assume without loss of generality that $x=e$. By Proposition~\ref{tauGexplicit}(iii) one can write $y=\tau(u_1\cdots u_kt)$ with $t \in A$ and 
$u_i\in U_{\beta_i}$ for some $\beta_i\in\Phi^+$.
 
 For $\alpha\in\Phi^+$ and $u\in U_\alpha$ and $t \in A$, the element \[\tau(ut) = ut\theta(t)^{-1}\theta(u)^{-1} = ut^2\theta(u)^{-1}\in A \langle U_\alpha, U_{-\alpha} \rangle\] (cf.~Lemma~\ref{tauonT}) stabilizes two opposite spherical residues (in fact, two opposite panels), whence is diagonalizable by Lemma~\ref{lem:sym-ss}.
 Applying this to $u_it_i$ with $t_i=1$ for $1\leq i < k$ and $t_k=t$, one obtains standard flats $F_{i}'$ containing $e$ and $\tau(u_it_i)$ and, thus, geodesic segments joining $e$ and $\tau(u_it_i)$. Then $F_{i}:=u_1t_1\cdots u_{i-1}t_{i-1} \ast F_{i}'$ is a standard flat containing $\tau(u_1t_1\cdots u_{i-1}t_{i-1})$ and $\tau(u_1t_1\cdots u_it_i)$. Setting $x_0=e$ and $x_{i}:=\tau(u_1t_1\cdots u_it_i)$ for $0\leq i\leq k$, one has 
 $x_i \in F_i\cap F_{i+1}$ and, moreover, $x_{k}=y$.
The claim follows.
\end{proof}

We have proved Theorem~\ref{thm1.8}.

\begin{remark} \label{5.14}
  Note that in the proof of the preceding proposition one actually has quite some freedom in choosing the individual geodesic segments. For instance, by \ref{RGD6} for any factorization $t=t_1\cdots t_k$ within $A$ there exist $u'_i \in U_{\beta_i}$ such that \[u_1\cdots u_kt = u'_1t_1 \cdots u'_kt_k.\]
  Of course, the argument in the proof applies to any such factorization.
\end{remark}

\subsection{The classification of maximal flats} \label{classmaxflats}

The methods for analyzing flats developed so far allow one to characterize the maximal (weak) flats in Kac--Moody symmetric spaces. The proof of the following theorem makes use of the various different models of the Kac--Moody symmetric space, in particular the group model. We recall from Convention \ref{ConventionTopologies} that we always equip the group model with the external topology (which we can show to coincide with the internal topology in the two-spherical case, but currently not in general). This has the effect that the coset model and the group model are isomorphic as topological reflection spaces, and in particular flats in one of them correspond to flats in the other.

\begin{theorem} \label{maximalflatsareeuclidean}
Every weak flat in a Kac--Moody symmetric space (reduced or unreduced) is contained in a standard flat. In particular, 
\begin{enumerate}
\item standard flats are exactly the maximal (weak) flats;
\item all weak flats are Euclidean, hence all weak flats are flats;
\item $G$, respectively $\ol{G}$ acts transitively on maximal (weak) flats.
\end{enumerate}
\end{theorem}

\begin{proof}
Let $F\subset \tau(G)$ be a weak flat. It suffices to show that $F$ is contained in a standard flat corresponding to some maximal (split) torus of $G$. Since $G$ acts transitively on $\tau(G)$, one may additionally assume without loss of generality that $e\in F$. Note that this assumption will in fact enable us to prove that the flat $F$ is contained in a standard flat of a {\em $\theta$-split} maximal torus, i.e., that $\theta$ acts by inversion on that maximal torus.

\medskip
From now on assume $e \in F$, let $x,y\in F$, and use the notation for the group model from Section~\ref{model23}.

\begin{claim} \label{claim:[xy,yx]=1} $[xy,yx]=e$, or equivalently, $xy^2x = yx^2y$. \end{claim}
One computes 
\[ \mT(x, \mT(e, \mT(y,e)) = \mT(x, \mT(e, y^2)) = \mT(x, y^{-2}) = xy^2x \]
and, similarly,
\[ \mT( y, \mT(e, \mT(x,e)) = yx^2y. \]
Hence, as $F$ is weakly abelian, 
\begin{alignat*}{3}
&&\ & xy^2x = \mT(x, \mT(e, \mT(y,e)) = \mT( y, \mT(e, \mT(x,e)) = yx^2y \\
&\iff&& (xy)(yx) = (yx)(xy) \\
&\iff&& [xy,yx]=e.
\end{alignat*}

\begin{claim} \label{claim:xy-diag} $xy$ is diagonalizable. \end{claim}
By midpoint convexity of $F$, there is a midpoint $x'\in F$ between $e$
and $x$, whence $x'^2 = s_{x'}(e) = x$. Moreover, $s_e(y)=y^{-1}\in F$, and so
$s_{x'}(y^{-1})= x'yx' \in F$.
By Corollary~\ref{lem:mc-ss} the element $x'yx'$ is diagonalizable.
Hence, by definition, there exists a twin apartment $\Sigma$ of the twin building of $G$ which is fixed chamberwise by $x'yx'$.
Let $c$ be a chamber of $\Sigma$ and set $(\Sigma',c'):=x'.(\Sigma,c)$. Then 
\[
xy.(\Sigma',c') = xyx'.(\Sigma,c) = x' (x' y x'). (\Sigma,c) = x' (\Sigma,c) = (\Sigma',c'),
\]
and so $xy$ stabilizes $\Sigma'$ and fixes $c'$. That is, $xy$ fixes $\Sigma'$ pointwise and, by definition, is diagonalizable.

\begin{claim} \label{claim:xy-yx-stab-R} In each half $\Delta_\pm$ of the twin building there exist opposite spherical residues $R_+ \subset \Delta_+$ and $R_- = \theta(R_+) \in \Delta_-$ stabilized by both $xy$ and $yx$. \end{claim}

Both $xy$ and $yx=\theta(y)^{-1}\theta(x)^{-1}=\theta(xy)^{-1}$ (by Lemma~\ref{lem:tau-K-saturation}(i) plus $x, y \in F \subset \tau(G)$) are diagonalizable and, thus, both fix some twin apartment
chamberwise. In particular, both admit fixed points in the CAT(0) realizations $X_\pm$
of either half $\Delta_\pm$ of the twin building. (See \cite{Davis98}, also \cite[Section~2.1]{Caprace09}.)

One can now find a common fixed point of $xy$ and $yx$ in $X_+$ by a standard commutation argument as follows:
For $p\in\Fix(xy)$, one has 
\[ yx.p = yx.(xy.p) \overset{\text{Claim~\ref{claim:[xy,yx]=1}}}= xy.(yx.p), \]
whence $yx.p\in\Fix(xy)$.
Thus the convex set $\Fix(xy)$ is preserved by the isometry $yx$. Let $q$
be a point fixed by $yx$ and $r_+$ its (unique) projection to $\Fix(xy)$ in the CAT(0) space $X_+$.
Since $\Fix(xy)$ is preserved by $yx$, it follows that $r_+$ is also fixed by $yx$.
The point $r_+ \in X_+$ corresponds to a spherical residue $R_+$ of $\Delta_+$ stabilized by both $yx$ and $xy$.

Consequently, the residue $R_- := \theta(R_+)$ opposite $R_+$ is stabilized by both $\theta(xy)=(yx)^{-1}$ and $\theta(yx)=(xy)^{-1}$ and, hence, also by $yx \in \left\langle (yx)^{-1} \right\rangle < G$ and $xy \in \left\langle (xy)^{-1} \right\rangle < G$. 

\begin{claim} \label{claimopposite} $xy$ fixes a chamber $d\in R_+$ and $yx$
fixes $\widehat d:=\proj_{R_+} \theta(d)$ opposite $d$ in $R_+$. \end{claim}
Since $xy$ is diagonalizable, it fixes a twin apartment chamberwise and so there is a chamber $c\in\Delta_+$ fixed by $xy$. Thus
the chamber $d:=\proj_{R_+}(c)$ is also fixed by $xy$. The involution $\theta$ induces an involution $\theta_{R_+}(c):=\proj_{R_+}(\theta(c))$ on $R_+$,
which maps every chamber in $R_+$ to a chamber opposite in $R_+$.
The chamber $\widehat d:=\theta(d)$ is fixed by $yx=\theta(xy)^{-1}$:
\[ yx. \widehat d  = 
\proj_{yx.R_+}(yx.\theta(d))
\overset{\text{Claim~\ref{claim:xy-yx-stab-R}}}= 
\proj_{R_+}(yx.\theta(d))
= \proj_{R_+}(\theta((xy)^{-1}.d))
= \proj_{R_+}(\theta(d)) = \widehat d.
\]

\begin{claim} There exists a chamber $d'\in R_+$ fixed by both $xy$ and $yx$.  \end{claim}

By Claims~\ref{claim:xy-yx-stab-R} and \ref{claimopposite} the elements $xy$ and $yx$ are contained in opposite
Borel subgroups of the reductive split real Lie group stabilizing the opposite spherical residues $R_+$ and $\theta(R_+)$ (cf.\ \cite[Corollary~7.16]{HartnickKoehlMars}).

This reductive Lie group is a subgroup of $\GL_{n+1}(\R)$. By \cite[Proposition~16.1.5]{HilgertNeeb12} one can model the stabilizer of $d$ as lower triangular matrices, the stabilizer of $\widehat d$ as upper triangular matrices,
and $\theta$ as transpose-inverse. Thus $yx = \theta(xy)^{-1} = (xy)^T$.
One concludes that both $xy$ and $yx$ are diagonal
in this coordinatization:
Suppose 
\[xy=
\begin{pmatrix}
v_1    & v_2    & \ldots & v_n    \\
0      & *      & \ldots & *      \\
\vdots & \vdots & \ddots & \vdots \\
0      & 0      & \ldots & *
\end{pmatrix}
\quad\text{ and thus }\quad
yx = (xy)^T =
\begin{pmatrix}
v_1    & 0      & \ldots & 0      \\
v_2    & *      & \ldots & 0      \\
\vdots & \vdots & \ddots & \vdots \\
v_n    & *      & \ldots & *
\end{pmatrix}.
\]
Computing the product $xy\cdot yx$ yields the top left entry $v_1^2+\cdots+v_n^2$.
On the other hand, the top left entry of $yx\cdot xy$ is $v_1^2$. By Claim~\ref{claim:[xy,yx]=1} one has $[xy,yx]=e$ and hence $v_2^2+\cdots+v_n^2=0$ and so $v_2=\ldots=v_n=0$.
Inductively one obtains that $xy$ and $yx$ act by the same diagonal matrix on $R_+$.
Therefore there is a chamber $d'$ stabilized by both $xy$ and $yx$.

\begin{claim} \label{claim: xy=yx} $xy = yx$. \end{claim}

Since $xy$ and $yx$ stabilize a chamber $d'$, one has
$\theta(d') = \theta(xy.d') =  x^{-1} y^{-1}.\theta(d')$, thus $\theta(d')=yx.\theta(d')$.
It follows that $xy$ and $yx$ stabilize both $d'$ and $\theta(d')$ and, hence, fix a $\theta$-stable twin apartment. 
Thus they are contained in a common $\theta$-split torus. As an immediate consequence, 
$(xy)^{-1}=\theta(xy)=\theta(x)\theta(y)=x^{-1} y^{-1}$.
Hence $xy=yx$.

\begin{claim} \label{claim:commutative} For each $x, y \in F$ one has $y^{-1}, xy \in F$. That is, $F$ is a commutative subgroup of $G$. \end{claim}

Recall that $e \in F$ by assumption. Let $x'\in F$ be a midpoint of $x$ and $e$, i.e., $x=s_{x'}(e)=x'^2$. Then $s_e(y)=y^{-1}\in F$ and
$s_{x'}(y^{-1})=x'yx'\overset{\text{Claim~\ref{claim: xy=yx}}}= x'^2y=xy \in F$. This finishes the proof of Claim~\ref{claim:commutative}.

\medskip
Concerning the statement of the following claim, we observe that the intersection of midpoint convex reflection subspaces of $F$ is again a midpoint convex reflection subspace. In particular, if $ x_1, \ldots, x_t \in F$ then there exists a unique smallest midpoint convex reflection subspace of $F$ containing $x_1, \dots, x_t$. We denote this subspace by $\langle x_1, \ldots, x_t \rangle$.

\begin{claim} \label{claim:torus} For every finite subset $\{ x_1, \ldots, x_t \} \subset F$ the topological closure $\ol{\langle x_1, \ldots, x_t \rangle}$ of ${\langle x_1, \ldots, x_t \rangle}$ is a diagonalizable subgroup of $G$, i.e., is contained in a maximal torus of $G$. Moreover, there exists $m \in \N$ such that $\ol{\langle x_1, \ldots, x_t \rangle} \cong (\R_{>0}^m,\cdot) \cong (\R^m,+)$.  \end{claim}

Since the weak flat $F$ is closed, one has $\ol{\langle x_1, \ldots, x_t \rangle} \subseteq F$. Since $F$ is a commutative group, any reflection subspace is a subgroup, and so is its closure.
By Corollary~\ref{lem:mc-ss} and the fact that maximal tori of $G$ are closed (see \cite[Corollary~7.17]{HartnickKoehlMars}), each of the subgroups $H_i := \ol{\langle x_i \rangle } \leq F$ is diagonalizable. Moreover, $H_i \cong (\R,+)$ by direct computation in any torus containing $x_i$ (see also Proposition~\ref{PropTranslationGroups}). The groups $H_i$ commute with one another by Claim~\ref{claim:commutative}, whence \cite[Proposition~4.4]{Caprace09} implies that $\ol{\langle x_1, \ldots, x_t \rangle}$ normalizes a maximal torus $T$ of $G$. Moreover, since $W = N_G(T)/T$ is discrete and $\ol{\langle x_1, \ldots, x_t \rangle}$ is connected, one actually has $\ol{\langle x_1, \ldots, x_t \rangle} \leq T$. Connectedness then additionally implies that $\ol{\langle x_1, \ldots, x_t \rangle} \leq A = \tau(T) \cong (\R^n,+)$, where $n$ is the rank of $G$. (Note that in $\ol{G}$ one obtains a torus isomorphic to $(\R^{\rk(\mathbf{A})},+)$ instead.) The final statement follows from the classification of closed connected subgroups of $(\R^n,+)$.

\begin{claim} $F$ is contained in a standard flat. \end{claim}

Let $m := \max \left\{ \dim_\R\left(\ol{\langle x_1, \ldots, x_t \rangle}\right) \mid t \in \N, x_1, \ldots, x_t \in F \right\} \leq n$, where $n$ is the rank of $G$, and let $\{ x_1, \ldots, x_t \} \subset F$ such that $\dim_\R\left(\ol{\langle x_1, \ldots, x_t \rangle}\right)=m$.
Then $F = \ol{\langle x_1, \ldots, x_t \rangle}$: indeed, otherwise there exists $x_{t+1} \in F \backslash \ol{\langle x_1, \ldots, x_t \rangle}$ and $\dim_\R\left(\ol{\langle x_1, \ldots, x_t, x_{t+1} \rangle}\right) = m+1$, a contradiction.

\medskip \noindent
The proof for $\ol{G}$ is essentially the same.
\end{proof}

\begin{corollary}\label{ConsequencesFlatClassification}
\begin{enumerate}
\item $G$ acts strongly transitively on $G/K$.
 \item $G$ and $\ol{G}$ and $\Ad(G)$ act strongly transitively on $\ol{G}/\ol{K}$.
\item Maximal flats in $\ol{G}/\ol{K}$ lift uniquely to maximal flats in $G/K$.
\end{enumerate}
\end{corollary}

\begin{proof} In view of Theorem \ref{maximalflatsareeuclidean} this follows from Corollary \ref{StdFlatTori} and Proposition~\ref{PropAutomaticallyStrong}. Note that $\Ad(G)$ indeed acts on $\ol{G}/\ol{K}$ by Lemma~\ref{LemmaKernelofAction}.
\end{proof}

We have established Theorem~\ref{IntroFlats}.

\section{Local and global automorphisms of Kac--Moody symmetric spaces} \label{aut}

We keep the notation of the previous section, i.e., $G$ denotes a simply connected centred split real Kac--Moody group of irreducible symmetrizable type, the group $\ol{G}$ denotes its semisimple adjoint quotient, and $\Ad(G)$ its adjoint quotient. Moreover, $\Delta = \Delta^-\sqcup \Delta^+$ denotes the twin building associated to the RGD systems of these groups.

\subsection{Automorphisms of Kac--Moody groups}

The abstract automorphisms of the groups $G$, $\ol{G}$ and $\Ad(G)$ have been classified in \cite{Caprace09}. Since $\R$ does not admit any non-trivial field automorphism, this classification can be stated as follows.

\begin{theorem}[{Caprace \cite[Theorem~4.2]{Caprace09}}]\label{ThmAut} Let $\GGG \in \{G, \ol{G}, \Ad(G)\}$. Then every automorphism of $\GGG$ can be written as a product of an inner automorphism of $\GGG$, a diagram automorphism, a diagonal automorphism and a power of the Cartan--Chevalley involution $\theta$.\qed
\end{theorem}

This result has  several immediate consequences. Firstly, every automorphism of $G$ preserves the center and hence descends to an automorphism of $\Ad(G)$. It also descends to an automorphism of $\ol{G}=G/C$, since $C\leq Z(G)$ is the unique complement to the torsion subgroup of $Z(G)$, cf. Proposition~\ref{kernelsemisimpleadjoint}, and hence $C$ is a characteristic subgroup of $Z(G)$ and thus of $G$.
One thus obtains homomorphisms
\begin{equation}\label{AutomorphismsDescend}
\Aut(G) \to \Aut(\ol{G}) \to \Aut(\Ad(G)).
\end{equation}
Secondly, these homomorphism are injective, because for each automorphism, there is a root subgroup not centralized by it; but the root subgroups can be seen in any central quotient of $G$.

Thirdly, it follows from the concrete description of automorphisms in Theorem~\ref{ThmAut} that every automorphism of $\Ad(G)$ or $\ol{G}$ can be extended to $G$. That is, the homomorphisms in \eqref{AutomorphismsDescend} are also surjective, and hence isomorphisms.

We compare these automorphisms to combinatorial automorphisms of the twin building $\Delta$. By an \Defn{automorphisms} of $\Delta$ we shall mean a self-map of the chamber set $\Delta = \Delta^+ \cup \Delta^-$ which preserves adjacency and opposition of chambers (but may swap the two halves of $\Delta$). An automorphism will be called \Defn{type-preserving} if it  preserves distances and codistances (and hence the underlying chamber system). Denote by ${\rm Aut}(\Delta)$ and ${\rm Aut}_S(\Delta)$ the groups of all automorphisms, respectively all type-preserving automorphisms of $\Delta$.

We can identify chambers of $\Delta$ with Borel subgroups of $G$, i.e., conjugates of $B^+$ or $B^-$. Every inner automorphism of $G$ certainly maps $B$ to a Borel subgroup; the same holds for diagram and diagonal automorphisms. Also, the Cartan--Chevalley involution swaps $B^+$ and $B^-$ and thus preserves the set of Borel subgroups. We deduce with Theorem \ref{ThmAut} that every automorphism of $G$ induces an automorphism of $\Delta$, and hence we obtain a homomorphism $\Aut(G) \to \Aut(\Delta)$. Moreover, all of the basic types of automorphisms of $G$ except the diagram automorphisms induced type-preserving automorphisms. We denote the group generated by all such automorphisms by ${\rm Aut}_S(G) < {\rm Aut}(G)$.

\begin{remark}\label{RemarkCoxeterVsDynkin}
  Recall from Definition~\ref{GCMdef} that the Dynkin diagram $\Gamma_{\mathbf{A}}$ arises from the Coxeter diagram of $(W,S)$ by adding certain labels. In particular, the automorphism group $\Aut(\Gamma_{\mathbf{A}})$ of the Dynkin diagram is a subgroup of the automorphism group $\Aut(W,S)$ of the Coxeter diagram.
\end{remark}

\begin{proposition}\label{AutBuilding}
The homomorphism $\Aut(G) \to \Aut(\Delta)$ is injective, and thus
\[ \Aut(G)  \cong \Aut(\ol{G}) \cong \Aut(\Ad(G)) \into \Aut(\Delta). \]
If $G$ is two-spherical and if $\Aut(\Gamma_{\mathbf{A}})=\Aut(W,S)$, then it is an isomorphism, and thus
\[ \Aut(G)  \cong \Aut(\ol{G}) \cong \Aut(\Ad(G)) \cong \Aut(\Delta). \]
\end{proposition}

\begin{proof}
 Every diagram automorphism of $G$ induces a non-trivial automorphism on each of the two halves of $\Delta$, and the Cartan-Chevalley involution swaps these two halves. It thus follows from Theorem \ref{ThmAut} that the kernel of the homomorphism $\Aut(G) \to \Aut(\Delta)$ necessarily lies in the subgroup of $\Aut(G)$ generated by the inner and the diagonal automorphisms of $G$, which is a group with an RGD system with abelian maximal torus, trivial center, and the same twin building $\Delta$. By \cite[Proposition~8.82]{AbramenkoBrown2008} therefore the kernel of the homomorphism $\Aut(G) \to \Aut(\Delta)$ is trivial. 

Now assume that $G$ is two-spherical and that $\Aut(\Gamma_{\mathbf{A}})=\Aut(W,S)$. To prove surjectivity, one needs to prove that any automorphism $\alpha \in \Aut(\Delta)$ is induced by an automorphism of $G$.

Each automorphism $\alpha$ of $\Delta$ induces a well-defined permutation of the diagram of $\Delta$, which necessarily has to be an automorphism of the underlying Coxeter diagram. Hence the automorphism $\alpha$ is the product of a type-preserving automorphism of $\Delta$ and a Coxeter diagram automorphism. If $G$ and $\Delta$ admit the same diagram automorphisms, i.e., if the automorphisms of the Dynkin diagram equal the automorphisms of the Coxeter diagram, one may assume that $\alpha$ is type-preserving.

Let $C_+$ and $C_- = \theta(C_+)$ be opposite chambers of the twin building $\Delta$. By the strongly transitive action of $G$ on $\Delta$ (see \cite[Lemma~6.70 and Theorem~8.9]{AbramenkoBrown2008}) there exists an inner automorphism of $G$ that maps the set $\{ \alpha(C_+), \alpha(C_-) \}$ onto the set $\{ C_+, C_- \}$. By composing $\alpha$ with this inner automorphism and, if necessary, the Cartan--Chevalley involution $\theta$, one may actually assume that the type-preserving automorphism $\alpha$ fixes the chambers $C_+$ and $C_-$.

If the diagram is two-spherical, then the extension theorem by M\"uhlherr and Ronan \cite[Theorem~1.2]{Muehlherr/Ronan:1995} (see also \cite[Theorem~5.213]{AbramenkoBrown2008}) implies that the type-preserving automorphism $\alpha$ is the unique extension that fixes $C_-$ of its restrictions to the residues of rank two containing $C_+$ in the positive half $\Delta_+$ of the twin building $\Delta$.

By inspection, those local rank-two restrictions are all induced by automorphisms of the corresponding split real Lie groups of rank two that as a family together provide an automorphism of the amalgam $\AAA(\mathbf{A})$ of fundamental subgroups of rank two of $G$. This amalgam automorphism, again using two-sphericity, induces a unique automorphism of $G$ by \cite[Theorem~7.22]{HartnickKoehlMars} (see Section~\ref{Subsection3B}) whose image under the natural map is $\alpha$.
\end{proof}
Note that in the two-spherical case we always have an isomorphism ${\rm Aut}_S(G) \cong {\rm Aut}_S(\Delta)$, irrespective of the automorphism group of the Dynkin diagram.
\begin{prop}\label{StructureOfAutG}
The group $\Out(G) := \Aut(G)/\Inn(G)$ is finite.
More explicitly, we have
\[ \Aut(G) \cong (\Ad(G) \rtimes (D \times \gen{\theta})) \rtimes \Aut(\Gamma_{\mathbf{A}}) \]
where $\Inn(G)\cong G/Z(G) = \Ad(G)$ and $D$ is a finite elementary abelian 2-group of diagonal automorphisms.
Moreover, $\theta$ commutes with the elements of $\Aut(\Gamma_{\mathbf{A}})$.
\end{prop}

\begin{proof}
Note that $\Aut(\Gamma_{\mathbf{A}})$, as the automorphism group of a finite graph, is finite and that the group generated by the involution $\theta$ has order 2. Thus the first claim follows from the second claim.

Now consider the subgroup $H$ of $\Aut(G)$ generated by the inner and diagonal automorphisms. By \cite[Theorem~4.2]{Caprace09}, it suffices to consider the diagonal automorphisms coming from a diagonal automorphism of fundamental rank-1 subgroups $G_{\alpha_i}\cong\SL_2(\R)$, with $\alpha_i\in\Pi$. But these are either inner, or are inner followed by conjugation with \[ \mtr{-1}001\in\GL_2(\R). \]
Denote the latter automorphisms of $G$ by $d_i$. Then $H$ is generated by $\Inn(G)$ together with $\tilde{D}:=\gen{d_1,\dots,d_n}$. Clearly $\tilde{D}$ is an elementary abelian 2-group, i.e., it is an $\F_2$-vector space. Hence it contains a complement $D$ to $\tilde{D}\cap \Inn(G)$, and one has $H\cong \Inn(G)\rtimes D$.

The Cartan--Chevalley involution $\theta$ centralizes the $d_i$, hence $\theta$ commutes with $D$ and of course normalizes $\Inn(G)$. Thus $H':=\gen{H,\theta} \cong \Inn(G) \rtimes (D\times \theta)$.

Finally, any diagram automorphism permute the $G_{\alpha_i}$ and hence normalize $\tilde{D}$; it also commutes with $\theta$. Thus $\Aut(\Gamma_{\mathbf{A}})$ normalizes $H'$. Since all elements of $H'$ centralize the Weyl group one moreover has $\Aut(\Gamma_{\mathbf{A}})\cap H' = \{1\}$.
\end{proof}

Since $\theta$ commutes with the diagonal automorphisms in $D$ and with diagram automorphisms, we also conclude the following:

\begin{corollary}\label{ConjClassOfTheta}
The $\Aut(G)$-conjugacy class of $\theta$ in $\Aut(G)$ coincides with its $\Ad(G)$-conjugacy class.
\end{corollary}

\begin{remark}
A large part of the proof of Proposition~\ref{StructureOfAutG} is dedicated to proving
that there are only finitely many diagonal automorphisms modulo inner automorphisms.  
This is also implied (but non-constructively)
by the fact, which is also true for (rational points of) algebraic groups, that the index of the adjoint quotient $\Ad(G)$ inside the adjoint split real Kac--Moody group of type $\mathbf{A}$ is finite: Indeed, the exact sequence $0 \to \FFF \to \ol{\TTT}^{sc} \to \TTT^{ad} \to 0$ of torus schemes (where $\ol{\TTT}^{sc}$ is the simply connected torus of the semisimple adjoint quotient $\ol{G}$ and $\TTT^{ad}$ is the torus of the adjoint Kac--Moody group of type $\mathbf{A}$ and $\FFF$ is defined as the kernel of $\ol{\TTT}^{sc} \to \TTT^{ad}$) yields an exact sequence $0 \to \FFF(\R) \to \ol{\TTT}^{sc}(\R) \to \TTT^{ad}(\R) \to H^1_{et}(\R,\FFF) \to 0$ of $\R$-points; since \'etale cohomology is finite over $\R$ (see, e.g., \cite{Milne:1980}), the claim follows.
\end{remark}

\subsection{Automorphisms of the main group}
The next main goal is to describe the automorphism group of the reduced Kac--Moody symmetric space $\ol{\XXX} = \ol{G}/\ol{K}$. Recall from Proposition~\ref{PropMainTransKM} that
\[
 \Trans(\ol{\XXX}) = \Ad(G) \quad \text{and} \quad G(\ol{\XXX})=\Ad(G)  \rtimes \gen{\ol\theta}.\]
Also recall from Remark \ref{AutEmbedding} that there exists an embedding
\[
 c: \Aut(\ol{\XXX}) \to \Aut( G(\ol{\XXX})) = \Aut(\Ad(G)  \rtimes \gen{\ol\theta}), \quad \alpha \mapsto c_\alpha,
\]
where $c_\alpha(g) := \alpha \circ g \circ \alpha^{-1}$. Thus in order to determine $\Aut(\ol{\XXX})$ one first needs to determine $\Aut(\Ad(G)  \rtimes \gen{\ol\theta})$. For this we will use the following general lemma.

\begin{lemma}\label{StructureOfAut(H x theta)}
Let $H$ be a perfect group and $\theta\in\Aut(H)$ an involution.
Endow $\Aut(H)$ with the multiplication $\alpha\beta:=\beta\circ\alpha$.
Then the following hold:
\begin{enumerate}
\item The following is a well-defined subgroup of the holomorph $\Aut(H)\ltimes H$ of $H$:
\[ \Hol_\theta(H) := \{ (\beta,b) \in \Aut(H)\ltimes H \mid \beta\circ\theta = c_b\circ\theta\circ\beta,\ \theta(b)=b^{-1} \}, \]
where $c_b$ is the inner automorphism $x\mapsto bxb^{-1}$.
\item There is an isomorphism $\phi: \Aut(H \rtimes \gen{\theta}) \to \Hol_\theta(H)$.
\item Let $\pi:\Hol_\theta(H)\to\Aut(H)$ be the restriction of the natural projection $\Aut(H)\ltimes H\to\Aut(H)$. Then \[
\ker\pi = \{ (\id,z) \mid z\in Z(H),\ \theta(z)=z^{-1} \}.
\]
\item
For the inner automorphisms of $\Aut(H\rtimes\gen{\theta})$, we have
\[ \phi( c_{(h,\id)} ) = (c_h, \tau(h))
   \quad\text{ and }\quad
   \phi( c_{(1,\theta)}) = (\theta, 1). \]
\end{enumerate}
\end{lemma}

\begin{proof}
\begin{enumerate}
\item
Clearly $\Hol_\theta(H)$ contains the identity $(\id_H,1_H)$. Let $(\beta,b),\ (\gamma,c)\in\Hol_\theta(H)$. To see that $(\beta,b)(\gamma,c)=(\beta\gamma, \gamma(b)c)$ is contained in $\Hol_\theta(H)$, we verify that
\[ \theta(\gamma(b)c) = (\theta\circ\gamma)(b) \theta(c) = (c_c^{-1}\circ\gamma\circ\theta)(b) c^{-1} = c^{-1}\gamma(b^{-1}) = (\gamma(b)c)^{-1} \]
as well as
\[ (\beta\gamma)\circ\theta
= \gamma\circ\beta\circ\theta
= \gamma\circ c_b\circ\theta\circ\beta
= c_{\gamma(b)}\circ c_c\circ\theta\circ\gamma\circ\beta
= c_{\gamma(b)c}\circ\theta\circ (\beta\gamma).
\]
Finally, $(\beta,b)^{-1}=(\beta^{-1}, \beta^{-1}(b^{-1}))$ is in $\Hol_\theta(H)$ because
\[ \theta(\beta^{-1}(b^{-1})) = (\beta^{-1}\circ c_b \circ \theta)(b^{-1}) = \beta^{-1}(b) \]
and furthermore, $\beta\circ\theta = c_b\circ\theta\circ\beta$ implies
\[
\theta\circ\beta^{-1} = \beta^{-1}\circ c_b\circ\theta = c_{\beta^{-1}(b)}\circ\beta^{-1}\circ\theta
\implies
c_{\beta^{-1}(b^{-1})} \circ\theta\circ\beta^{-1} = \beta^{-1}\circ\theta
.
\]

\item
By slight abuse of notation, we identify $H$ with the subgroup $\{(h,1) \mid h\in H\}\leq H\rtimes\gen\theta$.
Then the commutator subgroup of $H\rtimes\gen\theta$ is contained in $H$.
But $H$ is perfect, hence it equals that commutator subgroup, which is characteristic.
It follows that every $\alpha\in \Aut(H \rtimes \gen{\theta})$ normalizes $H$, thus its restriction to $H$ is an automorphism $\beta\in\Aut(H)$. 

Since $\alpha$ is surjective and normalizes $H$, we must have $\alpha( (1,\theta) ) = (b,\theta)$ for some $b\in H$, and then $\alpha$ is uniquely determined by the pair $(\beta,b)$.
The fact that $\alpha$ is a homomorphism puts restrictions on the pair $(\beta,b)$: For $g,h\in H$, we must have
\begin{align*}
(\beta(g\theta(h)),\id)
&= \alpha( (g\theta(h),\id) )
= \alpha( (g,\theta)(h,\theta) )
\overset!= \alpha( (g,\theta))\alpha((h,\theta)) \\
&= (\beta(g)b,\theta)(\beta(h)b,\theta)
 = (\beta(g)b\theta(\beta(h))\theta(b),\id),
\end{align*}
hence $\beta(\theta(h)) = b\theta(\beta(h))\theta(b)$.
For $h=1$ this yields $\theta(b)=b^{-1}$. We thus find
\[ \beta\circ\theta = c_b\circ\theta\circ\beta
. \]
One readily checks that these two conditions are necessary and sufficient to make $\alpha$ a homomorphism.
This yields the desired map $\phi$, which by construction is bijective. It remains to verify that $\phi$ is indeed a group homomorphism to verify the claim.

To this end, suppose we have $\alpha,\alpha'\in \Aut(H \rtimes \gen{\theta})$
satisfying $\phi(\alpha)=(\beta,b)$ resp. $\phi(\alpha')=(\gamma,c)$.
Let $\alpha\alpha':=\alpha'\circ\alpha$ and let $\alpha''$ be the preimage under $\psi$ of $\psi(\alpha)\psi(\alpha')) = (\beta\gamma, \gamma(b)c)$.
Then for $h\in H$ and $e\in\{0,1\}$ we have as desired
\[
(\alpha\alpha')(h,\theta^e)
= \alpha'((\beta(h)b^e,\theta^e))
= (\gamma(\beta(h)) (\gamma(b)c)^e, \theta^e)
= \alpha''(h,\theta^e).
\]

\item
Suppose $\alpha\in\ker\phi$. Then $\beta=\id$, and we get $c_b=[\id,\theta]=\id$, i.e., $b\in Z(H)$.

\item Follows from elementary computations.
\qedhere
\end{enumerate}
\end{proof}

\begin{proposition}\ \label{Aut(AdG x theta) = Aut(AdG)}
\begin{enumerate}
\item
There is a short exact sequence
\[ 1 \to Z(G) \to \Aut(G\rtimes\gen{\theta}) \xto{\pi} \Aut(G) \to 1. \]
\item There is an isomorphism
\[ \Aut(\Ad(G)\rtimes\gen{\ol\theta}) \to \Aut(\Ad(G)) \cong \Aut(G). \]
\item Following the notation in Definition~\ref{Geff} and Proposition~\ref{StructureOfAutG}, there is an isomorphism
\[ 
\Aut(\Geff\rtimes\gen{\ol\theta}) \to 
\AutEff(G) := (\Geff \rtimes (D \times \gen{\theta})) \rtimes \Aut(\Gamma_{\mathbf{A}})
\]\end{enumerate}
\end{proposition}

\begin{proof}
\begin{enumerate}
\item
We choose for $\pi$ the composition of the maps from Lemma~\ref{StructureOfAut(H x theta)} (ii) and (iii). 
To see that this map is surjective, let $\beta\in\Aut(G)$. By Proposition~\ref{StructureOfAutG} one then has $\beta=c_g \circ d \circ \theta^r \circ \gamma$ for an inner automorphism $c_g$, some $d\in D$, $r\in\{0,1\}$ and $\gamma\in\Aut(\Gamma_{\mathbf{A}})$.
Recall that $\theta$ commutes with the elements of $D$ and $\Aut(\Gamma_{\mathbf{A}})$ and also with itself.
Therefore
\[
\beta\circ\theta = (c_g \circ d \circ \theta^r \circ \gamma) \circ \theta
= c_g \circ \theta \circ d \circ \theta^r \circ \gamma
= c_{\tau(g)} \circ \theta \circ \beta.
\]
Since also $\theta(\tau(g))=\tau(g)^{-1}$, we have $(\beta,\tau(g))\in\Hol_\theta(H)$. By Lemma~\ref{StructureOfAut(H x theta)} (ii) it follows that $\beta$ is in the image of $\pi$, which thus is surjective.

The center of $G$ is contained in $T$. But $\theta$ acts on $T$ and hence on $Z(G)$ by inversion, thus by Lemma~\ref{StructureOfAut(H x theta)}(iii) we have $\ker(\pi)\cong Z(G)$ as claimed.

\item
This follows similar to (i) together with the fact that $\Ad(G)$ is center-free.

\item
Consider the natural epimorphism
\[ p: \AutEff(G)\to\Aut(G),\quad (g, d, \theta^r, \gamma) \mapsto c_g\circ d\circ \theta^r\circ\gamma, \]
which maps $g\in\Geff$ to $c_g$ (here, we equip $\Aut(G)$ with composition $\circ$ as multiplication).
Then the claim follows by showing that the following map is an anti-isomorphism (recall that anti-isomorphic groups are isomorphic):
\[ \kappa: \AutEff(G)\to\Hol_\theta(\Geff),\quad
\alpha = (g, d, \theta^r, \gamma) \mapsto (p(\alpha), \tau(g))
. \]
The map is well-defined by an argument  similar to that in (i).
To verify that it is an anti-homomorphism, let
\[ \alpha := (g, d, \theta^r, \gamma),\quad \beta := (h, d', \theta^s, \gamma') \]
be arbitrary elements of $\AutEff(G)$ and set
\[ \tilde\alpha := d \circ \theta^r \circ \gamma, \quad \tilde\beta := d' \circ \theta^s \circ \gamma'. \]
Then we get
\begin{align*}
\kappa(\alpha)\kappa(\beta)
  &= (p(\alpha),\tau(g))\cdot (p(\beta),\tau(h)) \\
  &= (p(\alpha)p(\beta),\  c_h(\tilde\beta(\tau(g))) \cdot h\theta(h^{-1}) )
         & (\text{since } {p(\beta) = c_h\circ\tilde\beta}) \\
  &= (p(\alpha)p(\beta),\  h \tau(\tilde\beta(g))) \theta(h^{-1}) )
         & (\text{since } {\tilde\beta\circ\theta = \theta\circ\tilde\beta}) \\
  &= (p(\beta)\circ p(\alpha),\  \tau(h\tilde\beta(g))),\\
  &= ( p(\beta\alpha),\ \tau(h\tilde\beta(g))) \\
  &= \kappa( (h \tilde\beta(g),\ d' \gamma'(d),\ \theta^{r+s},\ \gamma'\gamma )) \\
  &= \kappa(\beta\alpha).
\end{align*}

Finally, we compute
\begin{align*}
\ker(\kappa)
&= \{ \alpha=(g, d, \theta^r, \gamma) \in\AutEff(G) \mid \alpha\in\ker(p) \text{ and } \tau(g)=1 \} \\
&= \{ (g,\id,\id,\id) \mid g\in Z(\Geff) \text{ and } \theta(g)=g \}.
\end{align*}
But $\theta$ acts by inversion on $Z(\Geff)$, hence $g\in Z(\Geff)$ with $\theta(g)=g$ satisfies $g^2=1$, i.e., is torsion. But $\Geff=G/C_K(G)$, and $C_K(G)=Z(G)\cap M$ is precisely the torsion part of $Z(G)$, hence $Z(\Geff)$ is torsion free. It follows that the kernel of $\kappa$ is trivial, whence $\kappa$ is an isomorphism.
\qedhere
\end{enumerate}
\end{proof}

\subsection{Global automorphisms of reduced Kac--Moody symmetric spaces}

In this section we prove Theorem~\ref{thm1.9}.

\begin{theorem}\ \label{ThmGlobalAut}
Consider the reduced resp.\ non-reduced Kac--Moody symmetric spaces $\ol{\XXX} = \ol{G}/\ol{K}$ resp.\ $\XXX = G/K$. Then the following are true:
\begin{enumerate}
\item $\Aut(\XXX) \cong \AutEff(G).$
\item $\Aut(\ol{\XXX}) \cong \Aut(G) \cong \Aut(\ol{G}) \cong \Aut(\Ad(G)).$
\end{enumerate}
\end{theorem}

\begin{proof}
\begin{enumerate}
\item
By Proposition~\ref{PropMainTransKM} (iii) we have $G(\XXX,\mu)=\Geff\rtimes\gen{\theta}$.
Combining the maps $c$ from Remark~\ref{AutEmbedding}, $\phi$ from Lemma~\ref{StructureOfAut(H x theta)} (ii) and $\kappa$ from \ref{Aut(AdG x theta) = Aut(AdG)} (iii), we get an embedding
\[
\Aut(\XXX,\mu) \xto{c} \Aut(\Geff\rtimes\gen{\theta}) \xto{\phi} \Hol_\theta(\Geff) \xto{\kappa^{-1}} \AutEff(G).
\]
It suffices to show that the composition $\psi:=\kappa^{-1}\circ\phi\circ c$ is surjective. Theorem~\ref{ThmAut} allows us to reduce a case-by-case analysis of the different types of automorphisms which together generate $\AutEff(G)$. 

\begin{enumerate}
\item Let $h\in\Geff$. Then $h$ acts on $\XXX$ as an automorphism, which $c$ maps to the inner automorphism $c_{(h,\id)}$ of $\Aut(\Geff\ltimes\gen\theta)$. By Lemma~\ref{StructureOfAut(H x theta)} (iii), $\phi$ maps this to $(c_h,\tau(h))\in\Hol_\theta(\Geff)$, thus $\psi(h) = \kappa^{-1}(c_h,\tau(h)) = (h,\id,\id,\id)$, whence $\Geff$ considered as a subgroup of $\AutEff(G)$ is contained in the image of $\psi$.

\item Let $f\in C_{\Aut(G)}(\theta)$, so in particular we may consider $f=\theta$ or $f \in D$ or $f \in \Aut(\Gamma_{\mathbf{A}})$. Then $f(K)=K$ and thus $f$ induces a permutation of $G/K$. Moreover, $f$ is compatible with $\mu$; indeed,
\[ f(\mu(gK, hK)) = f( g\theta(g^{-1}h)K ) = f(g)\theta(f(g)^{-1}f(h))K = \mu(f(gK), f(hK)). \]
It follows that $f$ induces an automorphism of $\XXX$. Then for $(g,\theta^e)\in \Geff\rtimes\gen\theta = G(\XXX,\mu) \unlhd \Aut(\XXX,\mu)$, and any $hK\in G/K = \XXX$, we get
\[ c(f)(g,\theta^e)(hK)
= f( g\theta^e( f^{-1}(hK)))
= f(g)\theta^e(hK)
= (f(g), \theta^e)(hK).
\]
Thus $\phi(c(f))=(f,1)\in\Hol_\theta(\Geff)$, and $\psi(f) = \kappa^{-1}(f,1) = f \in \AutEff(G)$.
\end{enumerate}

\item
By Proposition~\ref{AutBuilding} one has $\Aut(G) \cong \Aut(\ol{G}) \cong \Aut(\Ad(G))$.
It therefore suffices to construct an isomorphism $\Aut(\ol{\XXX}) \to \Aut(\Ad(G))$, which follows analogous to case (i).
\qedhere
\end{enumerate}
\end{proof}

\begin{remark} By Theorem \ref{ThmAut} one can write $\Aut(G)$ as a semidirect product $\Aut(G) = \Aut^+(G) \rtimes \gen{\theta}$, where $\Aut^+(G)$ denotes the index two subgroup generated by all inner automorphisms, diagram automorphisms and diagonal automorphisms. In the sequel denote by $\Aut^+(\ol{\XXX})$ the image of $\Aut^+(G)$ under the isomorphism $\Aut(G) \to \Aut(\ol{\XXX})$ from Theorem \ref{ThmGlobalAut}. Then 
\begin{equation}\label{Aut+}
\Aut(\ol{\XXX}) = \Aut^+(\ol{\XXX}) \rtimes \langle s_o \rangle,
\end{equation}
where $o \in \ol{\XXX}$ is an arbitrary basepoint. We also denote by $\Aut_S(\ol{\XXX})$ the subgroup of $\Aut(\ol{\XXX})$ which corresponds to $\Aut_S(G)$ under the isomorphism from Theorem \ref{ThmGlobalAut}.(ii). Note that both $\Aut^+(\ol{\XXX})$ and $\Aut(\ol{\XXX})$ contain the transvection group $\Ad(G)$. 
\end{remark}

\subsection{Local transformations and the Coxeter complex}\label{SubsecLocalAut}

\begin{convention} For the remainder of Section \ref{aut} we assume that ${\mathbf A}$ is non-affine.
\end{convention}
In this section we investigate the local transformations of $\ol{\XXX}$. Recall that for a pointed maximal flat $(p, F)$ the set $F^\mathrm{sing}(p)$ of singular points of $F$ with respect to $p$ and the group $\GL(p,F, F^{\rm sing}(p))$ of local transformations of $(p,F)$ were defined in Definition~\ref{defsingularregular}. By strong transitivity, these notions do not depend on the choice of pointed maximal flat up to isomorphism, and we will work with the \Defn{standard pointed flat} $(e, \ol{A}\ol{K})$ of the coset model.

By Proposition~\ref{StdFlatsEuclidean} a chart of the flat $\ol{A}\ol{K}$ centred at $e$ is given by
\begin{equation}\label{StandardChart}
\varphi_e:  \ol{\mathfrak a} \to \ol{A}\ol{K}, \quad X \mapsto \exp(X)\ol{K}.
\end{equation}
Recall from Definition~\ref{defsingularregular} that $\ol{AK}^{\rm sing}(e)$ denotes the
subset of $\ol{AK}$ consisting of points singular with respect to $e$. If one defines 
\[
\ol{\mathfrak a}^{\rm sing} :=\varphi_e^{-1}(\ol{AK}^{\rm sing}(e)) \quad \text{and} \quad \GL(\ol{\mathfrak a}, \ol{\mathfrak a}^{\rm sing}) := \{f \in \GL(\ol{\mathfrak a})\mid f(\ol{\mathfrak a}^{\rm sing} )= \ol{\mathfrak a}^{\rm sing} \},
\]
then one obtains an isomorphism
\begin{equation}\label{AutGlobalToLocal}
\GL(e, \ol{A}) \xto{\cong} \GL(\ol{\mathfrak a}, \ol{\mathfrak a}^{\rm sing}), \quad F \mapsto f:= \varphi_e^{-1}\circ F \circ \varphi_e.
\end{equation}
By strong transitivity of $\ol{G}$ (see Corollary~\ref{ConsequencesFlatClassification}), any two maximal flats through $e$ in $\ol{G}/\ol{K}$ are $\ol{K}$-conjugate, i.e.
\begin{equation}\label{asing1}
\ol{\mathfrak a}^{\rm sing} = \bigcup_{\{k \in K \mid \Ad(k)\ol{\mathfrak a} \neq \ol{\mathfrak a}\}} \ol{\mathfrak a} \cap \Ad(k)\ol{\mathfrak a},
\end{equation}
The results from the appendix allow one to describe this set in a more combinatorial way. Recall from Definition \ref{DefKMRep} the definition of the Kac--Moody representations $\rho_{KM}: W \to \GL({\mathfrak a})$ and the reduced Kac--Moody representation $\ol{\rho}_{KM}: W \to \GL(\ol{\mathfrak a})$ of the Weyl group. As ${\bf A}$ is assumed to be non-affine, both of these representations are faithful by Corollary~\ref{CorollaryRedKMRepFaithful}, and reflections in $W$ act as linear reflections under these representations.

Given a real root $\alpha \in \Phi$ with associated root reflection $\check r_\alpha \in W$ denote by  \[{H}_{\alpha} := \Fix({\rho}_{KM}(\check r_\alpha)) < {\mathfrak a} \quad \text{and} \quad \ol{H}_{\alpha} := \Fix(\ol{\rho}_{KM}(\check r_\alpha)) < \ol{\mathfrak a}\] the corresponding reflection hyperplanes in $\mathfrak a$ and $\ol{\mathfrak a}$ respectively (cf.~Definition~\ref{overlinehDEF}). Recall from \eqref{formcoroots} on page \pageref{formcoroots} that the reflections ${\rho}_{KM}(\check r_\alpha)$ and $\ol{\rho}_{KM}(\check r_\alpha)$ are orthogonal with respect to
suitable choices of invariant bilinear forms on $\mathfrak a$ and $\ol{\mathfrak a}$. Since the invariant form on $\ol{\mathfrak{a}}$ is non-degenerate (cf.\ the proof of Proposition~\ref{IntertwinerWeyl1}), the map $\ol{\rho}_{KM}(\check r_\alpha)$ is in fact the unique orthogonal reflection at the hyperplane $\ol{H}_\alpha$. This implies in particular that the map $\alpha \mapsto \ol{H}_\alpha$ defines a one-to-one correspondence between positive reals roots $\alpha \in \Phi^+$ and reflection hyperplanes $\ol{H}_\alpha$.  
\begin{proposition}\label{hyperplanelog}
  Assume that $G$ is of non-affine type.
Under the chart $\varphi_e$ the singular set of the pointed maximal flat $(e, \ol{A})$ in $\ol{G}/\ol{K}$ corresponds to the union of the reflection hyperplanes of root reflections under $\widehat{\rho}_{KM}$, i.e.
 \[\varphi_e^{-1}(\ol{A}^{\rm sing}(e)) = \ol{\mathfrak a}^{\rm sing} = \bigcup_{\alpha \in \Phi^+} \ol{H}_{\alpha}.\]
\end{proposition}

\begin{proof} Assume first that $X \in \ol{\mathfrak a}^{\rm sing}$. By \eqref{asing1} there exists $k \in K$ with $\Ad(k) \ol{\mathfrak a}\neq  \ol{\mathfrak a}$ such that $X\in  \ol{\mathfrak a}\cap \Ad(k) \ol{\mathfrak a}$. Recall that $\ol{T}$ is the unique maximal torus of $\ol{G}$ such that $\ol{A} = \ol{T} \cap \ol{\tau}(\ol{G})$. Consequently, $\ol{T}^k$ is the unique maximal torus of $\ol{G}$ such that $\ol{A}^k = \ol{T}^k \cap \ol{\tau}(\ol{G})$. Assuming $X \neq 0$, one obtains a non-trivial intersection $H := \ol{T} \cap \ol{T}^k \ni \exp(X)$. 

As in \cite[Proposition~4.6]{Caprace09} let
\[\Phi^H = \{ \alpha \in \Phi \mid [U_\alpha, H] = 1 \} \qquad \text{ and } \qquad \ol{G}^H = \ol{T}.\langle U_\alpha \mid \alpha \in \Phi^H \rangle.\] Since $H$ is contained in the distinct tori $\ol{T}$ and $\ol{T}^k$, it is not regular in the sense of \cite[Section~4.2.3]{Caprace09}, i.e., $H$ fixes more than a single twin apartment of the twin building $\Delta$. Hence \cite[Proposition~4.6(i)(ii)]{Caprace09} imply that $(\ol{G}^H,(U_\alpha)_{\alpha \in \Phi^H})$ is a locally $\R$-split twin root datum with Weyl group $W^H = \langle s_\alpha \mid \alpha \in \Phi^H \rangle$ and maximal torus $\ol{T}$. Also $\ol{T}^k$ is a maximal torus of $\ol{G}^H$ by \cite[Proposition~4.6(v)]{Caprace09}, and $\ol{G}^H$ centralizes $H$. Since $\ol{G}^H$ acts transitively on twin apartments of the twin building associated with the twin root datum $(\ol{G}^H,(U_\alpha)_{\alpha \in \Phi^H})$ and these correspond to maximal tori in $\ol{G}^H$  (see e.g.\ \cite[Corollary~8.78]{AbramenkoBrown2008}), one deduces that $\ol{T}$ and $\ol{T}^k$ are conjugate in $\ol{G}^H$. 

Next observe that $H$ is $\theta$-invariant as $\ol{T}$ and $\ol{T}^k$ are. It then follows that for each $\alpha \in \Phi^H$ one has $-\alpha \in \Phi^H$, because \[[U_\alpha, H] = 1 \qquad \Longleftrightarrow \qquad [U_{-\alpha}, H] = [\theta(U_\alpha), \theta(H)] = 1.\] Therefore $\theta$ leaves $\langle U_\alpha, U_{-\alpha} \rangle$, $\alpha \in \Phi^H$ invariant and acts as an automorphism on $\ol{G}^H$. Consequently the group $\ol{G}^H$ admits an Iwasawa decomposition $\ol{G}^H = \ol{K}^H \ol{A}\, \ol{U}^H$, where $\ol{K}^H\leq \ol{K}\cap \ol{G}^H$ and $\ol{U}^H\leq \ol{U}_+\cap \ol{G}^H$. 

By \cite[Theorem 1.2]{deMedtsGramlichHorn09} the commutator subgroup $[\ol{K}^H,\ol{K}^H]$ is generated by the family $(\ol{K} \cap \langle U_\alpha, U_{-\alpha} \rangle)_{\alpha \in \Phi^H}$ and acts chamber-transitively on each half of the twin building of $\ol{G}^H$. In particular, there exist suitable $\beta_i \in \Pi^H$ and $k_i \in \ol{K} \cap \langle U_{\beta_i}, U_{-\beta_i} \rangle$ such that \[\prod_{i=1}^t k_i =: k^H \in [\ol{K}^H,\ol{K}^H] \] maps any chosen pair $(c, \theta(c))$ of opposite chambers of the twin building of $\ol{G}^H$ with stabilizer $\ol{T}$ onto some pair $(d,\theta(d))$ of opposite chambers with stabilizer $\ol{T}^k$. Consequently, the groups $\ol{T}$ and $\ol{T}^k$ are conjugate by the element $k^H \in [\ol{K}^H,\ol{K}^H]$.

%
%
%
Not all elements $k_i$ can normalize $\ol{T}$, for otherwise $\ol{T}^k = \ol{T}$.
Pick $i \in \{ 1, \dots, t \}$ such that $\ol{T}^{k_i} \neq \ol{T}$. Then $H \leq \ol{T}^{k_i} \cap \ol{T}$, as $k_i \in \ol{G}^H$. Furthermore, $\ol{T}^{k_i} \cap \ol{T}$ has corank $1$ in $\ol{T}$, because $\beta_i \in \Pi^H \subset \Phi^H \subset \Phi$ and $k_i \in \langle U_{\beta_i},U_{-\beta_i} \rangle$. Now  $\widetilde s_{\beta_i} \in N_{\ol{K^H}}(A) \leq N_{\ol{K}}(\ol{A}) \leq N_{\ol{G}}(\ol{T})$ (cf.\ Section~\ref{extendedWeylgroup} and Lemma~\ref{lem:NGT=A NKT}) fixes the intersection $\ol{T}^{k_i} \cap \ol{T}$ and, as it has corank one in $\ol{T}$, this intersection must be the exponential of the reflection hyperplane of $H_{\beta_i}$. 
This shows that $X \in H_{\beta_i}$ and, since $X$ was arbitrary, one obtains $ \ol{\mathfrak a}^{\rm sing} \subset \bigcup_{\alpha \in \Phi} H_{\alpha}$.

Conversely, if $X \in H_{\alpha}$, then $\exp(X) \in \ol{A} \cap \ol{A}^{k}$, where $k\in K \cap \langle U_{\alpha}, U_{-\alpha}\rangle$ is any element not normalizing $\ol{T}$.
\end{proof}

One concludes from Proposition~\ref{hyperplanelog} that the subset $\ol{\mathfrak a}^{\rm sing} \subset \ol{\mathfrak a}$ is precisely the hyperplane arrangement which is denoted by the same symbol $\ol{\mathfrak a}^{\rm sing}$ in the appendix. Note in passing that  Proposition~\ref{hyperplanelog} carries over to non-reduced Kac--Moody symmetric space as follows: 

\begin{corollary}\label{hyperplanelogcor} Assume that $G$ is of non-affine type. Under the chart $\varphi_e: \mathfrak a \to AK$, $X \mapsto \exp(X)K$ the singular set of the pointed maximal flat $(e, {A})$ in $G/K$ is given by
 \[\varphi_e^{-1}({A}^{\rm sing}(e)) = {\mathfrak a}^{\rm sing} = \bigcup_{\alpha \in \Phi^+} H_{\alpha}.\]
\end{corollary}

\begin{proof}
This follows from Corollary~\ref{ConsequencesFlatClassification} and Proposition~\ref{hyperplanelog}; see also Proposition~\ref{IntertwinerWeyl1}.
\end{proof}

We also record the following consequence of the proof of Proposition \ref{hyperplanelog}:

\begin{corollary}\label{twopointsintwoflats}
Assume that $G$ is of non-affine type and let $F_1$, $F_2$ be maximal flats of $\ol{\XXX}$. Then there exists $g \in \ol{G}$ stabilizing $F_1 \cap F_2$ elementwise with $g(F_1) = F_2$.
\end{corollary}

\begin{proof}
If $|F_1 \cap F_2| \leq 1$, this is an immediate consequence of strong transitivity. If $|F_1 \cap F_2| \geq 2$, then strong transitivity allows one to assume $e \in F_1 \cap F_2$. Then, as in the proof of Proposition~\ref{hyperplanelog}, the maximal flats $F_1$ and $F_2$ correspond to maximal tori $\ol{T}$ and $\ol{T}^k$ with non-trivial intersection $H$. By the arguments given in that proof the maximal tori $\ol{T}$ and $\ol{T}^k$ are in fact conjugate by an element of $\ol{K} = \Stab_{\ol{G}}(e)$ centralizing $H$.
\end{proof}
The same argument also applies to the $G$-action on maximal flats in $\XXX$.

\medskip
Proposition \ref{hyperplanelog} allows one to compute the group $\GL(p, F, F^{\rm sing}(p))$ of local transformations of $(p, F)$ for a pointed flat $(p, F)$. By strong transitivity, every pointed maximal flat $(p, F)$ can be mapped by an automorphism of $\ol{\XXX}$ to the standard pointed flat $(e, \overline{AK})$. Composing the chart $\varphi_e$ defined in \eqref{StandardChart} with this automorphism provides a chart $\varphi: \overline{\mathfrak a} \to F$ centred at $p$ which by Proposition \ref{hyperplanelog} identifies $ \ol{\mathfrak a}^{\rm sing}$ with $F^{\rm sing}(p)$. Hence for every pointed flat $(p, F)$ there exists an isomorphism
\[\GL(p, F) \cong  \GL(\ol{\mathfrak a}, \ol{\mathfrak a}^{\rm sing}) = \left\{f \in \GL(\ol{\mathfrak a})\mid f\left( \bigcup_{\alpha \in \Phi^+} \ol{H}_{\alpha}\right) = \bigcup_{\alpha \in \Phi^+} \ol{H}_{\alpha} \right\}.\]
Since ${\bf A}$ is irreducible, symmetrizable and non-affine, it follows from Corollary \ref{CanHomothetyClass} that there exists a non-degenerate bilinear form  $B$ on ${\ol{\mathfrak a}}$ such that every element of $\GL(\ol{\mathfrak a}, \ol{\mathfrak a}^{\rm sing})$ is a similarity with respect to $B$, i.e.\ a product of a $B$-orthogonal linear transformation and a homothety. Moreover, this form $B$ is unique up to multiples. Since this result depends only on the hyperplane arrangement $(\ol{\mathfrak a}, \ol{\mathfrak a}^{\rm sing})$ one concludes the following:
\begin{corollary}\label{CanonicalFormpF} Assume that ${\bf A}$ is non-affine. Let $(p, F)$ be a pointed flat and let $\varphi: \R^{\rk({\bf A})} \to F$ be a chart centred at $p$. Then $\varphi^{-1}(F^{\rm sing}(p))$ is a hyperplane arrangement in $ \R^{\rk({\bf A})} $, and there exists a non-degenerate bilinear form $B$ on $ \R^{\rk({\bf A})} $, unique up to multiples, such that every linear transformation of $\R^{\rk({\bf A})}$ preserving this hyperplane arrangement is a similarity with respect to $B$. \qed
\end{corollary}
Let $f \in \GL(p, F, F^{\rm sing}(p))$ be a local transformation. Then for every chart $\varphi: \R^{\rk({\bf A})} \to F$ centred at $p$ the map $\varphi^{-1} \circ f \circ \varphi$ is a linear map preserving the hyperplane arrangement $\varphi^{-1}(F^{\rm sing}(p))$. By the preceding corollary it can be written as a product of a homothety and a $B$-orthogonal transformation.

The map $f$ is called a \Defn{local automorphism} of $(p, F)$ if it is $B$-orthogonal, i.e., if it does not involve a non-trivial homothety. This notion does not depend on the choice of $B$ (since $B$ is unique up to multiples), nor on the choice of chart (since a change of charts maps the corresponding hyperplane arrangements and, thus, the associated forms to each other). Denoting the group of local automorphisms of $(p, F)$  by $\Aut(p, F)$, leads to a splitting
\[
\GL(p, F, F^{\rm sing}(p)) \cong \R^{>0} \times \Aut(p, F),
\]
where $\R^{>0}$ acts on $F$ by homotheties.

It is possible to describe the right-hand side explicitly. Under a suitable chart, $$\Aut(p, F) = \mathrm{O}(\overline{\mathfrak a}, \overline{\mathfrak a}^{\rm sing}) := {\rm O}(\ol{\mathfrak a},B) \cap \GL(\ol{\mathfrak a} , \ol{\mathfrak a}^{\rm sing}),$$ where $B$ is a bilinear form in the canonical homothety class for $(\ol{\mathfrak a}, \ol{\mathfrak a}^{\rm sing})$ (cf.\ Definition \ref{DefCanHomothetyClass}) that, under a suitable isomorphism $\ol{\mathfrak{a}} \cong \R^{\rk({\bf A})}$, can in fact be chosen to be the bilinear form $B$ from Corollary~\ref{CanonicalFormpF}.

Fix a simplicial Coxeter complex $\Sigma$ for $(W, S)$ (see Subsection \ref{appendixAA}), denote by $\Aut(\Sigma)$ its group of simplicial automorphisms
and by $\Aut(W, S)$ the automorphism group of the Coxeter graph of $W$ with respect to $S$. Proposition~\ref{autocoxeter} and Remark~\ref{SphericalAffineNonTrichotomy} imply the following:
\begin{corollary}\label{LocalAutomorphismClassification} Assume that ${\bf A}$ is  non-spherical and non-affine. Then for every pointed flat $(p, F)$ one has
\[
\Aut(p, F) \cong \Aut(\Sigma) \times \Z/2\Z \cong (W \rtimes \Aut(W,S)) \times \Z/2\Z,
\]
and hence
\[
\GL(p, F, F^{\rm sing}(p)) \cong \R^{>0} \times(W \rtimes \Aut(W,S)) \times \Z/2\Z.
\]
\end{corollary}
\begin{remark}
\begin{enumerate}
\item In the spherical case, the same result holds, except that the $\Z/2\Z$-factor is missing (see Remark \ref{SphericalAffineNonTrichotomy}).
\item The isomorphisms in Corollary~\ref{LocalAutomorphismClassification} can be made more explicit: Let $g$ be an automorphism of $\ol{\XXX}$ which maps $(p, F)$ to the standard pointed flat $(e\ol{K}, \ol{A}\ol{K})$ and let $\varphi_e:  \ol{\mathfrak a} \to \ol{A}\ol{K}$ as in \eqref{StandardChart}. Then $\varphi: = g \circ \varphi_e :  \ol{\mathfrak a} \to F$ is a chart for $F$ centred at $p$ with $\varphi(F^{\rm sing}(p)) = \ol{\mathfrak a}^{\rm sing}$. In particular, if $f \in \GL(p, F, F^{\rm sing}(p)) $, then $\phi \circ f \circ \phi^{-1} \in \GL(\ol{\mathfrak a}, \ol{\mathfrak a}^{\rm sing}) < \GL(\ol{\mathfrak a})$. This linear map can then be written as a product of a homothety, an element of the Weyl group acting on $\overline{\mathfrak a}$ by the reduced Kac--Moody representation (see Definition \ref{DefKMRep}), a Cayley graph isomorphism of $(W,S)$ and possibly the antipode map $v \mapsto -v$. Here the action of $\Aut(W,S)$ on $\ol{\mathfrak a}$ is given as follows: By the discussion in Subsection \ref{TitsCones}, the reduced Tits cone $\ol{\mathcal C} \subset \ol{\mathfrak a}$ is a cone over a coloured polyhedral complex whose dual graph is isomorphic to the Cayley graph of $W$ with respect to $(W,S)$, and hence $\Aut(W,S)$ acts on the reduced Tits cone by combinatorial automorphisms, which can be realized uniquely by linear automorphisms of the ambient vector space $\ol{\mathfrak a}$.
\item All homotheties and all elements of $\Aut(\Sigma)$ preserve the Tits cone when acting on $\overline{\mathfrak{a}}$, whereas the antipodal map exchanges the Tits cone and its negative. In particular, all elements of $\GL(p, F, F^{\rm sing}(p))$ preserve the \Defn{Tits double cone}, i.e., the union of the Tits cone and its negative.
\end{enumerate}
\end{remark}

\subsection{Local vs.\ global automorphisms}

By Corollary~\ref{ConsequencesFlatClassification} the Kac--Moody group $\ol{G}$ and hence the full automorphism group $\Aut(\XXX)$ act strongly transitively on $\XXX$. In particular, the corresponding Weyl groups $W(\Aut(\ol{\XXX}) \curvearrowright \ol{\XXX})$ and $W(\ol{G}\curvearrowright\ol{\XXX})$ and local actions are well-defined (see Definition~\ref{DefLocalAction}). A priori, these local actions take values in the group $\GL(p, F, F^{\rm sing}(p))$ of local transformation of a given pointed flat. If ${\bf A}$ is non-affine, then they actually take value in the subgroup $\Aut(p, F) < \GL(p, F, F^{\rm sing}(p))$ of local automorphisms, as we will discuss in this section.

Recall that $\ol{M} < \ol{T}$ denotes the torsion subgroup of $\ol{T}$ so that $\ol{T} = \ol{A} \times \ol{M}$.
\begin{proposition}\label{GeometricWeylGroup1} Assume that ${\bf A}$ is non-affine and let $(p, F)$ be a pointed flat in $\ol{\XXX}$.
\begin{enumerate}
\item $\Stab_{\ol{G}}(p, F) \cong N_{\ol{K}}(\ol{T})$ and  $\Fix_{\ol{G}}(p, F) \cong \ol{M}$.
\item The geometric Weyl group $W(\ol{G}\curvearrowright\ol{\XXX})$ is isomorphic to the algebraic Weyl group $W$ of $\ol{G}$.
\item There exists a chart $\phi: \overline{\mathfrak a} \to F$ centred at $p$ which intertwines the action of $W$ on $F$ via the isomorphism in (ii) and the reduced Kac--Moody representation.
\end{enumerate}
\end{proposition}
\begin{proof} By Proposition \ref{2.29} one may assume without loss of generality that $(p, F)$ is given by the standard pointed flat $(e, \ol{AK})$. By Remark \ref{5.4} the stabilizer in $\overline{G}$ of the standard flat $\ol{AK}$ is given by $N_{\ol{G}}(\ol{T})$. Since the fixator of $e$ is given by $\ol{K}$ one has $\Stab_{\ol{G}}(e, \ol{AK}) = N_{\ol{K}}(\ol{T})$. Recall from Corollary~\ref{AbstractExtendedWeylGroup} that if $\pi: G \to \ol{G}$ denotes the canonical projection, then $N_{\ol{K}}(\ol{T}) = \pi(\widetilde{W})$ is the image of the extended Weyl group. In particular, since $M < \widetilde{W}$, the stabilizer $\Stab_{\ol{G}}(e, \ol{AK})$ contains $\ol{M} = \pi(M)$.

Consider the action of $\Stab_{\ol{G}}(e, \ol{AK}) $ on the standard flat $\ol{AK}$. The subgroup $\ol{M}$ centralizes $\ol{A}$ and is contained in $\ol{K}$, hence acts trivially on $\ol{AK}$, i.e., $\ol{M}<\Fix_{\ol{G}}(e, \ol{AK})$. Consequently, the action of $\Stab_{\ol{G}}(e, \ol{AK})$ factors through the group
\[
\Stab_{\ol{G}}(e, \ol{AK})/\ol{M} = \pi(\widetilde{W})/\pi(M) = \pi(\widetilde{W}/M).
\]
The standard chart $\phi: \ol{\rm a} \to \ol{AK}$ from \ref{StandardChart} intertwines the action of this group on $\ol{AK}$ with the restriction of the adjoint action on $\ol{\mathfrak a}$. 
As discussed in Subsection \ref{extendedWeylgroup} there exists an isomorphism $\widetilde{W}/M \cong W$ and under this isomorphism the adjoint action of $\widetilde{W}/M$ on $\mathfrak a$ is given by the Kac--Moody representation of $W$. It follows that the adjoint action of $\Stab_{\ol{G}}(e, \ol{AK})/\ol{M}$ on $\ol{a}$ identifies $\Stab_{\ol{G}}(e, \ol{AK})/\ol{M}$ with the subgroup $\ol{\rho}_{KM}(W) < \GL(\ol a)$. In particular, since the reduced Kac--Moody representation is faithful, one obtains an isomorphism $\Stab_{\ol{G}}(e, \ol{AK})/\ol{M} \cong W$. Moreover, since every element of $W$ acts non-trivially on $\ol{\mathfrak a}$ the inclusion $\ol{M}\hookrightarrow \Fix_{\ol{G}}(e, \ol{AK})$ is actually an equality. This finishes the proof.
\end{proof}
As before denote by $\Sigma = \Sigma(W,S)$ a simplicial Coxeter complex of the Coxeter system $(W,S)$ underlying {\bf A}. Recall from Corollary \ref{LocalAutomorphismClassification} (or Lemma~\ref{AutWS}) that the simplicial automorphism group  $\Aut(\Sigma)$ splits as the semidirect product $\Aut(\Sigma) = W \rtimes \Aut(W,S)$, where $\Aut(W,S)$ denotes the group of automorphisms of the Coxeter diagram, and that $ \GL(p, F, F^{\rm sing}(p)) \cong \R^{>0} \times(W \rtimes \Aut(W,S)) \times \Z/2\Z$ as long as ${\bf A}$ is non-spherical and non-affine. Proposition~\ref{GeometricWeylGroup1} and Corollary~\ref{LocalAutomorphismClassification} imply the following:
\begin{corollary}\label{GeometricWeylGroup2}
Assume that ${\bf A}$ is non-spherical and non-affine and let $(p, F)$ be a pointed flat in $\ol{\XXX}$. Then the local action of $ W(\ol{G}\curvearrowright\ol{\XXX})$ on $(p, F)$ is intertwined by the isomorphisms from Corollary \ref{LocalAutomorphismClassification} and Proposition \ref{GeometricWeylGroup1} with the canonical inclusion 
\[
W \hookrightarrow \R^{>0} \times(W \rtimes \Aut(W,S)) \times \Z/2\Z,
\]
i.e.\ the local action fits into a commutative diagram of the form
\[\begin{xy}\xymatrix{ W(\ol{G}\curvearrowright\ol{\XXX}) \ar[d]_\cong \ar[r]& \GL(p, F, F^{\rm sing}(p))\ar[d]^\cong \\
W \ar[r]&\R^{>0} \times(W \rtimes \Aut(W,S)) \times \Z/2\Z,
}\end{xy}\]
In particular, the local action takes values in the subgroup $\Aut(p, F) <  \GL(p, F, F^{\rm sing}(p))$ of local automorphisms.\qed
\end{corollary}
We conclude this section by analyzing the action of the group $W(\Aut(\ol{\XXX}) \curvearrowright \ol{\XXX})$. Recall from Remark~\ref{RemarkCoxeterVsDynkin} that $\Aut(\Gamma_{\mathbf{A}})<\Aut(W,S)$. 
\begin{theorem}\label{ExtendLocalAutomorphisms} 
Assume that ${\bf A}$ is non-spherical and non-affine and let $(p, F)$ be a pointed flat in $\ol{\XXX}$. Then the local action of $W(\Aut(\ol{\XXX}) \curvearrowright \ol{\XXX})$ fits into a commutative diagram of the form
\[\begin{xy}\xymatrix{ W(\Aut(\ol{\XXX})\curvearrowright\ol{\XXX}) \ar[d]_\cong \ar[r]& \GL(p, F, F^{\rm sing}(p))\ar[d]^\cong \\
(W \rtimes \Aut(\Gamma_{\mathbf{A}}))  \times \Z/2\Z \ar[r]&\R^{>0} \times(W \rtimes \Aut(W,S)) \times \Z/2\Z.
}\end{xy}\]
In particular, the local action takes values in the group of local automorphisms. Moreover, every local automorphism extends to a global automorphism if and only if $\Aut(\Gamma_{\mathbf{A}})= \Aut(W,S)$.
\end{theorem}
\begin{proof} By Theorem \ref{ThmGlobalAut} one has $\Aut(\ol{\XXX}) \cong \Aut(\ol{G})$ and by Theorem \ref{ThmAut} every automorphism of $\ol{G}$ can be written as a product of an inner automorphism of $\ol{G}$, a diagram automorphism, a diagonal automorphism and a power of the Cartan--Chevalley involution $\theta$. One needs to determine which of these automorphisms stabilize the standard pointed flat $(p, F) = (e, \ol{AK})$, and how they act on $F$. Among the inner automorphism, these are precisely the elements of $N_{\ol{K}}(\ol{T})$, and these correspond to the elements of $W$ by Corollary \ref{GeometricWeylGroup2}. In addition, all diagram automorphisms stabilize the standard pointed flat and act as Coxeter automorphisms, and all diagonal automorphisms fix the standard pointed flat pointwise. Finally, the Cartan-Chevalley involution preserves the standard pointed flat and acts on it by inversion, hence it corresponds to the generator of $\Z/2\Z$. The theorem follows.
\end{proof}
The same argument also shows that $W(\Aut^+(\ol{\XXX})\curvearrowright\ol{\XXX}) \cong W \rtimes \Aut(\Gamma_{\mathbf{A}})$. We have shown Theorem~\ref{thm1.11} and Corollary~\ref{thm1.12}.

\section{Causal structures and the causal boundary}\label{SecCausal}

We keep the notation of the previous section. That is, $G$ denotes a simply connected centred split real Kac--Moody group with semisimple adjoint quotient $\ol{G}$ and adjoint quotient $\Ad(G)$. We are going to consider the reduced Kac--Moody symmetric space $\ol{\XXX}$ in its group model $\ol{G}/\ol{K}$. 
\begin{convention}\label{ConventionSecCausal} Throughout Section \ref{SecCausal} we will assume that $G$ is of non-spherical and non-affine type.
\end{convention}

\subsection{Invariant causal structures}\label{SecCausalStructureEx}

The goal of this subsection is to introduce an $\Aut(\ol{\XXX})$-invariant field of double cones in $\ol{\XXX}$. Our starting point is the observation that the vector space $\ol{\mathfrak a}$ contains a canonical cone $\ol{\CCC} \subset \ol{\mathfrak a}$ with open interior $\ol{\CCC}^o$ and tip $0$, called the \Defn{reduced Tits cone}, see Section~\ref{TitsCones} in the appendix. Since the generalized Cartan matrix {\bf A} is irreducible non-spherical and non-affine, this cone is \Defn{pointed} in the sense that
\[
\ol{\CCC} \cap (-\ol{\CCC}) = \{0\}.
\]
Refer to the union $\ol{\CCC}^o \cup (- \ol{\CCC}^o)$ as the \Defn{open Tits double cone} in $\ol{\mathfrak a}$. Denote by $\ol{A}^o_\pm := \exp(\pm\ol{\CCC}^o)$ the corresponding subsemigroups of $\ol{A}$ and refer to $\ol{A}^o_+ \cup \ol{A}^o_-$ as the \Defn{canonical (open) double cone} in $\ol{A}$.

\begin{remark}\label{ConeWellDefined}
Let $F$ be an arbitrary flat through $e$ in the group model of $\ol{\XXX}$. By strong transitivity, there exists $k \in \Aut(\ol{\XXX})_e$ such that $k.\ol{A} = F$. Moreover, the subset $\ol{C}^{+, -}_e(F) := k.(\ol{A}^o_+ \cup \ol{A}^o_-) \subset F$ is independent of the choice of $k$. Indeed, if $k'$ is a different choice, then $k^{-1}k'$ acts on $\ol{A}$ by a local automorphism, and by Theorem \ref{ExtendLocalAutomorphisms} any such automorphism leaves the canonical double cone invariant. 
\end{remark}
Define
\[
\ol{C}^{+, -}_e := \bigcup \ol{C}^{+, -}_e(F),
\]
where the union is taken over all flats containing the basepoint $e$. 
\begin{proposition} Assume that ${\bf A}$ is non-spherical and non-affine. Then for every flat $F$ containing $e$ one has \[\ol{C}^{+, -}_e \cap F = \ol{C}^{+, -}_e(F).\] In particular, $\ol{C}^{+, -}_e$ intersects each flat in a double cone, whose two halves do not intersect.
\end{proposition}
\begin{proof} One needs to show that, if $F_1$ and $F_2$ are flats containing $e$ and $x \in F_1 \cap \ol{C}_e^\pm(F_2)$, then $x \in \ol{C}^\pm(F_1)$. By Corollary \ref{twopointsintwoflats} there exists $\alpha \in \Aut(\ol{\XXX})$ which fixes $F_1 \cap F_2$ pointwise and maps $F_2$ to $F_1$. In particular, since $x \in F_1 \cap F_2$ one has $x = \alpha(x) \in \alpha(\ol{C}_e^\pm(F_2))$. Moreover, since $e \in F_1 \cap F_2$ one has $\alpha \in \Aut(\ol{\XXX})_e$ and hence $\alpha(\ol{C}_e^\pm(F_2)) = \ol{C}_e^\pm(F_1)$ by the argument above. This shows $x \in \ol{C}_e^\pm(F_1)$ and finishes the proof.
\end{proof}

By abuse of language, we will also call $\ol{C}^{+, -}_e$ a double cone. By construction, this double cone is invariant under all automorphisms in $ \Aut(\ol{\XXX})_e$. In particular, if $x \in \ol{\XXX}$ and if $\alpha \in \Aut(\ol{\XXX})$ maps $e$ to $x$, then
\[
\ol{C}^{+, -}_x := \alpha(\ol{C}^{+, -}_e)
\]
is independent of the choice of $\alpha$. Moreover, if $\varphi: \ol{\mathfrak a}\to F$ is any chart centred at $x$, then
\[
\ol{C}^{+, -}_x [F] := \ol{C}^{+, -}_x \cap F = \varphi(\ol{\CCC}^o \cup (- \ol{\CCC}^o)).
\]
Note also that by construction the family $(\ol{C}^{+, -}_x)_{x \in \ol{\XXX}}$ of double cones is $\Aut(\ol{\XXX})$-invariant in the sense that
\[
\alpha(\ol{C}^{+, -}_x) = \ol{C}^{+, -}_{\alpha(x)} \quad (\alpha \in \Aut(\ol{\XXX}), x \in \ol{\XXX}).
\]
Refer to $(\ol{C}^{+, -}_x)_{x \in \ol{\XXX}}$ as the \Defn{canonical double cone field} on $\ol{\XXX}$.

If $\phi, \phi': (0, \ol{\mathfrak a}) \to (p, F)$ are charts, then $\phi^{-1} \circ \phi'$ is a linear map preserving the decomposition $\ol{\mathfrak a}  =\ol{\mathfrak a}^{\rm reg} \sqcup \ol{\mathfrak a}^{\rm sing}$ as well as the open double Tits cone $\ol{\CCC}^o \cup (- \ol{\CCC}^o) \subset \ol{\mathfrak a}$. There are thus two possibilities: Either $\phi^{-1} \circ \phi'$ preserves the open Tits cone or it maps the open Tits cone to its negative.

\begin{definition} Two charts $\phi, \phi': (0, \ol{\mathfrak a}) \to (p, F)$ of $F$ centred at the same point $p$ are called \Defn{causally equivalent} if $\phi^{-1} \circ \phi'$ preserves the open Tits cone. A \Defn{causal orientation} of $\ol{\XXX}$ is a choice of one of the two causal equivalence classes of charts for every maximal pointed flat $(p, F)$.

If a group $H$ acts by automorphisms on $\ol{\XXX}$, then a causal orientation is called \Defn{$H$-invariant} if for every $h \in H$ and every chart $\phi$ in the chosen causal equivalence class, also the chart $h \circ \phi$ is in the chosen equivalence class.
\end{definition}
\begin{proposition}\label{2causalstructures} There exist exactly two $\Aut^+(\ol{\XXX})$-invariant causal orientations on $\ol{\XXX}$.
\end{proposition}

\begin{proof} Since $\Aut^+(\ol{\XXX})$ acts strongly transitively on $\ol{\XXX}$ and since every pointed maximal flat admits only two causal equivalence classes, there are at most two $G$-invariant causal structures on $\ol{\XXX}$. By Theorem~\ref{ExtendLocalAutomorphisms} one has $W(\Aut(\ol{\XXX}) \curvearrowright \ol{\XXX}) \cong (W \rtimes \Aut(\Gamma_{\mathbf{A}}))  \times \Z/2\Z$, where the first factor acts on the Tits cone, and the second factor swaps the Tits cone and its negative. Moreover, $W(\Aut^+(\ol{\XXX}) \curvearrowright \ol{\XXX})$ is given by the subgroup $(W \rtimes \Aut(\Gamma_{\mathbf{A}})) \times \{e\}$. One thus obtains two distinct $\Aut^+(\ol{\XXX})$-invariant causal orientations, one for which the charts $\{\alpha \circ \exp\mid \alpha \in  \Aut^+(\ol{\XXX})\}$ are positive, and one for which the charts $\{-\alpha \circ \exp\mid \alpha \in  \Aut^+(\ol{\XXX})\}$ are positive.
\end{proof}
Charts in the unique $\Aut^+(\ol{\XXX})$-invariant causal orientation containing $\exp$ are called \Defn{positive charts}, charts in the unique $\Aut^+(\ol{\XXX})$-invariant causal orientation containing $-\exp$ \Defn{negative charts}. Given a pointed maximal flat $(x, F)$ in $\ol{\XXX}$ and a positive chart $\varphi: \ol{\mathfrak a} \to F$ centred at $x$ define
\[
\ol{C}^{+}_x [F] := \varphi(\ol{\CCC}^o) \quad \text{and} \quad \ol{C}^{-}_x [F] := \varphi(-\ol{\CCC}^o).
\]
By definitions, these cones do not depend on the choice of positive chart, and if one defines
\[
\ol{C}^{\pm}_x := \bigcup_{F \ni x} \ol{C}^{\pm}_x[F],
\]
then $\ol{C}^{+,-}_x = \ol{C}^{+}_x \cup \ol{C}^{-}_x$. This decomposes the canonical double cone field on $\ol{\XXX}$ into two cone fields.
\begin{definition} The cone field $(\ol{C}^{+}_x)_{x \in \ol{\XXX}}$ is called the \Defn{positive causal structure} on $\ol{\XXX}$, and the cone field $(\ol{C}^{-}_x)_{x \in \ol{\XXX}}$ is called the \Defn{negative causal structure} on $\ol{\XXX}$.
\end{definition}
Note that the positive and negative causal structure are invariant under $\Aut^+(\ol{\XXX})$, and in particular $G$-invariant. At this point we have established Proposition~\ref{PropCausalIntro}.
\begin{remark} In Lorentzian geometry, invariant causal structures arise naturally. Namely, if $(g_x)_{x \in X}$ is a Lorentzian metric on a manifold $X$, then the associated field of light cones $(\ol{\CCC}_x \subset T_xX)_{x \in X}$ is invariant under all Lorentzian automorphisms. In our setting, there is always an invariant bilinear form on $\ol{\mathfrak a}$, since ${\bf A}$ is assumed to be symmetrizable. However, this bilinear form need not be Lorentzian, and even if it is Lorentzian it may happen that the Tits cone is not contained in the light cone of the invariant Lorentzian form (see e.g. \cite{FKN}). We emphasize that our $G$-invariant causal structures are modelled on the Tits cone, rather than the light cone of a bilinear form, hence our geometry here is {\em causal} rather than {\em Lorentzian}. This being said, in certain hyperbolic examples, including $E_{10}$, the interiors of the Tits cone and the light cone coincide according to \cite{FKN}, \cite{CFF}; hence in these cases our results do admit a Lorentzian interpretation. In these examples our construction of causal boundaries below is a global version of the lightcone embedding provided in \cite{CFF}.
\end{remark}

\subsection{Causal geodesic rays and the municipality}

The positive causal structure gives rise to a notion of causal curves in the following standard way:
\begin{definition} Let $I\subseteq \R$ be an interval which is open on the right, i.e., for every $t\in I$ there is $\epsilon>0$ such that $t+\epsilon\in I$. A continuous map $\gamma:I \to \ol{\XXX}$ is called a \Defn{causal curve} if for every $t\in I$ there exists $\epsilon>0$ such that
\[
\gamma((t, t + \epsilon)) \subset\ol{C}^{+}_{\gamma(t)}.
\]
If instead for every $t \in I$ there exists $\epsilon>0$ such that
\[
\gamma((t, t + \epsilon)) \subset\ol{C}^{-}_{\gamma(t)},
\]
then $\gamma$ is called an \Defn{anti-causal curve}.
\end{definition}
A (anti-)causal curve, which is also a geodesic ray, respectively a geodesic segment, will be called a \Defn{(anti-)causal ray}, respectively \Defn{(anti-)causal segment}.

\begin{lemma}\label{CausalGeodesics} Let $r:[0, \infty) \to \ol{\XXX}$ be a geodesic ray, let $0< S < T < \infty$ and let $\gamma : [S, T] \to \ol{\XXX}$ be the geodesic segment obtained by restricting $r$ to $[S, T]$. Then the following are equivalent.
\begin{enumerate}
\item $\gamma$ is a causal segment.
\item $r(t) \in \ol{C}^+_{r(0)}$ for some $t \in \R$.
\item $r(t) \in \ol{C}^+_{r(s)}$ for all $0 \leq s < t < \infty$.
\item $\gamma(t) = r(t) \in \ol{C}^+_{r(s)}$ for all $S \leq s < t \leq T$.
\item $r$ is a causal ray.
\end{enumerate}
\end{lemma}
\begin{proof} The implications (i)$\Rightarrow$(ii), (iii)$\Rightarrow$(iv)$\Rightarrow$(i) and (iii)$\Rightarrow$(v)$\Rightarrow$(i) are immediate from the definitions. To show the remaining implication (ii)$\Rightarrow$(iii) one may assume by strong transitivity, that $r$ is contained in $\ol{A}$ and emanates from $e$, i.e.,  $r(t) = \exp(tX)$ for some $X \in \ol{\mathfrak a}$. Under this assumption, (ii) amounts to $tX \in \ol{\CCC}$ for some $t \in \R$. This implies that $(t-s)X \in \ol{\CCC}$ for all $0 \leq s < t < \infty$, which is (iii). 
\end{proof}

In the sequel $\partial_{\bullet}\ol{\XXX}$ denotes the collection of all geodesic rays $r: [0, \infty) \to \ol{\XXX}$. Then $\partial_{\bullet}\ol{\XXX}$ fibers over $\ol{\XXX}$ by the map
\begin{equation}\label{ev0}
{\rm ev}_0: \partial_{\bullet}\ol{\XXX} \to {\ol{\XXX}}, \quad r \mapsto r(0),
\end{equation}
and we refer to the  fiber $\partial_x \ol{\XXX} := {\rm ev}_0^{-1}(x)$ over $x$ as the \Defn{point horizon} of $x$. Given a flat $F$ containing $x$ we also denote by $\partial_x{\ol{\XXX}}[F] \subset \partial_\bullet{\ol{\XXX}}$ the subset of rays emanating from $x$ and contained in $F$. The action of the automorphism group $\Aut(\ol{\XXX})$ preserves geodesic rays and thus induces an action on $\partial_{\bullet}\ol{\XXX}$, for which the projection ${\rm ev}_0$ is equivariant. In particular, for every $x \in \ol{\XXX}$ the point stabilizer $\Aut(\ol{\XXX})_x$ acts on $\partial_x \ol{\XXX}$, and $\Aut(x, F)$ acts on $\partial_x{\ol{\XXX}}[F]$.

To explicitly parametrize geodesic rays in $\ol{\XXX}$, consider again the standard pointed maximal flat $(e, \ol{A})$ in the group model of $\ol{\XXX}$. Then the geodesic rays contained in $\ol{A}$ and emanating from $e$ are given by $r_{e, X}(t) := \exp(tX)$, where $X$ runs through the Lie algebra $\ol{\mathfrak a}$. Since $\ol{\XXX}$ is strongly transitive, every geodesic ray in $\ol{\XXX}$ is of the form  $r_{g, X}(t) := g.\exp(tX)$ for some $g \in \Aut(\ol{\XXX})$ and $X \in \ol{\mathfrak a}$. One thus obtains a surjective map
\[ \Aut(\ol{\XXX}) \times \ol{\mathfrak a} \to \partial_\bullet \ol{\XXX}, \quad (g, X) \mapsto r_{g,X}.\]
Note that this map is not injective, i.e.\ the ray $r_{g, X}$ does not determine the parameters $g$ and $X$.
\begin{definition} A geodesic ray $r:[0, \infty) \to \ol{\XXX}$ is called \Defn{regular} if it is contained in a unique maximal flat of $\ol{\XXX}$ and \Defn{singular} otherwise. 
\end{definition}

Note that by Lemma~\ref{LemmaBasicFlats} these notions are invariant under automorphisms of $\ol{\XXX}$. Recall the notation $\ol{\mathfrak a}^{\rm sing} := \log(\ol{A}^{\rm sing}(e))$ for the logarithm of the singular set of $(e, \ol{A})$ from Subsection \ref{SubsecLocalAut}; denote by $\ol{\mathfrak a}^{\rm reg} := \ol{\mathfrak a} \setminus\ol{\mathfrak a}^{\rm sing}$ its complement. In terms of the parametrization above, regular and singular geodesic rays can be characterized as follows.

\begin{lemma}\label{LemmaRegSing} The geodesic ray $r_{g, X}$ is singular if $X \in \ol{\mathfrak a}^{\rm sing}$ and regular if $X \in \ol{\mathfrak a}^{\rm reg}$.
\end{lemma}

\begin{proof} By invariance of regular/singular rays under automorphisms it suffices to show this for $g = e$. It therefore remains to prove that, if $X \in \ol{\mathfrak a}^{\rm reg}$, then the whole open ray $\{tX \mid t \in (0, \infty)\}$ is contained in $\ol{\mathfrak a}^{\rm reg}$. This, however, follows from the fact that $\ol{\mathfrak a}^{\rm sing}$ is a hyperplane arrangement by Proposition~\ref{hyperplanelog}.
\end{proof}

\begin{definition}\label{Muni} The subset $\Delta_\bullet \subset \partial_\bullet \ol{\XXX}$ consisting of all causal and anti-causal rays is called the \Defn{municipality} of $\ol{\XXX}$.
\end{definition}

The terminology refers to the fact, to be proved in Proposition~\ref{LocalMunicipality} below, that the fibers of $\Delta_\bullet$ with respect to the map ${\rm ev}_0$ are geometric realizations of the twin building of $G$, hence we will think of the municipality as a collection of (mutually isomorphic) twin buildings parametrized by $\ol{\XXX}$. By construction, $\Delta_\bullet \subset \partial_\bullet \ol{\XXX}$ is $\Aut(\ol{\XXX})$-invariant, and if one denotes by $\Delta^\pm_\bullet \subset \Delta_\bullet$ the collections of causal/anti-causal rays, then these are invariant under $\Aut^+(\ol{\XXX})$. Also note that one can characterize causal/anti-causal rays in terms of the standard parametrization as follows.
\begin{proposition} The ray $r_{g, X}$ is contained in $\Delta_\bullet$ if and only if $X \in \ol{\CCC}^o \cup - \ol{\CCC}^o$. \qed
\end{proposition}
Denote by
\[
\Delta^{\rm reg}_\bullet = \{ r_{g,X} \in \Delta_\bullet \mid X \in \ol{\mathfrak{a}}^{\mathrm{reg}} \}, \text{ respectively } \Delta^{\rm sing}_\bullet = \{ r_{g,X} \in \Delta_\bullet \mid X \in \ol{\mathfrak{a}}^{\mathrm{sing}} \},
\]
the subsets of regular, respectively singular rays in the municipality. Furthermore, given $x \in \ol{\XXX}$, denote by  $\Delta_x$,  $\Delta^{\rm reg}_x$ and $\Delta^{\rm sing}_x$ the corresponding fibers over $x$ by the map ${\rm ev}_0$.

Since the notion of a municipality ray is invariant under automorphisms, the subset $\Delta_\bullet \subset \partial_\bullet\ol{\XXX}$ is $\Aut(\ol{\XXX})$-invariant, and the induced $\Aut(\ol{\XXX})$-action preserves the decomposition $\Delta_\bullet = \Delta^{\rm reg}_\bullet \sqcup  \Delta^{\rm sing}_\bullet$. Consequently, for every $x \in \ol{\XXX}$ the point stabilizer $\Aut(\ol{\XXX})_x$ acts on $\Delta_x$ preserving the decomposition $\Delta_x= \Delta^{\rm reg}_x \sqcup  \Delta^{\rm sing}_x$.

\subsection{Ideal polyhedral complexes and their combinatorial descriptions}\label{SecIdealPoly}
Our next goal is to equip the municipality with a certain polyhedral structure and to compare this structure with a certain polyhedral realization of the twin building. In contrast to the spherical case, the classical language of simplicial complexes is not sufficient to discuss these matters for a number of reasons: Firstly, our complexes will be built from more general polyhedra than simplices. Secondly, we also need to discuss cones over polyhedral complexes, which have a polyhedral structure with \emph{non-compact} cells. Finally, we will need to work with subsets of polyhedral complexes in which some faces (of codimension $\geq 2$) are missing. We thus need to develop a framework which deals with all of these complications.

A \Defn{halfspace} in $\R^n$ is defined as a connected component of the complement of an affine hyperplane in $\R^n$. An intersection of finitely many half-spaces will be called a \Defn{polyhedron} if it is non-emtpy. Thus by definition polyhedra are always closed and convex, but not necessarily compact. A non-empty convex subset $F$ of a polyhedron $P$ is called a \Defn{face} if for every $x \in F$ and every $y, z \in P$ such that $x$ lies on the line segment between $y$ and $z$ we have $\{y,z\} \subset F$. Every face of a polyhedron is again a polyhedron.

\begin{defn}
A pair $(X, (\varphi_i)_{i \in I})$ consisting of a set $X$ and a family of injective maps $\varphi_i: P_i \to X$ from polyhedra $P_i$ to $X$ is called a \Defn{polyhedral complex} if the following two conditions are satisfied for all $i, j \in I$:
\begin{enumerate}
\item If $F$ is a face of $P_i$, then there exists $k \in I$ such that $\varphi_k(P_k) = \varphi_i(F)$ and $\varphi_k^{-1} \circ \varphi_i|_F$ is an isometry.
\item If $\varphi_i(P_i) \cap \varphi_j(P_j) \neq \emptyset$, then there exist $k \in I$ and $\varphi_k: P_k \to X$ such that $\varphi_k(P_k) = \varphi_i(P_i) \cap \varphi_j(P_j)$ and $\varphi_i^{-1} \circ\varphi_k$ and $\varphi_j^{-1}\circ \varphi_k$ are isometric embeddings.
\end{enumerate}
The \Defn{weak topology} on $X$ is the weakest topology which makes all the inclusions $\varphi_i$ continuous. 
\end{defn}
In the sequel we will always equip polyhedral complexes with their weak topology unless mentioned otherwise. If all the polyhedra $P_i$ are simplices, then we recover the notion of a \Defn{simplicial complex}. We refer to the images $\varphi(P_i)$ as \Defn{closed cells} of $X$; the interior of a closed cell is called an \Defn{open cell}, and an open or closed cell which is not the face of any other cell is called an open or closed \Defn{chamber}. An open or closed cell $\sigma$ is called a \Defn{face} of an open or closed cell $\tau$ if $\sigma \subset \overline{\tau}$. In this case the difference between the dimension of $\tau$ and the (covering) dimension of $\sigma$ is called the \Defn{codimension} of $\sigma$ in $\tau$.

If $X$ is a polyhedral complex, then the \Defn{polyhedral cone} $CX$ over $X$ is the following polyhedral complex: The underlying set of $CX$ is the quotient of $([0, \infty) \times X)$ obtained by collapsing $\{0\} \times X$. The polyhedral structure is then obtained by declaring the basepoint $[(0,x)]$ to be a closed cell and declaring the image of $[0, \infty) \times \sigma$ to be a closed cell for every closed cell $\sigma \subset X$.  

\begin{defn} If $\overline{X}$ is a polyhedral complex, then a subset $X\subset \overline{X}$ is called an \Defn{ideal polyhedral complex} with \Defn{completion} $\overline{X}$ if $X$ is 
a union of open cells and contains all open chambers of $\overline{X}$ and their codimension $1$ faces.  If $X_1$ and $X_2$ are two ideal polyhedral complexes, then a bijection $f: X_1 \to X_2$ will be called a \Defn{geometric isomorphism} if $f$ is a homeomorphism with respect to the respective weak topologies and maps open cells homeomorphically onto open cells. An action of a group on an ideal polyhedral complex is called \Defn{cellular} if it is by geometric automorphisms.
\end{defn}

Typical examples of ideal polyhedral complexes are given by ``ideal tesselations'' of the hyperbolic plane with some vertices at infinity, hence the name. 

\begin{remark}\label{LinkComplex}
Assume that $V$ is a vector space and that $\mathcal C \subset V$ is a subset which carries the structure of an ideal polyhedral complex such that every closed cell contains $0$ and is invariant under the action of $\R^{>0}$ on $V$ by homotheties. Let $\mathbb S(V) := (V \setminus \{0\})/\R^{>0}$ and denote by $\mathbb S: (V \setminus \{0\}) \to \mathbb S(V)$ the canonical projection. Then $\mathbb S(\mathcal C)$ has an ideal polyhedral structure whose closed cells are of the form $\mathbb S(\tau)$, where $\tau$ is a closed cell of $\mathcal C$ different from $\{0\}$. Alternatively, one can realize $\mathbb S(\mathcal C)$ as the intersection of $\mathcal C$ with an arbitrary sphere centered at $0$, hence we refer to $\mathbb S(\mathcal C)$ as the \Defn{link complex} of $\mathcal C$. Furthermore, if $\mathcal{C}$ is contained in an open half-space of $V$, then one can also realize $\mathbb S(\mathcal C)$ as the intersection of $\mathcal{C}$ with a suitable affine hyperplane of $V$. 
\end{remark}

If $X$ is an ideal polyhedral complex, then we define a partial order on the set of open cells of $X$ by setting $\sigma \leq \tau$ if $\sigma$ is a face of $\tau$. We denote by $\Sigma(X)$ the resulting poset (partially ordered set). Posets together with order preserving maps form a category, and we say that two ideal polyhedral complexes are \Defn{combinatorially isomorphic} if their underlying posets are isomorphic in this category. If $X$ is an ideal polyhedral complex and $\Sigma$ is a poset with $\Sigma \cong \Sigma(X)$, then we say that $X$ is a \Defn{polyhedral realization} of $\Sigma$. For example, the $r$-dimensional simplex $\Delta_r$ realizes the poset $\blacktriangle_r$ given by all non-empty subsets of $\{0, \dots, r\}$.

A poset is called \Defn{polyhedral}, respectively \Defn{simplicial} if it can be realized by a polyhedral or simplicial complex. \Defn{Ideal polyhedral posets} and \Defn{ideal simplicial posets} are defined similary.  If $\Sigma$ is a poset, then its \Defn{augmentation} $\Sigma^+$ is the poset obtained from $\Sigma$ by adjoining a minimum $\emptyset_\Sigma$.
If $\Sigma$ can be realized by a polyhedral complex $X$, then $\Sigma^+$ can be realized by the polyhedral cone $CX$; in particular, augmentations of polyhedral posets are polyhedral.

A poset $\Sigma$ is simplicial if and only if for all $\sigma, \tau \in \Sigma$ there exists a greatest lower bound $\sigma \wedge \tau$ and for every $\sigma \in \Sigma$ the \Defn{downward link}
\[
\Sigma_{\leq \sigma} = \{\tau \in \Sigma \mid \tau \leq \sigma\}
\]
is isomorphic to $\blacktriangle_r$ for some $r \in \mathbb N_0$, cf.\ \cite[p.\ 661]{AbramenkoBrown2008}. We then call $\sigma$ an \Defn{(abstract) $r$-simplex} of $\Sigma$ and refer to $r$ as its \Defn{dimension}. The $0$-simplices of $\Sigma$ are also called its \Defn{vertices} and if $\sigma$ is an $r_1$-simplex, $\tau$ is an $r_2$-simplex and $\sigma \leq \tau$, then $\sigma$ is called a \Defn{face} of $\tau$ of \Defn{codimension} $r_2-r_1$. Ideal simplicial posets are the subposets of simplicial cosets which contain all maximal and comaximal elements. 

Every simplicial poset $\Sigma$ admits a canonical simplicial realization $|\Sigma|$ called its \Defn{geometric realization}, cf.\ \cite[p.\ 662]{AbramenkoBrown2008}; denote by $|\Sigma^+| := C|\Sigma|$ the realization of $\Sigma^+$ given by the simplicial cone over $|\Sigma|$. In the geometric realization of a simplicial poset $\Sigma$, every abstract $r$-simplex is realized by an $r$-dimensional simplex. While every simplicial realization of $\Sigma$ is geometrically isomorphic to $|\Sigma|$, there may exist other (non-simplicial) ideal polyhedral realizations of $\Sigma$ which are not geometrically isomorphic to $|\Sigma|$, and in which abstract $r$-simplices are realized by polyhedra of dimensions different from $r$ (cf.\ Example \ref{simplicialcox}). 

If $S$ is a finite set, then an \Defn{$S$-coloring} of a polyhedral poset $\Sigma$ is a map from the comaximal elements of $\Sigma$ to $S$ which restricts to a bijection on the codimension $1$ faces of each given chamber. If $\Sigma = \Sigma(X)$ for an ideal polyhedral complex $X$, then such a map is also called a coloring of $X$, and we call $\Sigma$ together with this map an \Defn{$S$-coloured ideal polyhedral poset}. In this case every $s \in S$ defines an equivalence relation $\sigma_s$ on the set ${\rm Ch}(\Sigma)$ of chambers of $\Sigma$ by setting $\sigma \sim_s \tau$ if $\sigma \wedge \tau$ is a codimension $1$ face coloured by $s$. The pair $({\rm Ch}(\Sigma), (\sim_s)_{s \in S})$ is then a chamber system in the sense of  \cite[Def. 5.21]{AbramenkoBrown2008}, called the \Defn{underlying chamber system} of $\Sigma$ (or of $X$). If $X$ is an ideal polyhedral complex with completion $\overline{X}$, then every $S$-colouring of $\overline{X}$ restricts to an $S$-colouring of $X$, and this restriction determines the underlying chamber system uniquely. We say that two $S$-coloured polyhedral complexes or posets are \Defn{chamber isomorphic} if the underlying chamber systems are isomorphic in the sense that there is a bijection between chambers preserving all of the equivalence relations.

For coloured polyhedral complexes we thus have three notions of isomorphism: geometric isomorphism (the strongest), combinatorial isomorphism and chamber isomorphism (the weakest).

\subsection{Ideal polyhedral realizations of the twin building}\label{TwinBuildinRealization}

So far we have considered the twin building $\Delta$ associated with $G$ as a chamber system. Indeed, in our previous notation we have $\Delta = \Delta^-\sqcup \Delta^+$, where $\Delta^\pm = G/B_\pm$ are the sets of chambers of the two halves. Since $B_\pm$ are self-normalizing we can identify $\Delta^+$ and $\Delta^-$ with the set of conjugates of $B_+$ and $B_-$ respectively. More generally, we can consider the sets $\Sigma(\Delta^+)$ and $\Sigma(\Delta^-)$ of \emph{all} parabolic subgroups of $G$ (excluding $G$) which contain a conjugate of $B_+$, respectively $B_-$. If we define partial orders on these sets by reverse inclusion, then $\Delta^\pm$ can be seen as the underlying chamber systems of the posets $\Sigma(\Delta^\pm)$ with respect to a suitable colouring. Note that the augmentations $\Sigma^+(\Delta^\pm) = \Sigma(\Delta^\pm) \cup \{G\}$ give rise to the same chamber complex.

The posets $\Sigma(\Delta^\pm)$ are in fact simplicial, hence admit simplicial geometric realizations $|\Sigma(\Delta^\pm)|$ with underlying chamber systems $\Delta^\pm$. See \cite[Chapter 4]{AbramenkoBrown2008} for a discussion of these simplicial complexes. In the context of our municipalities we will be interested in different realizations of the chamber systems $\Delta^\pm$.

\begin{definition} The \Defn{positive/negative Davis poset} is the subposet $\Sigma_{\rm sph}(\Delta^\pm) \subset \Sigma(\Delta^\pm)$ consisting of all parabolic subgroups of spherical type. We also define the \Defn{Davis poset} by  $\Sigma_{\rm sph}(\Delta) :=  \Sigma_{\rm sph}(\Delta^-) \sqcup  \Sigma_{\rm sph}(\Delta^+)$, and use similar notation for the augmented versions.
\end{definition}
The significance of these posets was pointed out by Davis who provided a CAT(0) ideal polyhedral realization $|\Sigma_{\rm sph}(\Delta^\pm)|_{\rm Davis}$ of $\Sigma_{\rm sph}(\Delta^\pm)$ in \cite{Davis98}, which is now called the \Defn{Davis realization}. Such a CAT(0) realization was known previously in the affine case by classical results of Bruhat and Tits. To obtain a CAT(0) realization in the spherical case one has to replace the positive/negative Davis poset by the \Defn{augmented positive/negative Davis poset} $\Sigma^+_{\rm sph}(\Delta^\pm) := \Sigma_{\rm sph}(\Delta^\pm) \cup \{G\}$ in order to avoid positive curvature.

By construction, the positive/negative Davis posets realize the chamber systems $\Delta^\pm$. In particular, the Davis posets inherit a canonical colouring, and $G$ acts on them via conjugation by colouring-preserving automorphisms. If we consider only those parabolic subgroups in $\Sigma_{\rm sph}(\Delta^\pm)$ which contain a fixed split torus, then we obtain a poset which is isomorphic to the Davis--Moussong poset of the underlying Coxeter poset (see Definition \ref{DavisMoussongPoset} and the subsequent discussions).

We now describe the alternative realization of the Davis poset which we will use in our study of the municipality of a reduced Kac--Moody. Recall from Subsection \ref{TitsCones} that the closed fundamental chamber of the reduced Tits cone is given by
\[
\overline{C} = \{X \in \ol{\mathfrak a} \mid \forall\, i = 1, \dots, n:\; \alpha_i(X) \geq 0\} \subset \ol{\mathfrak a},
\]
and that it is a polyhedral cone which is bounded by the root hyperplanes $\ol{H}_{\alpha_i}$ of the simple roots in $\ol{\mathfrak a}$; its polyhedral cells are given by intersections of these hyperplanes, and there is a natural colouring of the faces by $S$, which colours each reflection hyperplane with the corresponding simple reflection. The intersection of the fundamental chamber with the interior of the Tits cone is the ideal subcomplex given by the union of those cells which have finite stabilizer under the reduced Kac--Moody representation of the Weyl group on $\ol{\mathfrak a}$. We denote  by $P_{\mathbb S(\ol{\mathfrak a})}$ the image of this ideal polyhedral cone under the projection $\mathbb S: \ol{\mathfrak a}\setminus \{0\} \to \mathbb S(\ol{\mathfrak a})$ from Remark \ref{LinkComplex} and call it the \Defn{reduced ideal fundamental cell}. It is an ideal polyhedral complex with a single chamber whose faces are coloured by $S$.

We now form the quotient of direct products $\Delta^\pm \times P_{\mathbb S(\ol{\mathfrak a})}$ by identifying $(C, x)$ and $(C', x)$ in each half provided $C$ and $C'$ are $s$-adjacent for some $s \in S$ and $x$ is contained in the closure of the face of $P_{\mathbb S(\ol{\mathfrak a})}$ labelled by $s$. This yields a coloured ideal polyhedral complex $|\Delta^\pm|_{\overline{\mathfrak a}}$ and we set $|\Delta|_{\ol{\mathfrak a}} := |\Delta^+|_{\overline{\mathfrak a}} \sqcup |\Delta^-|_{\overline{\mathfrak a}}$. By construction, $|\Delta^\pm|_{\ol{\mathfrak a}}$ are realizations of the chamber systems $\Delta^\pm$, and hence we refer to $|\Delta|_{\ol{\mathfrak a}}$ as the \Defn{$\ol{\mathfrak a}$-realization} of the twin building $\Delta$. Note that the action of the combinatorial automorphism group ${\rm Aut}(\Delta)$ on the first factor of $\Delta \times P_{\mathbb S(\ol{\mathfrak a})}$ descends to an action by polyhedral automorphisms on $|\Delta^\pm|_{\overline{\mathfrak a}}$. In particular, ${\rm Aut}(\ol{\XXX})$ acts on $|\Delta^\pm|_{\overline{\mathfrak a}}$ via the embedding ${\rm Aut}(\ol{\XXX}) \hookrightarrow {\rm Aut}(\Delta)$.

We have thus obtained three realizations of the twin building $\Delta$: The simplicial realization $|\Delta|$, the Davis realization $|\Sigma_{\rm sph}(\Delta)|_{\rm Davis}$ and the $\ol{\mathfrak a}$-realization $|\Delta|_{\ol{\mathfrak a}}$. All three realizations are chamber isomorphic, but in general not geometrically isomorphic. We will see below that (under our standing assumption that ${\bf A}$ is non-spherical and non-affine) $|\Sigma_{\rm sph}(\Delta^\pm)|_{\rm Davis}$ and $|\Delta|_{\ol{\mathfrak a}}$ are combinatorially isomorphic, namely they both realize the Davis poset $\Sigma_{\rm sph}(\Delta^\pm)$. On the other hand, the simplicial realization is not combinatorially isomorphic to either of them. Our polyhedral structure on the municipality will be modelled on the $\ol{\mathfrak a}$-realization.

\subsection{The polyhedral cell structure of the municipality}\label{SecMuni}
Recall from Definition \ref{Muni} the definition of the municipality $\Delta_\bullet$. Throughout this subsection we fix $x \in \ol{\XXX}$ and denote by $\Delta_x$ the fiber of $\Delta_\bullet$ under the surjection ${\rm ev}_0$ from \eqref{ev0}. By definition, $\Aut(\ol{\XXX})$ acts on $\Delta_\bullet$ and the stabilizer $\Aut(\ol{\XXX})_x$ preserves $\Delta_x$. The goal of this subsection is to define an ideal polyhedral structure on $\Delta_x$ and to show that the resulting ideal polyhedral complex is $\Aut(\ol{\XXX})_x$-equivariantly geometrically isomorphic to the $\ol{\mathfrak a}$-realization $|\Delta|_{\ol{\mathfrak a}}$, and combinatorially isomorphic to the Davis realization $|\Sigma_{\rm sph}(\Delta)|_{\rm Davis}$. As a by-product we will obtain that $|\Delta|_{\ol{\mathfrak a}}$ realizes $\Sigma_{\rm sph}(\Delta)$.

Our first goal is to construct a polyhedral structure on $\Delta_x$ for a fixed $x \in \ol{\XXX}$. For this we will need to recall some results from the appendix. Firstly, the Weyl group $W$ acts on $\ol{\mathfrak a}$ by the reduced Kac--Moody representation, and in view of Convention \ref{ConventionSecCausal} we deduce from Corollary \ref{CorollaryRedKMRepFaithful} that $\left(\ol{\mathfrak{a}}, (-|-), \pi(\check \Pi_{\rm nor})\right)$ is a root basis for $(W,S)$ under this action. By \eqref{TitsCone} and the discussion in \ref{TitsCones}, the \Defn{reduced Tits cone}
\[
{\ol{\CCC}} = \{X \in \ol{\mathfrak a}\mid\alpha(X) \geq 0 \text{ for almost all }\alpha \in \Phi^+\} \subset \ol{\mathfrak a}
\]
is isomorphic to the dual Tits cone associated with this root basis, hence provides a polyhedral realization of the augmented Coxeter poset $\Sigma^+(W,S)$. Note that, in fact, $\Sigma^+(W,S)= \Sigma(W,S)$, since $W$ is assumed to be non-spherical. With this observation it then follows from Proposition~\ref{TitsInterior} that the interior of the Tits cone is a polyhedral realization of the Davis--Moussong poset $\Sigma_{\rm sph}(W,S)$. In particular, this polyhedral complex admits a canonical colouring by $S$.

We call a ray in $\ol{\mathfrak a}$ with origin $0$ a \Defn{Tits ray} if it is contained in the interior of the Tits cone. We can then identify the sets of all Tits rays with the link complex of the interior of the Tits cone in the sense of Remark \ref{LinkComplex}, and thereby define an ideal polyhedral structure on the set of all Tits rays. This complex then realizes the Davis--Moussong poset $\Sigma_{\rm sph}(W,S)$. Geometrically it is isomorphic to a twin apartment in $|\Delta|_{\ol{\mathfrak a}}$, since the closed fundamental chambers in both complexes carry the same geometry by definition, and since $W$ acts chamber-transitively on both complexes preserving the geometry. We refer to open chambers in this complex as \Defn{open Weyl chambers} in $\mathbb S(\ol{\mathfrak a})$.

Given $x \in \ol{\XXX}$ and a flat $F$ containing $x$ pick a positive chart $\phi: \ol{\mathfrak a} \to F$ so that $\phi(\pm \ol{\CCC}^o) = \ol{C}^{\pm}_x$. If one denotes by $\Delta^\pm_x(F) \subset \Delta^\pm_x$ the subset of rays contained in $F$, then $\phi$ sends regular Tits rays (respectively, their negatives) to geodesic rays in $\Delta^\pm_x(F)$, hence induces bijections
\[
\phi_*: \mathbb S(\pm {\ol{\CCC}}) \to  \Delta_x^{\pm}(F).
\]
By transport of structure one can thus turn $\Delta_x^{+}(F)$ and $\Delta_x^{-}(F)$ into $S$-coloured ideal polyhedral complexes. We thus refer to these polyhedral structures on $\Delta^\pm_x(F)$ as the \Defn{canonical ideal polyhedral structure} on $\Delta_x^{\pm}(F)$ and refer to their open chambers as \Defn{open Weyl chambers} in $\Delta_x^{\pm}(F)$. We also denote by $\Sigma(\Delta_x^\pm(F))$ the underlying posets.
\begin{remark} In order to define the canonical ideal polyhedral structure we have chosen a positive chart $\phi: \ol{\mathfrak a} \to F$. If $\psi$ is any other positive chart, then by Remark \ref{ConeWellDefined} and Theorem \ref{ExtendLocalAutomorphisms} we have a commutative diagram 
\[\begin{xy}\xymatrix{
\overline{\mathfrak a} \ar[r]^{\phi} & F,\\
\overline{\mathfrak a} \ar[u]^{\alpha} \ar[ur]_\psi
}\end{xy}\]
where up to an automorphism of the Coxeter diagram which changes only the labelling $\alpha$ is given by an element of the Weyl group, acting via the reduced Kac--Moody representation of $\overline{\mathfrak a}$. Since the Weyl group acts on the double Tits cone and its interior by geometric automorphisms, we deduce that the canonical ideal polyhedral structure is independent of the choice of positive chart $\phi$ used to define it.
\end{remark}
We can summarize the properties of our construction so far as follows:
\begin{corollary}\label{ApartmentsRight} For every pointed flat $(x, F)$ the set $\Delta_x(F) := \Delta^-_x(F) \sqcup \Delta^+_x(F)$ with its canonical ideal polyhedral cell structure is geometrically isomorphic to a twin apartment in the $\ol{\mathfrak a}$-realization $|\Delta|_{\ol{\mathfrak a}}$ and combinatorially isomorphic to a twin apartment in the Davis realization $|\Sigma_{\rm sph}(W,S)|_{\rm Davis}$ of the twin building $\Delta$ of $G$, hence $\Sigma(\Delta_x^\pm(F))$ is isomorphic to the Davis--Moussong poset $\Sigma_{\rm sph}(W,S)$.
Under any such isomorphisms the subsets $\Delta^\pm_x(F) \subset \Delta_x(F)$ correspond to the two halves of the twin apartment.\qed
\end{corollary}
Our next goal is to establish a global equivariant version of this result. To formulate our result, we first discuss the relevant actions of ${\rm Aut}(\ol{\XXX})$.
\begin{remark}\label{AutomorpshismsBuilding}
We have the following actions of subgroups of $\Aut(\ol{\XXX})$:
\begin{enumerate}
\item The action of $\Aut(\ol{\XXX})$ on $\ol{\XXX}$ induces an action of the stabilizer $\Aut(\ol{\XXX})_x$ on $\Delta_x$. This action is by geometric automorphisms, hence induces an action on the underlying poset $\Sigma_x$.
\item Every colouring-preserving automorphism $\alpha$ of the coloured poset $\Sigma(\Delta)$ preserves $\Sigma_{\rm sph}(\Delta)$ and induces a type-preserving automorphism of the chamber system $\Delta$. Moreover, $\alpha$ can be recovered from the corresponding automorphism of $\Delta$ by \cite[Cor.\ 4.11]{AbramenkoBrown2008}. We may thus identify colouring-preserving automorphisms of $\Sigma(\Delta)$ or $\Sigma_{\rm sph}(\Delta)$ and type-preserving automorphisms of $\Delta$. Similarly, we can (and will) identify ${\rm Aut}(\Delta)$ with the automorphism groups of the poset $\Sigma(\Delta)$ or of the poset $\Sigma_{\rm sph}(\Delta)$. Any such automorphism induces a geometric automorphism of the $\overline{\mathfrak a}$-realization $|\Delta|_{\ol{\mathfrak a}}$.
\item Via the canonical embedding $\Aut(\ol{\XXX}) \into \Aut(\Delta)$, the group  $\Aut(\ol{\XXX})$ acts on the chamber complex $\Delta$. By the previous remark this induces embeddings
\[
\Aut(\ol{\XXX}) \hookrightarrow \Aut(\Sigma_{\rm sph}(\Delta)) \quad \text{and} \quad \Aut(\ol{\XXX}) \hookrightarrow \Aut(|\Delta|_{\ol{\mathfrak a}}).
\] 
\end{enumerate}
\end{remark}
\begin{proposition}\label{LocalMunicipality} There exists a unique ideal polyhedral structure on $\Delta_x$ such that for every pointed flat $(x, F)$ the subset $\Delta_x(F) \subset \Delta_x$ is an ideal polyhedral subcomplex and carries its canonical cell structure. With this structure, 
$\Delta_x$ is ${\rm Aut}(\XXX)_x$-equivariantly geometrically isomorphic to $|\Delta|_{\ol{\mathfrak a}}$ and ${\rm Aut}(\XXX)_x$-equivariantly combinatorially isomorphic to the Davis realization $|\Sigma_{\rm sph}(W,S)|$ of the twin building $\Delta$. In particular, it is an ideal polyhedral realization of $\Sigma_{\rm sph}(\Delta)$. 
\end{proposition}
In view of the proposition we refer to $\Delta_x \subset \partial_x \ol{\XXX}$ as the \Defn{twin building at the horizon of $x$}. We will refer to the polyhedral structure on $\Delta_x$ given by the proposition as the \Defn{canonical polyhedral structure}. Since $\Delta_x$ is covered by the subsets $\Delta_x(F)$ there is clearly at most one such structure. In order to obtain existence of the canonical polyhedral structure and to deduce Proposition \ref{LocalMunicipality} from Corollary \ref{ApartmentsRight} we need to discuss the effect of automorphisms on the various complexes above.

First note that if $\alpha \in \Aut(\ol{\XXX})$ maps the pointed flat $(x, F)$ to a pointed flat $(x', F')$, then by equivariance of our construction $\alpha$ induces a geometric isomorphisms $\Delta^{\pm}_x(F) \to \Delta^{\pm}_x(F')$, which in turn induces a combinatorial isomorphism $\Sigma(\Delta_x^\pm(F)) \to \Sigma(\Delta_x^\pm(F))$. Moreover, this maps preserves the respective colourings if $\alpha \in  \Aut_S(\ol{\XXX})$.

Now assume that $F, F'$ are two flats through $x$ and denote by $I := F \cap F' \subset F$ their intersection. Also set $\Sigma_x(F, I) := \{C \in \Sigma(\Delta_x(F))\mid C \subset I\}$ and define $\Sigma_x(F', I)$ accordingly. Finally, let $\Delta_x(I) := \{r\in \Delta_x \mid r((0, \infty)) \subset I\}$ and note that $\Delta_x(I)= \Delta_x(F) \cap \Delta_x(F')$.

By Corollary \ref{twopointsintwoflats} there exists an automorphism $\alpha \in {\rm Aut}^+_S(\ol{\XXX})$ which maps $F$ to $F'$ and fixes $I$, hence in particular $x$. By construction $\alpha$ induces a geometric isomorphism between the polyhedral complexes $\Delta^\pm_x(F)$ and $\Delta^\pm_x(F')$. In particular, since $\alpha$ fixes $I$ one obtains $\Sigma_x(F, I) = \Sigma_x(F', I)$; moreover, this set is the underlying poset of an ideal polyhedral structure on  $\Delta^{\pm}_x(I)$. With this polyhedral structure, $\Delta^{\pm}_x(I)$ is an ideal polyhedral subcomplex of both $\Delta_x(F)$ and $\Delta_x(F')$. It is therefore possible to glue $\Delta_x(F)$ and $\Delta_x(F')$ along $\Delta_x(I)$ to obtain an ideal  polyhedral structure on $\Delta_x(F) \cup \Delta_x(F')$. All these glueing are compatible, and hence one obtains the desired canonical ideal polyhedral structure on
\[
\Delta_x = \bigcup_{F \ni x} \Delta_x(F).
\]
The underlying poset is
\[
\Sigma_x = \bigcup_{F \ni x} \Sigma_x(F),
\]
and it inherits a colouring from the $\Sigma_x(F)$. It remains to show that there exists an ${\Aut}(\XXX)_x$-equivariant isomorphism $\Sigma_x \to \Sigma_{\rm sph}(\Delta)$. This will establish the desired combinatorial statement, and the geometric statement will follow, since the geometries of the cells are matched in each apartment and since these cover the buildings in question. To prove the combinatorial statement we compare flats through $x$ in $\ol{\XXX}$ to twin apartments of the twin building $\Delta$. For this we denote by $\FFF_x$ the set of maximal flats in $\ol{\XXX}$ containing $x$ and by $\AAA_x$ the set of twin apartments of $\Delta$ which are invariant under $s_x$. Then one observes the following:
\begin{lemma}\label{GlueingLemma} For every $x \in \ol{\XXX}$ there exists an $\Aut(\ol{\XXX})_x$-equivariant bijection $\phi_x: \FFF_x \to \AAA_x$.
\end{lemma}
\begin{proof} Argue in the group model. By transitivity of $\Aut(\ol{\XXX})$ on $\ol{\XXX}$ it suffices to establish the lemma for the basepoint $x = e$. Recall from 
\eqref{TorusParametrization} that maximal flats are in $G$-equivariant bijection with maximal tori of $G$, hence with twin apartments in $\Delta$. Since the point reflection $s_e$ is induced by $\theta$, the flats through $e$ correspond to $\theta$-stable tori and thus to twin apartments which are invariant under $s_e$, and this correspondence is equivariant with respect to the point stabilizer $K$ of $e$ in $G$. 

One can argue as follows to establish that the correspondence is $\Aut(\ol{\XXX})_e$-equivariant: By Theorem \ref{ThmGlobalAut} one has $\Aut(\ol{\XXX})_e \cong \Aut(G)_e$ and every element of $\Aut(G)_e$ is a product of an element of $K$ with an automorphism which fixes both the flat $\ol{A}$ through $e$ and the corresponding twin apartment $\phi_e(\ol{A})$ of $\Delta$. It follows that the given bijection is not only $K$-equivariant, but moreover $\Aut(\ol{\XXX})_e$-equivariant.
\end{proof}
\begin{proof}[Proof of Proposition \ref{LocalMunicipality}] 
  Choose a bijection $\phi_x$ as in Lemma~\ref{GlueingLemma} and a flat $F \in \FFF_x$. Set $A := \phi_x(F)$ and let $\Sigma_{\rm sph}(A) \subset \Sigma_{\rm sph}(\Delta)$ the subset of the Davis poset of $\Delta$ corresponding to the twin apartment $A$.  By Corollary~\ref{ApartmentsRight} there is a poset isomorphism $\iota_o: \Sigma_x(F) \to \Sigma_{\rm sph}(A)$, which one may choose to be colouring preserving. It remains to show that $\iota_o$ can be extended to an $\Aut(\ol{\XXX})_x$-equivariant combinatorial isomorphism $\iota: \Sigma_x \to \Sigma_{\rm sph}(\Delta)$. 
  
For this let $c$ be a polyhedral cell in $\Delta_x$ which is contained in $\Delta_x(F')$ some flat $F'$. By Corollary~\ref{twopointsintwoflats} there exists an automorphism $\alpha \in \Aut(\ol{\XXX})_c$ which maps $F$ to $F'$ and fixes $F \cap F'$, hence in particular fixes $x$. Define $\iota(c') := \alpha \circ \iota_o \circ \alpha^{-1} (c)$ and observe that this definition does not depend on the choice of automorphism $\alpha$. Indeed, assume that $\beta$ is another automorphism which maps $F$ to $F'$ and fixes $F \cap F'$. Then
\[
\beta \circ \iota_o \circ \beta^{-1} = \alpha \circ (\alpha^{-1} \circ \beta \circ \iota_o \circ \beta^{-1} \circ \alpha) \circ \alpha^{-1}.
\]
Now $\beta^{-1} \circ \alpha$ is an automorphism which stabilizes $F$ and fixes $x$; by Theorem \ref{ExtendLocalAutomorphisms} it is thus given by an automorphism of the Coxeter complex up to possibly swapping the two halves of the apartment. It thus follows that $\beta^{-1} \circ \alpha$ commutes with $\iota_o$, and one concludes
\[
\beta \circ \iota_o \circ \beta^{-1} = \alpha \circ\iota \circ \alpha^{-1}.
\]
This proves that $\iota$ is well-defined, and it is $\Aut(\ol{\XXX})_x$-equivariant by construction.
\end{proof}

\subsection{The global structure of the municipality}

In the previous subsection we have constructed an ideal polyhedral structure on $\Delta_x$ for every $x \in \ol{\XXX}$. Combining these structures we also obtain an ideal polyhedral structure on the disjoint union $\Delta_\bullet = \bigsqcup \Delta_x$. We denote the underlying poset by $\Sigma_\bullet = \bigsqcup \Sigma_x$; by definition elements of $\Sigma_\bullet$ are open cells in $\Delta_\bullet$. Also turn $\ol{\XXX} \times \Sigma_{\rm sph}(\Delta)$ into a poset by setting $(x, c) \leq (x', c')$ if and only if $x = x'$ and $c \leq c'$, and turn 
\[\ol{\XXX} \times |\Delta|_{\ol{\mathfrak a}} = \bigsqcup_{x \in \ol{\XXX}} |\Delta|_{\ol{\mathfrak a}}
\] into a (disconnected) ideal polyhedral complex by equipping each $|\Delta|_{\ol{\mathfrak a}}$ with its canonical polyhedral structure. By Remark \ref{AutomorpshismsBuilding} the group ${\rm Aut}(\ol{\XXX})$ acts on $|\Delta|_{\ol{\mathfrak a}}$, and hence diagonally on $\ol{\XXX} \times |\Delta|_{\ol{\mathfrak a}}$ by polyhedral automorphisms. This action then induces an action by combinatorial automorphisms on the underlying poset $\ol{\XXX} \times \Sigma_{\rm sph}(\Delta)$.

For the remainder of this section we fix a basepoint $o \in \ol{\XXX}$ and an $\Aut(\ol{\XXX})_o$-equivariant geometric isomorphism $|\iota_o|: |\Delta|_{\mathfrak a} \to \Delta_o$ with underlying combinatorial isomorphism $\iota_o: \Sigma_{\rm sph}(\Delta) \to \Sigma_o$. We may and will assume that $\iota_o$ maps $\Sigma_{\rm sph}(\Delta^+)$ to $\Sigma^+_o$ and preserves colourings. A geometric isomorphism $|\iota|: \ol{\XXX} \times |\Delta|_{\ol{\mathfrak a}} \to \Delta_\bullet$ will be called an \Defn{extension} of $|\iota_o|$ if $|\iota|(o, \cdot) = |\iota_o|$.
\begin{proposition}\label{PropGlobalEquiv} $|\iota_o|$ admits a unique colouring-preserving $\Aut(\ol{\XXX})$-equivariant extension
\[
|\iota|: \ol{\XXX} \times |\Delta|_{\ol{\mathfrak a}} \to \Delta_\bullet, \quad (x, \xi) \mapsto |\iota_x|(\xi),
\]
such that $|\iota_x|:|\Delta|_{\ol{\mathfrak a}}  \to \Sigma_x$ is an isomorphism of posets for every $x \in \ol{\XXX}$. 
\end{proposition}
\begin{remark}\begin{enumerate}
\item It is immediate from Proposition \ref{LocalMunicipality} that $\ol{\XXX} \times |\Delta|_{\ol{\mathfrak a}} \cong \Delta_\bullet$ as polyhedral complexes. The point of the proposition is that the extension can be chosen \emph{equivariantly}. 
\item The proposition implies on the combinatorial level that $\iota_o$ can be extended to a unique $\Aut(\ol{\XXX})$-equivariant automorphism of posets
\[
\iota:\ol{\XXX} \times \Sigma_{\rm sph}(\Delta) \to \Sigma_\bullet, \quad (x, C) \mapsto \iota_x(C),
\]
such that $\iota_x: \Sigma_{\rm sph}(\Delta) \to \Sigma_x$ is an isomorphism of posets for every $x \in \ol{\XXX}$. 
\item Conversely, the combinatorial statement in (ii) implies the geometric one, since the cells have the same geometry by Proposition \ref{LocalMunicipality}.
\end{enumerate}
\end{remark}
\begin{proof}[Proof of Proposition \ref{PropGlobalEquiv}] By the previous remark, it suffices to show the combinatorial statement. Every equivariant extension $\iota$ clearly has to satisfy $\iota_x(C) := \alpha^{-1}(\iota_o(\ol{\alpha}(C)))$ for any $\alpha$ mapping $x$ to $o$. We now show that this formula defines indeed a map with the desired properties.

Thus let $(x, C) \in \ol{\XXX} \times \Sigma_{\rm sph}(\Delta)$. To define $\iota(x, C)$ we pick $\alpha \in \Aut_S(\XXX)$ with $\alpha(x) = o$.
Via the isomorphism $\Aut_S(\XXX) \to \Aut_S(\Ad(G))$, this $\alpha$ corresponds to an automorphism $\ol{\alpha}$ of $\Ad(G)$. Think of the chamber $C$ as a parabolic subgroup of $\Ad(G)$; then $\ol{\alpha}(C)$ is also a parabolic subgroup (of the same type, since $\alpha$ and hence $\ol{\alpha}$ are type-preserving) and hence one may define
\[
\iota_x(C) := \alpha^{-1}(\iota_o(\ol{\alpha}(C))).
\]
This does not depend on the choice of $\alpha$. Indeed, let $\beta \in \Aut(\XXX)$ with $\beta(x) = o$. Then $\beta\alpha^{-1} \in \Aut(\XXX)_o$, and thus by $\Aut(\XXX)_o$-equivariance of $\iota_o$, 
\begin{eqnarray*}
 \beta^{-1}(\iota_o(\ol{\beta}(C))) =  \beta^{-1}(\iota_o(\ol{\beta\alpha^{-1}} \; \ol{\alpha}(C)) = \beta^{-1}\beta\alpha^{-1}\iota_o(\ol{\alpha}(C)) =  \alpha^{-1}(\iota_o(\ol{\alpha}(C))).
\end{eqnarray*}
We have now established that $\iota$ is well-defined; by construction it is order- and colouring-preserving. Finally, it is $\Aut(\XXX)$-equivariant by the following argument. Let $(x, C) \in \ol{\XXX} \times \Delta$ and $\beta \in \Aut(\ol{\XXX})$. If $\alpha \in \Aut(\XXX)$ satisfies $\alpha(x) = o$, then $\gamma := \alpha \circ \beta^{-1}$ satisfies $\gamma(\beta(x)) = o$. One thus gets
\[
\iota_{\beta(x)}(\ol{\beta}(C)) =  \gamma^{-1}(\iota_o(\ol{\gamma}(\ol{\beta}(C)))) = \beta \alpha^{-1}(\iota_o(\ol{\alpha \circ \beta^{-1}}(\ol{\beta}(C)))) = \beta(\iota_x(C)).
\]
This finishes the proof.
\end{proof}
From now on we will use the notations $|\iota|$ and $|\iota_x|$ for the maps defined by Proposition \ref{PropGlobalEquiv} and write $\iota$ and $\iota_x$ for the corresponding combinatorial maps. Note that, if $\alpha \in \Aut(\ol{\XXX}) \subset \Aut(\Sigma_{\rm sph}(\Delta))$ satisfies $\alpha(x) = y$ for some $x, y \in \ol{\XXX}$,  then the diagrams
\begin{equation}\label{IotaDefiningSquare}
\begin{xy}\xymatrix{
\Delta_x \ar[rr]^{\alpha} && \Delta_y&&\Sigma_x \ar[rr]^{\alpha} && \Sigma_y\\
|\Delta|_{\ol{\mathfrak a}} \ar[u]^{|\iota_x|}  \ar[rr]_{\alpha}&&|\Delta|_{\ol{\mathfrak a}}\ar[u]_{|\iota_y|}&&\Sigma_{\rm sph}(\Delta) \ar[u]^{\iota_x}  \ar[rr]_{\alpha}&&\Sigma_{\rm sph}(\Delta)\ar[u]_{\iota_y}
}\end{xy}
\end{equation}
commute, and this property together with the choice of $|\iota_o|$ determines all the isomorphisms $|\iota_x|$. 

Given $x, y \in \ol{\XXX}$ we denote $|\iota_{x,y}| := |\iota_y| \circ |\iota_x^{-1}|$ and $\iota_{x,y} := \iota_y \circ \iota_x^{-1}$. Then, by definition, $|\iota_{x,y}|$ is a colouring-preserving isomorphism of ideal polyhedral complexes with underlying combinatorial isomorphism $\iota_{x,y}$ and the following diagrams commute:
\begin{equation}\label{IotaTriangle}
\begin{xy}\xymatrix{
\Delta_x \ar[rr]^{|\iota_{x,y}|} & &\Delta_y,&\Sigma_x \ar[rr]^{\iota_{x,y}} && \Sigma_y,\\
&|\Delta|_{\ol{\mathfrak a}} \ar[lu]^{|\iota_x|} \ar[ru]_{|\iota_{y}|}&&&\Sigma_{\rm sph}(\Delta) \ar[lu]^{\iota_x} \ar[ru]_{\iota_{y}}
}\end{xy}.
\end{equation}
\begin{remark}\label{PositivityPreserved} Each $\Delta_x$ decomposes into the subsets $\Delta^+_x$ and $\Delta^-_x$ of causal and anti-causal rays emanating from $x$. Accordingly, $|\iota_{o,x}|$ splits into two combinatorial isomorphisms
\begin{equation}
|\iota_{o,x}|^+: \Delta_o^+ \to \Delta^+_x \quad \text{and} \quad |\iota_{o,x}|^-: \Delta_o^- \to \Delta^-_x,
\end{equation}
as a consequence of the following simple observation.
\end{remark}
\begin{lemma}\label{SeePositivityAtInfty} Let $r \in \Delta_\bullet$. Then $r$ is causal if and only if $|\iota_{r(0), o}|(r) \in \Delta_o^+$.
\end{lemma}
\begin{proof} Note that the action of the subgroup $\Aut^+(\ol{\XXX})$ on $\Delta$ and $\Delta_\bullet$ preserves the two halves. Since $\Aut^+(\ol{\XXX})$ acts strongly transitively on $\ol{\XXX}$ and in view of the commuting diagram \eqref{IotaDefiningSquare} one may thus assume that $x = o$, whence the lemma follows from our choice of $\iota_o$.
\end{proof}
\subsection{Asymptoticity of causal and anti-causal rays}\label{SubsecAsymptotic}

Recall that two geodesic rays in a Riemannian symmetric space are called \Defn{asymptotic} if they are at bounded Hausdorff distance. For example, geodesic rays $r_1, r_2$ in $\E^n$ are called \Defn{asymptotic} provided they are parallel and point in the same direction, i.e., they are of the form $r_1(t) = x + tv$ and $r_2(t) = y + tv$ for some $x,y \in \R^n$ and a unit vector $v$. Similarly, two geodesics rays in the hyperbolic plane $\HH^2$ are asymptotic if they converge to the same point in $\partial \HH^2 \cong S^1$. 
Our goal is to define similar notions of asymptoticity for causal and anti-causal rays in Kac--Moody symmetric spaces. We keep the notation of the previous subsection and define:
\begin{definition} Two rays $r_1 \in \Delta_{x}$ and $r_2 \in \Delta_{y}$ are \Defn{asymptotic}, denoted $r_1 \parallel r_2$, provided $|\iota_{x,y}|(r_1) = r_2$.
\end{definition}
\begin{remark} By Remark \ref{PositivityPreserved} the isomorphisms $|\iota_{x,y}|: \Delta_x \to \Delta_y$ preserve the two halves, and thus induce combinatorial isomorphisms
\[
|\iota_{x,y}|^{\pm}: \Delta_x^\pm \to \Delta_y^\pm.
\]
In particular, causal rays can only be asymptotic to causal ray, and similarly anti-causal rays can only be asymptotic to anti-causal rays.
\end{remark}
The following proposition summarizes the main properties of the equivalence relation $\parallel$. Concerning the statement of the proposition we observe that if $G_i < G$ is a standard rank one subgroup, then the orbit $G_i.o \subset \ol{\XXX}$ is an embedded hyperbolic plane $\HH^2_{(i)} \subset \ol{\XXX}$. We then refer to a subset of $\ol{\XXX}$ of the form $g.\HH^2_{(i)}$ for some $g \in G$ and $i \in \{1, \dots, n\}$ as a \Defn{standard hyperbolic plane} in $\ol{\XXX}$.

\begin{proposition}\label{PropAsy} 
Let $x, y \in \ol{\XXX}$ and let $r_1 \in \Delta_x$ and $r_2 \in \Delta_y$. Then the equivalence relation $\parallel$ satisfies the following properties:
\begin{axioms}{A}
\item For every $r \in \Delta_x$ there exists a unique $r' \in \Delta_y$ with $r \parallel r'$.
\item $\parallel$ is invariant under $\Aut(\ol{\XXX})$, i.e., if $r_1 \parallel r_2$, then $\alpha(r_1)\parallel \alpha(r_2)$ for all $\alpha \in \Aut(\ol{\XXX})$.
\item  If $r_1, r_2$ are contained in a standard hyperbolic plane, then $r_1 \parallel r_2$ if and only if they are asymptotic in the hyperbolic sense.
\item If $r_1, r_2$ are contained in a common maximal flat $F$, then $r_1 \parallel r_2$ if and only if they are asymptotic in the Euclidean sense.
\end{axioms}
\end{proposition}
\begin{proof} (A1) is immediate from the fact that $|\iota_{x,y}|$ is a bijection.

(A2) If $r_1\parallel r_2$, then there is a $\xi \in |\Delta|_{\ol{\mathfrak a}}$ such that $r_1 = |\iota|(x, \xi)$ and $r_2 = |\iota|(y, \xi)$. Since $|\iota|$ is $\Aut(\ol{\XXX})$-equivariant one thus has
\[
\alpha(r_1) = \alpha( |\iota|(x, \xi))= |\iota|(\alpha(x), \alpha(\xi)) \quad \text{and} \quad \alpha(r_2) = \alpha( |\iota|(y, \xi))= |\iota|(\alpha(y), \alpha(\xi)),
\]
which implies $\alpha(r_1)\parallel \alpha(r_2)$.

(A3) By (A2) it suffices to prove the statement under the assumption that $r_1$ and $r_2$ are contained in $\HH^2_{(j)}$ for some $j=1, \dots, n$. One can identify $\HH^2_{(j)}$ with the upper half-plane model $\HH^2$ of the hyperbolic plane in such a way that the base point $o$ gets identified with $i$. Furthermore, one can identify the image $\iota_o^{-1}(\HH^2_{(j)}) \subset |\Delta|_{\ol{\mathfrak a}}$ with $\R\cup \{\infty\}$ in such a way that $\iota_o^{-1}|_{\HH^2_{(j)}}$ identifies geodesics in $\HH^2_{(j)}$ emanating from $o$ with the endpoint of the corresponding geodesic in $\HH^2$ in $\R\cup \{\infty\}$.

Fix this identification and work in the upper half plane model from now on. If $x = i+\lambda \in \HH^2$ for some $\lambda \in \R$, then an automorphism of $\HH^2$ mapping $o$ to $x$ is given by $\tau_\lambda: z \mapsto z+\lambda$. This automorphism is induced by an element of the corresponding rank one subgroup $G_j$, hence extends to $\ol{\XXX}$. Given $r \in\Delta_x$, by \eqref{IotaDefiningSquare} one has $|\iota_x|^{-1}(r) = (\tau_\lambda \circ |\iota_o|^{-1} \circ \tau_\lambda^{-1})(r)$. In other words, $\iota_x^{-1}(r)$ is obtained by translating $r$ by $\lambda$ to the left, taking the endpoint and then shifting it by $\lambda$ to the right. This, however, is the same as just taking the endpoint of $r$, since this is the case for vertical geodesic rays emanating from $x$ and the construction is equivariant with respect to the point stabilizer of $x$ in the automorphism group. One deduces that $r \in \Delta_x \HH$ is asymptotic to $r' \in \Delta_o$ if and only if $r$ and $r'$ have the same endpoint. Since every pair of points in $\HH^2$ can be mapped by an automorphism of $\HH^2$ to $(i, i+\lambda)$ for a suitable $\lambda$, and since any such automorphism extends to an automorphism of $\ol{\XXX}$, one deduces that our notion of asymptoticity restricts to usual hyperbolic asymptoticity on $\HH^2_{(j)}$.

(A4) In view of (A2) one may assume that $F = \ol{A}$ is the standard maximal flat in the group model and that $x = o$. Let $\vec\sigma$ be the unique oriented geodesic segment from $o$ to $y$ and let $\tau:= t[\vec \sigma]$ be parallel transport along $\vec\sigma$. Then $\tau$ acts on $F$ as a Euclidean translation. By \eqref{IotaDefiningSquare} one has a commuting diagram
\begin{equation}
\begin{xy}\xymatrix{
\Delta_o \ar[rr]^{\tau} && \Delta_y\\
 |\Delta|_{\ol{\mathfrak a}} \ar[u]^{|\iota_o|}  \ar[rr]_{\tau}&& |\Delta|_{\ol{\mathfrak a}} \ar[u]_{|\iota_y|}
}\end{xy}.
\end{equation}
Now the map  $\tau: |\Delta|_{\ol{\mathfrak a}} \to |\Delta|_{\ol{\mathfrak a}}$ is given by an element of the maximal torus $\ol{T} \subset \ol{A}$, which fixes pointwise the realization of the apartment corresponding to $\ol{A}$. Thus if one denotes by $\Delta_o(\ol{A})$, respectively $\Delta_y(\ol{A})$, the subsets of $\Delta_o$ and $\Delta_y$ consisting of causal or anti-causal rays in $\ol{A}$, then one has a commuting diagram
\[
\begin{xy}\xymatrix{
\Delta_o(\ol{A}) \ar[rr]^{\tau} && \Delta_y(\ol{A})\\
  &|\Delta|_{\ol{\mathfrak a}} \ar[ul]^{|\iota_o|}  \ar[ur]_{|\iota_y|}&
}\end{xy}.
\]
This shows that the restriction of $|\iota_{o,y}|$ to $\Delta_o(\ol{A})$ is induced by $\tau$, i.e.\ $r_o \in \Delta_o(\ol{A})$ is parallel to $r_y \in \Delta_y(\ol{A})$ if and only if $r_y$ is obtained from $r_o$ by a Euclidean translation, i.e., $r_y$ is parallel to $r_o$ in the Euclidean sense.
\end{proof}
One can also describe the equivalence relation $\parallel$ in group-theoretic terms. For this we introduce some notations concerning parabolic subgroups of $\ol{G}$. Given an element $\xi \in |\Delta|_{\ol{\mathfrak a}}$, we denote by $P_\xi$ the stabilizer of $\xi$ in $\Ad(G)$ and by ${\rm supp}(\xi) \in \Delta$ the smallest open cell containing $\xi$. Then $P_\xi$ is the stabilizer of ${\rm supp}(\xi)$ in $\Delta$, whence a parabolic subgroup, and in particular acts transitively on $\ol{\XXX}$ by the Iwasawa decomposition. Depending on whether $\xi \in |\Delta^+|$ or $\xi \in |\Delta|_{\ol{\mathfrak a}}^-$ we call the parabolic $P_\xi$ a \Defn{positive} or a \Defn{negative parabolic}. By \cite[Theorem 6.4.1]{Remy02} every parabolic subgroup $P_\xi$ splits as a semidirect product $P_\xi = M_\xi \ltimes U_\xi$, where $M_\xi$ is a Levi factor and $U_\xi$ is generated by the appropriate positive root subgroups. Given a point $x \in \ol{\XXX}$ and $\xi \in |\Delta|_{\ol{\mathfrak a}}$ we refer to the orbit $H_\xi(x) := U_\xi.x$ as the \Defn{horosphere} with \Defn{center} $\xi$ through $x$, and we call the horosphere positive or negative according to whether $\xi \in |\Delta^\pm|$.
\begin{proposition} \label{parabolicatinfinity}
  Let $x \in \ol{\XXX}$, let $r_x \in \Delta_x$ and let $\xi = |\iota_x|^{-1}(r_x) \in |\Delta|_{\ol{\mathfrak a}}$. Then $r \in \Delta_\bullet$ is asymptotic to $r_x$ if and only if there exists $p \in P_\xi$ such that $r = p.r_x$.
 \end{proposition}
\begin{proof} Let $y \in \ol{\XXX}$, denote by $K_y$ the stabilizer of $y$ in $G$ let $g \in G$ with $g.x= y$. Then $g.r_x \in \Delta_y$ and
\[
|\iota_y|^{-1}(g.r_x) = |\iota_{gx}|^{-1}(g.r_x) = g.|\iota_x|^{-1}(r_x) = g.\xi.
\]
Recall that $|\Delta|_{\ol{\mathfrak a}}$ was defined as a certain quotient of $\Delta \times P_{\mathbb S(\ol{\mathfrak a})}$ on which $G$ acts only on the first factor; there thus exist $C, C'\ \in \Delta$ and $p \in P_{\mathbb S(\ol{\mathfrak a})}$ such that $\xi = [(C, p)]$ and $g.\xi = [(C', p)]$. Since $K_y$ acts transitively on $\Delta$ there exists $k \in K_y$ such that $k.C' = C$. If we define $p:= kg$, then
\[
p.\xi = k.(g.\xi) = k.[(C', p)] = [(k.C', p)] = [(C, p)] = \xi \quad \text{and} \quad p.x = k.(g.x)) = k.y = y.
\]
From the former we deduce that $p \in P_\xi$, and from the latter we deduce that
\[
|\iota_y|^{-1}(p.r_x) = |\iota_{p.x}|^{-1}(p.r_x) = |\iota_{x}|^{-1}(r_x) = \xi.
\]
Thus the unique ray $r_y \in \Delta_y$ with $r_y \parallel r_x$ is given by $r_y = p.r_x$. 
This show that the asymptoticity class of $r_x$ is contained in $P_\xi.r_x$. Conversely, if $r = pr_x$ for some $p \in P_\xi$ and $y := p.x$, then 
\[
|\iota_y|^{-1}(r) = |\iota_{p.x}|^{-1}(p.r_x) = p.|\iota_x|^{-1}(r_x) = p.\xi = \xi,
\]
showing that $r \parallel r_x$.
\end{proof}
\begin{remark}
In conjunction with Lemma~\ref{SeePositivityAtInfty}, Proposition~\ref{PropAsy} implies that parallel classes of causal rays are orbits of positive parabolic subgroups, and parallel classes of anti-causal rays are orbits of negative parabolic subgroups. In particular, parallel classes of regular causal rays are orbits of positive Borel subgroups. Geometrically this means that one can obtain all rays parallel to a given regular causal ray $r$ by translating $r$ inside a flat and then sliding along a suitable positive horosphere.
\end{remark}

\subsection{The causal boundary}

\begin{definition} The space $\Delta_\parallel := \Delta _\bullet/ \parallel$ of asymptoticity classes of causal and anti-causal rays in $\ol{\XXX}$ is called the \Defn{causal boundary} of $\ol{\XXX}_G$. Its subset $\Delta_\parallel^+ := \Delta^+_\bullet/ \parallel$ is called the \Defn{future boundary} of $\ol{\XXX}$, and the complement $\Delta_\parallel^- := \Delta^-_\bullet/ \parallel$ is called the \Defn{past boundary} of $\ol{\XXX}$
\end{definition}

By Proposition~\ref{PropAsy} the $\Aut(\ol{\XXX})$-action on $\Delta_\bullet$ descends to an $\Aut(\ol{\XXX})$-action on $\Delta_\parallel$, and similarly the subgroup $\Aut^+(\ol{\XXX})$ acts on the future and the past boundary, whereas each point reflection swaps the two boundaries. 

\begin{corollary}\ \label{theend}
\begin{enumerate}
\item There exists a unique ideal polyhedral structure on $\Delta_\parallel$ such that for every $x \in \ol{\XXX}$ the map
\[
\varphi_x: \Delta_x \into \Delta_\bullet \to \Delta_\parallel
\]
is a geometric isomorphism.
\item The group $\Aut(\ol{\XXX})$ acts on $\Delta_\parallel$ by geometric automorphisms with respect to this structure.
\item The ideal polyhedral complex $\Delta_\parallel$ is $\Aut(\ol{\XXX}_G)$-equivariantly geometrically isomorphic to the $\ol{\mathfrak a}$-realization $|\Delta|_{\ol{\mathfrak a}}$ of the twin building $\Delta$. In particular, it is combinatorially isomorphic to the Davis realization $|\Delta|_{\rm Davis}$ and an ideal polyhedral geometric realization of the chamber system $\Delta$.
\item Every automorphism of $\ol{\XXX}$ is uniquely determined by the induced combinatorial automorphism of the causal boundary.
\end{enumerate}
\end{corollary}
\begin{proof} Given $x, y \in \ol{\XXX}$ we consider the diagram
\[\begin{xy}\xymatrix{
&\Delta_x \ar[dd]^{|\iota_{x,y}|} \ar[r]& \Delta_\bullet \ar[rd]&\\
|\Delta|_{\ol{\mathfrak a}}\ar[ru]^{|\iota_x|}\ar[rd]_{|\iota_y|}& &&\Delta_\parallel\\
&\Delta_y \ar[r]& \Delta_\bullet \ar[ru]&,
}\end{xy}
\]
where the horizontal maps are the canonical inclusions and the final maps are the canonical quotient maps. We observe that the diagram commutes, since the left-hand side commutes by \eqref{IotaTriangle} and the right-hand side commutes by definition of asymptoticity. In particular, the map $\varphi :=  \varphi_x \circ |\iota_x|: |\Delta|_{\ol{\mathfrak a}} \to \Delta_\parallel$ is independent of the choice of $x$.
Moreover, $\varphi$ is a bijection by Property (A1) of Proposition \ref{PropAsy}. Since both $|\iota_x|$ and $\varphi_x$ are $\Aut(\ol{\XXX})_x$-equivariant, so is the map $\varphi = \varphi_x \circ |\iota_x|$, and since the groups $\Aut(\ol{\XXX})_x$ generate $\Aut(\ol{\XXX})$ it follows that $\varphi$ is ${\rm Aut}(\XXX)$-equivariant. 

Now the ideal polyhedral structure on $|\Delta|_{\ol{\mathfrak a}}$ defines an ideal polyhedral structure on $\Delta_\parallel$ by transport of structure via $\varphi$. Since $|\iota_x|$ is a geometric isomorphism for every $x \in \ol{\XXX}$, we conclude that also $\varphi_x  = |\iota_x|^{-1} \circ \varphi$ is a geometric isomorphism. Since $\varphi_x$ is surjective this implies (i), and since $\varphi$ is  ${\rm Aut}(\XXX)$-equivariant, we deduce that (ii) and (iii) hold.

Unravelling definitions one now checks that the composition
\[
\Aut(\ol{\XXX}_G) \to \Aut(\Delta_\parallel) \to \Aut(\Delta)
\] 
coincides with the inclusion $\Aut(\ol{\XXX}_G) \cong{\Aut(G)} \to \Aut(\Delta)$ given by Theorem \ref{ThmGlobalAut} and Proposition~\ref{AutBuilding}. This implies (iv) and finishes the proof.
\end{proof}

We have shown Theorems~\ref{thm1.14} and \ref{thm1.15}. 

\subsection{Causal curves and the causal pre-order}

\begin{definition} A \Defn{piecewise geodesic causal curve} is a causal curve $\gamma: [S, T] \to \ol{\XXX}$ with $0< S < T< \infty$ for which there exist $S=t_0 < t_1 < \dots < t_N = T$ such that $\gamma|_{[t_i, t_{i+1}]}$ is a causal segment for every $i=0, \dots, N-1$.

Given $x, y \in \ol{\XXX}$ write $x \prec y$ and say that $x$ \Defn{strictly causally precedes} $y$ if there exists a piecewise geodesic causal curve $\gamma: [S, T] \to \ol{\XXX}$ with $\gamma(S) = x$ and $\gamma(T) = y$. Write $x \preceq y$ if $x \prec y$ or $x = y$ and say that $x$ \Defn{causally precedes} $y$.
\end{definition}

By definition, $\preceq$ is a pre-order, i.e.\ a reflexive and transitive relation, called the \Defn{causal pre-order} on $\ol{\XXX}$.
Since the group $\Aut^+(\ol{\XXX})$ preserves the class of piecewise geodesic causal curves, it also preserves the causal pre-order $\preceq$ in the sense that
\begin{equation}
x\preceq y \; \Rightarrow \;\alpha(x) \preceq \alpha(y) \quad (x,y \in \ol{\XXX}, \alpha \in \Aut^+(\ol{\XXX})).
\end{equation}
\begin{definition}
Let $x \in \ol{\XXX}$. The \Defn{strict causal future}, respectively \Defn{strict causal past} of $x$ are defined as 
\[
\ol{\XXX}^+_x := \{y \in \ol{\XXX}\mid y \succ x\} \quad \text{and} \quad \ol{\XXX}^-_x := \{y \in \ol{\XXX}\mid y \prec x\}.
\]
\end{definition}
\begin{remark}
If one denotes by $\ol{S}^\pm \subset \ol{G}$ the semigroups generated by $\ol{A}_\pm$ (as defined in Section~\ref{SecCausalStructureEx}) and $\ol{K}$, then 
$\ol{\XXX}^\pm_e$ is simply the $\ol{S}^\pm$--orbit through $e$. Note that, by definition, 
\[
\ol{S}^\pm = \bigcup_{n=1}^\infty \left(\ol{K}\,\ol{A}_\pm\ol{K}\right)^n
\]
and that $\ol{A}_+$ is a subsemigroup of $\ol{G}$. Since $\ol{S}^{\pm}$ contains $\ol{K}$, the semigroups $\ol{S}^{\pm}$ can also be characterized as
\[
\ol{S}^+ = \{g\in \ol{G}\mid e \preceq g.e\}  \quad \text{and} \quad \ol{S}^- = \{g\in \ol{G}\mid g.e \preceq e\} 
\]
In particular, $\preceq$ is a partial order if and only if $\ol{S}^+ \cap \ol{S}^- = \ol{K}$.
\end{remark}
\begin{proposition}\label{OrderDichotomy} Exactly one of the following two options holds in $\ol{\XXX}$:
\begin{enumerate}
\item $\preceq$ is a partial order and $\ol{S}^+ \cap \ol{S}^- = \ol{K}$.
\item $g \preceq h \preceq g$ for all $g, h \in \ol{\XXX}$ and $\ol{S}^+ = \ol{S}^- = \ol{G}$.
\end{enumerate}
\end{proposition}

\begin{proof} Obviously the two conditions are mutually exclusive. Assume that (i) fails, i.e., that $\preceq$ is not anti-symmetric. By $G$-invariance one then finds $x \in \ol{\XXX}$ such that
\[
e \prec x \prec e.
\]
By definition, this means that there exist points $x_1, \dots, x_n = x$, $y_1, \dots, y_n = y$ and causal geodesic segments from $e$ to $x_1$, $x_1$ to $x_2$, \dots, $x_{n-1}$ to $x_n$ and $x_n$ to $y_1$, \dots, $y_{n-1}$ to $y_n$ and $y_n$ to $e$. In particular, $y \prec e$ is contained in a common flat $F$ with $e$ and the geodesic ray in $F$ emanating from $y$ and through $e$ is causal. Since $K$ acts transitively on flats through $e$ there exists $k \in K$ which maps $F$ to the standard flat $\ol{A}\,\ol{K}$. Then $z := k . y$ has the following properties: Firstly, since $y \prec e \prec y$ and $k.e = e$ one has
\[
z \prec e \prec z.
\]
Moreover, $z$ lies in $\ol{A}$ and the geodesic ray emanating from $e$ through $z$ is anti-causal. In other words, $z = \exp(-X)$ for some $X \in \CCC^o$. Now consider parallel transport $\tau$ along the geodesic segment from $e$ to $z$. One has $\tau(\exp(-nX)) = \exp(-(n+1)X)$ for all $n \geq 0$. Therefore $e \prec z$ implies that for all $n \geq 0$,
\[
\exp(-nX) = \tau^n(e) \prec \tau^n(z) = \exp(-(n+1)X).
\]
Thus transitivity of $\prec$ yields
\[
e \prec \exp(-nX) \quad \text{for all }n \geq 1,
\]
and thus
\[
\ol{\XXX}^+_e \supset \bigcup_{n \geq 1} \ol{C}^+_{\exp(-nX)}.
\]
In particular,
\[
\ol{\XXX}^+_e \cap \ol{A}\,\ol{K} \supset  \bigcup_{n \geq 1} (\ol{C}^+_{\exp(-nX)} \cap \ol{A}\,\ol{K}) \supset \bigcup_{n \geq 1} \exp(\CCC^o - nX)\,\ol{K} = \exp\left(\bigcup_{n \geq 1} \CCC^o - nX\right)\,\ol{K} = \exp(\ol{\mathfrak a})\,\ol{K} = \ol{A}\,\ol{K},
\]
i.e., $\ol{A}\,\ol{K} \subset \ol{\XXX}^+_e= \{x \in \ol{\XXX}\mid x \succ e\}$. Since \[\ol{S}^+ = \bigcup_{n=1}^\infty (\ol{K}\,\ol{A}_+\ol{K})^n = \{g\in \ol{G}\mid g.e \succ e\},\] the semigroup $\ol{S}^+$ contains $\ol{A}$ and $\ol{K}$. It therefore contains each of the finite products $\ol{K}\,\ol{A}\,\ol{K} \cdots \ol{A}\, \ol{K}$. Proposition~\ref{lem:geod-conn} implies $\ol{\XXX}^+_e = \ol{\XXX}$, i.e.\ $x \succ e$ for all $x \in \ol{\XXX}$, and thus (ii) holds.
\end{proof}

We have shown Proposition~\ref{dichintro}.
\appendix

\section{Complex Kac--Moody algebras and their Weyl groups}\label{AppendixAlgebrasAndWeyl}

\subsection{Ideal polyhedral complexes associated with Coxter systems} \label{appendixAA}
Recall that a \Defn{Coxeter system} is a pair $(W,S)$ consisting of a group $W$ and a (finite) generating system $S = \{r_1, \dots, r_n\}$ such that
\[
W = \left\langle r_1, \dots, r_n \mid r_i^2  = 1, \; (r_ir_j)^{m_{ij}} = 1 \right\rangle
\]
for suitable $(m_{ij})_{i,j} \subset \Z \cup \{ \infty \}$ is a presentation of $W$ by generators and relations. The matrix $M = (m_{ij})_{i,j}$ is called the \Defn{Coxeter matrix} of the Coxeter 
system $(W,S)$. The group $W$ is then called a \Defn{Coxeter group}. Finite Coxeter groups are also called \Defn{spherical} Coxeter groups.

With each Coxeter system one can associate certain canonical ideal polyhedral complexes; concerning such complexes and the underlying posets we will use the language and notation from Subsection \ref{SecIdealPoly}.

If $T \subset S$ is a subset, then $W_T := \langle T \rangle < W$ is called a \Defn{standard parabolic subgroup} of $W$, and any conjugate of a standard parabolic subgroup is called a \Defn{parabolic subgroup} of $W$. Given a Coxeter system $(W,S)$ we denote by $\Sigma^+(W,S)$ the poset of all left-cosets of standard parabolic subgroups of $W$, ordered by reverse inclusion, and set $\Sigma(W,S) := \Sigma^+(W, S) \setminus \{W\}$. Then $\Sigma(W,S)$ is a simplicial poset, called the  \Defn{Coxeter poset} of $(W,S)$, and its augmentation $\Sigma^+(W,S)$ is called the \Defn{augmented Coxter poset}. The geometric realization $|\Sigma(W,S)|$ of $\Sigma(W,S)$ is called the \Defn{Coxeter complex} of $(W,S)$; the cone $|\Sigma^+(W,S)|$ over $|\Sigma(W,S)|$ is called the \Defn{augmented Coxeter complex}.

The comaximal elements in $\Sigma(W,S)$ and $\Sigma^+(W,S)$ are of the form $\sigma = w\langle s \rangle$ for some $w \in W$ and $s \in S$, and if we colour each coset of $\langle s \rangle$ by $s$, then we obtain \Defn{canonical colourings} of $\Sigma(W,S)$ and $\Sigma^+(W,S)$ by $S$. The group $W$ acts on $\Sigma(W,S)$ and $\Sigma^+(W,S)$, and this action is both order- and colouring-preserving. Similarly, $W$ acts cellularly on $|\Sigma(W,S)|$ and $|\Sigma^+(W,S)|$, preserving the colouring.

By definition, the Coxeter complex is a simplicial realization of $\Sigma(W,S)$; however, it is sometimes convenient to work with non-simplicial realizations. 
\begin{example} \label{simplicialcox}
For $1 \leq i \leq n$ given $m_i \in \left(\mathbb{N} \cup \{ \infty \}\right) \backslash \{ 1 \}$ with $\sum_{i = 1}^n \frac{1}{m_i} < n-2$ there exists a (possibly ideal) hyperbolic $n$-gon embedded in the Poincar\'e disc with interior angles $\frac{\pi}{m_i}$.  According to Poincar\'e's theorem (see \cite[Theorem~6.4.3]{DavisBook}) the hyperbolic reflections in the sides of the hyperbolic polygon generate a Coxeter group $W$, called a hyperbolic $n$-gon group. The $W$-translates of the closure of the hyperbolic polygon in the closed disc provide a polyhedral realization of $\Sigma(W,S)$. For $n \geq 4$ this realization certainly is not simplicial, and its dimension is always $2$ (and hence smaller than the dimension of $|\Sigma(W,S)|$).
\end{example}
In the previous example, it is natural to consider also the ideal polyhedral complex obtained by removing ideal vertices, i.e.\ intersecting all cells with the open disc, since this ideal polyhedral complex admits a CAT(-1) metric. The search for CAT(0)-realizations of Coxeter complexes led Moussong in his thesis \cite{Moussong} to consider the following subposets of the extended Coxeter complex:
\begin{defn}\label{DavisMoussongPoset} The \Defn{augmented Davis--Moussong poset} of the coloured subposet  $\Sigma_{\rm sph}^+(W,S)$ of  $\Sigma^+(W,S)$ consists of all cosets of spherical parabolic subgroups, and we define the \Defn{Davis--Moussong poset} $\Sigma_{\rm sph}(W,S) := \Sigma_{\rm sph}^+(W,S) \setminus \{W\} \subset \Sigma(W,S)$.
\end{defn}
Note that if $W$ is spherical, then $\Sigma_{\rm sph}^+(W,S) = \Sigma_{\rm sph}(W,S)^+$; otherwise we have $\Sigma_{\rm sph}^+(W,S) = \Sigma_{\rm sph}(W,S)$.
\begin{remark} The Davis--Moussong poset $\Sigma_{\rm sph}(W,S)$ and its augmentation $\Sigma^+_{\rm sph}(W,S)$ have the same underlying chamber system as the Coxeter poset $\Sigma(W, S)$. By \cite[Proposition~A.20]{AbramenkoBrown2008}, the Coxeter poset can be recovered from this chamber system as the residue poset. This implies that every automorphism of the chamber system, and in particular every automorphims of the Davis--Moussong poset extends to an automorphism of the Coxeter poset.
\end{remark}

Moussong has established in his thesis that the {augmented Davis--Moussong poset} always admits a CAT(0)-realization (cf.\ \cite{Moussong}, \cite[Chapter 12]{DavisBook}). A variant of this construction was later given by Krammer 
in his thesis \cite[Appendix B]{Krammer94}. Krammer's construction is based on the notion of a root basis {\cite[Definition~1.2.1]{Krammer94}}, which we briefly recall:
\begin{defn}\label{rootbasisKrammer} A triple $(E, (-|-), \Pi)$ is called a \Defn{root basis} if $E$ is a real vector space, $(-|-)$ is a symmetric bilinear form on $E$ and $\Pi \subset E$ is a finite set such that the following hold:
\begin{enumerate}
\item For every $\xi \in \Pi$ one has $(\xi|\xi) = 1$.
\item For any pair of distinct $\xi_1, \xi_2 \in \Pi$ one has \[(\xi_1|\xi_2) \in \{-\cos(\pi/m)\mid m \in \N\} \cup (-\infty, -1].\]
\item There exists $\lambda \in E^*$ such that $\lambda(\xi) >0$ for all $\xi \in \Pi$.
\end{enumerate}
\end{defn}

\begin{lemma}
Let $(E, (-|-), \Pi)$ be a root basis. Then $0$ is not contained in the set $$\left\{ \sum_{ \xi \in \Pi} \lambda_\xi \xi \mid \lambda_\xi \geq 0 \text{ for all $\xi \in \Pi$, and } \lambda_\xi \neq 0 \text{ for some $\xi \in \Pi$} \right\}.$$  
\end{lemma}

\begin{proof}
  Any $\lambda \in E^*$ satisfies $$\lambda\left(\sum_{ \xi \in \Pi} \lambda_\xi \xi\right) =  \sum_{ \xi \in \Pi} \lambda_\xi \lambda(\xi).$$
  If the $\lambda_\xi$ are non-negative with at least one positive and $\lambda$ is as in Definition \ref{rootbasisKrammer}.(iii), then the hypothesis $\lambda(\xi) > 0$ for all $\xi \in \Pi$ implies that the right-hand side of this equality is greater than $0$, which by linearity of $\lambda$ implies that the linear combination on the left-hand side is different from zero.  
\end{proof}

\begin{remark}\label{HowlettApplies}
  The preceding lemma implies that a root basis in the sense of Definition~\ref{rootbasisKrammer} is also a root basis in the sense of \cite[Definition~2.1]{Howlett/Rowley/Taylor:1997}.
\end{remark}

The relation to Coxeter groups is as follows:
\begin{proposition}[cf.\ {\cite[Theorem~1.2.2]{Krammer94}}] \label{rootbasiscoxeter}
Let  $(E, (-|-), \Pi)$ be a root basis, for each $\xi \in \Pi$ define $s_\xi \in \GL(E)$ via
\[
s_\xi(v) = v- 2(\xi|v)\xi,
\]
let ${S} := \{s_{\xi}\mid \xi \in \Pi\}$, and let ${W} := \langle {S} \rangle < \GL(E)$. Then $({W}, {S})$ is a Coxeter system and ${W} < {\rm O}(E,  (-|-))$ is a discrete subgroup.\qed
\end{proposition}
In the situation of Proposition \ref{rootbasiscoxeter} we say that $(E, (-|-), \Pi)$ is a root basis for the Coxeter system $(W,S)$. The following example shows that every Coxeter group admits a root basis:
\begin{example}
Let $(W,S)$ a Coxeter group with Coxter system, $S = \{r_1, \dots, r_n\}$ with Coxeter matrix $(m_{ij})_{1 \leq i,j \leq n}$. Then the \Defn{classical root basis} given by $E = \R^n$, $\Pi = \{e_1, \dots, e_n\}$ and $(e_i\mid e_j) = -\cos(\pi/m_{ij})$ if $m_{ij} < \infty$ and $(e_i \mid e_j) = -1$ if $m_{ij} = \infty$ is a root basis for $(W,S)$.
\end{example}
In the case of the classical root basis, $\Pi$ is a basis of $E$. In general, we do not assume that $\Pi$ is linearly independent, and we explicitly allow $(-|-)$ to be degenerate. From now on, $(W,S)$ denotes a Coxeter system and $(E, (-|-), \Pi)$ denotes a root basis for $(W,S)$. For $I \subset \Pi$ we defined a subset of the dual space $E^*$ of $E$ by
\[
C^*_I := \{\varphi \in E^*\mid (\forall \alpha \in I:\, \varphi(\alpha) = 0) \text{ and } (\forall \alpha \in \Pi \setminus I:\, \varphi(\alpha) > 0)\}.
\]
We say that $I \subset \Pi$ is \Defn{facial} if $C_I \neq \emptyset$. The set 
\[
C^* := \bigsqcup_{I \subset \Pi \text{ facial}} C^*_I
\]
is called the  \Defn{closed fundamental chamber} of the given root basis; it has dense interior given by ${\rm Int}(C^*) = C^*_\emptyset$. The group $W$ acts on $E^*$ by the dual action, i.e.\ $w.\varphi(x) = \varphi(w^{-1}.x)$, and we refer to the translates $w.C^*$ of the closed fundamental chambers (respectively their interiors) as \Defn{closed (resp. open) Tits chambers}; the union 
\[
\mathcal C^* = \bigcup_{w \in W} w.C^* \subset E^*
\]
of the closed Tits chambers is called the \Defn{dual Tits cone} of the root bases. By \cite[Thm. 1.2.2]{Krammer94} the dual Tits cone is naturally a polyhedral complex whose cells are given by the subsets of the form $w.C^*_I$ with $w \in W$ and $I \subset \Pi$. Note, however, that the weak-topology on $\mathcal C^*$ is finer than the subspace topology from $E^*$ and that $\mathcal C^*$ is not closed in $E^*$.

If we identify $\Pi$ with $S$ by identifying $\alpha \in \Pi$ with the corresponding reflections $r_\alpha$, then the stabilizer of $x \in w.C^*_I$ is given by the parabolic subgroup $wW_Iw^{-1}$. Its underlying poset is the subposet $\Sigma_{\rm fac}^+(W,S) \subset \Sigma^+(W,S)$ consisting of those $wW_I$ for which $I$ is facial. 
\begin{proposition}[{\cite[Cor. 2.2.5]{Krammer94}}]\label{TitsInterior} The interior ${\rm Int}(\mathcal C^*)$ of the dual Tits cone is the union of the cells $w.C^*_I$ such that $W_I$ is spherical. In particular, the interior of the dual Tits cone is a realization of the augmented David-Moussong poset $\Sigma^+_{\rm sph}(W,S)$.\qed
\end{proposition}
\begin{remark} As explained in \cite{Krammer94}, the interior of the dual Tits cone has several advantages over the full dual Tits cone. Firstly, it is open by definition, whereas the dual Tits cone is neither open nor closed in general. Secondly, the $W$-action on ${\rm Int}(\mathcal C^*)$ is proper, whereas the action on $\mathcal C^*$ is in general not proper. Thirdly, the interior of the dual Tits cone admits a CAT(0) metric (namely, the Moussong metric). Finally, while the subspace topology of $\mathcal C^*$ is finer than the weak topology in general, the subspace topology on ${\rm Int}(\mathcal C^*)$ actually coincides with the weak topology. \end{remark}
From the dual Tits cone we can also construct a realization of the (non-augmented) Davis--Moussong poset $\Sigma_{\rm sph}(W,S)$ by passing to the link complex as in Remark \ref{LinkComplex}:
\begin{corollary}\label{SphericalTitsCone} Let $W$ be non-spherical. If $\mathcal C^*$ is the dual Tits cone of a root basis $(E, (-|-), \Pi)$ for $(W,S)$, then the link complex $\mathbb S({\rm Int(\mathcal C^*)})$ is an ideal polyhedral realization of the Davis--Moussong poset $\Sigma_{\rm sph}(W, S)$.\qed
\end{corollary}

\subsection{Complex Kac--Moody algebras} \label{AppendixAlgebrasAndWeylAB}

Throughout this appendix let $\mathbf{A}$ be an irreducible generalized Cartan matrix in the sense of Definition~\ref{GCMdef} (see also \cite[\S1.1]{Kac90}). We will mostly be interested in the case where $\mathbf{A}$ is neither of spherical nor of affine type, and starting from Subsection \ref{bases} we will have to assume that $\mathbf{A}$ is symmetrizable. However, for the moment no such assumptions are necessary.

One can associate to ${\mathbf A}$ several complex Lie algebras as follows. In  \cite[\S1.3]{Kac90} Kac defines a quadruple
\begin{equation}
  ({\mathfrak g}({\mathbf{A}}), \mathfrak h({\mathbf{A}}), \Pi, \check \Pi)   \label{quadruple}
\end{equation}
  consisting of a complex Lie algebra $\mathfrak g({\mathbf{A}})$, an abelian subalgebra $\mathfrak h({\mathbf{A}})$ and finite subsets $\Pi = \{\alpha_1, \dots, \alpha_n\} \subset \mathfrak h({\mathbf{A}})^*$ and $\check \Pi = (\check \alpha_1, \dots, \check \alpha_n) \subset \mathfrak h({\mathbf{A}})$ called \Defn{simple roots} and \Defn{simple coroots} respectively. A useful characterization of this quadruple $({\mathfrak g}({\mathbf{A}}), \mathfrak h({\mathbf{A}}), \Pi, \check \Pi)$ is given in \cite[Proposition~1.4]{Kac90}. In the present article $\mathfrak g({\mathbf{A}})$ is called the \Defn{complex Kac-Moody algebra} associated with ${\mathbf{A}}$. If one denotes by
\[\Delta := \left\{ \alpha \in \sum_{i=1}^n \Z \alpha_i \mid \mathfrak{g}_\alpha \neq \{ 0 \} \right\}\] the set of $\mathfrak h(\mathbf{A})$-roots in $\mathfrak g(\mathbf{A})$, then
by \cite[(1.3.1)]{Kac90} one has the \Defn{root space decomposition}
\begin{equation} \label{rootspacedec}
\mathfrak g({\mathbf{A}}) = \bigoplus_{\alpha \in \Delta} \mathfrak{g}_\alpha.
\end{equation}
Denote by $\mathfrak g_i < \mathfrak g({\mathbf{A}})$ the complex subalgebra generated by the root spaces $\mathfrak{g}_{\alpha_i}$ and $\mathfrak{g}_{-\alpha_i}$. By \cite[(1.3.3), (1.4.1), (1.4.2)]{Kac90} one has \[\mathfrak{g}_i = \langle \mathfrak{g}_{\alpha_i}, \mathfrak{g}_{-\alpha_i} \rangle \cong \mathfrak{sl}(2,\C).\]
Given $I\subset \{1, \dots, n\}$, define \[\mathfrak g_I := \langle \mathfrak g_i\mid i \in I \rangle\] and call $\mathfrak g_I$ a \Defn{standard rank $|I|$} subalgebra of $\mathfrak g({\mathbf{A}})$.

The main object of interest in this appendix is the derived subalgebra
\begin{equation}\label{defKMalg}
\mathfrak g := [\mathfrak g({\mathbf{A}}), \mathfrak g({\mathbf{A}})] < \mathfrak g({\mathbf{A}}),
\end{equation}
which is called the \Defn{derived complex Kac-Moody algebra} associated with ${\mathbf{A}}$. It is denoted by $\mathfrak g'({\mathbf{A}})$ in \cite[\S1.3]{Kac90}. The Lie algebra $\mathfrak g$ contains all the standard rank one subalgebras, as $\mathfrak{sl}(2,\C)$ is perfect, and in fact is generated by these by \cite[Proposition~1.4]{Kac90}. The intersection
\begin{equation}\label{hdef}
  \mathfrak h := \mathfrak h({\mathbf{A}}) \cap \mathfrak g = \sum_{i = 1}^n \C \check \alpha_n
\end{equation}
  is given by the complex span of the simple coroots, see \cite[\S1.3]{Kac90} (where it is denoted by $\mathfrak{h}'$). By \cite[Proposition~1.6]{Kac90} the Lie algebra $\mathfrak{h}$ contains the center of $\mathfrak g({\mathbf{A}})$ and of $\mathfrak{g}$, which is given by
\begin{equation}\label{KMAcenter}
\mathfrak{z}(\mathfrak g({\mathbf{A}})) = \mathfrak{z}(\mathfrak{g}) = \mathfrak c := \{h \in \mathfrak h({\mathbf{A}}) \mid \forall i = 1, \dots, n: \; \alpha_i(h) = 0\}.
\end{equation}
The third Lie algebra of interest in this appendix is the quotient \[\ol{\mathfrak g} := \mathfrak g/\mathfrak c,\] called the \Defn{adjoint complex Kac-Moody algebra} associated with ${\mathbf{A}}$. Since $\mathfrak{sl}(2,\C)$ is simple, the standard rank one subalgebras $\mathfrak{g}_i$ embed into $\ol{\mathfrak g}$ and so do in fact all root spaces $\mathfrak{g}_\alpha$ for $\alpha \neq 0$, from \eqref{rootspacedec}, whereas the image of $\mathfrak{g}_0 = \mathfrak h$ in $\overline{\mathfrak g}$ is given by $\ol{\mathfrak h} := \mathfrak h/ \mathfrak c$.

If ${\mathbf{A}}$ is of size $n \times n$ and of rank $l$, then the complex dimensions of the abelian subalgebras are given by
\begin{equation} \label{dimofh}
\dim_\C \mathfrak h({\mathbf{A}}) = 2n-l, \quad \dim_\C \mathfrak h = n, \quad \dim_\C \ol{\mathfrak h}  = l,
\end{equation}
cf.\ \cite[(1.1.3), resp.\ \S1.3, resp.\ Proposition~1.6]{Kac90}. In particular, on one hand one has the following:

\begin{observation}
Let $\mathbf{A}$ be an invertible generalized Cartan matrix. Then \[\mathfrak{g}(\mathbf{A}) = \mathfrak{g} = \ol{\mathfrak{g}}.\]
\end{observation}

On the other hand, the following example illustrates the differences between the Lie algebras $\mathfrak g({\mathbf{A}})$, $\mathfrak g$ and $\ol{\mathfrak g}$ for an irreducible generalized Cartan matrix of affine type.
\begin{example}
Let ${\mathbf{A}}$ be an irreducible generalized Cartan matrix of affine type and denote by $\overset{o}{\mathfrak g}$ the finite-dimensional simple Lie algebra associated with the corresponding Cartan matrix of finite type. Then in the notation of \cite[Chapter 7]{Kac90} the Lie algebra \[\ol{\mathfrak g} = \LLL(\overset{o}{\mathfrak g})\] is the loop algebra of $\overset{o}{\mathfrak g}$, whereas \[\mathfrak g = \LLL(\overset{o}{\mathfrak g}) \oplus \C K\] is a one-dimensional central extension of the loop algebra and \[\mathfrak g({\mathbf{A}}) = \LLL(\overset{o}{\mathfrak g}) \oplus \C K \oplus \C d\] for a certain derivation $d$. The complex dimensions of $\mathfrak h({\mathbf{A}})$, $\mathfrak h$ and $\ol{\mathfrak h}$ are given by $\rk(\overset{o}{\mathfrak g})+2$, $\rk(\overset{o}{\mathfrak g})+1$ and $\rk(\overset{o}{\mathfrak g})$, respectively.
\end{example}

For a symmetrizable generalized Cartan matrix ${\mathbf{A}}$ as defined in \cite[\S2.1]{Kac90} (see also Definition~\ref{GCMdef}) the Gabber--Kac Theorem provides a very efficient way of defining the derived Kac--Moody algebra $\mathfrak{g}$.

\begin{theorem}[Gabber--Kac Theorem] \label{GabberKac}
  Let $\mathbf{A} = (a_{ij})_{1 \leq i, j \leq n}$ be a symmetrizable generalized Cartan matrix of size $n \times n$. Then the derived complex Kac--Moody algebra $\mathfrak{g}$ is isomorphic to the quotient of the free complex Lie algebra on $3n$ generators $e_i$, $f_i$, $h_i$, $1 \leq i \leq n$, modulo the following relations:
  \begin{eqnarray*}
    & & [h_i,h_j]=0, \qquad [e_i,f_i]=h_i, \qquad [e_i,f_j]=0 \,\,\, (i \neq j), \\
    & & [h_i,e_j]=a_{ij}e_j, \qquad [h_i,f_j] = -a_{ij}f_j, \\
    & & (\ad e_i)^{1-a_{ij}}e_j = 0 \,\,\, (i \neq j), \qquad (\ad e_i)^{1-a_{ij}}f_j = 0 \,\,\, (i \neq j),
  \end{eqnarray*}
via the homomorphism that maps $h_i$ to $\check \alpha_i$ and transforms $a_{ij}$ into $\alpha_j(\check \alpha_i)$.
  
In particular, $\mathfrak{g}$ is the colimit of the amalgam of Lie algebras consisting of its standard subalgebras $\mathfrak{g}_i$, $\mathfrak{g}_{i,j}$ of rank one and two.
\end{theorem}

\begin{proof}
See \cite{GabberKac81}, \cite[Theorem~9.11]{Kac90} plus \cite[Remark~1.5]{Kac90}.
\end{proof}

The presentation of $\mathfrak{g}$ from the preceding theorem is called the \Defn{Gabber--Kac presentation}. Of course, one obtains a presentation of $\ol{\mathfrak g}$ by adding the elements of $\mathfrak c$ as relators to the Gabber--Kac presentation.

\begin{notation} \label{defa} \label{defofa}
Since $\mathfrak h = \sum_{i = 1}^n \C \check \alpha_n$, one can define a real form $\mathfrak a$ of $\mathfrak h$ by setting
\[
\mathfrak a := {\mathrm{span}}_\R(\check \alpha_1, \dots, \check \alpha_n).
\] 
Dually, also define a subspace $V \subset \mathfrak h({\bf A})^*$ by
\[
V := {\rm span}_\R(\alpha_1, \dots, \alpha_n)
\]

Then the image of $\mathfrak a$ under the canonical projection $\mathfrak h \to \ol{\mathfrak h}$ defines a real form of $\ol{\mathfrak h}$ which is denoted by $\ol{\mathfrak a}$. One then has the following commutative diagram, where all maps are the obvious inclusions/projections, respectively their dual maps:
\begin{equation}\label{BigDiagram}
\begin{xy}
\xymatrix{
  & \mathfrak h({\bf A}) &&& \mathfrak h({\bf A})^* \ar@{->>}[dl]_{\iota^*} \ar@{->>}[d]^{\iota_\C^*}&V\ar@{->>}[d] \ar@{_{(}->}[l]\\
\mathfrak a \ar@{^{(}->}[ur]^\iota \ar@{->>}[d]_\pi  \ar@{^{(}->} [r]^{j}& \mathfrak h \ar@{->>}[d]^{\pi_\C}\ar@{^{(}->}[u]_{\iota_\C}&&\mathfrak a^*& \mathfrak h^* \ar@{->>}[l]_{j^*} & \iota_\C^*(V)\ar@{_{(}->}[l] \\
\ol{\mathfrak a}\ar@{^{(}->}[r] & \ol{\mathfrak h}&&\ol{\mathfrak a}^*\ar@{^{(}->}[u]^{\pi^*}& \ol{\mathfrak h}^*\ar@{^{(}->}[u]_{\pi_\C^*}\ar@{->>}[l] &
}
\end{xy}\end{equation}
\end{notation}
All of these maps are linear (over $\R$ and $\C$ respectively) and injective/surjective as indicated by the arrows. The dual maps $\iota^*$ and $j^*$ are defined by considering $\iota$ and $j$ as linear maps between real vector spaces. 

\subsection{The Weyl group, its Coxeter system and its Kac--Moody representation} \label{appendixweylgroup}

\begin{definition}\label{Weylgroup}
Following \cite[\S3.7]{Kac90}), given $i\in \{1, \dots, n\}$ define $r_{\alpha_i} \in \GL(\mathfrak h({\mathbf{A}})^*)$ by
\begin{equation}\label{WeylReflDef2}
r_{\alpha_i}(\lambda) = \lambda - \lambda(\check \alpha_i)\alpha_i;
\end{equation}
dually, define $\check r_{\alpha_i} \in \GL(\mathfrak h({\mathbf{A}}))$ by
\begin{equation}\label{WeylReflDef}
\check r_{\alpha_i}(h) = h - \alpha_i(h)\check \alpha_i.
\end{equation}
The groups $W:=\langle \check r_{\alpha_1}, \dots, \check r_{\alpha_n} \rangle <  \GL(\mathfrak h({\mathbf{A}}))$ and $\langle r_{\alpha_1}, \dots, r_{\alpha_n} \rangle <  \GL(\mathfrak h({\mathbf{A}})^*)$ are canonically isomorphic via $\check r_{\alpha_i} \mapsto r_{\alpha_i}$; the group $W$ is called the \Defn{Weyl group} associated with the generalized Cartan matrix ${\mathbf{A}}$. 
\end{definition}

For $S := \{r_{\alpha_1}, \dots, r_{\alpha_n}\}$ the pair $(W,S)$ is a Coxeter system by \cite[\S3.13]{Kac90}. According to \cite[Proposition~3.13]{Kac90} (see also Definition~\ref{GCMdef}) its Coxeter matrix $M = (m_{ij})_{i, j}$ is given by $m_{ii} = 1$ and $m_{ij}$ for $i \neq j$ by
\[
m_{ij} = \left\{\begin{array}{ll}2, & a_{ij}a_{ji} = 0,\\
3, & a_{ij}a_{ji} = 1,\\
4, & a_{ij}a_{ji} = 2,\\
6, & a_{ij}a_{ji} = 3, \\
\infty, & a_{ij}a_{ji} \geq 4;
\end{array} \right.
\]
recall here from \cite[(1.1.2)]{Kac90} that $a_{ij} = \alpha_j(\check \alpha_i)$.

The action of the Weyl group $W$ on $\mathfrak h({\mathbf{A}})^*$ defined in \eqref{WeylReflDef2} preserves the set $\Delta$ of $\mathfrak h(\mathbf{A})$-roots in $\mathfrak g(\mathbf{A})$,
and the elements of $\Phi = W.\Pi \subset \Delta$ are called the \Defn{real roots} of $\mathfrak g(\mathbf{A})$. To a real root $\alpha = w.\alpha_i \in \Phi$, $w \in W$, corresponds the \Defn{root reflection} $\check r_\alpha : = w\check r_{\alpha_i}w^{-1} \in W$, which depends only on $\alpha$; see \cite[proof of Lemma~3.10]{Kac90}.

The tuple $(W, S, \Phi, \Pi)$ is called the \Defn{Coxeter datum} associated with the generalized Cartan matrix ${\mathbf{A}}$. Note that the Coxeter datum determines uniquely a system $\Phi^+ \subset \Phi$ of \Defn{positive roots} by demanding that $\Phi^+$ contains $\Pi$.

With the notation introduced in Notation \ref{defa} one has:

\begin{proposition} \label{KMRepWeyl}
\begin{enumerate}
\item The action of the Weyl group $W$ defined in \eqref{WeylReflDef} stabilizes the complex subalgebra $\mathfrak h < \mathfrak{h}(\mathbf{A})$ and its real form
$\mathfrak a$, acts trivially on $\mathfrak{c}$ and thus induces actions of $W$ on $\ol{\mathfrak h}$ and $\ol{\mathfrak{a}}$.
\item The action of the Weyl group $W$ defined in \eqref{WeylReflDef2} stabilizes the real subspace $V <\mathfrak{h}(\mathbf{A})^*$.
\item The map $j^*$ induces an isomorphism
\[
 \iota_\C^*(V) \xto{\cong} \pi^*(\ol{\mathfrak{a}}^*)
\]
\item The action of the Weyl group $W$ from assertion (ii) acts trivially on $\ker(\iota^*)$ and, thus, induces an action of $W$ on $\pi^*(\ol{\mathfrak{a}}^*) \cong \iota_\C^*(V)$ and, by transport of structure, on $\ol{\mathfrak{a}}^*$.
\end{enumerate}
\end{proposition}

\begin{proof} It is immediate from \eqref{WeylReflDef2} and \eqref{WeylReflDef} that each $r_{\alpha_i}$ maps simple roots to $\R$-linear combinations of simple roots and each $\check r_{\alpha_i}$ maps simple coroots to $\R$-linear combinations of simple coroots. Moreover, each $r_{\alpha_i}$ acts trivially on $\mathfrak c$ by \eqref{KMAcenter} and \eqref{WeylReflDef}. Assertions (i) and (ii) follow.

In order to prove (iii) recall from \eqref{dimofh} that the quotient $\ol{\mathfrak{a}}$ has $\R$-dimension $l$, and so do $\ol{\mathfrak{a}}^*$ and $\pi^*(\ol{\mathfrak{a}}^*)$. For each $h \in \mathfrak{c} \cap \mathfrak{a}$ one has $(\iota^*(\alpha_i))(h) = (\alpha_i \circ \iota)(h) = (\alpha_i)_{\mathfrak{a}}(h) = 0$. That is, each $\iota^*(\alpha_i) = {\alpha_i}_{|\mathfrak{a}}$ in fact is of the form $\pi^*(\ol{\alpha}_i) = \ol{\alpha}_i \circ \pi$ for a uniquely determined $\ol{\alpha}_i \in \ol{\mathfrak{a}}^*$; in other words, $\iota^*(\alpha_i) \in \pi^*(\ol{\mathfrak{a}}^*)$. Since $V$ equals the $\R$-span of the simple roots $\alpha_1,\dots,\alpha_n$, the image $\iota^*(V)$ equals the $\R$-span of the images $\iota^*(\alpha_1),\dots,\iota^*(\alpha_n)$. In particular, $\iota^*(V) \leq \pi^*(\ol{\mathfrak{a}}^*)$.

Since $\mathfrak{a}$ is the $\R$-span of the simple coroots $\check\alpha_1,\dots,\check\alpha_n$, the $\R$-dimension of the image $\iota^*(V)$ equals the rank of the generalized Cartan matrix $\mathbf{A}$, i.e., $\dim_\R(\iota^*(V)) = l$. One concludes $\pi^*(\ol{\mathfrak{a}}^*) = \iota^*(V)$.

In order to prove (iv), observe that $\lambda \in \ker(\iota^*)$ if and only if for each $1 \leq i \leq n$ one has $\lambda(\check \alpha_i) = 0$. Therefore for any $\lambda \in \ker(\iota^*)$ one has $r_{\alpha_i}(\lambda) = \lambda - \lambda(\check \alpha_i) \alpha_i = \lambda$ by \eqref{WeylReflDef2}; that is, $W$ acts trivially on $\ker(\iota^*)$.
\end{proof}
\begin{definition}\label{DefKMRep}  The real representations
 \[{\rho}_{KM}: W \to \GL(\mathfrak a) \qquad \text{ and } \qquad \ol{\rho}_{KM}: W \to \GL(\ol{\mathfrak a})\]
discussed in Proposition~\ref{KMRepWeyl}  are called the \Defn{Kac--Moody representation} of $W$ and the \Defn{reduced Kac--Moody representation} of $W$, respectively. The real representation
\[
W \to \GL(\ol{\mathfrak a}^*)
\]
is called the \Defn{dual reduced Kac--Moody representation}.
\end{definition}
Note that the Kac--Moody representation of $W$ depends on the generalized Cartan matrix $\mathbf{A}$ and not just on the Coxeter system $(W,S)$ (or the Coxeter matrix $M$), hence the name.

\subsection{Existence of root bases for Weyl groups and symmetrizability} \label{bases}

In general, given an irreducible generalized Cartan matrix $\mathbf{A}$, one cannot find a root basis for $\rho_{KM}(W)$ in $\mathfrak a$ or for $\ol{\rho}_{KM}(W)$ in $\ol{\mathfrak a}$. For instance, if $\mathbf{A}$ is not symmetrizable, then by \cite[Exercise~1.5E(2)]{Kumar02} (also \cite[\S 2.10, Exercise~2.3]{Kac90}) there simply does not exist a suitable $\rho_{KM}(W)$-invariant bilinear form on $\mathfrak a$. We will see in this section that non-symmetrizability actually is the only obstruction for the existence of a root basis in $\mathfrak a$. The case of the quotient $\ol{\mathfrak a}$ is a bit more subtle; however, if one excludes the affine case, it is also possible to construct a root basis for $W$ in $\ol{\mathfrak a}$, as we will discuss below.

For a symmetrizable generalized Cartan matrix $\mathbf{A}$ and a diagonal matrix $D = {\rm diag}(\epsilon_1, \dots, \epsilon_n)$ with positive entries such that $D^{-1}\mathbf{A} = (b_{ij})$ is symmetric, following \cite[(2.1.4)]{Kac90} one defines an \Defn{invariant symmetric bilinear form} on $\mathfrak a$ via
\[
(\check \alpha_i|\check \alpha_j) := b_{ij}\epsilon_i\epsilon_j.
\]
Note that $b_{jj}\epsilon_j = a_{jj} = 2$, whence
\begin{eqnarray}
\check r_{\alpha_j}(\check \alpha_i) & \stackrel{\eqref{WeylReflDef}}{=} & \check \alpha_i - \alpha_j(\check \alpha_i)\check \alpha_j
=
\check \alpha_i - a_{ij}\check \alpha_j \notag \\ & = & \check \alpha_i - \epsilon_{i}b_{ij}\check \alpha_j= \check \alpha_i - 2\frac{b_{ij}\epsilon_i\epsilon_j}{b_{jj}\epsilon_j^2}\check \alpha_j = \check \alpha_i - 2 \frac{( \check \alpha_j|\check\alpha_i)}{(\check \alpha_j|\check \alpha_j)}\check \alpha_j, \label{formcoroots}
\end{eqnarray}
i.e., ${{\check r}_{\alpha_j}}|_{\mathfrak{a}}$ is the $(-|-)$-orthogonal reflection associated with $\check \alpha_j$, in particular $(-|-)$ is invariant under the action of $W$ on $\mathfrak a$.

Define the \Defn{normalized coroots} by 
\[\check n_j := \frac{\check \alpha_j}{(\check \alpha_j|\check \alpha_j)^{\frac 1 2}} = \frac{1}{\sqrt{2\epsilon_j}} \check \alpha_j\]
and set $\check \Pi_{\rm nor} := \{ \check n_1, \dots, \check n_n \}$.

Following \cite[(2.1.6)]{Kac90}, one dually defines an \Defn{invariant symmetric bilinear form} on $V$ via \[
(\alpha_i|\alpha_j) := b_{ij} = \frac{a_{ij}}{\epsilon_i}.
\] 

As above one computes
\begin{eqnarray}
r_{\alpha_j}(\alpha_i) & \stackrel{\eqref{WeylReflDef2}}{=} & \alpha_i - \alpha_i(\check \alpha_j)\alpha_j=
\alpha_i - a_{ji}\alpha_j \notag \\ & = & \alpha_i - b_{ji}\epsilon_{j}\alpha_j= \alpha_i - 2\frac{b_{ji}\epsilon_j}{a_{jj}}\alpha_j = \alpha_i - 2 \frac{( \alpha_j|\alpha_i)}{(\alpha_j|\alpha_j)}\alpha_j. \label{formroots}
\end{eqnarray}
Define the \Defn{normalized roots} by 
\[n_j := \frac{\alpha_j}{(\alpha_j|\alpha_j)^{\frac 1 2}} = \frac{\sqrt{\epsilon_j}}{\sqrt{2}} \alpha_j\]
and set $\Pi_{\rm nor} := \{ n_1, \dots, n_n \}$.

\begin{proposition}\label{IntertwinerWeyl1}
  Let $\mathbf{A}$ be a symmetrizable irreducible generalized Cartan matrix and let $(W, S)$ be the associated Coxter system (cf.\ Proposition~\ref{rootbasiscoxeter}).
  
  Then the triples \[\left(\mathfrak{a}, (-|-), \check \Pi_{\rm nor}\right) \qquad \text{and} \qquad \left(V, (-|-), \Pi_{\rm nor}\right)\] are root bases for $(W,S)$. If $\mathbf{A}$ is non-affine, then also their images \[\left(\pi(\mathfrak{a})=\ol{\mathfrak{a}}, (-|-)/\ker(\pi), \pi(\check \Pi_{\rm nor})\right) \qquad \text{and} \qquad  \left(\pi^*(\ol{\mathfrak{a}}^*) = \iota^*(V), (-|-)/\ker(\iota^*),\iota^*(\Pi_{\rm nor})\right)\] are root bases for $(W,S)$.
\end{proposition}

\begin{proof}
  One computes \[(\check n_i|\check n_i) = (\frac{1}{\sqrt{2\epsilon_i}} \check \alpha_i|\frac{1}{\sqrt{2\epsilon_i}} \check \alpha_i) = \frac{1}{2\epsilon_i}(\check\alpha_i|\check\alpha_i) = \frac{1}{2\epsilon_i}b_{ii}\epsilon_i\epsilon_i = \frac{1}{2\epsilon_i}a_{ii}\epsilon_i = 1\] and \[(\check n_i|\check n_j) = (\frac{1}{\sqrt{2\epsilon_i}} \check \alpha_i|\frac{1}{\sqrt{2\epsilon_j}} \check \alpha_j) = \frac{1}{2\sqrt{\epsilon_i\epsilon_j}}(\check \alpha_i|\check\alpha_j) = \frac{1}{2} \frac{\sqrt{\epsilon_j}}{\sqrt{\epsilon_i}}a_{ij} = -\frac{1}{2}\sqrt{\frac{a_{ji}}{a_{ij}}}|a_{ij}| = -\frac{1}{2}\sqrt{a_{ij}a_{ji}}.\]

Moreover, \[(n_i|n_i) = (\frac{\sqrt{\epsilon_i}}{\sqrt{2}} \alpha_i|\frac{\sqrt{\epsilon_i}}{\sqrt{2}} \alpha_i) = \frac{\epsilon_i}{2}(\alpha_i|\alpha_i) = 1\] and \[(n_i|n_j) = (\frac{\sqrt{\epsilon_i}}{\sqrt{2}} \alpha_i|\frac{\sqrt{\epsilon_j}}{\sqrt{2}} \alpha_j) = \frac{\sqrt{\epsilon_i\epsilon_j}}{2}(\alpha_i|\alpha_j) = \frac{1}{2} \frac{\sqrt{\epsilon_j}}{\sqrt{\epsilon_i}}a_{ij} = -\frac{1}{2}\sqrt{\frac{a_{ji}}{a_{ij}}} |a_{ij}| = -\frac{1}{2}\sqrt{a_{ij}a_{ji}}.\]
It follows that $(\check n_i|\check n_j), (n_i|n_i) \in \{-\cos(\pi/m)\mid m \in \N\}\,\cup\,\, ]-\infty,-1]$.
    Altogether, $(\mathfrak{a}, (-|-), \check \Pi_{\rm nor})$ and $(V, (-|-), \Pi_{\rm nor})$ satisfy axioms (i) and (ii) of the definition of a root basis.
    Linear independence of $\check \Pi_{\rm nor}$ and $\Pi_{\rm nor}$ furthermore imply axiom (iii). Thus $(\mathfrak{a}, (-|-), \check \Pi_{\rm nor})$ and $(V, (-|-), \Pi_{\rm nor})$  are root bases, and in view of Proposition~\ref{rootbasiscoxeter} it follows from the explicit formulas \eqref{formcoroots} and \eqref{formroots} that the corresponding Coxeter systems are isomorphic to $(W,S)$. 
    
    Equality \eqref{formcoroots} moreover implies that the radical of the invariant bilinear form on $\mathfrak{a}$ equals $\ker(\pi) = \mathfrak{c} \cap \mathfrak{a}$ (see also~\cite[Lemma~2.1]{Kac90}). Equality \eqref{formroots} implies that the radical of the invariant bilinear form on $V$ equals $\ker(\iota^*)$, as for any $\lambda \in \ker(\iota^*)$ one has $r_{\alpha_i}(\lambda) = \lambda - \lambda(\check \alpha_i) \alpha_i = \lambda$ and $\dim_\R \ker(\iota^*) = n-l$, where $l$ is the rank of $\mathbf{A}$. Thus if ${\bf A}$ is non-affine, then \cite[Proposition~6.1.3]{Krammer94} applies, and the images of $(\mathfrak{a}, (-|-), \check \Pi_{\rm nor})$ and $(V, (-|-), \Pi_{\rm nor})$ on $\ol{\mathfrak{a}}$ and $\iota^*(V)$ are root bases as well, for the same Coxeter system.
\end{proof}
Note that the bilinear form $(-|-)/\ker(\pi)$ on $\ol{\mathfrak a}$ is always non-degenerate, since the radical of the invariant bilinear form on $\mathfrak{a}$ equals $\ker(\pi) = \mathfrak{c} \cap \mathfrak{a} <  \{h \in \mathfrak h({\mathbf{A}}) \mid \forall i = 1, \dots, n: \; \alpha_i(h) = 0\}$. In the sequel we will usually denote the bilinear form $(-|-)/\ker(\pi)$ on $\ol{\mathfrak a}$ simply by $(- | -)$, unless we want to distinguish it explicitly from the form $(-|-)$ on $\mathfrak a$. We will also write $\sigma_i := \ol{\rho}_{KM}(\check r_{\alpha_i})$ for the Coxeter generators of $\ol{\rho}_{KM}(W)$.

\begin{corollary}\label{CorollaryRedKMRepFaithful} If ${\bf A}$ is irreducible, symmetrizable and non-affine, then $\left(\ol{\mathfrak{a}}, (-|-), \pi(\check \Pi_{\rm nor})\right)$ is a root basis for the Coxeter system $(\ol{\rho}_{KM}(W), \{\sigma_1, \dots, \sigma_n\}) \cong (W,S)$ and the reduced Kac--Moody representation
$\ol{\rho}_{KM}: W \to \GL(\ol{\mathfrak{a}})$ is faithful. \qed
\end{corollary}
Note that the statement of the corollary does not hold in the affine case. Here the image $\ol{\rho}_{KM}(W)$ is just the canonical finite quotient of $W$ given by the underlying spherical Coxeter diagram, and thus the reduced Kac-Moody representation is \emph{not} faithful.

\subsection{The unreduced and reduced Tits cone}\label{TitsCones}
From now on we will always assume that our irreducible generalized Cartan matrix ${\bf A}$ is symmetrizable. As before we denote by $(W,S)$ the associated Coxeter system. By Proposition \ref{IntertwinerWeyl1} we then have a root basis for $(W,S)$ given by $\left(\mathfrak{a}, (-|-), \check \Pi_{\rm nor}\right)$. We refer to the associated dual Tits cone $\mathcal C^* \subset \mathfrak a^*$ as the \Defn{unreduced dual Tits cone} of ${\bf A}$. Explicitly, the fundamental chamber of the unreduced dual Tits cone is given by
\[C^*:=\{ \varphi \in {\mathfrak{a}}^* \mid \varphi(\check n_i) \geq 0 \text{ for $1 \leq i \leq n$} \} = \{ \varphi \in {\mathfrak{a}}^* \mid \varphi(\check \alpha_i) \geq 0 \text{ for $1 \leq i \leq n$} \} \subset \mathfrak{a}^*.\] 
If ${\bf A}$ is non-affine, then Proposition \ref{IntertwinerWeyl1} also provides another root basis for $(W,S)$, given by $(\ol{\mathfrak{a}}, (-|-), \pi(\check \Pi_{\rm nor}))$, where $\pi: \mathfrak a \to \overline{\mathfrak a}$ denotes the canonical projection as before. If ${\bf A}$ is affine, then $(\ol{\mathfrak{a}}, (-|-), \pi(\check \Pi_{\rm nor}))$ is still a root basis, but the associated Coxeter system is no longer $(W,S)$, but rather the underlying spherical Coxeter system. Either way we refer to the associated dual Tits cone  $\overline{\mathcal C}^* \subset \overline{\mathfrak a}^*$ as the \Defn{reduced dual Tits cone} of ${\bf A}$. Explicitly, the fundamental chamber of the reduced dual Tits cone is given by
\[\overline{C}^*:=\{ \varphi \in \overline{\mathfrak{a}}^* \mid \varphi(\pi(\check n_i)) \geq 0 \text{ for $1 \leq i \leq n$} \} = \{ \varphi \in \overline{\mathfrak{a}}^* \mid \varphi(\pi(\check \alpha_i)) \geq 0 \text{ for $1 \leq i \leq n$} \} \subset \overline{\mathfrak{a}}^*.\] 
Note that the form $(-|-)$ on $\overline{\mathfrak a}$ is non-degenerate and $W$-invariant, it induces a $W$-equivariant linear isomorphism $\overline{\mathfrak a} \to \overline{\mathfrak a}^*$ by $v \mapsto (v|-)$. We denote by $\overline{\mathcal C}$ and $\overline{C}$ respectively the preimages of $\overline{\mathcal C}^*$ and $\overline{C}^*$ under this linear isomorphism, which we call the \Defn{reduced Tits cone} of ${\bf A}$, respectively its \Defn{fundamental chamber}. By definition,
\[
\overline{C} = \{v \in \overline{\mathfrak{a}} \mid (v|\pi(\check n_i)) \geq 0  \text{ for $1 \leq i \leq n$}\} \quad \text{and} \quad \overline{\mathcal C} = W.\overline{C}.
\]
To describe these sets more explicitly we observe:
\begin{lemma} \label{allisthesame}
  If ${\bf A}$ is symmetrizable, 
 then the linear map
\begin{eqnarray*}
  \ol{\phi} : \left(\pi(\mathfrak{a})=\ol{\mathfrak{a}}, (-|-)/\ker(\pi), \pi(\check \Pi_{\rm nor})\right) & \to & \left(\pi^*(\ol{\mathfrak{a}}^*) = \iota^*(V), (-|-)/\ker(\iota^*),\iota^*(\Pi_{\rm nor})\right) \\
  \pi(\check n_j) & \mapsto & \iota^*(n_j)
\end{eqnarray*}
is a well-defined isometry. Furthermore, \[\big(\pi(\check n_i)|\pi(\check n_j)\big) = \big(\iota^*(n_i)|\iota^*(n_j)\big) = n_j(\check n_i).\]
\end{lemma}

\begin{proof}
  Note that the family $(\pi(\check n_j))_{1 \leq j \leq n}$ is not necessarily linearly independent and so, a priori, it is not even clear that there exists a linear map at all such that $\pi(\check n_j)  \mapsto  \iota^*(n_j)$. However, there certainly exists a linear map
  \begin{eqnarray*}
    \phi : \left(\mathfrak{a}, (-|-), \check \Pi_{\rm nor}\right) & \to & \left(V, (-|-), \Pi_{\rm nor}\right) \\
    \check n_j & \mapsto & n_j.
  \end{eqnarray*}
By the computation in the proof of Proposition~\ref{IntertwinerWeyl1} one has \[(\phi(\check n_i)|\phi(\check n_j)) = (n_i|n_j) = \frac{1}{2}\frac{\sqrt{\epsilon_j}}{\sqrt{\epsilon_i}}a_{ij} = \frac{1}{2}\frac{\sqrt{\epsilon_j}}{\sqrt{\epsilon_i}} \alpha_j(\check \alpha_i) = n_j(\check n_i).\] By that proof, moreover, $\ker(\pi)$ equals the radical of the bilinear form on $\mathfrak{a}$ and $\ker(\iota^*)$ equals the radical of the bilinear form on $V$, so that factoring out the respective radicals induces the desired isometry between $\pi(\mathfrak{a})$ and $\iota^*(V)$.
\end{proof}
Now every element $v \in \overline{\mathfrak a}$ can be written as $v = \sum v_j \pi(\check n_j)$ for some $v_j \in \R$, and since by Lemma \ref{allisthesame} 
\[
(\pi(\check n_i)|v) = \sum_{j=1}^n v_j (\pi(\check n_i)|\pi(\check n_j)) = \sum_{j=1}^n v_j \iota^*n_i(\pi(\check n_j) = \iota^*n_i(v),
\]
we have $(\pi(\check n_i)|-) = \iota^*n_i$, and thus
\[
\overline{C} = \{v \in \overline{\mathfrak{a}} \mid \iota^*n_i(v) \geq 0  \text{ for $1 \leq i \leq n$}\} =  \{v \in \overline{\mathfrak{a}} \mid \iota^*\alpha_i(v) \geq 0  \text{ for $1 \leq i \leq n$}\}, 
\]
and hence we have $\overline{C} = \pi(C)$, where
\[
C = \{v \in \mathfrak a \mid \alpha_i(v) \geq 0  \text{ for $1 \leq i \leq n$}\}.
\]
Since $\pi: \mathfrak a \to \overline{\mathfrak a}$ is $W$-equivariant we thus have $\overline{\mathcal C} = \pi(\mathcal C)$, where $\mathcal C = \bigcup_{w \in W} w.C$. The set $\mathcal C \subset \mathfrak a$ is precisely the intersection of $\mathfrak a$ with the set which Kac calls the Tits cone in \cite[§ 3.12]{Kac90}; we will refer to it as the \Defn{unreduced Tits cone}. To summarize, our reduced Tits cone is the projection to $\overline{\mathfrak a}$ of the intersection of Tits cone (in the sense of Kac) with $\mathfrak a$, and it is geometrically isomorphic to the dual Tits cone (in the sense of Krammer) of the root basis $\left(\mathfrak{a}, (-|-), \check \Pi_{\rm nor}\right)$. If ${\bf A}$ is non-affine, then the latter is a root basis for $(W,S)$, and hence Proposition \ref{TitsInterior} and Corollary \ref{SphericalTitsCone} imply:
\begin{corollary}\label{TitsConeInterior} Assume that ${\bf A}$ is of non-spherical and non-affine type. Then the interior ${\rm Int}(\overline{\mathcal C})$ of the reduced Tits cone is an ideal polyhedral realization of the augmented Davis--Moussong poset $\Sigma^+_{\rm sph}(W,S)$, and its link complex $\mathbb S({\rm Int}(\overline{\mathcal C}))$ is an ideal polyhedral realization of the Davis--Moussong poset $\Sigma_{\rm sph}(W,S)$. In particular, both have the same underlying chamber system which is isomorphic to the chamber system of the Coxeter complex $|\Sigma(W,S)|$.\qed
\end{corollary}
We close this subsection by discussing various alternative descriptions of the reduced and unreduced Tits cone. These descriptions apply both in the affine and the non-affine case (although the reduced Tits cone is less interesting in the affine case).

Firstly, since every simple root reflection turns precisely one positive root negative, the unreduced Tits cone can be characterized as
\begin{equation}\label{TitsCone}
\CCC = \{X \in {\mathfrak a}\mid\alpha(X) \geq 0 \text{ for almost all }\alpha \in \Phi^+\};
\end{equation}
cf.\ \cite[Proposition~3.12(c)]{Kac90}. Since $\overline{\CCC} = \pi(\CCC)$ this also yields a description of the reduced Tits cone. Secondly we can obtain a description of the reduced and unreduced Tits cone in terms of the following hyperplane arrangements. 
\begin{definition} \label{overlinehDEF} Let $\alpha \in \Phi$ be a real root. Then $H_{\alpha} :=  \ker(\alpha|_{\mathfrak a}) \subset \mathfrak a$ and  $\ol{H}_{\alpha} := \pi(H_\alpha)\subset \ol{\mathfrak a}$  are called the \Defn{root hyperplanes} of $\alpha$ in $\mathfrak a$ and $\ol{\mathfrak a}$ respectively.
\end{definition}
Since $\alpha|_{\mathfrak a} \neq 0$ for all $\alpha \in \Phi$, the subspaces $H_\alpha$ are indeed hyperplanes, and since by \eqref{KMAcenter} one has
\[
\mathfrak c \cap \mathfrak a = \bigcap_{i=1}^n H_{\alpha_i},
\]
the subspaces $\ol{H}_\alpha$ are hyperplanes as well. By definition, $H_\alpha$ and $\ol{H}_\alpha$ are precisely the fixpoint sets of the roots reflections ${\rho}_{KM}(\check r_\alpha)$ and $\ol{\rho}_{KM}(\check r_\alpha)$ respectively. We refer to elements in the unions
 \[{\mathfrak{a}}^{\mathrm{sing}} :=  \bigcup_{\alpha \in \Phi} {H}_\alpha, \quad \text{respectively}\quad\ol{\mathfrak{a}}^{\mathrm{sing}} :=  \bigcup_{\alpha \in \Phi} \ol{H}_\alpha\]
 as \Defn{singular points} of $\mathfrak a$ and $\ol{\mathfrak a}$ respectively. Non-singular points are called \Defn{regular points} and we write 
\[{\mathfrak{a}}^{\mathrm{reg}} := {\mathfrak{a}} \setminus {\mathfrak{a}}^{\mathrm{sing}}, \; \mathcal C^{\rm reg} := \mathcal C \cap \mathfrak a^{\rm reg} \quad \text{and} \quad \ol{\mathfrak{a}}^{\mathrm{reg}} := \ol{\mathfrak{a}} \setminus \ol{\mathfrak{a}}^{\mathrm{sing}}, \; \ol{\mathcal C}^{\rm reg} := \ol{\mathcal C} \cap \ol{\mathfrak a}^{\rm reg}.\]  
Since the fundamental chamber is bounded by root hyperplanes, and the arrangement of root hyperplanes is $W$-invariant by construction, we deduce that the connected components of $\mathcal C^{\mathrm{reg}}$ and $\ol{\mathcal C}^{\mathrm{reg}}$ are precisely the open chambers of the respective Tits cones. We thus refer to these connected components as \Defn{open Tits chambers} and to their closures as \Defn{closed Tits chambers}.

Note that the hyperplanes bounding the fundamental chamber are preicsely the root hyperplanes $\ol{H}_{\alpha_i}$ corresponding to the simple roots. Thus if we fix a chamber $C_o$ of $|\Sigma(W,S)|$ and denote by $H_i$ the face of $C_o$ labelled by the element $\check r_{\alpha_i}\in S$, then we can restate Corollary \ref{TitsConeInterior} as follows:
\begin{corollary}\label{TitsCoxeterIso} If ${\bf A}$ is of non-spherical and non-affine type, then there is a unique incidence-preserving bijection
$\ol{\phi}$ between the set of chambers and co-dimension $1$ faces of the Coxeter complex $|\Sigma(W,S)|$ and the set of chambers and co-dimension one faces of the interior of the reduced Tits cone (respectively, its link complex) such that the following hold:
\begin{enumerate}
\item $\ol{\phi}$ is $\ol{\rho}_{KM}(W)$-equivariant and inclusion preserving.
\item $\ol{\phi}$ maps the chamber $C_o$ of $|\Sigma(W,S)|$ to the fundamental chamber $\overline{C}$ of $\overline{\CCC}$ (respectively, to $\mathbb S(\overline{C})$).
\item $\ol{\phi}(H_i) = \ol{H}_{\alpha_i}$ (respectively, $\ol{\phi}(H_i) =\mathbb S(\ol{H}_{\alpha_i})$).\qed
\end{enumerate}
\end{corollary}

\subsection{Automorphisms of the unreduced and reduced Tits cone}\label{AppendixCoxeterAut}

Keep the assumption that ${\bf A}$ be a symmetrizable irreducible generalized Coxeter matrix with associated Coxeter system $(W,S)$ and denote by $\Sigma := |\Sigma(W,S)|$ the underlying Coxeter complex. We are interested in the group $\Aut(\Sigma)$ of simplicial automorphisms of the Coxeter complex, which do not necessarily preserve the colouring. Equivalently, one can think of $\Aut(\Sigma)$ as the automorphisms of the Cayley graph ${\rm Cay}(W,S)$ (not necessarily preserving the edge colouring). Denote by $\Aut(W,S) < \Aut(W)$ the subgroup of automorphisms of $W$ which preserve $S$ as a set. This subgroup acts faithfully by automorphisms on the Cayley graph of $(W,S)$ and thus $\Aut(W, S) < \Aut(\Sigma)$. Also, $W$ acts by automorphisms on $\Sigma$ and thus can be considered as a subgroup of $\Aut(\Sigma)$. 
\begin{lemma}[{\cite[3.34, 3.35]{AbramenkoBrown2008}}]\label{AutWS} The automorphism group $\Aut(\Sigma) $ splits as a semidirect product $\Aut(\Sigma) = W \rtimes \Aut(W,S)$. Moreover, $\Aut(W,S)$ is isomorphic to the group of automorphisms of the Coxeter diagram of $(W,S)$.\qed
\end{lemma}
Now assume that ${\bf A}$ is non-affine so that the reduced Kac--Moody representation $\ol{\rho}_{KM}: W \to \GL(\ol{\mathfrak a})$ is faithful (Corollary \ref{CorollaryRedKMRepFaithful}). Every diagram automorphism $\alpha \in \Aut(W,S)$ then corresponds to a permutation of the walls of the fundamental chamber which preserves angles, and any such permutation can be realized by a unique linear map $\ol{\alpha}$ of the ambient vector space $\ol{\mathfrak a}$. One thus obtains a monomorphism 
\[
\ol{\rho}: \Aut(\Sigma) = W \rtimes \Aut(W,S) \to \GL(\ol{\mathfrak a})
\]
which maps each diagram automorphism $\alpha$ to $\ol{\alpha}$ and restricts to  $\ol{\rho}_{KM}$ on $W$. Refer to $\ol{\rho}$ as the \Defn{canonical linear realization} of $\Aut(\Sigma)$ over $\ol{\mathfrak a}$. By construction, this representation takes values in the group
\[
\GL(\ol{\mathfrak a} , \ol{\mathfrak a}^{\rm sing}) := \{f \in \GL(\ol{\mathfrak a}) \mid f( \ol{\mathfrak a}^{\rm sing}) =  \ol{\mathfrak a}^{\rm sing}\}
\]
of those linear automorphisms of $\ol{\mathfrak a}$ which preserve the hyperplane arrangement $\ol{\mathfrak a}^{\rm sing}$.

The semidirect product $\Aut(\Sigma) = W \rtimes \Aut(W,S)$ certainly also preserves the non-degenerate symmetric bilinear form $(-|-)$ on $\ol{\mathfrak a}$ from Section~\ref{bases}. One concludes that the representation $\ol{\phi}$ actually takes values in
\[
{\rm O}(\ol{\mathfrak a} , \ol{\mathfrak a}^{\rm sing}) := {\rm O}(\ol{\mathfrak a},(-|-)) \cap \GL(\ol{\mathfrak a} , \ol{\mathfrak a}^{\rm sing}).
\]
Both the hyperplane arrangement and the bilinear form are also invariant under $-{\rm Id}_{\ol{\mathfrak a}}$, which may or may not be contained in the image of $\ol{\rho}$. One can thus extend the canonical linear realization to a homomorphism
\[
\ol{\rho}:  \Aut(\Sigma) \times \Z/2\Z \to {\rm O}(\ol{\mathfrak a} , \ol{\mathfrak a}^{\rm sing}),
\]
by letting the generator of $\Z/2\Z$ act by $-{\rm Id}_{\ol{\mathfrak a}}$. One then has the following rigidity result, which was pointed out to us by Bernhard~M\"uhlherr.

\begin{proposition}[M\"uhlherr, personal communication] \label{autocoxeter}
Let $\mathbf{A}$ be a non-affine irreducible symmetrizable generalized Cartan matrix of size $n \times n$ with $n \geq 2$, let $(W,S)$ be the associated Coxeter system and let $\Sigma $ be an associated Coxeter complex. Then the canonical linear realization defines a surjective homomorphism
\[
\ol{\rho}: \Aut(\Sigma) \times \Z/2\Z \to {\rm O}(\ol{\mathfrak a} , \ol{\mathfrak a}^{\rm sing}).
\]
If $-{\rm Id}_{\ol{\mathfrak a}} \not \in \ol{\rho}(\Aut(\Sigma))$, then this map is an isomorphism.
\end{proposition}

\begin{proof} Let $\phi \in \GL(\ol{\mathfrak a} , \ol{\mathfrak a}^{\rm sing})$. First establish that $\phi$ normalizes $\ol{W} := \ol{\rho}_{KM}(W)$ and that conjugation by $\varphi$ preserves reflections in $\ol{W}$. To this end, as before denote by $\sigma_i := \ol{\rho}_{KM}(r_{\alpha_i})$ the orthogonal reflection at the hyperplane $\ol{H}_i := \ol{H}_{\alpha_i}$. Recall that the hyperplanes $\ol{H}_1, \dots, \ol{H}_n$ bound the fundamental chamber $C \in \CCC$. 

 Since the pair $(\ol{W}, \{\sigma_1, \dots, \sigma_{n}\})$ is a Coxeter system, its conjugate $(\ol{W}^\phi, \{\sigma_1^\phi, \dots, \sigma_n^\phi\})$ by $\phi$ is also a Coxeter system. Each $\sigma_i^\phi$ is a reflection, because it has a $1$-eigenspace of codimension $1$ and is of order $2$. It follows that all reflections of the Coxeter system $(\ol{W}^\phi, \{\sigma_1^\phi, \dots, \sigma_n^\phi\})$ act by reflections on $\ol{\mathfrak{a}}$. These reflections preserve $\ol{\mathfrak{a}}^{\mathrm{sing}}$, since $\phi$ does. Moreover, every hyperplane in $\ol{\mathfrak{a}}^{\mathrm{sing}}$ is the set of fixed points of a unique reflection in $\ol{W}^\phi$, since $\phi(\ol{\mathfrak{a}}^{\mathrm{sing}}) = \ol{\mathfrak{a}}^{\mathrm{sing}}$. In particular for every $i=1, \dots, n$ there is a unique reflection $\widetilde{\sigma}_i$ in $\ol{W}^\phi$ with fixed-point set $\ol{H}_i$. 

Note that, by definition, $\widetilde{\sigma}_i$ exchanges $i$-adjacent Tits chambers. In particular, both $\sigma_i$ and $\widetilde{\sigma}_i$ map the fundamental chamber $C$ to its unique $i$-adjacent chamber. It follows that for $i=1, \dots, n$ the linear map $ \widetilde{\sigma_i} \sigma_i^{-1}$ preserves the hyperplane $\ol{H}_i$ pointwise and the fundamental chamber $C$ setwise. Since $\mathbf{A}$ is irreducible with $n \geq 2$, the product $\widetilde{\sigma_i} \sigma_i^{-1}$ therefore fixes a basis of $\ol{\mathfrak a}$ and hence $\widetilde{\sigma_i} =  \sigma_i$ for all $i = 1, \dots, n$. In particular, $\ol{W} =  \langle \widetilde{\sigma}_1, \dots, \widetilde{\sigma}_n \rangle$ is a subgroup of $\ol{W}^\phi$.

The reflections $\{\widetilde{\sigma}_1, \dots \widetilde{\sigma}_n\}$ actually generate $\ol{W}^\phi$. Indeed, since $\ol{W}^\phi$ is generated by reflections at certain hyperplanes $\ol{H}_\alpha$, it will suffice to show that $\ol{W} = \langle \widetilde{\sigma}_1, \dots, \widetilde{\sigma}_n \rangle$ contains reflections at all such hyperplanes. Since $\ol{W}$ acts sharply transitively on the Coxeter complex of $(W,S)$, it acts sharply transitively on chambers in the reduced Tits cone. In particular, it contains reflections at all hyperplanes in $\ol{\mathfrak a}^{\rm sing}$ which intersect the Tits cone. Since in fact every wall in $\ol{\mathfrak a}^{\rm sing}$ intersects the Tits cone, one deduces that $\ol{W}  = \ol{W}^\phi$.

That is, $\phi$ normalizes $\ol{W}$. Moreover, $\phi$ maps fundamental reflections, and thus arbitrary reflections, to reflections. If $\phi \in {\rm O}(\ol{\mathfrak a} , \ol{\mathfrak a}^{\rm sing})= {\rm O}(\ol{\mathfrak a},(-|-)/\ker(\pi)) \cap \GL(\ol{\mathfrak a} , \ol{\mathfrak a}^{\rm sing})$, then by \cite[Theorem~1.2]{Howlett/Rowley/Taylor:1997} for any root basis $\Pi$ there exists $w \in W$ such that $\phi(\Pi) = \pm w\Pi$.  (Note that this theorem applies by Remark \ref{HowlettApplies}.)

%
%
One may thus assume that  the automorphism $\ol{\phi}$ of $\ol{W}$ induced by conjugation with $\phi$ stabilizes $S$ and thus induces an automorphism $\alpha$ of $\Sigma$. Then $\ol{\rho}(\alpha)$ agrees with $\varphi$ up to $\lambda{\rm Id}_{\ol{\mathfrak a}}$ for $\lambda \in \mathbb{R} \backslash \{ 0 \}$. Indeed, by definition $\varphi$ and $\ol{\rho}(\alpha)$ are both linear maps preserving the hyperplane arrangement $\ol{\mathfrak a}^{\rm sing}$ and (since every hyperplane intersects the Tits cone) they map each hyperplane in $\ol{\mathfrak a}^{\rm sing}$ to the same hyperplane. This is only possible if they are linear multiples of one another, that is, if they coincide up to multiplication with $\pm{\rm Id}_{\ol{\mathfrak a}}$, by orthogonality.
%
\end{proof}
\begin{remark}\label{SphericalAffineNonTrichotomy}
For an irreducible symmetrizable generalized Cartan matrix ${\bf A}$ of size $\geq 2$ one has the following trichotomy concerning the isomorphy type of ${\rm O}(\ol{\mathfrak a} , \ol{\mathfrak a}^{\rm sing})$:
\begin{enumerate}
\item If ${\bf A}$ is spherical, then $\ol{\CCC} = \ol{\mathfrak a}$ and thus $-{\rm Id}_{\ol{\mathfrak a}} \in \ol{\rho}(\Aut(\Sigma))$. In this case, $\ol{\rho}$ yields an isomorphism $ \Aut(\Sigma)\cong  {\rm O}(\ol{\mathfrak a} , \ol{\mathfrak a}^{\rm sing})$.
\item If ${\bf A}$ is non-spherical and non-affine, then by \eqref{TitsCone} we have
\[
\ol{\CCC} \cap (- \ol{\CCC}) = \{\pi(X) \in \ol{\mathfrak a}\mid X \in \mathfrak a, \alpha(X) = 0 \text{ for almost all }\alpha \in \Phi^+\} = \{0\},
\]
i.e.\ the reduced Tits cone and its negative only meet at their tips. The action of $\ol{\rho}(\Aut(\Sigma))$ on $\ol{\mathfrak a}$ preserves the two cones $\ol{\CCC}$ and $-\ol{\CCC}$, whereas $-{\rm Id}_{\ol{\mathfrak a}}$ exchanges the two cones. In particular, $\ol{\rho}$ induces an isomorphism
\[
 \Aut(\Sigma) \times \Z/2\Z \cong  {\rm O}(\ol{\mathfrak a} , \ol{\mathfrak a}^{\rm sing})
\]
in this case.
\item If ${\bf A}$ is affine, then the action of $W$ on $\ol{\mathfrak a}$ is not faithful, and the $W$-module $\ol{\mathfrak a}$ is given by the Kac--Moody representation of the underlying spherical Coxeter system $(W_o, S_o)$. In this case one, thus, has $ {\rm O}(\ol{\mathfrak a} , \ol{\mathfrak a}^{\rm sing}) \cong \Aut(\Sigma(W_o, S_o))$ by (i).
\end{enumerate}
\end{remark}
The proof of Proposition \ref{autocoxeter} implies the following statement concerning arbitrary linear automorphisms preserving the hyperplane complement $\ol{\mathfrak a}^{\rm sing}$.
\begin{corollary} \label{rescalingtimesweyl} Let $\mathbf{A}$ be a non-affine irreducible symmetrizable generalized Cartan matrix of size $n \times n$ with $n \geq 2$, let $(W,S)$ be the associated Coxeter system and let $\Sigma$ be an associated Coxeter complex. Then every $\psi \in \GL(\ol{\mathfrak a}, \ol{\mathfrak a}^{\rm sing})$ can be written as $\psi = T_\lambda \circ \phi = \phi \circ T_\lambda$, where $\lambda \in \R^\times$, $T_\lambda$ is the homothety $x \mapsto \lambda x$ and $\phi \in \ol{\rho}(\Aut(\Sigma))$ is induced by an automorphism of $\Sigma$.\qed
\end{corollary}

\subsection{The canonical homothety class of bilinear forms of $(\ol{\mathfrak a},\ol{\mathfrak a}^{\rm sing})$}
We keep the notation of the previous subsection. The bilinear form $(-|-)$ on $\ol{\mathfrak a}$ is non-degenerate and invariant under the reduced Kac--Moody representation of the Weyl group. Moreover, reflections of the Weyl group act by reflections along the hyperplanes contained in $\ol{\mathfrak a}^{\rm sing}$ with respect to this bilinear form. Certainly, any non-zero multiple of this invariant form also satisfies these properties. The following proposition states that this actually characterizes this homothety class of bilinear forms. 

\begin{proposition} \label{uniquehomothety}
Let $(W,S)$ be a non-spherical non-affine irreducible Coxeter system, let $\ol{\mathfrak{a}}$ be the $W$-module afforded by the reduced Kac--Moody representation and let $b : \ol{\mathfrak a} \times \ol{\mathfrak a} \to \mathbb{R}$ be a nondegenerate symmetric bilinear form with the property that the reflections of the Weyl group act on $\ol{\mathfrak a}$ as $b$-orthogonal maps. Then $b$ is a multiple of the bilinear form $(-|-)$.   
\end{proposition}

\begin{proof}
  Let $s \in S$ and let $\sigma_s \in W$ be a simple reflection. Then by hypothesis the $(-|-)$-orthogonal eigenspace decomposition $\ol{\mathfrak a} = E_1(\sigma_s) \oplus E_{-1}(\sigma_s)$ is also $b$-orthogonal.
  Given two reflections $\sigma_s$ and $\sigma_t$ with orthogonal eigenspace decompositions $E_1(\sigma_s) \oplus E_{-1}(\sigma_s) = \ol{\mathfrak a} = E_1(\sigma_t) \oplus E_{-1}(\sigma_t)$, we conclude that the $(-|-)$-orthogonal projection of $E_{-1}(\sigma_s)$ onto $E_1(\sigma_t)$ also is $b$-orthogonal. By induction there exists a decomposition of $\ol{\mathfrak a}$ into a direct sum of one-dimensional subspaces that is both $(-|-)$-orthogonal and $b$-orthogonal. With other words, there exists a $(-|-)$-orthogonal basis $(b_i)_{1 \leq i \leq \dim(\ol{\mathfrak a})}$ of $\ol{\mathfrak a}$ that is also a $b$-orthogonal basis. We conclude that, with respect to this basis $(b_s)_{s \in T}$, the forms $b$ and $\ol{\mathfrak a}$ only differ by rescaling with a diagonal matrix $\mathrm{diag}(\lambda_1,\ldots,\lambda_t)$.

  Since $W$ is irreducible, to any given pair of vectors $b_i$, $b_j$ in this basis there exists a reflection hyperplane $H$ in $\ol{\mathfrak a}^{\rm sing}$ that contains neither $b_i$ nor $b_j$. Repeating the above construction with the reflection of $W$ that has $H$ as eigenspace with respect to the eigenvalue $1$, one necessarily has $\lambda_i = \lambda_j$. The claim follows.
\end{proof}

\begin{corollary}\label{CanHomothetyClass}
The  homothety class $[(-|-)]$ is the unique homothety class of nondegenerate symmetric bilinear forms on $\ol{\mathfrak{a}}$ such that $\GL(\ol{\mathfrak a}, \ol{\mathfrak a}^{\rm sing})$ is a subgroup of the group of linear similarities of that class.
\end{corollary}

\begin{proof}
By Corollary~\ref{rescalingtimesweyl} any element of $\GL(\ol{\mathfrak a}, \ol{\mathfrak a}^{\rm sing})$ is a scalar multiple of the image of an element of $\Aut(\Sigma) = W \rtimes \Aut(W,S)$ under its canonical linear realization over $\ol{\mathfrak a}$. By Proposition~\ref{uniquehomothety} this determines the homothety class $[(-|-)]$.
\end{proof}

\begin{definition}\label{DefCanHomothetyClass}
The unique homothety class of non-degenerate symmetric bilinear forms on $\ol{\mathfrak a}$ with the property that $\GL(\ol{\mathfrak a}, \ol{\mathfrak a}^{\rm sing})$ acts by linear similarities is called the \Defn{canonical homothety class of bilinear forms on $(\ol{\mathfrak a},\ol{\mathfrak a}^{\rm sing})$}.
\end{definition}

\end{document}